\documentclass[a4paper,11pt,twoside, eqno]{article}
\usepackage{mathrsfs}
\usepackage{amssymb}
\usepackage{amsmath}
\usepackage{amsthm}
\usepackage{amsfonts}
\usepackage{color}
\usepackage{latexsym}
\usepackage{mathtools}
\usepackage{indentfirst}
\usepackage{amscd}
\usepackage{bbm}
\usepackage{graphicx}
\usepackage[all]{xy}
\usepackage[mathscr]{eucal}
\usepackage{subfigure}

\newcommand{\N}{{\mathbb N}}

\newcommand{\R}{{\mathbb R}}

\allowdisplaybreaks
\newtheorem{theorem}{Theorem}[section]

\newtheorem{corollary}[theorem]{Corollary}

\newtheorem{definition}[theorem]{Definition}

\newtheorem{remark}[theorem]{Remark}

\newtheorem{lemma}[theorem]{Lemma}

\newtheorem{proposition}[theorem]{Proposition}
\newtheorem{claim}[theorem]{Claim}

\newtheorem{assumption}[theorem]{Assumption}

\numberwithin{equation}{section}

\begin{document}

\title{\bf\Large Bifurcations in Lagrangian systems and geodesics I
\footnotetext{\hspace{-0.35cm} 2020
{\it Mathematics Subject Classification}.
Primary 37J20, 34C23, 53C22.
\endgraf
{\it Key words and phrases.}
Bifurcation, Lagrangian system, Euler-Lagrange curves, focal point.
}}
\date{}
\author{Guangcun Lu\footnote{
E-mail: \texttt{gclu@bnu.edu.cn}}}

\maketitle

\vspace{-0.7cm}

\begin{center}

\begin{minipage}{13cm}
{\small {\bf Abstract}\quad
This paper is Part I of a two-part series.  We investigate bifurcation phenomena in Lagrangian systems with
various boundary conditions and constraints, focusing on the interplay between Morse theory and the existence of multiple solutions
 through three principal configurations: Lagrangian trajectories connecting two submanifolds or with endpoints  related by an isometry,
and brake orbits in Lagrangian systems. For each configuration, we establish necessary and sufficient conditions
for bifurcation using Morse index and nullity techniques,  including classification of Rabinowitz-type alternative bifurcation scenarios.
 For Euler-Lagrange curves emanating perpendicularly from a submanifold, we develop a unified Morse-theoretic framework that rigorously connects
geometric focal structure (e.g., conjugate points) and analytic bifurcation behavior (e.g., solution branching patterns).}
\end{minipage}
\end{center}
	
\vspace{0.2cm}

\tableofcontents

\vspace{0.2cm}

\section{Introduction and overview of the series}\label{sec:LgrResults}
\setcounter{equation}{0}

\subsection{Background and motivation} \label{ssec:background} 
The study of bifurcation phenomena in variational problems lies at the heart of nonlinear analysis and differential geometry. A fundamental question is to understand how families of solutions to Euler-Lagrange equations change their qualitative behavior as a parameter varies. Classical results, such as the Lyapunov-Schmidt reduction and the celebrated work of Rabinowitz on potential operators, provide abstract frameworks for analyzing such phenomena. In the geometric context, the Morse index theorem elegantly relates the index of a geodesic to its conjugate points, offering a powerful tool for understanding the stability and multiplicity of geodesics in Riemannian geometry.

However, extending this elegant theory to more general Lagrangian systems with constraints and various boundary conditions presents significant challenges. While substantial progress has been made for specific cases--such as periodic orbits, geodesics between fixed points, or brake orbits--a \textsf{unified framework} that systematically handles configurations like trajectories connecting two submanifolds, endpoints constrained by an isometry, or generalized periodic solutions has remained less developed. This paper aims to fill this gap by establishing a comprehensive Morse-theoretic approach to bifurcation theory for a broad class of Lagrangian boundary value problems.

Our work is motivated by the desire to create a common foundation that not only recovers known results in differential geometry (e.g., bifurcation of geodesics) as special cases but also applies to non-geometric, time-dependent Lagrangian systems. The key insight is to leverage the interplay between the analytic data (Morse index and nullity) and the geometric structure of the problem to derive sharp necessary and sufficient conditions for bifurcation.

\subsection{Main objectives and structure of the series} \label{ssec:objectives}

This two-part series has the following primary objectives:

1.  \textbf{Part I (Current Paper):} To develop a unified Morse-theoretic framework for analyzing bifurcations in Lagrangian systems under three general configurations:
    \begin{itemize}
        \item Trajectories connecting two submanifolds.
        \item Trajectories with endpoints related by an isometry (including generalized periodic solutions).
        \item Brake orbits (symmetric periodic trajectories).
    \end{itemize}
    For each configuration, we establish necessary and sufficient conditions for bifurcation, prove existence theorems, and classify the possible bifurcation scenarios following the tradition of Rabinowitz's alternatives.

2.  \textbf{Part II (companion paper):} To apply the framework developed in Part I to autonomous Lagrangian systems and delve deeper into bifurcation phenomena for geodesics in both Finsler and Riemannian settings, extending the results to broader geometric variational problems.

The core technical questions we address in this Part I are precisely formulated in Section 1.4 (see Questions (1)-(3)). The answers, summarized in Theorems 1.4--1.24, provide a complete characterization of bifurcation points in terms of changes in the Morse index and nullity along the trivial branch of solutions.

\subsection{Relation to existing literature} \label{ssec:literature}

Our work builds upon and extends several key strands of research. The abstract bifurcation theory for potential operators was pioneered by Rabinowitz \cite{Rab77}. In the context of geodesics, the connection between conjugate points and the Morse index is classical \cite{Mor1934, Mil1962}. More recent contributions by authors such as Duistermaat \cite{Du} (whose index theorem we rely upon in Theorem~\ref{th:DustMorse}), and Piccione et al. \cite{PiPoTa04} on the Maslov index and bifurcation, have significantly advanced the field.

The novelty of our approach lies in its systematicity and generality. While previous works often focused on specific settings, we provide a unified treatment that encompasses a wide range of boundary conditions. The techniques employed here combine the abstract critical point theory from 
\cite{Lu8, Lu9, Lu11, Lu10} with careful geometric reductions to Euclidean spaces, allowing us to handle the technical complexities arising from the manifold setting and constraints.

\subsection{Basic assumptions, conventions, and 
statement of the research problems} \label{ssec:basic_assumptions}

Throughout, $\mathbb{N}$, $\mathbb{Z}$, and $\mathbb{R}$ denote, respectively, the natural numbers, integers, and real numbers. The symbol $\tau$ is a fixed positive number, and $\Lambda$ denotes the parameter space (typically a real interval or a topological space).\\

\noindent\textbf{Assumption 1.0} \textsf{(Basic assumptions and conventions).}
Let $M$ be a $n$-dimensional, connected $C^7$ submanifold of $\R^N$.
Its tangent bundle $TM$ is a $C^6$-smooth manifold of dimension $2n$,
whose points are denoted by $(x,v)$, with $x\in M$ and $v\in T_xM$.
The bundle projection $\pi:TM\to M,\,(x,v)\mapsto x$ is $C^6$.
Let $g$ be a $C^6$ Riemannian metric and $\mathbb{I}_g$ a $C^7$ isometry  on $(M, g)$,
i.e., $\mathbb{I}_g:M\to M$ is $C^7$ and satisfies $g((\mathbb{I}_g)_\ast(u), (\mathbb{I}_g)_\ast(v))=g(u,v)$
for all $u,v\in TM$.
(Thus the Christoffel symbols $\Gamma^i_{jk}$ and the exponential map $\exp:TM\to M$ are $C^5$.)\\

Let ${\bf N}$ be a submanifold of $M\times M$ (specifically, either a product of two submanifolds of $M$ or the graph of a Riemannian isometry of $(M,g)$), and let $L:\Lambda\times [0,\tau]\times TM\to\mathbb{R}$ be as in Assumption~\ref{ass:Lagr6}.
  The current manuscript focuses on  bifurcations research for the following 
  Lagrangian  boundary value problem:
  \begin{equation}\label{e:LagrIntro}
\left.\begin{array}{ll}
\frac{d}{dt}\big(\partial_vL_\lambda(t, \gamma(t), \dot{\gamma}(t))\big)-\partial_q L_\lambda(t, \gamma(t), \dot{\gamma}(t))=0\;\;\forall
t\in [0,\tau],\\
(\gamma(0), \gamma(\tau))\in{\bf N}\quad\hbox{and}\quad\\
\partial_vL_\lambda(0, \gamma(0), \dot{\gamma}(0))[v_0]=\partial_vL_\lambda(\tau, \gamma(\tau), \dot{\gamma}(\tau))[v_1]\;\;\forall
(v_0,v_1)\in T_{(\gamma(0),\gamma(\tau))}{\bf N}
\end{array}\right\}
\end{equation} 
  with respect to a continuous family $\{\gamma_\lambda\,|\,\lambda\in\Lambda\}$ of solutions of this problem,
    where the superposed dot denotes the derivative with respect to $t$.
If every neighborhood of $(\mu, \gamma_\mu)$ in  $\Lambda\times C^1([0,\tau];\R^{2n})$ contains a point $(\lambda, \alpha_\lambda)\notin\{(\lambda,\gamma_\lambda)\,|\,\lambda\in\Lambda\}$
satisfying (\ref{e:LagrIntro}) we say $(\mu,\gamma_\mu)$ to be a \textsf{bifurcation point} of (\ref{e:LagrIntro})
in  $\Lambda\times C^1([0,\tau];\R^{2n})$ with respect to the trivial branch $\{(\lambda,\gamma_\lambda)\,|\,\lambda\in\Lambda\}$.

 Observe that solutions of (\ref{e:LagrIntro}) correspond to critical points of
 $C^2$ functionals defined by
   \begin{equation}\label{e:EnergyFunc}
 \mathcal{E}_\lambda(\gamma)=\int^\tau_0L_\lambda(t, \gamma(t), \dot{\gamma}(t))dt
\end{equation}
  on the Banach manifold
  \begin{equation}\label{e:BanachM}
  C^{1}_{\bf N}([0,\tau]; M)=\{\gamma\in C^{1}([0,\tau]; M)\,|\, (\gamma(0),\gamma(\tau))\in{\bf N}\}.
  \end{equation}
 Following \cite[p.175]{Du}, we define   the Morse index $m^-(\mathcal{E}_\lambda,\gamma_\lambda)$ and nullity
  $m^0(\mathcal{E}_\lambda,\gamma_\lambda)$ of  $\mathcal{E}_\lambda$ at
  $\gamma_\lambda\in C^{1}_{\bf N}([0,\tau]; M)$, respectively, as
   \begin{align}
m^-(\mathcal{E}_\lambda,\gamma_\lambda) &= \sup\{ \dim L \mid L \subset 
  T_{\gamma_\lambda}C^{1}_{\bf N}([0,\tau]; M) \text{ is a linear subspace such that} \nonumber\\
  &\qquad D^{2}\mathcal{E}_\lambda(\gamma_\lambda) \text{ is negative definite on } L \bigr\}, \label{e:DuMorseIndex} \\
m^0(\mathcal{E}_\lambda,\gamma_\lambda) &= \dim \ker \bigl( D^{2}\mathcal{E}_\lambda(\gamma_\lambda) \bigr),
 \label{e:DuNullity}
\end{align}
      where $T_{\gamma_\lambda}C^{1}_{\bf N}([0,\tau]; M)$ is 
 the tangent space of $C^{1}_{\bf N}([0,\tau]; M)$ at $\gamma_\lambda$, i.e.,
  the  Banach space of all $C^1$-sections $\xi$ of the pullback bundle \(\gamma_\lambda^*TM\)
  that satisfy the boundary condition  $(\xi(0), \xi(\tau))\in T_{(\gamma_\lambda(0), \gamma_\lambda(\tau))}\mathbf{N}$.
  Using $m^-(\mathcal{E}_\lambda,\gamma_\lambda)$ and $m^0(\mathcal{E}_\lambda,\gamma_\lambda)$, 
  which depend only on the second variation
    $D^{2}\mathcal{E}_\lambda(\gamma_\lambda)$
     of the functional  $\mathcal{E}_\lambda$ at
  $\gamma_\lambda\in C^{1}_{\bf N}([0,\tau]; M)$ (cf.~\cite[\S5]{BoTr} 
  for the precise expression of
  $D^{2}\mathcal{E}_\lambda(\gamma_\lambda)$ in terms of covariant derivatives),
     we address the following questions:
   \begin{description}
\item[(1)] Under what conditions there is  a bifurcation point $(\mu,\gamma_\mu)$ in the above sense?
\item[(2)] What are the necessary (resp. sufficient) condition for a given point $(\mu, \gamma_\mu)$ to be a bifurcation point in the above sense?
\item[(3)] How is the solutions of (\ref{e:LagrIntro}) distributed near a bifurcation point $(\mu, \gamma_\mu)$ as above ?
\end{description}
Roughly speaking, our answers are:
\begin{description}
\item[(a)] If $\Lambda$ is path-connected and there exist two  points $\lambda^+, \lambda^-\in\Lambda$ 
for which the Morse index intervals 
    $$
     [m^-({\mathcal{E}}_{\lambda^+}, \gamma_{\lambda^+}),
m^-({\mathcal{E}}_{\lambda^+}, \gamma_{\lambda^+})+ m^0({\mathcal{E}}_{\lambda^+}, \gamma_{\lambda^+})]
$$
and
$$
[m^-({\mathcal{E}}_{\lambda^-}, \gamma_{\lambda^-}),
m^-({\mathcal{E}}_{\lambda^-}, \gamma_{\lambda^-})+ m^0({\mathcal{E}}_{\lambda^-}, \gamma_{\lambda^-})]
    $$
    are disjoint, and either
     $m^0(\mathcal{E}_{\lambda^+}, \gamma_{\lambda^+})=0$ or $m^0(\mathcal{E}_{\lambda^-}, \gamma_{\lambda^-})=0$,
then there exists a point $\mu$ on any path connecting $\lambda^+$ to $\lambda^-$ in $\Lambda$
such that $(\mu,\gamma_\mu)$ is a bifurcation point  for the family as described as above.
\item[(b)] If $\Lambda$ is first countable and there exist two points $\lambda^+, \lambda^-\in\Lambda$ in any neighborhood of some $\mu\in\Lambda$
satisfying the properties as in (a), then $(\mu,\gamma_\mu)$ is a bifurcation point. Conversely, for
a bifurcation point $(\mu,\gamma_\mu)$ it must hold that $m^0({\mathcal{E}}_{\mu}, \gamma_{\mu})>0$.
\item[(c)] If $\Lambda$ is a real interval, $\mu\in{\rm Int}(\Lambda)$, then the solutions of (\ref{e:LagrIntro})
near a bifurcation point $(\mu, \gamma_\mu)$ have alternative bifurcations of Rabinowitz's type (as in \cite{Rab77}) provided that
 $m^0(\mathcal{E}_{\lambda}, \gamma_{\lambda})=0$  for each $\lambda\in\Lambda\setminus\{\mu\}$ near $\mu$, and
  $m^-(\mathcal{E}_{\lambda}, \gamma_{\lambda})$ take, respectively, values $m^-(\mathcal{E}_{\mu}, \gamma_\mu)$ and
  $m^-(\mathcal{E}_{\mu}, \gamma_\mu)+ m^0(\mathcal{E}_{\mu}, \gamma_\mu)$
 as $\lambda\in\Lambda$ varies in two deleted half neighborhoods  of $\mu$.
\end{description}
The precise assumptions and main results are organized in the following three subsections.
The proofs are built upon the abstract theory from \cite{Lu8, Lu10}, combined with techniques in Appendix A of \cite{Lu10}.
As an application, Theorem~\ref{th:MorseBif} establishes a rigorous one-to-one correspondence between focal points and bifurcation phenomena along Euler-Lagrange curves of (\ref{e:LagrCurve}), providing a complete classification of branching behaviors near critical instants.

\subsection{Bifurcations of Lagrangian system trajectories connecting two submanifolds}\label{sec:submanifolds}

\begin{assumption}\label{ass:Lagr6}
{\rm  Let $(M, g)$ be as in Assumption~1.0 in Introduction.
For a real $\tau>0$ and a topological space $\Lambda$,
let $L:\Lambda\times [0,\tau]\times TM\to\R$ be a continuous function such that
for each $C^3$ chart $\alpha:U_\alpha\to\alpha(U_\alpha)\subset\mathbb{R}^n$ and the induced bundle
chart $T\alpha:TM|_{U_\alpha}\to \alpha(U_\alpha)\times\mathbb{R}^n\subset\mathbb{R}^n\times\mathbb{R}^n$ the function
$$
L^\alpha:\Lambda\times[0,\tau]\times \alpha(U_\alpha)\times\mathbb{R}^n\to\mathbb{R},\;
(\lambda,t, q,v)\mapsto L(\lambda,t, (T\alpha)^{-1}(q,v))
$$
is $C^2$ with respect to $(t,q,v)$ and strictly convex with respect to $v$, and
all its partial derivatives also depend continuously on $(\lambda, t, q, v)$.
Let $S_0$ and $S_1$ be two boundaryless and connected submanifolds of $M$ of dimensions less than $\dim M$.  }
\end{assumption}

By \cite[Theorem~4.2]{PiTa01}, for each integer $1\le i\le 4$,  $C^{i}([0, \tau]; M)$  is
 a  $C^{5-i}$  Banach manifold modeled on the Banach space $C^i([0, \tau];\R^n)$ with the tangent space
 \begin{equation}\label{e:LagrGener000}
 T_\gamma C^i([0, \tau];M)=C^i(\gamma^\ast TM)=\{\xi\in C^i([0,\tau];\mathbb{R}^N)\,|\,\xi(t)\in T_{\gamma(t)}M\;\forall t\}
 \end{equation}
at $\gamma\in C^i([0, \tau];M)$.
Thus
\begin{equation}\label{e:LagrGener00}
C^{1}_{S_0\times S_1}([0,\tau]; M):=\left\{\gamma\in C^{1}([0,\tau]; M)\,|\, (\gamma(0),\gamma(\tau))\in S_0\times S_1\right\}
\end{equation}
 is a $C^4$ Banach submanifold of $C^{1}([0,\tau]; M)$. 
Its tangent space at $\gamma\in C^{1}_{S_0\times S_1}([0,\tau]; M)$ is
\begin{equation}\label{e:LagrGener00+}
C^{1}_{S_0\times S_1}(\gamma^\ast TM):=\left\{\xi\in C^1(\gamma^\ast TM)\,|\, (\xi(0),\xi(\tau))\in T_{(\gamma(0),\gamma(1))}(S_0\times S_1)\right\},
\end{equation}
which is dense in the Hilbert subspace
\begin{equation}\label{e:LagrGener0}
W^{1,2}_{S_0\times S_1}(\gamma^\ast TM):=\left\{\xi\in W^{1,2}(\gamma^\ast TM)\,|\, (\xi(0),\xi(\tau))\in T_{(\gamma(0),\gamma(1))}(S_0\times S_1)\right\}
\end{equation}
of $W^{1,2}(\gamma^\ast TM)$ (consisting of
all $W^{1,2}$-sections of the pull-back bundle $\gamma^\ast TM\to [0,\tau]$) with inner product
  given by
\begin{equation}\label{e:1.1}
\langle\xi,\eta\rangle_{1,2}=\int^\tau_0\langle\xi(t),\eta(t)\rangle_g dt+
\int^\tau_0\langle \frac{D\xi}{dt}(t), \frac{D\eta}{dt}(t)\rangle_g dt,
\end{equation}
where $\frac{D\xi}{dt}$ is the $W^{1,2}$ covariant
derivative of $\xi$ along $\gamma$,
defined via the Levi-Civita connection  of the metric $g$
in (\ref{e:covariantDerivative}).
Hereafter $\langle u,v\rangle_g=g(u,v)$ for $u,v\in TM$.

For each $\lambda\in\Lambda$,  as in the proof of the first claim in \cite[Proposition~4.2]{Lu9}  we get that
\begin{equation}\label{e:LagrGener-}
\mathcal{E}_\lambda: C^{1}_{S_0\times S_1}([0,\tau]; M)\to\mathbb{R},\;\gamma\mapsto \int^\tau_0L_\lambda(t, \gamma(t), \dot{\gamma}(t))dt
\end{equation}
is a $C^2$ functional.
A path $\gamma_0\in C^{1}_{S_0\times S_1}([0,\tau]; M)$ is a critical point of $\mathcal{E}_\lambda$
if and only if it belongs to $C^{2}([0,\tau]; M)$ and satisfies the Euler-Lagrange equation
\begin{equation}\label{e:LagrGener}
\frac{d}{dt}\big(\partial_vL_\lambda(t, \gamma(t), \dot{\gamma}(t))\big)-\partial_q L_\lambda(t, \gamma(t), \dot{\gamma}(t))=0\;\forall
t\in [0,\tau]
\end{equation}
and the boundary condition
\begin{equation}\label{e:LagrGenerB}
\left.\begin{array}{ll}
(\gamma(0), \gamma(\tau))\in S_0\times S_1\quad\hbox{and}\quad\\
\partial_vL_\lambda(0, \gamma(0), \dot{\gamma}(0))[v_0]=0\quad\forall v_0\in T_{\gamma(0)}S_0,\\
\partial_vL_\lambda(\tau, \gamma(\tau), \dot{\gamma}(\tau))[v_1]=0\quad\forall v_1\in T_{\gamma(\tau)}S_1.
\end{array}\right\}
\end{equation}
By \cite{Du}, the second-order differential $D^2\mathcal{E}_\lambda(\gamma_0)$ of $\mathcal{E}_\lambda$ at such a critical point $\gamma_0$
can be extended into a continuous symmetric bilinear form on $W^{1,2}_{S_0\times S_1}(\gamma^\ast TM)$
with finite Morse index and nullity
\begin{equation}\label{e:MorseIndexTwo}
m^-(\mathcal{E}_\lambda,\gamma_0)\quad\hbox{and}\quad m^0(\mathcal{E}_\lambda,\gamma_0),
\end{equation}
which coincide with those defined by (\ref{e:DuMorseIndex}) and (\ref{e:DuNullity}) 
for $\mathbf{N}=S_0\times S_1$, respectively.

\begin{assumption}\label{ass:LagrGenerB}
{\rm  Under Assumption~\ref{ass:Lagr6}, for each $\lambda\in\Lambda$ let
$\gamma_\lambda\in C^2([0,\tau];M)$ satisfy (\ref{e:LagrGener})-(\ref{e:LagrGenerB}).
It is also assumed that $\Lambda\times [0,\tau]\ni (\lambda,t)\mapsto\gamma_\lambda(t)\in M$ and
$\Lambda\times [0,\tau]\ni (\lambda,t)\mapsto\dot{\gamma}_\lambda(t)\in TM$ are continuous,
that is, for any $C^2$ coordinate chart $\phi:W\to\phi(W)\subset\mathbb{R}^n$, maps
$$
(\lambda,t)\mapsto(\phi\circ\gamma_\lambda)(t),\quad (\lambda,t)\mapsto\frac{d}{dt}(\phi\circ\gamma_\lambda)(t)
$$
are continuous.
}
\end{assumption}

\begin{definition}\label{def:LagrGenerBifur}
{\rm  
Let  $X=C^{1}_{S_0\times S_1}([0,\tau]; M)$ (or $C^{2}_{S_0\times S_1}([0,\tau]; M)$).
For $\mu\in\Lambda$,
we call $(\mu, \gamma_\mu)$  a \textsf{bifurcation point} of the problem (\ref{e:LagrGener})--(\ref{e:LagrGenerB})
in $\Lambda\times X$ with respect to the branch $\{(\lambda,\gamma_\lambda)\,|\,\lambda\in\Lambda\}$
  if  there exists a point  $(\lambda_0, \gamma_0)$ in any neighborhood of $(\mu,\gamma_\mu)$ in $\Lambda\times X$
    such that $\gamma_0\ne\gamma_{\lambda_0}$
  is a solution of (\ref{e:LagrGener})--(\ref{e:LagrGenerB}) with $\lambda=\lambda_0$.
  Moreover, $(\mu, \gamma_\mu)$ is said to be a \textsf{bifurcation point along sequences} of the problem (\ref{e:LagrGener})--(\ref{e:LagrGenerB})
in $\Lambda\times X$ with respect to the branch $\{(\lambda,\gamma_\lambda)\,|\,\lambda\in\Lambda\}$
  if  there exists a sequence  $\{(\lambda_k, \gamma^k)\}_{k\ge 1}$ in $\Lambda\times X$,
   converging to $(\mu,\gamma_\mu)$ such that each $\gamma^k\ne\gamma_{\lambda_k}$
  is a solution of (\ref{e:LagrGener})--(\ref{e:LagrGenerB}) with $\lambda=\lambda_k$, $k=1,2,\cdots$.
  (These two notions are equivalent if $\Lambda$ is first countable.)  }
\end{definition}

Recall that an isolated critical point $p$ of a $C^1$-functional $f$ on a Banach manifold $\mathcal{M}$
is called \textsf{homological visible} if there exists a nonzero critical group $C_m(f,p;\mathbf{K})$ for some
Abel group $\mathbf{K}$.

\begin{theorem}\label{th:bif-nessLagrGener}
Let  Assumptions~\ref{ass:Lagr6},~\ref{ass:LagrGenerB} be satisfied, and $\mu\in\Lambda$ be such that
\footnote{This assumption is to guarantee the existence of a Riemannian metric $g$ on $M$ such that
$S_0$ (resp. $S_1$) is totally geodesic near $\gamma_\mu(0)$ (resp. $\gamma_\mu(\tau)$)
when we reduce the problems to Euclidean spaces in Section~\ref{sec:LagrBound.1.1}.
Therefore it is not needed if $M$ is an open subset in  Euclidean spaces and $S_0$ and $S_1$ are linear subspaces.
Actually, when $\gamma_\mu(0)=\gamma_\mu(\tau)$ we only need a weaker condition that there exists a coordinate chart
$(U,\varphi)$ around this point on $M$ such that $\varphi(S_0\cap S_1\cap U)$ is the intersection $\varphi(U)$ of
the union of two linear subspaces in $\mathbb{R}^n$.}
 $\gamma_\mu(0)\ne\gamma_\mu(\tau)$
in the case $\dim S_0>0$ and $\dim S_1>0$.
\begin{enumerate}
\item[\rm (I)]{\rm (\textsf{Necessary condition}):}
Suppose  that $(\mu, \gamma_\mu)$  is a  bifurcation point along sequences of the problem (\ref{e:LagrGener})--(\ref{e:LagrGenerB})
with respect to the branch $\{(\lambda,\gamma_\lambda)\,|\,\lambda\in\Lambda\}$ in $\Lambda\times C^{1}_{S_0\times S_1}([0,\tau]; M)$.
Then $m^0(\mathcal{E}_{\mu}, \gamma_\mu)>0$.
\item[\rm (II)]{\rm (\textsf{Sufficient condition}):}
Suppose that $\Lambda$ is first countable and that there exist two sequences in  $\Lambda$ converging to $\mu$, $(\lambda_k^-)$ and
$(\lambda_k^+)$,  such that one of the following conditions is satisfied:
 \begin{enumerate}
 \item[\rm (II.1)] For each $k\in\mathbb{N}$, either $\gamma_{\lambda^+_k}$  is not an isolated critical point of ${\mathcal{E}}_{\lambda^+_k}$,
 or $\gamma_{\lambda^-_k}$ is not an isolated critical point of ${\mathcal{E}}_{\lambda^-_k}$,
 or $\gamma_{\lambda^+_k}$ {\rm (}resp. $\gamma_{\lambda^-_k}${\rm )}
  is an isolated critical point of $\mathcal{E}_{\lambda^+_k}$ {\rm (}resp. $\mathcal{E}_{\lambda^-_k}${\rm )} and
  $C_m(\mathcal{E}_{\lambda^+_k}, \gamma_{\lambda^+_k};{\bf K})$ and $C_m(\mathcal{E}_{\lambda^-_k}, \gamma_{\lambda^-_k};{\bf K})$
  are not isomorphic for some Abel group ${\bf K}$ and some $m\in\mathbb{Z}$.
\item[\rm (II.2)] For each $k\in\mathbb{N}$, there exists $\lambda\in\{\lambda^+_k, \lambda^-_k\}$
such that $\gamma_{\lambda}$  is an either non-isolated or homological visible critical point of
$\mathcal{E}_{\lambda}$ , and
$$
\left.\begin{array}{ll}
&[m^-(\mathcal{E}_{\lambda_k^-}, \gamma_{\lambda^-_k}), m^-(\mathcal{E}_{\lambda_k^-}, \gamma_{\lambda^-_k})+
m^0(\mathcal{E}_{\lambda_k^-}, \gamma_{\lambda^-_k})]\\
&\cap[m^-(\mathcal{E}_{\lambda_k^+}, \gamma_{\lambda^+_k}),
m^-(\mathcal{E}_{\lambda_k^+}, \gamma_{\lambda^+_k})+m^0(\mathcal{E}_{\lambda_k^+}, \gamma_{\lambda^+_k})]=\emptyset.
\end{array}\right\}\eqno(\hbox{$\ast_k$})
$$
\item[\rm (II.3)] For each $k\in\mathbb{N}$, (\hbox{$\ast_k$}) holds true,
and either $m^0(\mathcal{E}_{\lambda_k^-}, \gamma_{\lambda^-_k})=0$ or $m^0(\mathcal{E}_{\lambda_k^+}, \gamma_{\lambda^+_k})=0$.
 \end{enumerate}
 Then there exists a sequence $\{(\lambda_k,\gamma^k)\}_{k\ge 1}$ in
 $\hat\Lambda\times C^2_{S_0\times S_1}([0,\tau]; M)$  converging to
 $(\mu, \gamma_\mu)$  such that each $\gamma^k\ne\gamma_{\lambda_k}$ is a solution of the problem
 (\ref{e:LagrGener})--(\ref{e:LagrGenerB})  with $\lambda=\lambda_k$,
 $k=1,2,\cdots$,  where
 $\hat{\Lambda}=\{\mu,\lambda^+_k, \lambda^-_k\,|\,k\in\mathbb{N}\}$.
 In particular, $(\mu,\gamma_\mu)$ is a bifurcation point of the problem
 (\ref{e:LagrGener})--(\ref{e:LagrGenerB}) in $\hat\Lambda\times C^2_{S_0\times S_1}([0,\tau]; M)$ with  respect to the branch $\{(\lambda, \gamma_\lambda)\,|\,\lambda\in\hat\Lambda\}$
  {\rm (}and so $\{(\lambda, \gamma_\lambda)\,|\,\lambda\in\Lambda\}${\rm )}.
\end{enumerate}
  \end{theorem}

\begin{theorem}[\textsf{Existence for bifurcations}]\label{th:bif-existLagrGener}
Let  Assumptions~\ref{ass:Lagr6},~\ref{ass:LagrGenerB} be satisfied,
and let $\Lambda$ be path-connected. Suppose that
  there exist two  points $\lambda^+, \lambda^-\in\Lambda$ such that
  one of the following conditions is satisfied:
 \begin{enumerate}
 \item[\rm (i)] Either $\gamma_{\lambda^+}$  is not an isolated critical point of $\mathcal{E}_{\lambda^+}$,
 or $\gamma_{\lambda^-}$ is not an isolated critical point of $\mathcal{E}_{\lambda^-}$,
 or $\gamma_{\lambda^+}$ {\rm (}resp. $\gamma_{\lambda^-}${\rm )}  is an isolated critical point of
  $\mathcal{E}_{\lambda^+}$ {\rm (}resp. $\mathcal{E}_{\lambda^-}${\rm )} and
  $C_m(\mathcal{E}_{\lambda^+}, \gamma_{\lambda^+};{\bf K})$ and $C_m(\mathcal{E}_{\lambda^-}, \gamma_{\lambda^-};{\bf K})$ are not isomorphic for some Abel group ${\bf K}$ and some $m\in\mathbb{Z}$.

\item[\rm (ii)] %The intervals $[m^-({\mathcal{E}}_{\lambda^-}, \gamma_{\lambda^-}),
%m^-({\mathcal{E}}_{\lambda^-}, \gamma_{\lambda^-})+ m^0({\mathcal{E}}_{\lambda^-}, \gamma_{\lambda^-})]$ and
%$$
%[m^-({\mathcal{E}}_{\lambda^+}, \gamma_{\lambda^+}),
%m^-({\mathcal{E}}_{\lambda^+}, \gamma_{\lambda^+})+m^0({\mathcal{E}}_{\lambda^+}, \gamma_{\lambda^+})]
%$$
%is disjoint,
$[m^-({\mathcal{E}}_{\lambda^-}, \gamma_{\lambda^-}),
m^-({\mathcal{E}}_{\lambda^-}, \gamma_{\lambda^-})+ m^0({\mathcal{E}}_{\lambda^-}, \gamma_{\lambda^-})]\cap
[m^-({\mathcal{E}}_{\lambda^+}, \gamma_{\lambda^+}),
m^-({\mathcal{E}}_{\lambda^+}, \gamma_{\lambda^+})+m^0({\mathcal{E}}_{\lambda^+}, \gamma_{\lambda^+})]
\\=\emptyset$,
and there exists $\lambda\in\{\lambda^+, \lambda^-\}$ such that $\gamma_{\lambda}$  is an either non-isolated or homological visible critical point of
$\mathcal{E}_{\lambda}$.

\item[\rm (iii)] $[m^-({\mathcal{E}}_{\lambda^-}, \gamma_{\lambda^-}),
m^-({\mathcal{E}}_{\lambda^-}, \gamma_{\lambda^-})+ m^0({\mathcal{E}}_{\lambda^-}, \gamma_{\lambda^-})]\cap[m^-({\mathcal{E}}_{\lambda^+}, \gamma_{\lambda^+}),
m^-({\mathcal{E}}_{\lambda^+}, \gamma_{\lambda^+})+m^0({\mathcal{E}}_{\lambda^+}, \gamma_{\lambda^+})]\\=\emptyset$,
and either $m^0(\mathcal{E}_{\lambda^+}, \gamma_{\lambda^+})=0$ or $m^0(\mathcal{E}_{\lambda^-}, \gamma_{\lambda^-})=0$.
 \end{enumerate}
  Then for any path $\alpha:[0,1]\to\Lambda$ connecting $\lambda^+$ to $\lambda^-$ such that $\gamma_{\alpha(s)}(0)\ne \gamma_{\alpha(s)}(\tau)$
  for any $s\in [0,1]$ in the case $\dim S_0>0$ and $\dim S_1>0$,   there exists
 a sequence $(\lambda_k)\subset\alpha([0,1])$ converging to some $\mu\in\alpha([0,1])$, and
    solutions $\gamma^k\ne\gamma_{\lambda_k}$ of the problem (\ref{e:LagrGener})--(\ref{e:LagrGenerB})
  with $\lambda=\lambda_k$, $k=1,2,\cdots$, such that $\|\gamma^k-\gamma_{\lambda_k}\|_{C^2([0,\tau];\mathbb{R}^N)}\to 0$
  as $k\to\infty$. {\rm (}In particular,
  $(\mu, \gamma_{\mu})$  is a  bifurcation point along sequences  of the problem (\ref{e:LagrGener})--(\ref{e:LagrGenerB}) in
 $\Lambda\times C^2_{S_0\times S_1}([0,\tau]; M)$  with respect to the branch $\{(\lambda, \gamma_\lambda)\,|\,\lambda\in\Lambda\}$.{\rm )}
Moreover, $\mu$ is not equal to $\lambda^+$ {\rm (}resp. $\lambda^-${\rm )} if $m^0_\tau(\mathcal{E}_{\lambda^+}, \gamma_{\lambda^+})=0$
{\rm (}resp. $m^0_\tau(\mathcal{E}_{\lambda^-}, \gamma_{\lambda^-})=0${\rm )}.
  \end{theorem}

  \begin{theorem}[\textsf{Alternative bifurcations of Rabinowitz's type}]\label{th:bif-suffLagrGener}
  Under Assumptions~\ref{ass:Lagr6},~\ref{ass:LagrGenerB} with $\Lambda$ being a real interval,
let $\mu\in{\rm Int}(\Lambda)$ satisfy $\gamma_\mu(0)\ne\gamma_\mu(\tau)$
{\rm (}if $\dim S_0>0$ and $\dim S_1>0${\rm )} and $m^0(\mathcal{E}_{\mu}, \gamma_\mu)>0$.
   If  $m^0(\mathcal{E}_{\lambda}, \gamma_{\lambda})=0$  for each $\lambda\in\Lambda\setminus\{\mu\}$ near $\mu$, and
  $m^-(\mathcal{E}_{\lambda}, \gamma_{\lambda})$ take, respectively, values $m^-(\mathcal{E}_{\mu}, \gamma_\mu)$ and
  $m^-(\mathcal{E}_{\mu}, \gamma_\mu)+ m^0(\mathcal{E}_{\mu}, \gamma_\mu)$
 as $\lambda\in\Lambda$ varies in two deleted half neighborhoods  of $\mu$,
then  one of the following alternatives occurs:
\begin{enumerate}
\item[\rm (i)] The problem (\ref{e:LagrGener})--(\ref{e:LagrGenerB})
 with $\lambda=\mu$ has a sequence of solutions, $\gamma_k\ne \gamma_\mu$, $k=1,2,\cdots$,
which converges to $\gamma_\mu$ in $C^2([0,\tau], M)$.

\item[\rm (ii)]  For every $\lambda\in\Lambda\setminus\{\mu\}$ near $\mu$ there is a  solution $\alpha_\lambda\ne \gamma_\lambda$ of
(\ref{e:LagrGener})--(\ref{e:LagrGenerB}) with parameter value $\lambda$,
 such that  $\alpha_\lambda-\gamma_\lambda$
 converges to zero in  $C^2([0,\tau], \R^N)$ as $\lambda\to \mu$.
{\rm (}Recall that we have assumed $M\subset\R^N$.{\rm )}

\item[\rm (iii)] For a given neighborhood $\mathcal{W}$ of $\gamma_\mu$ in $C^1([0,\tau], M)$, 
there is a one-sided  neighborhood $\Lambda^0$ of $\mu$ such that
for any $\lambda\in\Lambda^0\setminus\{\mu\}$,  (\ref{e:LagrGener})--(\ref{e:LagrGenerB})
 with parameter value $\lambda$
has at least two distinct solutions in $\mathcal{W}$, $\gamma_\lambda^1\ne \gamma_\lambda$ and $\gamma_\lambda^2\ne \gamma_\lambda$,
which can also be chosen to satisfy $\mathcal{E}_{\lambda}(\gamma_\lambda^1)\ne \mathcal{E}_{\lambda}(\gamma_\lambda^2)$
provided that  $m^0(\mathcal{E}_{\mu}, \gamma_\mu)>1$ and (\ref{e:LagrGener})--(\ref{e:LagrGenerB}) with parameter value $\lambda$
has only finitely many distinct solutions in $\mathcal{W}$.
\end{enumerate}
\end{theorem}

When $M$ is an open subset in $\mathbb{R}^n$, the conditions in Assumptions~\ref{ass:Lagr6},~\ref{ass:LagrGenerB} in theorems above
can be weakened, see Theorems~\ref{th:bif-nessLagrGenerEu},~\ref{th:bif-existLagrGenerEu+},~\ref{th:bif-suffLagrGenerEu}. 

\begin{assumption}\label{ass:LagrGenerC}
{\rm Let $S_0$ be a boundaryless  submanifold of $M$ of dimension $\dim S_0<\dim M$,
 and let $L: [0,\tau]\times TM\to\R$ be $C^3$ and  fiberwise strictly convex,
 that is, for each $(t,q,v)\in [0,\tau]\times TM$ the bilinear form $\partial_{vv}L(t,q,v)$ is positive definite.}
\end{assumption}

Under Assumption~\ref{ass:LagrGenerC}, a $C^2$ curve $\gamma:[0, \lambda]\to M$ with $\lambda\in (0,\tau]$ is called
a \textsf{Euler-Lagrange curve of $L$ emanating perpendicularly from $S_0$} if it solves  the following boundary problem
\begin{equation}\label{e:LagrCurve}
\left.\begin{array}{ll}
&\frac{d}{dt}\big(\partial_vL(t, \gamma(t), \dot{\gamma}(t))\big)-\partial_q L(t, \gamma(t), \dot{\gamma}(t))=0,\;0\le t\le\lambda,\\
&\gamma(0)\in S_0\quad\hbox{and}\quad\partial_vL(0, \gamma(0), \dot{\gamma}(0))[v]=0
\;\forall v\in T_{\gamma(0)}S_0.
\end{array}\right\}
\end{equation}
In particular, if $S_0$ consists of a point $p$ we call 
$\gamma$ a \textsf{Euler-Lagrange curve of $L$  starting at $p$}.
 Since $L$ is $C^3$, with a local coordinate chart
it may follow from \cite[Proposition~4.3]{BuGiHi}  that the Euler-Lagrange curves of $L$ are $C^3$.
Clearly, for each $s\in (0, \lambda]$ the Euler-Lagrange curve $\gamma_s:=\gamma|_{[0,s]}$ of $L$
emanating perpendicularly from $S_0$ is a critical point of the $C^2$ functional
\begin{equation}\label{e:LagrFinEnergyPQ0}
\mathcal{L}_{S_0,s}(\alpha):=\int^s_0L(t, \alpha(t), \dot{\alpha}(t))dt
\end{equation}
 on the $C^4$ Banach manifold
\begin{equation}\label{e:BanachEnd}
C^{1}_{S_0\times\{\gamma(s)\}}([0, s]; M)=\left\{\alpha\in C^{1}([0,s]; M)\,\big|\, \alpha(0)\in S_0, \alpha(s)=\gamma(s))\right\}.
\end{equation}
We say $s\in (0, \lambda]$ to be a \textsf{$S_0$-focal point along $\gamma$}
if the linearization of (\ref{e:LagrCurve}) on $[0, s]$ (called
the the \textsf{Jacobi equation of the functional} $\mathcal{L}_{S_0, s}$) has nonzero solutions, i.e.,
the second order differential $D^2\mathcal{L}_{S_0, s}(\gamma_s)$ of
$\mathcal{L}_{S_0,s}$ at $\gamma_s$ is degenerate; moreover
 $\dim{\rm Ker}(D^2\mathcal{L}_{S_0,s}(\gamma_s))$ is called the \textsf{multiplicity} of $s$,
 denoted by 
 \begin{equation}\label{e:NullityFocus}
 \nu^{S_0}_\gamma(s)\quad\text{or}\quad m^0(\mathcal{L}_{S_0,s},\gamma_s).
\end{equation}

As done in \cite[Definition~6.1]{PiPoTa04} for geodesics,
similar to Jacobi's original
definition of conjugate points along an extremal of quadratic functionals (cf. \cite[Definition 4, page 114]{GeFo63}) we introduce:

\begin{definition}\label{def:bifurLagr}
{\rm Under Assumption~\ref{ass:LagrGenerC},  $\mu\in (0, \lambda)$
is called a \textsf{bifurcation instant for $(S_0, \gamma)$} if there exists a sequence $(t_k)\subset (0,\lambda]$ converging to $\mu$
and a sequence of Euler-Lagrange curves of $L$ emanating perpendicularly from $S_0$, $\gamma^k:[0, t_k]\to M$,   such that
\begin{eqnarray}\label{e:Lgeobifu1}
&&\gamma^k(t_k)=\gamma(t_k)\;\hbox{for all $k\in\N$},\\
&&0<\|\gamma^k-\gamma|_{[0,t_k]}\|_{C^1([0, t_k],\mathbb{R}^N)}\to 0\;\hbox{ as $k\to\infty$}.
\label{e:Lgeobifu2}
\end{eqnarray}}
\end{definition}

As proved in Lemma~\ref{lem:regu}(ii), using local coordinate charts
we can derive from the basic existence, uniqueness and smoothness theorem of ODE solutions that
the limit of (\ref{e:Lgeobifu2}) is equivalent to any one of the following two conditions:
\begin{description}
\item[$\bullet$] $\|\gamma^k-\gamma|_{[0,t_k]}\|_{C^2([0, t_k],\mathbb{R}^N)}\to 0$ as $k\to\infty$.
\item[$\bullet$] $\gamma^k(0)\to\gamma(0)$ and $\dot{\gamma}^k(0)\to\dot{\gamma}(0)$.
\end{description}

 \begin{theorem}\label{th:MorseBif}
 Under Assumption~\ref{ass:LagrGenerC}, let $\gamma:[0, \tau]\to M$ be
a Euler-Lagrange curve of $L$  emanating perpendicularly from $S_0$.
Then:
\begin{enumerate}
\item[\rm (i)] There exist only finitely many $S_0$-focal points  along $\gamma$.
\item[\rm (ii)] If $\mu\in (0, \tau]$  is a bifurcation instant for $(S_0, \gamma)$, then  it is a $S_0$-focal point along $\gamma$.
\item[\rm (iii)] If $\mu\in (0, \tau)$ is a $S_0$-focal point along $\gamma$, then it is a bifurcation instant for $(S_0, \gamma)$,
and  one of the following alternatives occurs:
\begin{enumerate}
\item[\rm (iii-1)] There exists a sequence of distinct $C^3$ Euler-Lagrange curves of $L$
 emanating perpendicularly from $S_0$ and ending at $\gamma(\mu)$,
$\alpha_k:[0,\mu]\to M$,  $\alpha_k\ne\gamma|_{[0,\mu]}$, $k=1,2,\cdots$, such that
 $\alpha_k\to\gamma|_{[0,\mu]}$ in $C^2([0,\mu],\mathbb{R}^N)$ as $k\to\infty$.

\item[\rm (iii-2)]  For every $\lambda\in (0, \tau)\setminus\{\mu\}$ near $\mu$ there
exists a $C^3$ Euler-Lagrange curves of $L$ emanating perpendicularly from $S_0$ and ending at $\gamma(\lambda)$,
$\alpha_\lambda:[0,\lambda]\to M$,  $\alpha_\lambda\ne\gamma|_{[0,\lambda]}$, such that
 $\|\alpha_\lambda-\gamma|_{[0,\lambda]}\|_{C^2([0,\lambda],\mathbb{R}^N)}\to 0$ as $\lambda\to\mu$.

\item[\rm (iii-3)]
For a given small $\epsilon>0$ there is a one-sided neighborhood $\Lambda^\ast$ of $\mu$ such that
for any $\lambda\in\Lambda^\ast\setminus\{\mu\}$, there exist at least two distinct $C^3$
Euler-Lagrange curves of $L$ emanating perpendicularly from $S_0$ and ending at $\gamma(\lambda)$,
$\beta^i_\lambda:[0,\lambda]\to M$,  $\beta_\lambda^i\ne\gamma|_{[0,\lambda]}$, $i=1,2$,
to satisfy the condition that  $\|\beta^i_\lambda-\gamma|_{[0,\lambda]}\|_{C^1([0,\lambda],\mathbb{R}^N)}<\epsilon$, $i=1,2$.
Moreover, if the multiplicity of $\gamma(\mu)$ as a $S_0$-focal point along $\gamma$ is greater than one and
there exist only finitely many distinct $C^3$ Euler-Lagrange curves of $L$
  emanating perpendicularly from $S_0$ and ending at $\gamma(\lambda)$,
$\alpha_1,\cdots,\alpha_m$, such that $\|\alpha_i-\gamma|_{[0,\lambda]}\|_{C^1([0,\lambda],\mathbb{R}^N)}<\epsilon$, $i=1,\cdots, m$,
then the above two distinct $C^3$ L-curves  $\beta_\lambda^i\ne\gamma|_{[0,\lambda]}$, $i=1,2$,
can also be chosen to satisfy
\begin{equation}\label{e:diffEnergy}
\int^\lambda_0L(t, \beta_\lambda^1(t), \dot{\beta}_\lambda^1(t))dt\ne \int^\lambda_0
L(t, \beta_\lambda^2(t), \dot{\beta}_\lambda^2(t))dt.
\end{equation}
\end{enumerate}
\end{enumerate}
\end{theorem}

\begin{remark}\label{rm:MorseBif}
{\rm If $M$ is an open subset in the Euclidean space, we may assume that $L$
in Assumption~\ref{ass:LagrGenerC} and Theorem~\ref{th:MorseBif}
is $C^2$, see Theorems~\ref{th:bif-nessLagrGenerEu},~\ref{th:bif-suffLagrGenerEu}.
}
\end{remark}

\subsection{Bifurcations of Lagrangian trajectories with endpoints constrained in  ${\rm Graph}(\mathbb{I}_{g})$}\label{sec:Graph}

\begin{assumption}\label{ass:Lagr7}
{\rm Let $(M, g, \mathbb{I}_g)$ be as in Assumption~1.0 in Introduction,
and let $(\tau, \Lambda, L)$ be as in Assumption~\ref{ass:Lagr6}.
%and also satisfy
%$L(\lambda, \tau, \mathbb{I}_g(x), d\mathbb{I}_g(x)[v])=L(\lambda, 0, x,v)$ for all $(\lambda, x,v)\in\Lambda\times TM$.
For each $\lambda\in\Lambda$ let $\gamma_\lambda\in C^2([0,\tau];M)$ satisfy
 the following boundary problem
\begin{equation}\label{e:Lagr2}
\left.\begin{array}{ll}
&\frac{d}{dt}\big(\partial_vL_\lambda(t, \gamma(t), \dot{\gamma}(t))\big)-\partial_q L_\lambda(t, \gamma(t), \dot{\gamma}(t))=0,\;0\le t\le\tau,\\
&\mathbb{I}_g(\gamma(0))=\gamma(\tau)\quad\hbox{and}\quad
d\mathbb{I}_g(\gamma(0))\big[\frac{\partial
L_\lambda}{\partial v}(0, \gamma(0), \dot{\gamma}(0))\big]=\frac{\partial
L_\lambda}{\partial v}(\tau, \gamma(\tau), \dot{\gamma}(\tau)).
\end{array}\right\}
\end{equation}
Suppose also that $\Lambda\times \R\ni (\lambda,t)\mapsto\gamma_\lambda(t)\in M$ and
$\Lambda\times \R\ni (\lambda,t)\mapsto\dot{\gamma}_\lambda(t)\in TM$ are continuous.
}
\end{assumption}

Consider the following  $C^{4}$  Banach submanifold of $C^{1}([0, \tau]; M)$ of codimension $n$,
\begin{equation}\label{e:Rot-space}
C^{1}_{\mathbb{I}_g}([0, \tau]; M):=\{\gamma\in C^{1}([0, \tau]; M)\,|\, \mathbb{I}_g(\gamma(0))=\gamma(\tau)\}.
\end{equation}
Its tangent space at $\gamma\in C^{1}_{\mathbb{I}_g}([0, \tau]; M)$ is
\begin{equation}\label{e:Rot-spaceA}
 C^1_{\mathbb{I}_g}(\gamma^\ast TM):=T_{\gamma}C^{1}_{\mathbb{I}_g}([0, \tau]; M)=\{\xi\in C^1(\gamma^\ast TM)\,|\, d\mathbb{I}_g(\gamma(0))[\xi(0)]=\xi(\tau)\},
\end{equation}
which is dense in the Hilbert subspace
\begin{equation}\label{e:Rot-spaceB}
W^{1,2}_{\mathbb{I}_g}(\gamma^\ast TM):=\{\xi\in H^1(\gamma^\ast TM)\,|\, d\mathbb{I}_g(\gamma(0))[\xi(0)]=\xi(\tau)\}
\end{equation}
of $W^{1,2}(\gamma^\ast TM)$.

For each $\lambda\in\Lambda$, by the proof of \cite[Proposition~4.2]{Lu9} we have a $C^2$ functional
\begin{equation}\label{e:Lagr0}
\mathscr{E}_\lambda: C^{1}_{\mathbb{I}_g}([0, \tau]; M)\to\mathbb{R},\;
\gamma\mapsto\int^\tau_0L_\lambda(t, \gamma(t), \dot{\gamma}(t))dt.
\end{equation}
 \cite[Proposition~4.2]{BuGiHi} shows that
$\gamma\in C^{1}_{\mathbb{I}_g}([0, \tau]; M)$ is
 a critical point of $\mathscr{E}_\lambda$
if and only if it belongs to $C^{3}([0, \tau]; M)$ and satisfy  (\ref{e:Lagr2}).
By \cite{Du}, the second-order differential $D^2\mathscr{E}_\lambda(\gamma)$ of $\mathscr{E}_\lambda$ at such a critical point $\gamma$
can be extended into a continuous symmetric bilinear form on $W^{1,2}_{\mathbb{I}_g}(\gamma^\ast TM)$
with finite Morse index and nullity
\begin{equation}\label{e:Lagr3}
m^-_\tau(\mathscr{E}_\lambda,\gamma)\quad\hbox{and}\quad m^0_\tau(\mathscr{E}_\lambda,\gamma),
\end{equation}
which coincide with those defined by (\ref{e:DuMorseIndex}) and (\ref{e:DuNullity}) 
for $\mathbf{N}=\text{Graph}(\mathbb{I}_g)$, respectively.

\begin{definition}\label{def:Bifur}
{\rm In Definition~\ref{def:LagrGenerBifur},  ``
$X=C^{1}_{S_0\times S_1}([0,\tau]; M)$ (or $C^{2}_{S_0\times S_1}([0,\tau]; M)$)''
is replaced by ``$X=C^{1}_{\mathbb{I}_g}([0,\tau]; M)$ (or $C^{2}_{\mathbb{I}_g}([0,\tau]; M)$)'',
and ``(\ref{e:LagrGener})--(\ref{e:LagrGenerB})''
is replaced by ``(\ref{e:Lagr2})''.
}
 \end{definition}

\begin{theorem}\label{th:bif-nessLagr*}
Let Assumption~\ref{ass:Lagr7} be satisfied, and $\mu\in\Lambda$.
\begin{enumerate}
\item[\rm (I)]{\rm (\textsf{Necessary condition}):}
Suppose  that $(\mu, \gamma_\mu)$  is a  bifurcation point along sequences of
 (\ref{e:Lagr2}) with respect to the branch $\{(\lambda,\gamma_\lambda)\,|\,\lambda\in\Lambda\}$ in $C^{1}_{\mathbb{I}_g}([0,\tau]; M)$.
Then $m^0(\mathscr{E}_{\mu}, \gamma_\mu)>0$.
\item[\rm (II)]{\rm (\textsf{Sufficient condition}):}
Suppose that $\Lambda$ is first countable and that there exist two sequences in  $\Lambda$ converging to $\mu$, $(\lambda_k^-)$ and
$(\lambda_k^+)$,  such that one of the following conditions is satisfied:
 \begin{enumerate}
 \item[\rm (II.1)] For each $k\in\mathbb{N}$, either $\gamma_{\lambda^+_k}$  is not an isolated critical point of ${\mathscr{E}}_{\lambda^+_k}$,
 or $\gamma_{\lambda^-_k}$ is not an isolated critical point of ${\mathscr{E}}_{\lambda^-_k}$,
 or $\gamma_{\lambda^+_k}$ {\rm (}resp. $\gamma_{\lambda^-_k}${\rm )}
  is an isolated critical point of $\mathscr{E}_{\lambda^+_k}$ {\rm (}resp. $\mathscr{E}_{\lambda^-_k}${\rm )} and
  $C_m(\mathscr{E}_{\lambda^+_k}, \gamma_{\lambda^+_k};{\bf K})$ and $C_m(\mathscr{E}_{\lambda^-_k}, \gamma_{\lambda^-_k};{\bf K})$
  are not isomorphic for some Abel group ${\bf K}$ and some $m\in\mathbb{Z}$.
\item[\rm (II.2)] For each $k\in\mathbb{N}$, there exists $\lambda\in\{\lambda^+_k, \lambda^-_k\}$
such that $\gamma_{\lambda}$  is an either nonisolated or homological visible critical point of
$\mathscr{E}_{\lambda}$ , and
$$
\left.\begin{array}{ll}
&[m^-(\mathscr{E}_{\lambda_k^-}, \gamma_{\lambda^-_k}), m^-(\mathscr{E}_{\lambda_k^-}, \gamma_{\lambda^-_k})+
m^0(\mathscr{E}_{\lambda_k^-}, \gamma_{\lambda^-_k})]\\
&\cap[m^-(\mathscr{E}_{\lambda_k^+}, \gamma_{\lambda^+_k}),
m^-(\mathscr{E}_{\lambda_k^+}, \gamma_{\lambda^+_k})+m^0(\mathscr{E}_{\lambda_k^+}, \gamma_{\lambda^+_k})]=\emptyset.
\end{array}\right\}\eqno(\hbox{$2\ast_k$})
$$
\item[\rm (II.3)] For each $k\in\mathbb{N}$, (\hbox{$2\ast_k$}) holds true,
and either $m^0(\mathscr{E}_{\lambda_k^-}, \gamma_{\lambda^-_k})=0$ or $m^0(\mathscr{E}_{\lambda_k^+}, \gamma_{\lambda^+_k})=0$.
 \end{enumerate}
 Then there exists a sequence $\{(\lambda_k,\gamma^k)\}_{k\ge 1}$ in
 $\hat\Lambda\times C^{2}_{\mathbb{I}_g}([0,\tau]; M)$   converging to
 $(\mu, \gamma_\mu)$  such that each $\gamma^k\ne\gamma_{\lambda_k}$ is a solution of the problem
 (\ref{e:Lagr2})  with $\lambda=\lambda_k$,
 $k=1,2,\cdots$,  where
 $\hat{\Lambda}=\{\mu,\lambda^+_k, \lambda^-_k\,|\,k\in\mathbb{N}\}$.
 In particular, $(\mu,\gamma_\mu)$ is a bifurcation point of the problem
 (\ref{e:Lagr2}) in $\hat\Lambda\times C^{2}_{\mathbb{I}_g}([0,\tau]; M)$
 with respect to the branch $\{(\lambda, \gamma_\lambda)\,|\,\lambda\in\hat\Lambda\}$
  {\rm (}and so $\{(\lambda, \gamma_\lambda)\,|\,\lambda\in\Lambda\}${\rm )}.
 \end{enumerate}
\end{theorem}

\begin{theorem}[\textsf{Existence for bifurcations}]\label{th:bif-existLagr*}
Let Assumption~\ref{ass:Lagr7} be satisfied,
and let $\Lambda$ be path-connected. Suppose that
 there exist two  points $\lambda^+, \lambda^-\in\Lambda$ such that
  one of the following conditions is satisfied:
 \begin{enumerate}
 \item[\rm (i)] Either $\gamma_{\lambda^+}$  is not an isolated critical point of $\mathscr{E}_{\lambda^+}$,
 or $\gamma_{\lambda^-}$ is not an isolated critical point of $\mathscr{E}_{\lambda^-}$,
 or $\gamma_{\lambda^+}$ {\rm (}resp. $\gamma_{\lambda^-}${\rm )}  is an isolated critical point of
  $\mathscr{E}_{\lambda^+}$ {\rm (}resp. $\mathscr{E}_{\lambda^-}${\rm )} and
  $C_m(\mathscr{E}_{\lambda^+}, \gamma_{\lambda^+};{\bf K})$ and $C_m(\mathscr{E}_{\lambda^-}, \gamma_{\lambda^-};{\bf K})$ are not isomorphic for some Abel group ${\bf K}$ and some $m\in\mathbb{Z}$.

\item[\rm (ii)] $[m^-({\mathscr{E}}_{\lambda^-}, \gamma_{\lambda^-}),
m^-({\mathscr{E}}_{\lambda^-}, \gamma_{\lambda^-})+ m^0({\mathscr{E}}_{\lambda^-}, \gamma_{\lambda^-})]\cap[m^-({\mathscr{E}}_{\lambda^+}, \gamma_{\lambda^+}),
m^-({\mathscr{E}}_{\lambda^+}, \gamma_{\lambda^+})+m^0({\mathscr{E}}_{\lambda^+}, \gamma_{\lambda^+})]\\=\emptyset$,
and there exists $\lambda\in\{\lambda^+, \lambda^-\}$ such that $\gamma_{\lambda}$  is an either non-isolated or homological visible critical point of $\mathscr{E}_{\lambda}$.

\item[\rm (iii)] $[m^-({\mathscr{E}}_{\lambda^-}, \gamma_{\lambda^-}),
m^-({\mathscr{E}}_{\lambda^-}, \gamma_{\lambda^-})+ m^0({\mathscr{E}}_{\lambda^-}, \gamma_{\lambda^-})]\cap[m^-({\mathscr{E}}_{\lambda^+}, \gamma_{\lambda^+}),
m^-({\mathscr{E}}_{\lambda^+}, \gamma_{\lambda^+})+m^0({\mathscr{E}}_{\lambda^+}, \gamma_{\lambda^+})]\\=\emptyset$,
and either $m^0(\mathscr{E}_{\lambda^+}, \gamma_{\lambda^+})=0$ or $m^0(\mathscr{E}_{\lambda^-}, \gamma_{\lambda^-})=0$.
 \end{enumerate}
  Then for any path $\alpha:[0,1]\to\Lambda$ connecting $\lambda^+$ to $\lambda^-$ there exists
 a sequence $(\lambda_k)\subset\alpha([0,1])$ converging to some $\mu\in\alpha([0,1])$, and
   solutions $\gamma^k\ne\gamma_{\lambda_k}$ of the problem (\ref{e:Lagr2})
  with $\lambda=\lambda_k$, $k=1,2,\cdots$, such that $\|\gamma^k-\gamma_{\lambda_k}\|_{C^2([0,\tau];\mathbb{R}^N)}\to 0$
  as $k\to\infty$. {\rm (}In particular,
  $(\mu, \gamma_{\mu})$  is a  bifurcation point along sequences  of the problem (\ref{e:Lagr2}) in
 $\Lambda\times C^{2}_{\mathbb{I}_g}([0,\tau]; M)$  with respect to the branch $\{(\lambda, \gamma_\lambda)\,|\,\lambda\in\Lambda\}$.{\rm )}
Moreover, $\mu$ is not equal to $\lambda^+$ {\rm (}resp. $\lambda^-${\rm )}
 if $m^0_\tau(\mathscr{E}_{\lambda^+}, \gamma_{\lambda^+})=0$ {\rm (}resp. $m^0_\tau(\mathscr{E}_{\lambda^-}, \gamma_{\lambda^-})=0${\rm )}.
\end{theorem}

\begin{theorem}[\textsf{Alternative bifurcations of Rabinowitz's type}]\label{th:bif-suffLagr*}
 Under Assumption~\ref{ass:Lagr7} with $\Lambda$ being a real interval,
let $\mu\in{\rm Int}(\Lambda)$ satisfy $m^0(\mathcal{E}_{\mu}, \gamma_\mu)>0$.
   If  $m^0(\mathcal{E}_{\lambda}, \gamma_{\lambda})=0$  for each $\lambda\in\Lambda\setminus\{\mu\}$ near $\mu$, and
  $m^-(\mathcal{E}_{\lambda}, \gamma_{\lambda})$ take, respectively, values $m^-(\mathcal{E}_{\mu}, \gamma_\mu)$ and
  $m^-(\mathcal{E}_{\mu}, \gamma_\mu)+ m^0(\mathcal{E}_{\mu}, \gamma_\mu)$
 as $\lambda\in\Lambda$ varies in two deleted half neighborhoods  of $\mu$,
then  one of the following alternatives occurs:
\begin{enumerate}
\item[\rm (i)] The problem (\ref{e:Lagr2})
 with $\lambda=\mu$ has a sequence of solutions, $\gamma_k\ne \gamma_\mu$, $k=1,2,\cdots$,
which converges to $\gamma_\mu$ in $C^2([0,\tau], M)$.

\item[\rm (ii)]  For every $\lambda\in\Lambda\setminus\{\mu\}$ near $\mu$ there is a  solution $\alpha_\lambda\ne \gamma_\lambda$ of
(\ref{e:Lagr2}) with parameter value $\lambda$,
 such that  $\alpha_\lambda-\gamma_\lambda$
 converges to zero in  $C^2([0,\tau], \R^N)$ as $\lambda\to \mu$.
{\rm (}Recall that we have assumed $M\subset\R^N$.{\rm )}

\item[\rm (iii)] For a given neighborhood $\mathcal{W}$ of $\gamma_\mu$ in $C^1_{\mathbb{I}_g}([0,\tau]; M)$,
there is a one-sided neighborhood $\Lambda^0$ of $\mu$ such that
for any $\lambda\in\Lambda^0\setminus\{\mu\}$, (\ref{e:Lagr2}) with parameter value $\lambda$
has at least two distinct solutions in $\mathcal{W}$, $\gamma_\lambda^1\ne \gamma_\lambda$ and $\gamma_\lambda^2\ne \gamma_\lambda$,
which can also be chosen to satisfy $\mathcal{E}_{\lambda}(\gamma_\lambda^1)\ne \mathcal{E}_{\lambda}(\gamma_\lambda^2)$
provided that  $m^0(\mathcal{E}_{\mu}, \gamma_\mu)>1$ and (\ref{e:Lagr2}) with parameter value $\lambda$
has only finitely many distinct solutions in $\mathcal{W}$.
\end{enumerate}
\end{theorem}

\subsection{Bifurcations of generalized periodic solutions  in %time dependent
non-autonomous Lagrangian systems}\label{secA}

\begin{assumption}\label{ass:Lagr8}
{\rm  In Assumption~\ref{ass:Lagr7}, the interval $[0, \tau]$ is replaced by $\mathbb{R}$, and
$L$ is also required to be truly dependent on time $t\in\mathbb{R}$ and $\mathbb{I}_g$-invariant in the
following sense:
\begin{equation}\label{e:Lagr8}
L(\lambda, t+\tau, \mathbb{I}_g(x), d\mathbb{I}_g(x)[v])=L(\lambda, t, x,v)\quad\forall (\lambda, t,x,v)\in\Lambda\times\mathbb{R}\times TM.
\end{equation}
The problem (\ref{e:Lagr2}) is replaced by
\begin{equation}\label{e:Lagr10}
\left.\begin{array}{ll}
&\frac{d}{dt}\big(\partial_vL_\lambda(t, \gamma(t), \dot{\gamma}(t))\big)-\partial_q L_\lambda(t, \gamma(t), \dot{\gamma}(t))=0\;\;\forall t\in\mathbb{R},\\
&\mathbb{I}_g(\gamma(t))=\gamma(t+\tau)\quad\forall t\in\mathbb{R}.
\end{array}\right\}
\end{equation}
 }
\end{assumption}

Solutions of (\ref{e:Lagr10}) are also called 
 \textsf{$\mathbb{I}_g$-periodic trajectories with period $\tau$} (\cite{Dav}).
When $\mathbb{I}_g$ generates a cyclic group, that is, it is of finite order $p\in\mathbb{N}$,
every $\mathbb{I}_g$-periodic trajectory is $p\tau$-periodic.

Since the functional in (\ref{e:Lagr0}) is $C^2$, 
(or by the proof of \cite[Proposition~4.2]{Lu9}) we have a $C^2$
functional on a $C^4$ Banach manifold $\mathcal{X}^1_{\tau}(M, \mathbb{I}_g)$,
  \begin{equation}\label{e:Lagr9}
\mathfrak{E}_\lambda(\gamma)=\int^\tau_0L_\lambda(t, \gamma(t), \dot{\gamma}(t))dt
\end{equation}
  where
\begin{equation}\label{e:BanachPerMani}
\mathcal{X}^i_{\tau}(M, \mathbb{I}_g):=\{\gamma\in C^i(\mathbb{R}, M)\,|\, \mathbb{I}_g(\gamma(t))=\gamma(t+\tau)\;\forall t\},\quad
i=0,1,2,\cdots.
\end{equation}
A $C^2$ curve $\gamma:\mathbb{R}\to M$ satisfies (\ref{e:Lagr10}) if and only if it is a critical point of $\mathfrak{E}_\lambda$.
(If $\gamma$ is a nonconstant solution of (\ref{e:Lagr10}), for any $\theta\in\mathbb{R}\setminus\{0\}$,  $\gamma(\theta+\cdot)$
cannot be a solution of (\ref{e:Lagr10}) again because $L$ is truly dependent on time $t\in\mathbb{R}$.)
Clearly, a solution $\gamma$ of the problem (\ref{e:Lagr10}) restricts to a solution
$\gamma|_{[0,\tau]}$ of (\ref{e:Lagr2}).
Conversely, any solution $\gamma^\ast$ of (\ref{e:Lagr2}) may extend into that of (\ref{e:Lagr10}),
$\gamma:\mathbb{R}\to M$, via
 \begin{equation}\label{e:extend*}
\gamma(t)=(\mathbb{I}_g)^k(\gamma^\ast(t-k\tau))\;\hbox{if $k\tau<t\le (k+1)\tau$ with $\pm k\in\mathbb{N}$}.
\end{equation}
Moreover, for a solution $\gamma$ of (\ref{e:Lagr10}) we call
\begin{equation}\label{e:Lagr12}
m^-_\tau(\mathfrak{E}_\lambda,\gamma):=m^-_\tau(\mathscr{E}_\lambda,\gamma|_{[0,\tau]})\quad\hbox{and}\quad
 m^0_\tau(\mathfrak{E}_\lambda,\gamma):=m^0_\tau(\mathscr{E}_\lambda, \gamma|_{[0,\tau]})
\end{equation}
the Morse index and nullity of $\mathfrak{E}_\lambda$ at $\gamma$, respectively, where
$m^-_\tau(\mathscr{E}_\lambda,\gamma|_{[0,\tau]})$ and $m^0_\tau(\mathscr{E}_\lambda, \gamma|_{[0,\tau]})$
are as in (\ref{e:Lagr3}). These are well-defined by \cite[\S4]{Dav}.

 Theorem~\ref{th:bif-nessLagr*},~\ref{th:bif-existLagr*},~\ref{th:bif-suffLagr*} 
 immediately leads to the following two results, respectively.

\begin{theorem}\label{th:bif-nessLagr**}
Let  Assumption~\ref{ass:Lagr8} be satisfied, and  $\mu\in\Lambda$.
\begin{enumerate}
\item[\rm (I)]{\rm (\textsf{Necessary condition}):}
Suppose  that $(\mu, \gamma_\mu)$  is a  bifurcation point along sequences of the problem
(\ref{e:Lagr10}) in $\Lambda\times\mathcal{X}^1_{\tau}(M, \mathbb{I}_g)$ with respect to the branch
$\{(\lambda,\gamma_\lambda)\,|\,\lambda\in\Lambda\}$.
Then $m^0_\tau(\mathfrak{E}_{\mu}, \gamma_\mu)\ne 0$.

\item[\rm (II)]{\rm (\textsf{Sufficient condition}):}
Suppose that $\Lambda$ is first countable and that there exist two sequences in  $\Lambda$ converging to $\mu$, $(\lambda_k^-)$ and
$(\lambda_k^+)$,  such that one of the following conditions is satisfied:
 \begin{enumerate}
 \item[\rm (II.1)] For each $k\in\mathbb{N}$, either $\gamma_{\lambda^+_k}$  is not an isolated critical point of
 ${\mathfrak{E}}_{\lambda^+_k}$,
 or $\gamma_{\lambda^-_k}$ is not an isolated critical point of ${\mathfrak{E}}_{\lambda^-_k}$,
 or $\gamma_{\lambda^+_k}$ {\rm (}resp. $\gamma_{\lambda^-_k}${\rm )} is an isolated critical point
 of $\mathfrak{E}_{\lambda^+_k}$ {\rm (}resp. $\mathfrak{E}_{\lambda^-_k}${\rm )} and
  $C_m(\mathfrak{E}_{\lambda^+_k}, \gamma_{\lambda^+_k};{\bf K})$ and $C_m(\mathfrak{E}_{\lambda^-_k}, \gamma_{\lambda^-_k};{\bf K})$
  are not isomorphic for some Abel group ${\bf K}$ and some $m\in\mathbb{Z}$.
\item[\rm (II.2)] For each $k\in\mathbb{N}$, there exists $\lambda\in\{\lambda^+_k, \lambda^-_k\}$
such that $\gamma_{\lambda}$  is an either nonisolated or homological visible critical point of
$\mathfrak{E}_{\lambda}$ , and
$$
\left.\begin{array}{ll}
&[m^-(\mathfrak{E}_{\lambda_k^-}, \gamma_{\lambda^-_k}), m^-(\mathfrak{E}_{\lambda_k^-}, \gamma_{\lambda^-_k})+
m^0(\mathfrak{E}_{\lambda_k^-}, \gamma_{\lambda^-_k})]\\
&\cap[m^-(\mathfrak{E}_{\lambda_k^+}, \gamma_{\lambda^+_k}),
m^-(\mathfrak{E}_{\lambda_k^+}, \gamma_{\lambda^+_k})+m^0(\mathfrak{E}_{\lambda_k^+}, \gamma_{\lambda^+_k})]=\emptyset.
\end{array}\right\}\eqno(\hbox{$3\ast_k$})
$$
\item[\rm (II.3)] For each $k\in\mathbb{N}$, (\hbox{$3\ast_k$}) holds true,
and either $m^0(\mathfrak{E}_{\lambda_k^-}, \gamma_{\lambda^-_k})=0$ or $m^0(\mathfrak{E}_{\lambda_k^+}, \gamma_{\lambda^+_k})=0$.
 \end{enumerate}
 Then there exists a sequence $\{(\lambda_k,\gamma^k)\}_{k\ge 1}$ in
 $\hat\Lambda\times \mathcal{X}^2_{\tau}(M, \mathbb{I}_g)$    converging to
 $(\mu, \gamma_\mu)$  such that each $\gamma^k\ne\gamma_{\lambda_k}$ is a solution of the problem
 (\ref{e:Lagr10})  with $\lambda=\lambda_k$,
 $k=1,2,\cdots$,  where
 $\hat{\Lambda}=\{\mu,\lambda^+_k, \lambda^-_k\,|\,k\in\mathbb{N}\}$.
 In particular, $(\mu,\gamma_\mu)$ is a bifurcation point of the problem
 (\ref{e:Lagr10}) in $\hat\Lambda\times \mathcal{X}^2_{\tau}(M, \mathbb{I}_g)$
 with respect to the branch $\{(\lambda, \gamma_\lambda)\,|\,\lambda\in\hat\Lambda\}$
 {\rm (}and so $\{(\lambda, \gamma_\lambda)\,|\,\lambda\in\Lambda\}${\rm )}.
   \end{enumerate}
\end{theorem}

  \begin{theorem}[\textsf{Existence for bifurcations}]\label{th:bif-existLagr**}
Let  Assumption~\ref{ass:Lagr8} be satisfied,  and let $\Lambda$ be path-connected.
Suppose that
 there exist two  points $\lambda^+, \lambda^-\in\Lambda$ such that
  one of the following conditions is satisfied:
 \begin{enumerate}
 \item[\rm (i)] Either $\gamma_{\lambda^+}$  is not an isolated critical point of $\mathfrak{E}_{\lambda^+}$,
 or $\gamma_{\lambda^-}$ is not an isolated critical point of $\mathfrak{E}_{\lambda^-}$,
 or $\gamma_{\lambda^+}$ {\rm (}resp. $\gamma_{\lambda^-}${\rm )}  is an isolated critical point of
  $\mathfrak{E}_{\lambda^+}$ {\rm (}resp. $\mathfrak{E}_{\lambda^-}${\rm )} and
  $C_m(\mathfrak{E}_{\lambda^+}, \gamma_{\lambda^+};{\bf K})$ and $C_m(\mathfrak{E}_{\lambda^-}, \gamma_{\lambda^-};{\bf K})$ are not isomorphic for some Abel group ${\bf K}$ and some $m\in\mathbb{Z}$.

\item[\rm (ii)] $[m^-({\mathfrak{E}}_{\lambda^-}, \gamma_{\lambda^-}),
m^-({\mathfrak{E}}_{\lambda^-}, \gamma_{\lambda^-})+ m^0({\mathfrak{E}}_{\lambda^-}, \gamma_{\lambda^-})]\cap[m^-({\mathfrak{E}}_{\lambda^+}, \gamma_{\lambda^+}),
m^-({\mathfrak{E}}_{\lambda^+}, \gamma_{\lambda^+})+m^0({\mathfrak{E}}_{\lambda^+}, \gamma_{\lambda^+})]\\=\emptyset$,
and there exists $\lambda\in\{\lambda^+, \lambda^-\}$ such that $\gamma_{\lambda}$  is an either non-isolated or homological visible critical point of $\mathfrak{E}_{\lambda}$.

\item[\rm (iii)] $[m^-({\mathfrak{E}}_{\lambda^-}, \gamma_{\lambda^-}),
m^-({\mathfrak{E}}_{\lambda^-}, \gamma_{\lambda^-})+ m^0({\mathfrak{E}}_{\lambda^-}, \gamma_{\lambda^-})]\cap[m^-({\mathfrak{E}}_{\lambda^+}, \gamma_{\lambda^+}),
m^-({\mathfrak{E}}_{\lambda^+}, \gamma_{\lambda^+})+m^0({\mathfrak{E}}_{\lambda^+}, \gamma_{\lambda^+})]\\=\emptyset$,
and either $m^0(\mathfrak{E}_{\lambda^+}, \gamma_{\lambda^+})=0$ or $m^0(\mathfrak{E}_{\lambda^-}, \gamma_{\lambda^-})=0$.
 \end{enumerate}
  Then for any path $\alpha:[0,1]\to\Lambda$ connecting $\lambda^+$ to $\lambda^-$ there exists
 a sequence $(\lambda_k)\subset\alpha([0,1])$ converging to $\mu\in\alpha([0,1])$, and
   solutions $\gamma^k\ne\gamma_{\lambda_k}$ of the problem  (\ref{e:Lagr10})
  with $\lambda=\lambda_k$, $k=1,2,\cdots$, such that $(\gamma^k-\gamma_{\lambda_k})|_{[0,\tau]}\to 0$ in $C^2([0,\tau];\mathbb{R}^N)$
  as $k\to\infty$. {\rm (}In particular,
  $(\mu, \gamma_{\mu})$  is a  bifurcation point along sequences of the problem  (\ref{e:Lagr10}) in
 $\Lambda\times \mathcal{X}^2_{\tau}(M, \mathbb{I}_g)$  with respect to the branch $\{(\lambda, \gamma_\lambda)\,|\,\lambda\in\Lambda\}$.
 {\rm )}
Moreover, $\mu$ is not equal to $\lambda^+$ {\rm (}resp. $\lambda^-${\rm )}
if $m^0_\tau(\mathfrak{E}_{\lambda^+}, \gamma_{\lambda^+})=0$ {\rm (}resp. $m^0_\tau(\mathfrak{E}_{\lambda^-}, \gamma_{\lambda^-})=0${\rm )}.
\end{theorem}

\begin{theorem}[\textsf{Alternative bifurcations of Rabinowitz's type}]\label{th:bif-suffLagr**}
Under Assumption~\ref{ass:Lagr8}
with $\Lambda$ being a real interval,
let $\mu\in{\rm Int}(\Lambda)$ satisfy $m^0_\tau(\mathfrak{E}_{\mu}, \gamma_\mu)\ne 0$.
  Suppose that $m^0_\tau(\mathfrak{E}_{\lambda}, \gamma_{\lambda})=0$  for each $\lambda\in\Lambda\setminus\{\mu\}$ near $\mu$, and
 that $m^-_\tau(\mathfrak{E}_{\lambda}, \gamma_{\lambda})$ take, respectively, values $m^-_\tau(\mathfrak{E}_{\mu}, \gamma_\mu)$ and
  $m^-_\tau(\mathfrak{E}_{\mu}, \gamma_\mu)+ m^0_\tau(\mathfrak{E}_{\mu}, \gamma_\mu)$
 as $\lambda\in\Lambda$ varies in two deleted half neighborhoods  of $\mu$.
Then  one of the following alternatives occurs:
\begin{enumerate}
\item[\rm (i)] The problem (\ref{e:Lagr10})
 with $\lambda=\mu$ has a sequence of solutions, $\gamma_k\ne \gamma_\mu$, $k=1,2,\cdots$,
 such that $\gamma_k\to\gamma_\mu$ in $\mathcal{X}^2_{\tau}(M, \mathbb{I}_g)$
 (or equivalently $\gamma_k|_{[0,\tau]}-\gamma_\mu|_{[0,\tau]}$  
 converges to zero in  $C^2([0,\tau], \R^N)$).

\item[\rm (ii)]  For every $\lambda\in\Lambda\setminus\{\mu\}$ near $\mu$ there exists a  solution $\alpha_\lambda\ne \gamma_\lambda$ of
(\ref{e:Lagr10}) with parameter value $\lambda$,
 such that  $\|(\alpha_\lambda-\gamma_\lambda)|_{[0,\tau]}\|_{C^2([0,\tau], \R^N)}\to 0$ 
  as $\lambda\to \mu$.

\item[\rm (iii)] For a given neighborhood $\mathcal{W}$ of 
$\gamma_\mu$ in $\mathcal{X}^1_{\tau}(M, \mathbb{I}_g)$,
there exists a one-sided neighborhood $\Lambda^0$ of $\mu$ such that
for any $\lambda\in\Lambda^0\setminus\{\mu\}$, the problem (\ref{e:Lagr10}) with parameter value $\lambda$
has at least two distinct solutions in $\mathcal{W}$, $\gamma_\lambda^1\ne \gamma_\lambda$ and $\gamma_\lambda^2\ne \gamma_\lambda$,
which can also be chosen to satisfy $\mathfrak{E}_{\lambda}(\gamma_\lambda^1)\ne\mathfrak{E}_{\lambda}(\gamma_\lambda^2)$
provided that  $m^0_\tau(\mathfrak{E}_{\mu}, \gamma_\mu)>1$ and the problem (\ref{e:Lagr10}) with parameter value
$\lambda\in\Lambda^0\setminus\{\mu\}$ has only finitely many distinct solutions in $\mathcal{W}$.
\end{enumerate}
\end{theorem}

\subsection{Bifurcations of brake orbits in Lagrangian systems}\label{sec:Lagrbroke}

\begin{assumption}\label{ass:Lagrbrake}
{\rm  Let $(M, g)$ be as in Assumption~1.0 in Introduction.
For a real $\tau>0$ and a topological space $\Lambda$,
let $L:\Lambda\times\R\times TM\to\R$ be a continuous function satisfying
 \begin{equation}\label{e:brakeLM}
 L(\lambda, -t, q, -v)=L(\lambda, t, q, v)=L(\lambda, t+\tau, q, v)\quad\forall (\lambda,t,q,v)\in\Lambda\times\R\times TM.
\end{equation}
Suppose that for each $C^3$ chart $\alpha:U_\alpha\to\alpha(U_\alpha)\subset\mathbb{R}^n$ and the induced bundle
chart $T\alpha:TM|_{U_\alpha}\to \alpha(U_\alpha)\times\mathbb{R}^n\subset\mathbb{R}^n\times\mathbb{R}^n$ the function
$$
L^\alpha:\Lambda\times\mathbb{R}\times \alpha(U_\alpha)\times\mathbb{R}^n\to\mathbb{R},\;
(\lambda,t, q,v)\mapsto L(\lambda,t, (T\alpha)^{-1}(q,v))
$$
is $C^2$ with respect to $(t,q,v)$ and strictly convex with respect to $v$, and
all its partial derivatives also depend continuously on $(\lambda, t, q, v)$. (Here $L$ is allowed to be independent of time.)
  }
\end{assumption}

Consider the following problem
\begin{equation}\label{e:PPerLagrBrakeorbit}
\left.\begin{array}{ll}
&\frac{d}{dt}\big(\partial_vL_\lambda(t, \gamma(t), \dot{\gamma}(t))\big)-
\partial_x L_\lambda(t, \gamma(t), \dot{\gamma}(t))=0\;\forall t\in\mathbb{R}\\
&\gamma(-t)=\gamma(t)=\gamma(t+\tau)\quad\forall t\in\mathbb{R}
\end{array}\right\}
\end{equation}
and  $C^4$ Banach manifolds
 \begin{equation}\label{e:EBanachM}
 EC^{1}(S_\tau; M):=\{\gamma\in C^1(\R;M)\,|\,\gamma(t+\tau)=\gamma(t)\,\&\, \gamma(-t)=\gamma(t)\;\forall t\in\R\}. 
 \end{equation}
Solutions of (\ref{e:PPerLagrBrakeorbit}) are called  \textsf{brake orbits}.
Assumption~\ref{ass:Lagrbrake} assures that the solutions of (\ref{e:PPerLagrBrakeorbit}) are  critical points of the $C^2$ functionals
\begin{equation}\label{e:BrakeFunct}
{\mathcal{L}}^E_\lambda:EC^{1}(S_\tau; M)\to\R,\; \gamma\mapsto \int^\tau_0{L}_\lambda(t, \gamma(t), \dot{\gamma}(t))dt\in\R,\quad\lambda\in\Lambda.
\end{equation}
For a critical point $\gamma$ of ${\mathcal{L}}^E_\lambda$,
the second-order differential $D^2\mathcal{L}^E_\lambda(\gamma)$ can be extended into a continuous symmetric bilinear form on $W^{1,2}(\gamma^\ast TM)$
with finite Morse index and nullity
\begin{equation}\label{e:BrakeLagrMorse}
m^-_\tau(\mathcal{L}^E_\lambda,\gamma)\quad\hbox{and}\quad m^0_\tau(\mathcal{L}^E_\lambda,\gamma).
\end{equation}

\begin{assumption}\label{ass:Lagr7brake}
{\rm For each $\lambda\in\Lambda$ let $\gamma_\lambda\in EC^{1}(S_\tau; M)\cap C^2(S_\tau;M)$  satisfy
 (\ref{e:PPerLagrBrakeorbit}) and  the maps $\Lambda\times \R\ni(\lambda,t)\to \gamma_\lambda(t)\in M$ and
$\Lambda\times\R\ni(\lambda, t)\mapsto \dot{\gamma}_\lambda(t)\in TM$ are continuous.}
\end{assumption}

For $\mu\in\Lambda$ we call $(\mu, \gamma_\mu)$  a \textsf{bifurcation point along sequences} of the problem (\ref{e:PPerLagrBrakeorbit})
in $\Lambda\times EC^{1}(S_\tau; M)$ with respect to the branch $\{(\lambda,\gamma_\lambda)\,|\,\lambda\in\Lambda\}$
  if  there exists a sequence $\{(\lambda_k, \gamma^k)\}_{k\ge 1}$ in $\Lambda\times EC^{1}(S_\tau; M)$
  converging to $(\mu,\gamma_\mu)$, such that each $\gamma^k\ne\gamma_{\lambda_k}$
  is a solution of (\ref{e:PPerLagrBrakeorbit}) with $\lambda=\lambda_k$, $k=1,2,\cdots$.

\begin{theorem}\label{th:bif-nessLagrBrakeM}
Let Assumptions~\ref{ass:Lagrbrake},\ref{ass:Lagr7brake} be satisfied.
\begin{enumerate}
\item[\rm (I)]{\rm (\textsf{Necessary condition}):}
Suppose   that $(\mu, \gamma_\mu)$  is a  bifurcation point along sequences of the problem (\ref{e:PPerLagrBrakeorbit}).
Then  $m^0_\tau(\mathcal{L}^E_{\mu}, \gamma_\mu)>0$.
\item[\rm (II)]{\rm (\textsf{Sufficient condition}):}
Let $\Lambda$ be first countable.
Suppose  that there exist two sequences in  $\Lambda$ converging to $\mu$, $(\lambda_k^-)$ and
$(\lambda_k^+)$,  such that
one of the following conditions is satisfied:
 \begin{enumerate}
 \item[\rm (II.1)] For each $k\in\mathbb{N}$, either $\gamma_{\lambda^+_k}$  is not an isolated critical point of
 $\mathcal{L}^E_{\lambda^+_k}$,
 or $\gamma_{\lambda^-_k}$ is not an isolated critical point of $\mathcal{L}^E_{\lambda^-_k}$,
 or $\gamma_{\lambda^+_k}$ {\rm (}resp. $\gamma_{\lambda^-_k}${\rm )}
  is an isolated critical point of $\mathcal{L}^E_{\lambda^+_k}$ {\rm (}resp. $\mathcal{L}^E_{\lambda^-_k}${\rm )} and
  $C_m(\mathcal{L}^E_{\lambda^+_k}, \gamma_{\lambda^+_k};{\bf K})$ and $C_m(\mathcal{L}^E_{\lambda^-_k}, \gamma_{\lambda^-_k};{\bf K})$
  are not isomorphic for some Abel group ${\bf K}$ and some $m\in\mathbb{Z}$.
\item[\rm (II.2)] For each $k\in\mathbb{N}$, there exists $\lambda\in\{\lambda^+_k, \lambda^-_k\}$
such that $\gamma_{\lambda}$  is an either nonisolated or homological visible critical point of
$\mathcal{L}^E_{\lambda}$ , and
$$
\left.\begin{array}{ll}
&[m^-(\mathcal{L}^E_{\lambda_k^-}, \gamma_{\lambda^-_k}), m^-(\mathcal{L}^E_{\lambda_k^-}, \gamma_{\lambda^-_k})+
m^0(\mathcal{L}^E_{\lambda_k^-}, \gamma_{\lambda^-_k})]\\
&\cap[m^-(\mathcal{L}^E_{\lambda_k^+}, \gamma_{\lambda^+_k}),
m^-(\mathcal{L}^E_{\lambda_k^+}, \gamma_{\lambda^+_k})+m^0(\mathcal{L}^E_{\lambda_k^+}, \gamma_{\lambda^+_k})]=\emptyset.
\end{array}\right\}\eqno(\hbox{$4\ast_k$})
$$
\item[\rm (II.3)] For each $k\in\mathbb{N}$, (\hbox{$4\ast_k$}) holds true,
and either $m^0(\mathcal{L}^E_{\lambda_k^-}, \gamma_{\lambda^-_k})=0$ or $m^0(\mathcal{L}^E_{\lambda_k^+}, \gamma_{\lambda^+_k})=0$.
 \end{enumerate}
Then there exists a sequence $\{(\lambda_k,\gamma^k)\}_{k\ge 1}$ in
 $\hat\Lambda\times EC^{1}(S_\tau; M)$    converging to
 $(\mu, \gamma_\mu)$  such that each $\gamma^k\ne\gamma_{\lambda_k}$ is a solution of the problem
(\ref{e:PPerLagrBrakeorbit})   with $\lambda=\lambda_k$,
 $k=1,2,\cdots$,  where
 $\hat{\Lambda}=\{\mu,\lambda^+_k, \lambda^-_k\,|\,k\in\mathbb{N}\}$.
 In particular, $(\mu,\gamma_\mu)$ is a bifurcation point of the problem
 (\ref{e:PPerLagrBrakeorbit}) in $\hat\Lambda\times EC^{1}(S_\tau; M)$
 with respect to the branch $\{(\lambda, \gamma_\lambda)\,|\,\lambda\in\hat\Lambda\}$
  {\rm (}and so $\{(\lambda, \gamma_\lambda)\,|\,\lambda\in\Lambda\}${\rm )}.
 \end{enumerate}
\end{theorem}

\begin{theorem}[\textsf{Existence for bifurcations}]\label{th:bif-existLagrBrakeM}
Let Assumptions~\ref{ass:Lagrbrake},\ref{ass:Lagr7brake} be satisfied,
and let $\Lambda$ be path-connected. Suppose that
  there exist two  points $\lambda^+, \lambda^-\in\Lambda$ such that
  one of the following conditions is satisfied:
 \begin{enumerate}
 \item[\rm (i)] Either $\gamma_{\lambda^+}$  is not an isolated critical point of
  $\mathcal{L}^E_{\lambda^+}$,
 or $\gamma_{\lambda^-}$ is not an isolated critical point of $\mathcal{L}^E_{\lambda^-}$,
 or $\gamma_{\lambda^+}$ {\rm (}resp. $\gamma_{\lambda^-}${\rm )} is an isolated critical point of
  $\mathcal{L}^E_{\lambda^+}$ {\rm (}resp. $\mathcal{L}^E_{\lambda^-}${\rm )} and
  $C_m(\mathcal{L}^E_{\lambda^+}, \gamma_{\lambda^+};{\bf K})$ and $C_m(\mathcal{L}^E_{\lambda^-}, \gamma_{\lambda^-};{\bf K})$
  are not isomorphic for some Abel group ${\bf K}$ and some $m\in\mathbb{Z}$.

\item[\rm (ii)] $[m^-(\mathcal{L}^E_{\lambda^-}, \gamma_{\lambda^-}),
m^-(\mathcal{L}^E_{\lambda^-}, \gamma_{\lambda^-})+ m^0(\mathcal{L}^E_{\lambda^-}, \gamma_{\lambda^-})]\cap[m^-(\mathcal{L}^E_{\lambda^+}, \gamma_{\lambda^+}),
m^-(\mathcal{L}^E_{\lambda^+}, \gamma_{\lambda^+})+m^0(\mathcal{L}^E_{\lambda^+}, \gamma_{\lambda^+})]\\=\emptyset$,
and there exists $\lambda\in\{\lambda^+, \lambda^-\}$ such that $\gamma_{\lambda}$  is an either non-isolated or homological visible critical point of $\mathcal{L}^E_{\lambda}$.

\item[\rm (iii)] $[m^-(\mathcal{L}^E_{\lambda^-}, \gamma_{\lambda^-}),
m^-(\mathcal{L}^E_{\lambda^-}, \gamma_{\lambda^-})+ m^0(\mathcal{L}^E_{\lambda^-}, \gamma_{\lambda^-})]\cap[m^-(\mathcal{L}^E_{\lambda^+}, \gamma_{\lambda^+}),
m^-(\mathcal{L}^E_{\lambda^+}, \gamma_{\lambda^+})+m^0(\mathcal{L}^E_{\lambda^+}, \gamma_{\lambda^+})]\\=\emptyset$,
and either $m^0(\mathcal{L}^E_{\lambda^+}, \gamma_{\lambda^+})=0$ or $m^0(\mathcal{L}^E_{\lambda^-}, \gamma_{\lambda^-})=0$.
 \end{enumerate}
  Then for any path $\alpha:[0,1]\to\Lambda$ connecting $\lambda^+$ to $\lambda^-$ there exists
 a sequence $(\lambda_k)\subset\alpha([0,1])$ converging to some $\mu\in\alpha([0,1])$, and
   solutions $\gamma^k\ne\gamma_{\lambda_k}$ of the problem (\ref{e:PPerLagrBrakeorbit})
  with $\lambda=\lambda_k$, $k=1,2,\cdots$,  such that $\|\gamma^k-\gamma_{\lambda_k}\|_{C^2(S_\tau;\mathbb{R}^N)}\to 0$
  as $k\to\infty$. {\rm (}In particular,  $(\mu, \lambda_\mu)$  is a  bifurcation point along sequences of the problem (\ref{e:PPerLagrBrakeorbit})
  in  $\Lambda\times EC^{1}(S_\tau; M)$  with respect to the branch $\{(\lambda, \gamma_\lambda)\,|\,\lambda\in\Lambda\}$.{\rm )}
Moreover, $\mu$ is not equal to $\lambda^+$ {\rm (}resp. $\lambda^-${\rm )} if $m^0_\tau(\mathcal{L}^E_{\lambda^+}, \gamma_{\lambda^+})=0$
{\rm (}resp. $m^0_\tau(\mathcal{L}^E_{\lambda^-}, \gamma_{\lambda^-})=0${\rm )}.
\end{theorem}

\begin{theorem}[\textsf{Alternative bifurcations of Rabinowitz's type}]\label{th:bif-suffLagrBrakeM}
Under Assumptions~\ref{ass:Lagrbrake},\ref{ass:Lagr7brake} with $\Lambda$ being a real interval,
let $\mu\in{\rm Int}(\Lambda)$ satisfy   $m^0_\tau(\mathcal{L}^E_{\mu}, \gamma_\mu)>0$.
 Suppose that $m^0_\tau(\mathcal{L}^E_{\lambda}, \gamma_\lambda)=0$  for each $\lambda\in\Lambda\setminus\{\mu\}$ near $\mu$, and
 that $m^-_\tau(\mathcal{L}^E_{\lambda}, \gamma_\lambda)$ take, respectively, values $m^-_\tau(\mathcal{L}^E_{\mu}, \gamma_\mu)$ and
  $m^-_\tau(\mathcal{L}^E_{\mu}, \gamma_\mu)+ m^0_\tau(\mathcal{L}^E_{\mu}, \gamma_\mu)$
 as $\lambda\in\Lambda$ varies in two deleted half neighborhoods  of $\mu$.
  Then  one of the following alternatives occurs:
\begin{enumerate}
\item[\rm (i)] The problem (\ref{e:PPerLagrBrakeorbit})
 with $\lambda=\mu$ has a sequence of solutions, $\gamma_k\ne \gamma_\mu$, $k=1,2,\cdots$,
which converges to $\gamma_\mu$ in $C^2(S_\tau, M)$.

\item[\rm (ii)]  For every $\lambda\in\Lambda\setminus\{\mu\}$ near $\mu$ there exists a  solution $\alpha_\lambda\ne \gamma_\lambda$ of
(\ref{e:PPerLagrBrakeorbit}) with parameter value $\lambda$, such that  $\alpha_\lambda-\gamma_\lambda$
 converges to zero in  $C^2(S_\tau, \R^N)$ as $\lambda\to \mu$.
{\rm (}Recall that $M\subset\R^N$.{\rm )}

\item[\rm (iii)] For a given neighborhood $\mathcal{W}$ of $\gamma_\mu$ in $C^2(S_\tau, M)$,
there exists a one-sided neighborhood $\Lambda^0$ of $\mu$ such that
for any $\lambda\in\Lambda^0\setminus\{\mu\}$, (\ref{e:PPerLagrBrakeorbit}) with parameter value $\lambda$
has at least two distinct solutions in $\mathcal{W}$, $\gamma_\lambda^1\ne \gamma_\lambda$ and $\gamma_\lambda^2\ne \gamma_\lambda$,
which can also be chosen to satisfy $\mathcal{L}^E_{\lambda}(\gamma_\lambda^1)\ne\mathcal{L}^E_{\lambda}(\gamma_\lambda^2)$
provided that $m^0_\tau(\mathcal{L}^E_{\mu}, \gamma_\mu)>1$ and (\ref{e:PPerLagrBrakeorbit}) with parameter value $\lambda$
has only finitely many solutions in $\mathcal{W}$.
\end{enumerate}
\end{theorem}

 As noted in Remark~\ref{rm:reduction},  when $M$ is an open subset in $\mathbb{R}^n$ the conditions in the above theorems may be weakened suitably.

\begin{remark}\label{rm:RiemannGeodesic}
{\rm
	Clearly, if the Lagrangian $L$ in Assumption~\ref{ass:Lagrbrake}
	comes from a family of $C^6$ Riemannian metrics $\{h_\lambda\,|\,\lambda\in\Lambda\}$
	on $M$, i.e., $L(\lambda,t, x,v)=(h_\lambda)_x(v,v)$ for all $(\lambda,t,x,v)$,
	as direct consequences of the above results we immediately obtain many bifurcation theorems of geodesics on Riemannian manifolds.
 See \cite[Part B]{Lu12}. 
 }
\end{remark}

\begin{remark}\label{rm:NewAdd1}
{\rm 
A natural question is how to verify the conditions required in Theorems~\ref{th:bif-suffLagrGener},~\ref{th:bif-suffLagr*},~\ref{th:bif-suffLagr**}, and~\ref{th:bif-suffLagrBrakeM}. We address this question in 
Section~\ref{sec:Cor-example} with Appendix~\ref{app:Abst}.
 }
\end{remark}

\subsection{Organization of the article}\label{sec:organ}

Section~\ref{sec:preliminaries} introduces the basic notation and presents some 
technical lemmas that will be used in the subsequent sections.

In Section~\ref{sec:LagrBound}, we discuss bifurcations of Lagrangian system trajectories connecting two submanifolds. By reducing the problem to Euclidean spaces, we prove the main results presented in Section~\ref{sec:submanifolds}, namely, Theorems~\ref{th:bif-nessLagrGener},~\ref{th:bif-existLagrGener},~\ref{th:bif-suffLagrGener}, and~\ref{th:MorseBif}. The proof of an important lemma, Lemma~\ref{lem:twoCont}, is deferred to Appendix~\ref{app:Exp}.

In Section~\ref{sec:LagrPPerio1}, we study 
bifurcations of Lagrangian trajectories with endpoints related by an isometry. 
 The main results of Section~\ref{sec:Graph}---Theorems~\ref{th:bif-nessLagr*},~\ref{th:bif-existLagr*} and \ref{th:bif-suffLagr*}---are proved by a reduction to Euclidean spaces.

Section~\ref{sec:Lagr7} deals with bifurcations of brake orbits in Lagrangian systems. Based on a reduction to Euclidean spaces (using Claim~\ref{cl:pallVEctorF}, to be proved in Appendix~\ref{app:Exp}), we prove the main results of Section~\ref{sec:Lagrbroke}: Theorems~\ref{th:bif-nessLagrBrakeM},~\ref{th:bif-existLagrBrakeM}, and~\ref{th:bif-suffLagrBrakeM}.

In Section~\ref{sec:Cor-example}, using the abstract results in Appendix~\ref{app:Abst}, we derive more readily applicable corollaries of the main results. We then apply them to bifurcation problems concerning physical Lagrangians with potential and electromagnetic forces, specifically presenting bifurcation results for the planar forced simple pendulum.

Appendix~\ref{app:Exp}  provides the proofs of Lemma~\ref{lem:twoCont} and Claim~\ref{cl:pallVEctorF}.

In Appendix~\ref{app:Abst}, we study the dependence on parameters of the Morse indices of a family of abstract functionals defined on a separable real Hilbert space at a common critical point.\\

\noindent{\bf Acknowledgements.}\; I would like to thank the anonymous referee for his/her many valuable comments and suggestions regarding the mathematical content and the writing, which have greatly improved the revised version.

\section{Preliminaries and notation} \label{sec:preliminaries} 
\setcounter{equation}{0}

This section collects the standard notations, function spaces, and basic technical lemmas used throughout the paper. It aims to make the paper self-contained and to clarify the meaning of symbols introduced in subsequent sections.

\subsection{Notation and Conventions}\label{sec:Notation}

\begin{itemize}
    \item Vectors in \(\mathbb{R}^m\) are understood as column vectors;
     \((\cdot, \cdot)_{\mathbb{R}^m}\) and \(|\cdot|\) denote, respectively, 
     the standard Euclidean inner product and norm in \(\mathbb{R}^m\).
            \item The transpose of a matrix \(M\) is denoted by \(M^\top\).
    \item For a Hilbert space $H$, $\mathscr{L}_s(H)$ denotes the space of continuous linear
    self-adjoint operators on $H$. We also use  \(\mathcal{L}_s(\mathbb{R}^m)\) 
    to denote the set of real symmetric matrices of order $m$).

\item For normed (real) linear spaces $X$ and $Y$, $\mathscr{L}(X,Y)$ is the space of
continuous linear mappings from $X$ to $Y$. Call $X^\ast:=\mathscr{L}(X,\mathbb{R})$ 
the dual (or adjoint) space of $X$.

\item For  a map $f$ from normed (real) linear spaces $X$ to $Y$,
let $Df(x)$  denote the G\^ateaux  derivative of $f$ at $x\in X$,
and let $df(x)$ (or $f'(x)$) denotes the  Fr\'echet derivative of $f$ at $x\in X$.
Both are elements of the space  $\mathscr{L}(X,Y)$. 
We may also use the notation  $f'(x)$ for the G\^ateaux derivative $Df(x)$
when no confusion arises.

\item For  a real functional $f$ on a Hilbert space $H$, the 
 the derivative $f'(x)$ of $f$ at $x\in H$ belongs to the dual space $H^\ast$. 
  By the Riesz representation theorem, $f'(x)$ corresponds to a unique 
  element in 
  $H$, which we call the gradient of $f$ at $x$, denoted by $\nabla f(x)$. 
 
  The Fr\'echet (or G\^ateaux) derivative of $\nabla f$ at $x\in H$
 is denoted by $f''(x)$. This is an element of $\mathscr{L}_s(H)$.
More precisely, $f''(x)$ is identified with the second derivative $(f')'(x)\in \mathscr{L}(H; H^\ast)$,
 which corresponds to a symmetric bilinear form on $H$ via the Riesz representation theorem
(identifying $H^\ast$ with $H$). This is equivalent to the derivative
 $D(\nabla f)(x)$. 

    \item $L^p([0, \tau]; \mathbb{R}^m)$ ($1\le p<\infty$) denotes the space of (equivalence classes of) $p$-integrable functions from $[0,\tau]$ to $\mathbb{R}^m$, equipped with the norm $\|u\|_p=(\int_0^\tau|u(t)|^p)^{1/p}$.
    In particular,   $L^2([0, \tau]; \mathbb{R}^m)$ 
     becomes a Hilbert space  equipped  with the inner product 
        \begin{equation}\label{e:innerP}
        (u,v)_{2}=\int_0^\tau(u(t),v(t))_{\mathbb{R}^{m}}dt.
        \end{equation}
                
    \item   $W^{1,p}([0, \tau]; \mathbb{R}^m)$ ($1\le p<\infty$)
    denotes the Sobolev space of absolutely continuous functions with $p$-integrable derivatives, 
    equipped with the norm 
    $$\|u\|_{1,p}=\left(\int_0^\tau|u(t)|^pdt+ \int_0^\tau|\dot{u}(t)|^pdt\right)^{1/p}.
    $$
    In particular, $W^{1,2}([0, \tau]; \mathbb{R}^m)$ becomes a Hilbert space   equipped with the inner product 
        \begin{equation}\label{e:innerP2}
          (u,v)_{1,2}=\int_0^\tau[(u(t),v(t))_{\mathbb{R}^{m}}+ (\dot{u},\dot{v})_{\mathbb{R}^{m}}]dt.
       \end{equation}
      \end{itemize}

\begin{itemize}
    \item \( C^k([0, \tau]; M) \): The Banach manifold of \( C^k \)-smooth curves from \([0, \tau]\) to \( M \), $k\in\mathbb{N}$. \( C^0([0, \tau]; M) \) denotes the topological space 
        of continuous curves from \([0, \tau]\) to \( M \). When $M\subset\mathbb{R}^N$,
        it has the induced subspace topology  from the Banach space \( C^0([0, \tau]; \mathbb{R}^N) \)
        with norm
        $\|u\|_{C^0}=\max\{|u(t)|_{\mathbb{R}^N}\mid t\in [0,\tau]\}$.

    \item \( C^1_{\mathbf{N}}([0, \tau]; M) \): The Banach submanifold of \( C^1([0, \tau]; M) \) defined by the boundary condition \( (\gamma(0), \gamma(\tau)) \in \mathbf{N} \), where \( \mathbf{N} \subset M \times M \) is a submanifold. See \eqref{e:BanachM}.

      \item $C_{\mathbb{I}_g}^1([0, \tau]; M)$ is defined as $C^1_{\mathbf{N}}([0, \tau]; M)$ 
      when $\mathbf{N}$ is the graph $\operatorname{Graph}(\mathbb{I}_g)$ of
      an isometry $\mathbb{I}_g$ on $(M, g)$. See (\ref{e:Rot-space}).

    \item \( \mathcal{X}^i_\tau(M, \mathbb{I}_g) \): The space of \( \mathbb{I}_g \)-periodic 
    $C^i$ curves (see (\ref{e:BanachPerMani})).

    \item \( EC^1(S_\tau; M) \): The space of 
   even and \( \tau \)-periodic $C^1$ curves (see (\ref{e:EBanachM})).
   
    \item \( W^{1,2}(\gamma^*TM) \): The Hilbert space of \( W^{1,2} \)-sections of the pullback bundle \(\gamma^*TM\), where $\gamma\in W^{1,2}([0,\tau]; M)$. It is equipped  with inner product
  given by (\ref{e:1.1}). When $\gamma$ belongs to 
       $W^{1,2}_{S_0\times S_1}([0,\tau]; M)$ (resp. $W^{1,2}_{\mathbb{I}_g}([0, \tau]; M)$),
     the space  \( W^{1,2}(\gamma^*TM) \) contains 
      $W^{1,2}_{S_0\times S_1}(\gamma^\ast TM)$ (resp.
      $W^{1,2}_{\mathbb{I}_g}(\gamma^\ast TM)$. These subspaces are defined in
           (\ref{e:LagrGener0})) and  (\ref{e:Rot-spaceB}), respectively. 

\item $C^1(\gamma^\ast TM)$: The Banach space of  $C^1$-sections of the pullback bundle \(\gamma^*TM\), where $\gamma\in C^1([0,\tau]; M)$. See (\ref{e:LagrGener000}).
 When $\gamma$ belongs to  $C^{1}_{S_0\times S_1}([0,\tau]; M)$ (resp. $C^{1}_{\mathbb{I}_g}([0, \tau]; M)$),
 the space $C^1(\gamma^\ast TM)$ contains the subspace $C^{1}_{S_0\times S_1}(\gamma^\ast TM)$ 
 (resp. $C^1_{\mathbb{I}_g}(\gamma^\ast TM)$). These subspaces are defined in
 (\ref{e:LagrGener00+}) and (\ref{e:Rot-spaceA}), respectively. 
\end{itemize}
\begin{itemize}
     \item \(\mathcal{E}_\lambda\): 
       The action functional associated with a parameterized Lagrangian function
       $L : \Lambda \times [0, \tau] \times TM \to \mathbb{R}$, 
    on $C^1_{\bf N}([0,\tau];M)$, specifically on the subspace 
     $C^1_{S_0\times S_1}([0,\tau];M)$. See (\ref{e:LagrGener-}) and (\ref{e:EnergyFunc}).
          
     \item \(\mathcal{L}_{S_{0},s}\): The action functional 
     associated with a Lagrangian function \( L :[0, \tau] \times TM \to \mathbb{R} \),
     defined on $C^1_{S_0\times \{\gamma(s)\}}([0,\tau];M)$
     for a given $C^2$ curve $\gamma:[0, \lambda]\to M$ with $\lambda\in (0,\tau]$ that solves (\ref{e:LagrCurve}), where $s\in (0,\lambda]$. See (\ref{e:LagrFinEnergyPQ0})
     for the explicit expression.    
      
    \item \(\mathscr{E}_\lambda\): 
       The action functional on $C^{1}_{\mathbb{I}_g}([0, \tau]; M)$
       associated with a parameterized Lagrangian function
       $L : \Lambda \times [0, \tau] \times TM \to \mathbb{R}$.
          See (\ref{e:Lagr0}).

    \item \(\mathfrak{E}_\lambda\): The action functional 
    on $\mathcal{X}^1_\tau(M, \mathbb{I}_g)$  associated with a parameterized Lagrangian function
       $L : \Lambda \times \mathbb{R} \times TM \to \mathbb{R}$ satisfying (\ref{e:Lagr8}).             
        See (\ref{e:Lagr9}).

    \item \(\mathcal{L}^E_\lambda\): The action functional on 
     $EC^1(S_\tau; M)$  associated with a parameterized Lagrangian function
       $L : \Lambda \times \mathbb{R} \times TM \to \mathbb{R}$ satisfying (\ref{e:brakeLM}).   
       See (\ref{e:BrakeFunct}).

\item $m^-(F, p)$ (resp.\ $m^0(F, p)$): The Morse index (resp.\ nullity) of a critical point $p$ of a functional $F$, where $F$ denotes one of the functionals $\mathcal{E}_\lambda$, $\mathcal{L}_{S_{0},s}$, $\mathscr{E}_\lambda$, $\mathfrak{E}_\lambda$, and $\mathcal{L}^E_\lambda$.
    See (\ref{e:MorseIndexTwo}), (\ref{e:NullityFocus}),
    (\ref{e:Lagr3}), (\ref{e:Lagr12}), and (\ref{e:BrakeLagrMorse}).        
\end{itemize}

\subsection{Technical lemmas}\label{sec:Tech}

For $1\le p<\infty$, we adopt the standard convention of identifying 
a function $u\in L^{p}([0,\tau];\R^{n})$ with a chosen representative of their equivalence class.
Consequently, we  identify every $u\in W^{1,p}([0,\tau];\mathbb{R}^{n})$ with its (unique) continuous representative. Therefore, we have
 $W^{1,p}([0,\tau];\mathbb{R}^{n})\hookrightarrow C^0([0,\tau];\mathbb{R}^{n})$ and
\begin{equation}\label{e:C-L}
\left.\begin{array}{ll}
\|u\|_{C^0}\le (\tau+1)\|u\|_{1,1},\quad\forall u\in W^{1,1}([0,\tau];\mathbb{R}^{n}),\\
\|u\|_{C^0}\le (\sqrt{\tau}+1/\sqrt{\tau})\|u\|_{1,2},\quad\forall u\in W^{1,2}([0,\tau];\mathbb{R}^{n}).
\end{array}\right\}
\end{equation}

 \begin{lemma}[\hbox{\cite[Lemma~2.1]{Lu4}}]\label{lem:Fin2.1}
Given positive numbers $c>0$ and $C_1\ge 1$, choose positive
parameters $0<\varepsilon<\delta<\frac{2c}{3C_1}$. Then:
\begin{enumerate}
\item[\rm (i)]  There exists
a $C^\infty$ function $\psi_{\varepsilon,\delta}:[0, \infty)\to\R$
such that: $\psi'_{\varepsilon,\delta}>0$ and $\psi_{\varepsilon,\delta}$ is convex on
$(\varepsilon, \infty)$, $\psi_{\varepsilon,\delta}$ vanishes in $[0, \varepsilon)$ and is equal to
the affine function $\kappa t+ \varrho_0$ on $[\delta, \infty)$,
where $\kappa>0$ and $\varrho_0<0$ are suitable constants.
\item[\rm (ii)] There exists a $C^\infty$ function
$\phi_{\mu,b}:[0, \infty)\to\R$ depending on parameters $\mu>0$ and
$b>0$, such that: $\phi_{\mu,b}$ is nondecreasing and concave (and
hence $\phi_{\mu,b}^{\prime\prime}\le 0$ ), and equal to the affine
function $\mu t-\mu\delta$ on $[0, \delta]$, and equal to constant
$b>0$ on $[\frac{2c}{3C_1}, \infty)$.
\item[\rm (iii)] Under the above assumptions, $\psi_{\varepsilon,\delta}(t)+
\phi_{\mu,b}(t)-b=\kappa t+ \varrho_0$ for any $t\ge
\frac{2c}{3C_1}$ (and hence for $t\ge\frac{2c}{3}$). Moreover,
$\psi_{\varepsilon,\delta}(t)+ \phi_{\mu,b}(t)-b\ge
-\mu\delta-b$ for any $t\ge 0$, and $\psi_{\varepsilon,\delta}(t)+
\phi_{\mu,b}(t)-b= -\mu\delta-b$ if and only if $t=0$.

\item[\rm (iv)] Under the assumptions (i)-(ii), suppose that the constant $\mu>0$ satisfies
\begin{eqnarray}\label{e:Fin1.0}
\mu+\frac{\varrho_0}{\delta-\varepsilon}>0\quad\hbox{and}\quad\mu\delta+b+\varrho_0>0.
\end{eqnarray}
Then $\psi_{\varepsilon,\delta}(t)+ \phi_{\mu,b}(t)-b\le\kappa t+
\varrho_0$ for all $t\ge\varepsilon$, and
$\psi_{\varepsilon,\delta}(t)+ \phi_{\mu,b}(t)-b\le\kappa t+
\varrho_0$ for all $t\in [0,\varepsilon]$ if $\kappa\ge\mu$.
\end{enumerate}
\end{lemma}

  \begin{assumption}\label{ass:Lagr0}
{\rm  For a real $\tau>0$, a topology space $\Lambda$, 
let $L:\Lambda\times [0,\tau]\times U\times \R^n\to\R$ be a continuous function such that
 the following  partial derivatives
$$
\partial_tL(\cdot),\;
\partial_qL(\cdot),\;\partial_vL(\cdot),
\;\partial_{tv}L(\cdot),\;\partial_{vt}L(\cdot), \;\partial_{qv}L(\cdot),\;\partial_{vq}L(\cdot),\;
\partial_{qq}L(\cdot),\;\partial_{vv}L(\cdot)
$$
exist and depend continuously on $(\lambda, t, q, v)\in\Lambda\times [0, \tau]\times U\times \R^n$.
 Moreover, $\Lambda\times [0,\tau]\times U\times\R^n\ni (\lambda, t, q, v)\mapsto {L}(\lambda, t, q, v)$ is
convex with respect to $v$, that is, the second partial derivative $\partial_{vv}{L}(\lambda, t, q,v)$ is positive semi-definite
as a quadratic form. }
\end{assumption}

Under Assumption~\ref{ass:Lagr0} let $E$ be a real orthogonal matrix of order $n$ such that $(EU)\cap U\ne\emptyset$.
Consider  the Lagrangian boundary value problem on $U$:
\begin{eqnarray}\label{e:L1}
&\frac{d}{dt}\left(\partial_v
L_\lambda(t, x(t), \dot{x}(t))\right)-\partial_q L_\lambda(t, x(t), \dot{x}(t))=0,\\
&E(x(0))=x(\tau)\quad\hbox{and}\quad
(E^\top)^{-1}\left[\partial_vL_\lambda(0, x(0), \dot{x}(0))\right]=\partial_vL_\lambda(\tau, x(\tau), \dot{x}(\tau)).\label{e:L2}
\end{eqnarray}

\begin{assumption}\label{ass:BasiAssLagr}
{\rm For a real $\tau>0$,  a topological space $\Lambda$,
a real orthogonal matrix $E$ of order $n$, and an $E$-invariant path-connected
 open subset $U\subset \mathbb{R}^n$ let $L:\Lambda\times \mathbb{R}\times U\times \mathbb{R}^n\to\R$ be a continuous function such that
  the following  partial derivatives
$$
\partial_tL(\cdot),\;
\partial_qL(\cdot),\;\partial_vL(\cdot),
\;\partial_{tv}L(\cdot),\;\partial_{vt}L(\cdot), \;\partial_{qv}L(\cdot),\;\partial_{vq}L(\cdot),\;
\partial_{qq}L(\cdot),\;\partial_{vv}L(\cdot)
$$
exist and depend continuously on $(\lambda, t, q, v)\in\Lambda\times [0, \tau]\times U\times \R^n$.
 Moreover, for each $(\lambda, t, q)\in \Lambda\times \R\times U$, ${L}(\lambda, t, q, v)$ is
convex in $v$,    and satisfies
\begin{eqnarray}\label{e:M-invariant1Lagr}
L(\lambda, t+\tau, Eq, Ev)=L(\lambda, t, q, v)\quad\forall (\lambda, t,q,v)\in\Lambda\times\R\times U\times\mathbb{R}^{n}.
\end{eqnarray}}
\end{assumption}

\begin{lemma}\label{lem:Lagr}
 Under Assumption~\ref{ass:Lagr0}, let the topological space $\Lambda$ be
 either {\rm compact} or {\rm sequentially compact}. Suppose  that for some real $\rho>0$
the function $B^n_{\rho}(0)\ni v\mapsto {L}_\lambda(t, q, v)$ is strictly convex
 for  each $(\lambda, t, q)\in \Lambda\times [0,\tau]\times U$.
 Then  for any given  real  $0<\rho_0<\rho$  there exists a  continuous function
 $\tilde{L}:\Lambda\times [0,\tau]\times U\times\mathbb{R}^n\to\R$ and a constant $\kappa>0$ satisfying the following properties:
 \begin{enumerate}
  \item[\rm (i)] $\tilde{L}$  is equal to $L$ on $\Lambda\times [0,\tau]\times U\times B^n_{\rho_0}(0)$.
  \item[\rm (ii)] The  partial derivatives
$$
\partial_t\tilde L(\cdot),\;
\partial_q\tilde L(\cdot),\;\partial_v\tilde L(\cdot),
\;\partial_{tv}\tilde L(\cdot),\;\partial_{vt}\tilde L(\cdot), \;\partial_{qv}\tilde L(\cdot),\;\partial_{vq}\tilde L(\cdot),\;
\partial_{qq}\tilde L(\cdot),\;\partial_{vv}\tilde L(\cdot)
$$
exist and depend continuously on $(\lambda, t, q, v)$.

   \item[\rm (iii)]  $\mathbb{R}^n\ni v\mapsto \tilde{L}_\lambda(t, q, v)$ 
   is  strictly convex for each $(\lambda, t, q)\in \Lambda\times [0,\tau]\times U$,
   that is, the second partial derivative $\partial_{vv}\tilde{L}_\lambda(t,x,v)$ 
   is positive definite as a quadratic form.
 \item[\rm (iv)] For any given compact subset  $S\subset U$,
 there exists a constant $C>0$ such that
    $$
    \tilde{L}_\lambda(t, q, v)\ge \kappa|v|^2-C,\quad\forall (\lambda, t, q, v)\in \Lambda\times [0,\tau]\times S\times\mathbb{R}^n.
    $$

 \item[\rm (v)] For each $(\lambda,t, q)$, if $L(\lambda, t, q, v)$ is even in $v$ then 
  $\tilde{L}_\lambda(t, q,v)$ can be required to be even in $v$.
  \item[\rm (vi)] If $U$ is a symmetric open neighborhood of the origin in $\mathbb{R}^n$,  and for each $(\lambda,t)$
 the function $L(\lambda, t, q, v)$ is even in $(q,v)$ then  $\tilde{L}(\lambda,t, q,v)$ can be also required to be even in $(q,v)$.
\item[\rm (vii)] For each $(\lambda, q)$, if $L(\lambda, t, q, v)$ is even in $(t,v)$, then $\tilde{L}(\lambda, t, q, v)$ can be chosen to be even in $(t,v)$.
\item[\rm (viii)] If $L$ is independent of time $t$, so is  $\tilde{L}$.
\item[\rm (ix)] If Assumption~\ref{ass:Lagr0} is replaced by Assumption~\ref{ass:BasiAssLagr},
the function $\tilde{L}$ given by (\ref{e:tildeL}) may be replaced by
  \begin{equation*}
\tilde{L}:\Lambda\times \R\times U\times\mathbb{R}^n\to\R,\;(\lambda,t,q,v)\mapsto {L}(\lambda,t,q,v)+ \psi_{\rho_0,\rho_1}(|v|^2),
 \end{equation*}
 which also satisfies (\ref{e:M-invariant1Lagr}) because $E$ is
  a real orthogonal matrix.
 \end{enumerate}
 \end{lemma}
\begin{proof}[\bf Proof]
Fix a positive real $\rho_1\in (\rho_0, \rho)$. By Lemma~\ref{lem:Fin2.1} (\cite[Lemma~2.1]{Lu4})
we have a $C^\infty$ convex function $\psi_{\rho_0,\rho_1}:[0, \infty)\to\R$
such that $\psi'_{\rho_0,\rho_1}(t)>0$ for
$t\in(\rho_0^2, \infty)$, $\psi_{\rho_0,\rho_1}(t)=0$ for $t\in[0, \rho_0^2)$ and
 $\psi_{\rho_0,\rho_1}(t)=\kappa t+ \varrho_0$ for $t\in [\rho_1^2, \infty)$,
where $\kappa>0$ and $\varrho_0<0$ are suitable constants. We conclude
  \begin{equation}\label{e:tildeL}
\tilde{L}:\Lambda\times [0,\tau]\times U\times\mathbb{R}^n\to\R,\;(\lambda,t,q,v)\mapsto {L}(\lambda,t,q,v)+ \psi_{\rho_0,\rho_1}(|v|^2)
 \end{equation}
 to satisfy the desired requirements. By Assumption~\ref{ass:Lagr0} and the choice of $\psi_{\rho_0,\rho_1}$
 it is clear that $\tilde{L}$ satisfies (i)-(ii).

 In order to see that $\tilde{L}$ satisfies (iii), note that
 \begin{eqnarray*}
\frac{\partial^2}{\partial t\partial
s}\psi_{\rho_0,\rho_1}(|v+su+ tu|^2)\Bigm|_{s=0,t=0}
=2\psi'_{\rho_0,\rho_1}(|v|^2)|u|^2
+4\psi''_{\rho_0,\rho_1}(|v|^2)\bigl(({v},{u})_{\mathbb{R}^{n}}\bigr)^2
\end{eqnarray*}
(cf. the proof of \cite[Lemma~2.1]{Lu4}) and therefore
 $$
\partial_{vv}\tilde{L}(\lambda,t,q,v)[u,u]=\partial_{vv}{L}(\lambda,t,q,v)[u,u]+ 2\psi'_{\rho_0,\rho_1}(|v|^2)|u|^2
+4\psi''_{\rho_0,\rho_1}(|v|^2)\bigl(({v},{u})_{\mathbb{R}^{n}}\bigr)^2.
 $$
 Since $\psi'_{\rho_0,\rho_1}\ge 0$ and $\psi''_{\rho_0,\rho_1}(t)\ge 0$,
 by Assumption~\ref{ass:Lagr0} we deduce
 $$
 \partial_{vv}\tilde{L}(\lambda,t,q,v)[u,u]\ge\partial_{vv}{L}(\lambda,t,q,v)[u,u]>0\quad\hbox{for $|v|<\rho$ and $u\ne 0$}.
 $$
 Moreover, if $|v|>\rho_1$ and $u\ne 0$ we get $\partial_{vv}\tilde{L}(\lambda,t,q,v)[u,u]\ge 2\kappa|u|^2$
 since $\partial_{vv}{L}(\lambda,t,q,v)[u,u]\ge 0$.
 Hence $\tilde{L}(\lambda,t,q,v)$ is strictly convex in $v$.

Let us prove (iv).
 Fixing  $v_0\in B^n_{\rho}(0)\setminus B^n_{\rho_1}(0)$,
 by \cite[Proposition~1.2.10]{Fa} we get
\begin{eqnarray*}
\tilde{L}(\lambda,t,q,v)&\ge& \tilde{L}(\lambda,t,q,v_0)+ \partial_v\tilde{L}(\lambda,t,q,v_0)[v-v_0]\\
&\ge& {L}(\lambda,t,q,v_0)+ \partial_v{L}(\lambda,t,q,v_0)[v-v_0]+
2\psi'_{\rho_0,\rho_1}(|v_0|^2)({v},v-v_0)_{\mathbb{R}^{n}}\\
=&&\hspace{-9mm}{L}(\lambda,t,q,v_0)-\partial_v{L}(\lambda,t,q,v_0)[v_0]+\partial_v{L}(\lambda,t,q,v_0)[v]
-2\kappa({v},{v}_0)_{\mathbb{R}^{n}}+2\kappa|v|^2
\end{eqnarray*}
for all $v\in\mathbb{R}^n$. Since $2\kappa|({v},{v}_0)_{\mathbb{R}^{n}}|\le \kappa|v|^2/2+ 2\kappa|v_0|^2$ and
$$
|\partial_v{L}(\lambda,t,q,v_0)[v]|=\bigg|\sum^n_{j=1}\frac{\partial L}{\partial v_j}(\lambda,t,q,v_0)v_j\bigg|\le
\frac{1}{\kappa}\sum^n_{j=1}\Big|\frac{\partial L}{\partial v_j}(\lambda,t,q,v_0)\Big|^2+
 \frac{\kappa}{4}|v|^2,
$$
we derive
\begin{eqnarray*}
\tilde{L}(\lambda,t,q,v)\ge {L}(\lambda,t,q,v_0)-\partial_v{L}(\lambda,t,q,v_0)[v_0]
-\frac{1}{\kappa}\sum^n_{j=1}\Big|\frac{\partial L}{\partial v_j}(\lambda,t,q,v_0)\Big|^2
-2\kappa|v_0|^2
+\frac{5\kappa}{4}|v|^2
\end{eqnarray*}
for all $v\in\mathbb{R}^n$. Since $\Lambda$ is either compact or sequential compact, so is $\Lambda\times [0,\tau]\times S$,
in either case we can always derive that both $L(\lambda,t,q,v_0)$ and $\partial_vL(\lambda,t,q,v_0)$ are bounded on
 $\Lambda\times [0,\tau]\times S$. Therefore there exists a constant $C>0$ such that
 $$
 \tilde{L}(\lambda,t,q,v)\ge \frac{5\kappa}{4}|v|^2-C,\quad\forall (\lambda,t,q,v)\in\Lambda\times [0,\tau]\times S\times\R^n.
 $$
Other conclusions are clear by the above construction.
\end{proof}

\begin{assumption}\label{ass:Lagr1}
{\rm Under Assumption~\ref{ass:Lagr0}, for each $\lambda\in\Lambda$, let $x_\lambda:[0, \tau]\to U$
be a  $C^2$ path  satisfying (\ref{e:L1}).  Suppose: (i) $\Lambda\times [0, \tau]\ni(\lambda,t)\to x_\lambda(t)\in U$ and $\Lambda\times [0, \tau]\ni(\lambda, t)\mapsto \dot{x}_\lambda(t)\in\R^n$ are continuous;
(ii) for any compact or sequential compact subset $\hat\Lambda\subset\Lambda$
 there exists $\rho>0$ such that
 $$
 \sup\big\{|\dot{x}_\lambda(t)|\;\big|\, (\lambda,t)\in\hat\Lambda\times [0,\tau]\big\}<\rho
 $$
 and that $\hat\Lambda\times [0,\tau]\times U\times B^n_{\rho}(0)\ni (\lambda, t, q,v)\mapsto {L}_\lambda(t, q, v)$ is
strictly convex with respect to $v$. }
\end{assumption}

\begin{lemma}\label{lem:regu}
Under Assumption~\ref{ass:Lagr1}, the following holds:
\begin{enumerate}
   \item[\rm (i)] $\Lambda\times [0, \tau]\ni(\lambda, t)\mapsto \ddot{x}_\lambda(t)\in\R^n$ is  continuous.
\item[\rm (ii)] If there exists a sequence $(\lambda_k)\subset\Lambda$ converging to $\mu\in\Lambda$
 and solutions $x_k\in C^2([0,\tau], U)\setminus\{x_{\lambda_k}\}$  of (\ref{e:L1}) with $\lambda=\lambda_k\in\Lambda$, $k=1,2,\cdots$, such that
$\|x_k-x_{\lambda_k}\|_{C^1}\to 0$, then $\|x_k-x_{\lambda_k}\|_{C^2}\to 0$ as $k\to\infty$.
\end{enumerate}
\end{lemma}

\begin{proof}[\bf Proof]
{\it Step~1}[\textsf{Prove (i)}]. Since $x_\lambda$ is $C^2$, we have
\begin{eqnarray*}
&&\partial_{vt}L_\lambda(t, x_\lambda(t), \dot{x}_\lambda(t))+
\partial_{vq}L_\lambda(t, x_\lambda(t), \dot{x}_\lambda(t))\dot{x}_\lambda(t)\\
&&+\partial_{vv}L_\lambda(t, x_\lambda(t), \dot{x}_\lambda(t))\ddot{x}_\lambda(t)
-\partial_q L_\lambda(t, x_\lambda(t), \dot{x}_\lambda(t))=0
\end{eqnarray*}
and therefore
\begin{eqnarray*}
\ddot{x}_\lambda(t)&=&[\partial_{vv}L_\lambda(t, x_\lambda(t), \dot{x}_\lambda(t))]^{-1}\partial_q L_\lambda(t, x_\lambda(t), \dot{x}_\lambda(t))\\
&&-[\partial_{vv}L_\lambda(t, x_\lambda(t), \dot{x}_\lambda(t))]^{-1}\partial_{vt}L_\lambda(t, x_\lambda(t), \dot{x}_\lambda(t))\nonumber\\
&&-[\partial_{vv}L_\lambda(t, x_\lambda(t), \dot{x}_\lambda(t))]^{-1}\partial_{vq}L_\lambda(t, x_\lambda(t), \dot{x}_\lambda(t))\dot{x}_\lambda(t)
\end{eqnarray*}
because  (ii) in Assumption~\ref{ass:Lagr1} implies that the matrixes $\partial_{vv}L_\lambda(t, x(t), \dot{x}(t))$ are positive definite and
therefore invertible.
 Moreover, by Assumption~\ref{ass:Lagr0} maps
\begin{eqnarray*}
&&(\lambda,t,q,v)\mapsto\partial_{vt}L_\lambda(t,q,v),\qquad  (\lambda,t,q,v)\mapsto\partial_{vq}L_\lambda(t,q,v),\\
&&(\lambda,t,q,v)\mapsto\partial_{vv}L_\lambda(t,q,v), \qquad (\lambda,t,q,v)\mapsto\partial_q L_\lambda(t,q,v)
\end{eqnarray*}
are continuous. The desired conclusion may follow from these,  (i) in Assumption~\ref{ass:Lagr1} and the above equality directly.

 {\it Step~2}[\textsf{Prove (ii)}]. Let $\hat\Lambda=\{\mu, \lambda_k\,|\,k\in\mathbb{N}\}$. It is a
 sequential compact subset of $\Lambda$. By the assumption (ii) in Assumption~\ref{ass:Lagr1} there exists $\rho>0$ such that
$\sup\{|\dot{x}_\lambda(t)|\,|\, (\lambda,t)\in\hat\Lambda\times [0,\tau]\}<\rho$
and that $\hat\Lambda\times [0,\tau]\times U\times B^n_{\rho}(0)\ni (\lambda, t, q,v)\mapsto {L}_\lambda(t, q, v)$ is
strictly convex with respect to $v$.
In particular,  we may obtain $0<M_1<M_2<\infty$ such that
$$
M_1I_n\le\partial_{vv}{L}_{\mu}(t, x_\mu(t), \dot{x}_\mu(t))\le M_2I_n,\quad\forall t\in [0,\tau].
$$
Suppose that there exists a sequence $(t_i)\subset [0,\tau]$ such that for each $i=1,2,\cdots$,
$$
\partial_{vv}{L}_{\lambda_{k_i}}(t_i, x_{\lambda_{k_i}}(t_i), \dot{x}_{\lambda_{k_i}}(t_i))<\frac{1}{2}M_1I_n
\quad\hbox{or}\quad \partial_{vv}{L}_{\lambda_{k_i}}(t_i, x_{\lambda_{k_i}}(t_i), \dot{x}_{\lambda_{k_i}}(t_i))>\frac{1}{2}M_2I_n.
$$
We can assume $t_i\to t_0\in [0,\tau]$. By  (i) in Assumption~\ref{ass:Lagr1}
and the continuity of the map
$$
\Lambda\times [0,\tau]\times U\times\R^n\ni (\lambda, t, q, v)\mapsto \partial_{vv}{L}(\lambda, t, q, v)\in\mathbb{R}^{n\times n},
$$
we derive
$$
\partial_{vv}{L}_{\mu}(t_0, x_{\mu}(t_0), \dot{x}_{\mu}(t_0))\le\frac{1}{2}M_1I_n
\quad\hbox{or}\quad \partial_{vv}{L}_{\mu}(t_0, x_{\mu}(t_0), \dot{x}_{\mu}(t_0))\ge\frac{1}{2}M_2I_n.
$$
This contradiction shows that if $k$ is large enough then
\begin{eqnarray}\label{e:BifOA}
\frac{1}{2}M_1I_n\le\partial_{vv}{L}_{\lambda_{k}}(t, x_{\lambda_{k}}(t), \dot{x}_{\lambda_{k}}(t))
\le\frac{1}{2}M_2I_n,\quad\forall t\in [0,\tau].
\end{eqnarray}
Similarly, since $\|x_k-x_{\lambda_k}\|_{C^1}\to 0$, for sufficiently large $k$ we have
\begin{eqnarray}\label{e:BifOB}
\frac{1}{2}M_1I_n\le\partial_{vv}{L}_{\lambda_{k}}(t, x_{k}(t), \dot{x}_{k}(t))
\le\frac{1}{2}M_2I_n,\quad\forall t\in [0,\tau].
\end{eqnarray}
Note that for each $k\in\mathbb{N}$, both $(\lambda_k, x_k)$ and $(\lambda_k, x_{\lambda_k})$ satisfy
(\ref{e:L1}), that is,
\begin{eqnarray}
&&\partial_{vv}L_{\lambda_k}(t, x_{\lambda_k}(t), \dot{x}_{\lambda_k}(t))\ddot{x}_{\lambda_k}(t)+
\partial_{tv}L_{\lambda_k}(t, x_{\lambda_k}(t), \dot{x}_{\lambda_k}(t))\nonumber\\
&&+\partial_{qv}L_{\lambda_k}(t, x_{\lambda_k}(t), \dot{x}_{\lambda_k}(t))\dot{x}_{\lambda_k}(t)
-\partial_q L_{\lambda_k}(t, x_{\lambda_k}(t), \dot{x}_{\lambda_k}(t))=0,\label{e:BifOC}\\
&&\partial_{vv}L_{\lambda_k}(t, x_k(t), \dot{x}_k(t))\ddot{x}_k(t)+
\partial_{tv}L_{\lambda_k}(t, x_k(t), \dot{x}_k(t))\nonumber\\
&&+\partial_{qv}L_{\lambda_k}(t, x_{k}(t), \dot{x}_{k}(t))\dot{x}_{k}(t)
-\partial_q L_{\lambda_k}(t, x_{k}(t), \dot{x}_{k}(t))=0.\label{e:BifOD}
\end{eqnarray}
By contradiction, passing to subsequences (if necessary) suppose that there exists $\varepsilon>0$ and a sequence $(t_k)\subset [0,\tau]$ converging to $t_0$
such that $|\ddot{x}_{\lambda_k}(t_k)-\ddot{x}_{k}(t_k)|\ge\varepsilon$ for all $k\in\mathbb{N}$. Then
for each large $k$, (\ref{e:BifOA}) and (\ref{e:BifOB}) imply that matrixes
$$\left(\partial_{vv}L_{\lambda_k}(t_k, x_{\lambda_k}(t_k), \dot{x}_{\lambda_k}(t_k))\right)\quad
\hbox{and}\quad \left(\partial_{vv}L_{\lambda_k}(t_k, x_{\lambda_k}(t_k), \dot{x}_{\lambda_k}(t_k))\right)
$$
are invertible.  It follows from (\ref{e:BifOC}) and (\ref{e:BifOD}) that for each large $k$,
\begin{eqnarray}\label{e:BifO}
\varepsilon&\le&|\ddot{x}_{\lambda_k}(t_k)-\ddot{x}_{k}(t_k)|\nonumber\\
&=&\bigl|\left(\partial_{vv}L_{\lambda_k}(t_k, x_{\lambda_k}(t_k), \dot{x}_{\lambda_k}(t_k))\right)^{-1}
\partial_{tv}L_{\lambda_k}(t_k, x_{\lambda_k}(t_k), \dot{x}_{\lambda_k}(t_k))\nonumber\\
&&-\left(\partial_{vv}L_{\lambda_k}(t_k, x_k(t_k), \dot{x}_k(t_k))\right)^{-1}
\partial_{tv}L_{\lambda_k}(t_k, x_k(t_k), \dot{x}_k(t_k))\nonumber\\
&&+\left(\partial_{vv}L_{\lambda_k}(t_k, x_{\lambda_k}(t_k), \dot{x}_{\lambda_k}(t_k))\right)^{-1}
\partial_{qv}L_{\lambda_k}(t_k, x_{\lambda_k}(t_k), \dot{x}_{\lambda_k}(t_k))\dot{x}_{\lambda_k}(t_k)\nonumber\\
&&-\left(\partial_{vv}L_{\lambda_k}(t_k, x_k(t_k), \dot{x}_k(t_k))\right)^{-1}
\partial_{qv}L_{\lambda_k}(t_k, x_{k}(t_k), \dot{x}_{k}(t_k))\dot{x}_{k}(t_k)\nonumber\\
&&-\left(\partial_{vv}L_{\lambda_k}(t_k, x_{\lambda_k}(t_k), \dot{x}_{\lambda_k}(t_k))\right)^{-1}
\partial_q L_{\lambda_k}(t_k, x_{\lambda_k}(t_k), \dot{x}_{\lambda_k}(t_k))\nonumber\\
&&+\left(\partial_{vv}L_{\lambda_k}(t_k, x_k(t_k), \dot{x}_k(t_k))\right)^{-1}\partial_q L_{\lambda_k}(t_k, x_{k}(t_k), \dot{x}_{k}(t_k))\bigr|.
\end{eqnarray}
By (i) in Assumption~\ref{ass:Lagr1}, $|x_{\lambda_k}(t_k)-x_\mu(t_0)|\to 0$ and $|\dot{x}_{\lambda_k}(t_k)-\dot{x}_{\mu}(t_0)|\to 0$. Moreover
\begin{eqnarray*}
 &&|x_{k}(t_k)-x_\mu(t_0)|\le |x_{\lambda_k}(t_k)-x_\mu(t_0)|+\|x_{\lambda_k}-x_k\|_{C^0}\to 0,\\
&&|\dot{x}_{k}(t_k)-\dot{x}_\mu(t_0)|\le |\dot{x}_{\lambda_k}(t_k)-\dot{x}_\mu(t_0)|+\|\dot{x}_{\lambda_k}-\dot{x}_k\|_{C^0}\to 0.
\end{eqnarray*}
Letting $k\to\infty$ in (\ref{e:BifO}), by Assumption~\ref{ass:Lagr0} we get $\varepsilon\le 0$.
This contradiction shows $\|\ddot{x}_{\lambda_k}-\ddot{x}_{k}\|_{C^0}\to 0$.
Combing the condition $\|x_k-x_{\lambda_k}\|_{C^1}\to 0$, we arrive at
 $\|x_k-x_{\lambda_k}\|_{C^2}\to 0$ as $k\to\infty$.
\end{proof}

\textsf{{Note}}: the continuity of $\partial_{tv}L$ is used in the proof of Lemma~\ref{lem:regu}.

Under Assumption~\ref{ass:Lagr1}, for a given compact or sequential compact subset $\hat\Lambda\subset\Lambda$
  there exist positive numbers $0<\rho_0<\rho$ such that
  $$
  \rho_{00}:= \sup\big\{|\dot{x}_\lambda(t)|\;\big|\, 
  (\lambda,t)\in\hat\Lambda\times [0,\tau]\big\}<\rho_0<\rho
  $$
  and that 
  $\hat\Lambda\times [0,\tau]\times U\times B^n_{\rho}(0)\ni (\lambda, t, q,v)\mapsto {L}_\lambda(t, q, v)$ 
  is strictly convex with respect to $v$.
By Lemma~\ref{lem:Lagr} we have an associated  continuous function
 $\tilde{L}:\hat\Lambda\times [0,\tau]\times U\times\mathbb{R}^n\to\R$.
Since a subset of $\mathbb{R}^n$ is compact if and only if it is sequential compact,
whether $\hat\Lambda$ is compact or sequential compact
the continuous map $\hat\Lambda\times [0, \tau]\ni (\lambda,t)\mapsto x_\lambda(t)\in\mathbb{R}^n$
has a compact image set and therefore
we can choose $\delta>0$ so small that the compact subset
$$
S:=Cl\bigl(\cup_{\lambda\in\hat\Lambda}\cup_{t\in [0,\tau]}(x_\lambda(t)+B^n_{\delta}(0))\bigr)
$$
is contained in $U$. 
Lemma~\ref{lem:Lagr}(iv) yields a constant $C>0$ such that
    $$
    \tilde{L}_\lambda(t, q, v)\ge \kappa|v|^2-C,\quad\forall (\lambda, t, q, v)\in \hat\Lambda\times [0,\tau]\times S\times\mathbb{R}^n.
    $$
   Define
\begin{equation}\label{e:ModifiedL}
\hat{L}:\hat\Lambda\times [0,\tau]\times B^n_{\delta}(0)\times \R^n\to\R,\;(\lambda, t, q,v)\mapsto \tilde{L}(\lambda, t, q+x_\lambda(t), v+\dot{x}_\lambda(t))
\end{equation}
and $\hat{L}_\lambda(\cdot)=\hat{L}(\lambda,\cdot)$ for $\lambda\in\hat\Lambda$.
By Lemmas~~\ref{lem:Lagr},~\ref{lem:regu},  we obtain:

\begin{lemma}\label{lem:reduction}
 \begin{enumerate}
  \item[\rm (A)] The function $\hat{L}$ is  continuous;   partial derivatives
$$
\partial_t\hat{L}(\cdot),\;
\partial_q\hat{L}(\cdot),\;\partial_v\hat{L}(\cdot),
\;\partial_{tv}\hat{L}(\cdot),\;\partial_{vt}\hat{L}(\cdot), \;\partial_{qv}\hat{L}(\cdot),\;\partial_{vq}\hat{L}(\cdot),\;
\partial_{qq}\hat{L}(\cdot),\;\partial_{vv}\hat{L}(\cdot)
$$
exist and depend continuously on $(\lambda, t, q, v)$.
\item[\rm (B)]  For each $(\lambda, t, q)\in \hat\Lambda\times [0,\tau]\times B^n_{\delta}(0)$, $\hat{L}_\lambda(t, q, v)$ is
strictly convex in $v$, and
 \begin{equation}\label{e:coercive}
   \hat{L}_\lambda(t, q, v)\ge \kappa|v+\dot{x}_\lambda(t)|^2-C\ge \frac{\kappa}{2}|v|^2-\kappa\rho^2-C
    \end{equation}
 for all $(\lambda, t, q, v)\in\hat\Lambda\times [0,\tau]\times B^n_{\delta}(0)\times \R^n$.

 \item[\rm (C)] For each $(\lambda,t, q)$, if each $x_\lambda$ is constant
 and $L(\lambda, t, q, v)$ is even in $v$ then
   $\hat{L}_\lambda(t, q,v)$ can be required to be even in $v$.
  \item[\rm (D)] If $U$ is a symmetric open neighborhood of the origin in $\mathbb{R}^n$, $x_\lambda\equiv 0\;\forall\lambda$, and for each $(\lambda,t)$
 the function $L(\lambda, t, q, v)$ is even in $(q,v)$, then 
  $\hat{L}(\lambda,t,q,v)$ can be also required to be even in $(q,v)$.
\item[\rm (E)] For each $(\lambda, q)$, if $L(\lambda, t, q, v)$ is even in $(t,v)$, 
and $x_\lambda\equiv 0\;\forall\lambda$,
then $\hat{L}(\lambda, t, q, v)$ can be chosen to be even in $(t,v)$.
\item[\rm (F)] If $L$ is independent of time $t$, so is  $\hat{L}$.
\item[\rm (G)] If Assumption~\ref{ass:Lagr0} is replaced by Assumption~\ref{ass:BasiAssLagr}, and $Ex_\lambda(t)=x_\lambda(t)\;\forall t\in\hat\Lambda$,
then the function $\hat{L}$ given by (\ref{e:ModifiedL}) may be replaced by
  \begin{equation*}
\hat{L}:\hat\Lambda\times \R\times U\times\mathbb{R}^n\to\R,\;(\lambda,t,q,v)\mapsto {L}(\lambda,t,q+x_\lambda(t),v+\dot{x}_\lambda(t))+
\psi_{\rho_0,\rho_1}(|v+\dot{x}_\lambda(t)|^2),
 \end{equation*}
 which also satisfies (\ref{e:M-invariant1Lagr}) because $E$ is
  a real orthogonal matrix.
  \end{enumerate}
\end{lemma}

\begin{remark}\label{rm:reduction}
{\rm For a given positive number $\rho_0>0$,  replacing 
 $\tilde{L}^\ast$ and $\iota$ by $\hat{L}$ and $\delta$ in the proof of Lemma~\ref{lem:Gener-modif}
we may obtain a  continuous function $\check{L}:\hat\Lambda\times [0, \tau]\times B^n_{3\delta/4}(0)\times\mathbb{R}^n\to\R$
 satisfying (L1)-(L6) in Lemma~\ref{lem:Gener-modif} with $\tilde{L}^\ast=\hat{L}$ and $\iota=\delta$,
 and Lemma~\ref{lem:Gener-modif}(L0) without $\partial_t\check{L}(\cdot)$.
Because of Remark~\ref{rm:funct-analy}, as in the proofs of Theorems~\ref{th:bif-nessLagrGener},~\ref{th:bif-existLagrGener},~\ref{th:bif-suffLagrGener}
in Section~\ref{sec:LagrBound} we may obtain the corresponding versions of these theorems under weaker Assumptions~\ref{ass:Lagr0},~\ref{ass:Lagr1}.
Similarly, when $M$ is an open subset in $\mathbb{R}^n$ the conditions in Theorem~\ref{th:MorseBif} may be weakened suitably.
}
\end{remark}

\section{Proofs of Theorems~\ref{th:bif-nessLagrGener},~\ref{th:bif-existLagrGener},~\ref{th:bif-suffLagrGener} and \ref{th:MorseBif}}\label{sec:LagrBound}
\setcounter{equation}{0}

\subsection{Proofs of Theorems~\ref{th:bif-nessLagrGener},~\ref{th:bif-existLagrGener},~\ref{th:bif-suffLagrGener}}\label{sec:LagrBound.1}

\subsubsection{Reduction to Euclidean spaces}\label{sec:LagrBound.1.1}

Since $\gamma_\mu(0)\ne\gamma_\mu(\tau)$, we may choose the $C^6$ Riemannian metric $g$ on $M$ so that
$S_0$ (resp. $S_1$) is totally geodesic near $\gamma_\mu(0)$ (resp. $\gamma_\mu(\tau)$).
(Indeed, by the definition of submanifolds there exists a coordinate chart $(U,\varphi)$ around
$\gamma_\mu(0)$ (resp. $\gamma_\mu(\tau)$) on $M$ such that $\varphi(S_0\cap U)=\varphi(U)\cap V_0$
(resp. $\varphi(S_1\cap U)=\varphi(U)\cap V_1$) for some linear subspace $V_0$ (resp. $V_1$) in $\mathbb{R}^n$.
Then extending the pullback of the standard metric on $\mathbb{R}^n$ to $U$ yields a required metric.)
There exists a fibrewise convex open neighborhood $\mathcal{U}(0_{TM})$ of the zero section of $TM$
such that the exponential map of $g$ gives rise to $C^5$ immersion
\begin{equation}\label{e:exp}
\mathbb{F}:\mathcal{U}(0_{TM})\to M\times M,\;(q,v)\mapsto (q,\exp_q(v)),
\end{equation}
(cf.~Appendix~\ref{app:Exp}). By (\ref{e:App.2}),
$d\mathbb{F}(q,0_q):T_{(q,0_q)}\to T_{(q,q)}(M\times M)=T_qM\times T_qM$ is an isomorphism for each $q\in M$.
Since $\mathbb{F}$ is injective on the closed subset $0_{TM}\subset TM$, it follows from Exercise~7 in
\cite[page 41]{Hir} that $\mathbb{F}|_{\mathcal{W}(0_{TM})}$ is a $C^5$ embedding of some smaller open neighborhood
$\mathcal{W}(0_{TM})\subset \mathcal{U}(0_{TM})$ of $0_{TM}$.
Note that $\mathbb{F}(0_{TM})$ is equal to the diagonal $\Delta_M$ in $M\times M$, and that $\gamma_\mu([0, \tau])$ is compact.
We may choose  a number $\iota>0$ such that
\begin{description}
\item[($\clubsuit$)] the closure $\bar{\bf U}_{3\iota}(\gamma_\mu([0, \tau]))$ of
${\bf U}_{3\iota}(\gamma_\mu([0, \tau])):=\{p\in M\,|\, d_g(p, \gamma_\mu([0, \tau]))< 3\iota\}$ is
a compact neighborhood of $\gamma_\mu([0, \tau])$ in $M$, and
$\bar{\bf U}_{3\iota}(\gamma_\mu([0, \tau]))\times\bar{\bf U}_{3\iota}(\gamma_\mu([0, \tau]))$
is contained in the image of $\mathbb{F}|_{\mathcal{W}(0_{TM})}$;
\item[($\spadesuit$)] $\{(q,v)\in TM\,|\,q\in \bar{\bf U}_{3\iota}(\gamma_\mu([0, \tau])),\;|v|_q\le 3\iota\}\subset \mathcal{W}(0_{TM})$.
\end{description}
Then $3\iota$ is less than the injectivity
radius of $g$ at each point on $\bar{\bf U}_{3\iota}(\gamma_\mu([0, \tau]))$.
Let us take a path  $\overline{\gamma}\in C^7([0,\tau];M)$  such that
 \begin{eqnarray}\label{e:gamma1}
\overline{\gamma}(0)=\gamma_\mu(0),\quad \overline{\gamma}(\tau)=\gamma_\mu(\tau),\quad\hbox{and}\quad
{\rm dist}_g(\gamma_\mu(t), \overline{\gamma}(t))<\iota\;\forall t\in
[0,\tau].
\end{eqnarray}
We first assume:
\begin{eqnarray}\label{e:gamma2}
d_g(\gamma_\lambda(t), \overline{\gamma}(t))<\iota,\quad\forall (\lambda,t)\in
\Lambda\times[0,\tau].
\end{eqnarray}
(For cases of Theorems~\ref{th:bif-nessLagrGener},~\ref{th:bif-suffLagrGener},
by contradiction we may use nets to prove that (\ref{e:gamma2}) 
is satisfied after shrinking $\Lambda$ toward $\mu$.)
Then (\ref{e:gamma1}) and (\ref{e:gamma2}) imply
\begin{eqnarray}\label{e:gamma3}
d_g(\gamma_\lambda(t), \gamma_\mu([0,\tau]))\le
d_g(\gamma_\lambda(t), \gamma_\mu(t))<2\iota,\quad\forall (\lambda,t)\in
\Lambda\times[0,\tau].
\end{eqnarray}
Using a unit orthogonal parallel $C^5$ frame field along 
$\overline{\gamma}$,   $[0,\tau]\ni t\mapsto (e_1(t),\cdots, e_n(t))$, we get a $C^5$ map
 \begin{eqnarray}\label{e:Lagr4}
\phi_{\overline{\gamma}}:[0,\tau]\times B^n_{2\iota}(0)\to
M,\;(t,x)\mapsto\exp_{\overline{\gamma}(t)}\Big(\sum^n_{i=1}x_i
e_i(t)\Big)
 \end{eqnarray}
 (Note that the tangent map $d\phi_{\overline{\gamma}}:T([0,\tau]\times B^n_{2\iota}(0))\to
TM$ is $C^4$.)
By Step 1 in \cite[\S4]{Lu4} there exist two linear subspaces
of $\R^n$, $V_0$ and $V_1$, such that $v\in V_0$ (resp. $v\in V_1$) if and only if
$\sum^n_{k=1}v_ke_k(0)\in T_{\gamma(0)}S_0$ (resp. $\sum^n_{k=1}v_ke_k(\tau)\in T_{\gamma(0)}S_1$).
By \cite[Theorem~4.2]{PiTa01},
$C^{1}_{S_0\times S_1}([0,\tau]; M)$ is a $C^4$ Banach manifold;
and it follows from \cite[Theorem~4.3]{PiTa01} that the map
  \begin{eqnarray}\label{e:chart5.1}
\Phi_{\overline{\gamma}}:C^1_{V_0\times V_1}([0,\tau]; B^n_{2\iota}(0))
 \to C^{1}_{S_0\times S_1}([0,\tau]; M)
\end{eqnarray}
defined by $\Phi_{\overline{\gamma}}(\xi)(t)=\phi_{\overline{\gamma}}(t,\xi(t))$
gives a $C^2$ coordinate chart around $\overline{\gamma}$ on $C^{1}_{S_0\times S_1}([0,\tau]; M)$,
where
$$
C^k_{V_0\times V_1}([0,\tau]; B^n_{2\iota}(0))=\{\xi\in
C^k([0,\tau]; B^n_{2\iota}(0)) \,|\,\xi(0)\in V_0,\;\xi(\tau)\in V_1\}\quad\hbox{with  $k\in\mathbb{N}\cup\{0\}$}.
$$
 Moreover, it is clear that
$$
\Phi_{\overline{\gamma}}\left(C^1_{V_0\times V_1}([0,\tau]; B^n_{2\iota}(0))\right)=
\big\{\gamma\in C^1_{S_0\times S_1}([0,\tau], M)\,\big|\, \sup_t{\rm dist}_g(\gamma(t),\overline{\gamma}(t))<2\iota\big\}.
$$
(\textsf{{Note}}: $\Phi_{\overline{\gamma}}$ also defines an at least $C^1$ map from $C^2_{V_0\times V_1}([0,\tau]; B^n_{2\iota}(0))$ to
$C^{2}_{S_0\times S_1}([0,\tau]; M)$.)

By  (\ref{e:gamma2}), for each $\lambda\in\Lambda$ there exists a unique map
 ${\bf u}_\lambda:[0,\tau]\to B^n_{\iota}(0)$  such that
$$
\gamma_\lambda(t)=\phi_{\overline{\gamma}}(t,{\bf u}_\lambda(t))=\exp_{\overline{\gamma}(t)}\Big(\sum^n_{i=1}{\bf u}_\lambda^i(t)
e_i(t)\Big),\quad t\in [0,\tau].
$$
Clearly,  ${\bf u}_\lambda$ satisfies the first assertion of the following lemma, 
whose proof will be given in Appendix~\ref{app:Exp}.

\begin{lemma}\label{lem:twoCont}
 ${\bf u}_\mu(0)=0={\bf u}_\mu(\tau)$, ${\bf u}_\lambda\in C^2_{V_0\times V_1}([0,\tau]; B^n_{\iota}(0))$, and
$$
 (\lambda,t)\mapsto {\bf u}_\lambda(t)\quad\hbox{and}\quad (\lambda,t)\mapsto \dot{\bf u}_\lambda(t)
 $$
 are continuous as maps from  $\Lambda\times [0,\tau]$ to $\R^n$.
 \end{lemma}

Define $L^{\ast}:\Lambda\times [0,\tau]\times B^n_{2\iota}(0)\times\R^n\to\R$ by
\begin{eqnarray}\label{e:3.15Gener}
 L^{\ast}(\lambda, t, q,  v)=L^{\ast}_\lambda(t, q,  v)
=L_\lambda\left(t, \phi_{\overline{\gamma}}(t, q),
D_t\phi_{\overline{\gamma}}(t, q)+ D_q\phi_{\overline{\gamma}}(t, q)[v]\right)
\end{eqnarray}
Since $\phi_{\overline{\gamma}}$ is $C^5$, by Assumption~\ref{ass:Lagr6}, $L^\ast$
is $C^2$ with respect to $(t,q,v)$ and strictly convex with respect to $v$, and
all its partial derivatives also depend continuously on $(\lambda, t, q, v)$.
 Moreover, ${\bf u}_\lambda$ solves the following boundary problem:
\begin{eqnarray*}
&&\frac{d}{dt}\big(\partial_vL^\ast_\lambda(t, x(t), \dot{x}(t))\big)-\partial_q{L}^\ast_\lambda(t, x(t), \dot{x}(t))=0,\\
&&\left.\begin{array}{ll}
x\in C^2([0,\tau]; B^n_{\iota}(0)),\;(x(0), x(\tau))\in V_0\times V_1\quad\hbox{and}\quad\\
\partial_v{L}^\ast_\lambda(0, x(0), \dot{x}(0))[v_0]=\partial_v{L}^\ast_\lambda(\tau, x(\tau), \dot{x}(\tau))[v_1]\\
\qquad\forall (v_0,v_1)\in V_0\times V_1.
\end{array}\right\}
\end{eqnarray*}

By Lemmas~\ref{lem:regu},~\ref{lem:twoCont} we directly obtain:

\begin{lemma}\label{lem:twoCont+}
 $\Lambda\times [0,\tau]\ni(\lambda,t)\mapsto \ddot{\bf u}_\lambda(t)\in\R^n$
is continuous.
 \end{lemma}
({\it Note}: This result is necessary for us to derive that $\partial_{vt}\tilde L^\ast$ is continuous in Proposition~\ref{th:bif-LagrGenerEu}.)

Define $\tilde L^\ast:\Lambda\times [0,\tau]\times B^n_{\iota}(0)\times\R^n\to\R$ by
\begin{equation}\label{e:ModifiedLEu}
\tilde L^\ast(\lambda, t, q,v)=\tilde L^\ast_{\lambda}(t, q,v)= {L}^\ast(\lambda, t, q+ {\bf u}_\lambda(t), v+\dot{\bf u}_\lambda(t)).
\end{equation}
Then
\begin{eqnarray*}
\partial_t\tilde L^\ast(\lambda, t, q,v)&=& \partial_t{L}^\ast(\lambda, t, q+ {\bf u}_\lambda(t), v+\dot{\bf u}_\lambda(t))\\
 &&+\partial_q{L}^\ast(\lambda, t, q+ {\bf u}_\lambda(t), v+\dot{\bf u}_\lambda(t))\dot{\bf u}_\lambda(t)\\
&&+\partial_v{L}^\ast(\lambda, t, q+ {\bf u}_\lambda(t), v+\dot{\bf u}_\lambda(t))\ddot{\bf u}_\lambda(t),\\
\partial_q\tilde L^\ast(\lambda, t, q,v)&=& \partial_q{L}^\ast(\lambda, t, q+ {\bf u}_\lambda(t), v+\dot{\bf u}_\lambda(t))\\
\partial_v\tilde L^\ast(\lambda, t, q,v)&=& \partial_v{L}^\ast(\lambda, t, q+ {\bf u}_\lambda(t), v+\dot{\bf u}_\lambda(t)).
\end{eqnarray*}
  By these and Lemmas~\ref{lem:twoCont},~\ref{lem:twoCont+} 
  it is not hard to see that $\tilde L^\ast$ satisfies  the following:

\begin{proposition}\label{th:bif-LagrGenerEu}
\begin{enumerate}
   \item[\rm (a)] $\tilde L^\ast$ is  continuous, and the following  partial derivatives
$$
\partial_t\tilde L^\ast(\cdot),\;\partial_q\tilde L^\ast(\cdot),\;\partial_v\tilde L^\ast(\cdot),
\;\partial_{tv}\tilde L^\ast(\cdot),\;\partial_{vt}\tilde L^\ast(\cdot), \;\partial_{qv}\tilde L^\ast(\cdot),\;\partial_{vq}\tilde L^\ast(\cdot),\;
\partial_{qq}\tilde L^\ast(\cdot),\;\partial_{vv}\tilde L^\ast(\cdot)
$$
exist and depend continuously on $(\lambda, t, q, v)$.
  \item[\rm (b)] For each $(\lambda, t, q)\in \Lambda\times [0,\tau]\times B^n_{\iota}(0)$, $\tilde L^\ast_\lambda(t, q, v)$ is
strictly convex in $v$.
  \end{enumerate}
  \end{proposition}

Clearly, $\tilde L^\ast$ satisfies Assumption~\ref{ass:Lagr0}, and Assumption~\ref{ass:Lagr1} with $x_\lambda\equiv 0\;\forall\lambda$.

\begin{remark}\label{rm:bif-LagrGenerEu}
{\rm Actually, for our next arguments in this section it suffices that $\tilde{L}^\ast$ satisfies (a) and the following weaker condition:
\begin{enumerate}
 \item[\rm (b')]  $\tilde L^\ast(\lambda, t, q, v)$ is
 convex in $v$, and for any compact or sequential compact subset $\hat\Lambda\subset\Lambda$ there exists $\rho>0$ such that
 $\hat\Lambda\times [0,\tau]\times B^n_{\iota}(0)\times B^n_{\rho}(0)\ni (\lambda, t, q,v)\mapsto \tilde L^\ast(\lambda, t, q, v)$ is
strictly convex with respect to $v$. 
 \end{enumerate}
This means: In Assumption~\ref{ass:Lagr6} we may only require that
 $L$ is fiberwise convex; but in Assumption~\ref{ass:LagrGenerB}
we need to add the condition: \textsf{for any compact or sequential compact subset $\hat\Lambda\subset\Lambda$ there exist $0<\rho_0<\rho$
such that $\sup\{|\dot{\gamma}_\lambda(t)|_g\,|\,(\lambda,t)\in\hat\Lambda\times [0,\tau]\}<\rho_0$ and
 $L$ is fiberwise strictly convex in $(\lambda,t,q, v)\in\hat\Lambda\times [0,\tau]\times TM\,|\, |v|_g<\rho\}$.}
 }
\end{remark}

  The  condition (a)  in Proposition~\ref{th:bif-LagrGenerEu} assures that each functional
\begin{equation}\label{e:brakefunctGener*}
\tilde{\mathcal{E}}^\ast_\lambda: C^1_{V_0\times V_1}([0,\tau]; B^n_{\iota}(0))\to\R,\;x\mapsto\int^1_0\tilde{L}^\ast_\lambda(t, x(t), \dot{x}(t))dt
\end{equation}
is $C^2$, and satisfies
\begin{equation}\label{e:twofunctional}
\tilde{\cal E}^\ast_{\lambda}(x)={\cal E}_{\lambda}\left(\Phi_{\overline{\gamma}}(x+{\bf u}_\lambda)\right)\;\forall x\in C^1_{V_0\times V_1}([0,\tau]; B^n_{\iota}(0))\quad\hbox{and}\quad d\tilde{\mathcal{E}}^\ast_\lambda(0)=0.
\end{equation}
Hence for each $\lambda\in\Lambda$,
 $x\in C^1_{V_0\times V_1}([0,\tau]; B^n_{\iota}(0))$ satisfies
 $d\tilde{\mathcal{E}}^\ast_\lambda(x)=0$ if and only if $\gamma:=\Phi_{\overline{\gamma}}(x+{\bf u}_\lambda)$ satisfies
$d{\mathcal{E}}_\lambda(\gamma)=0$; and in this case %it holds that
$m^-(\tilde{\mathcal{E}}^\ast_\lambda, x)=m^-({\mathcal{E}}_\lambda, \gamma)$ and
$m^0(\tilde{\mathcal{E}}^\ast_\lambda, x)=m^0({\mathcal{E}}_\lambda, \gamma)$.
In particular, we have 
\begin{equation}\label{e:twoMorse}
m^-(\tilde{\mathcal{E}}^\ast_\lambda, 0)=m^-(\mathcal{E}_\lambda, \gamma_\lambda)\quad
\hbox{and}\quad
m^0(\tilde{\mathcal{E}}^\ast_\lambda, 0)=m^0(\mathcal{E}_\lambda, \gamma_\lambda).
\end{equation}

By \cite[Proposition~4.2]{BuGiHi}
 the critical points of $\tilde{\mathcal{E}}^\ast_\lambda$ 
 correspond to the solutions of the following boundary problem:
\begin{eqnarray}\label{e:LagrGenerEu}
&&\frac{d}{dt}\big(\partial_v\tilde L^\ast_\lambda(t, x(t), \dot{x}(t))\big)-
\partial_q\tilde{L}^\ast_\lambda(t, x(t), \dot{x}(t))=0,\\
&&\left.\begin{array}{ll}
x\in C^2([0,\tau]; B^n_{\iota}(0)),\;(x(0), x(\tau))\in V_0\times V_1\quad\hbox{and}\quad\\
\partial_v\tilde{L}^\ast_\lambda(0, x(0), \dot{x}(0))[v_0]=0\quad\forall v_0\in V_0,\\
   \partial_v\tilde{L}^\ast_\lambda(\tau, x(\tau), \dot{x}(\tau))[v_1]=0\quad\forall v_1\in V_1.
\end{array}\right\}\label{e:LagrGenerBEu}
\end{eqnarray}

Let $W^{1,2}_{V_0\times V_1}([0,\tau]; B^n_{\iota}(0))=\{\xi\in W^{1,2}([0,\tau];
B^n_{\iota}(0))\,|\,(\xi(0),\xi(\tau))\in V_0\times V_1 \}$.
The following three theorems may be, respectively, viewed as corresponding results of Theorems~\ref{th:bif-nessLagrGener},~\ref{th:bif-existLagrGener},
~\ref{th:bif-suffLagrGener} provided that $\tilde{L}^\ast$ satisfies (a) and (b) in Proposition~\ref{th:bif-LagrGenerEu}.

\begin{theorem}\label{th:bif-nessLagrGenerEu}
\begin{enumerate}
\item[\rm (I)]{\rm (\textsf{Necessary condition}):}
  Suppose that $(\mu, 0)\in\Lambda\times C^1_{V_0\times V_1}([0,\tau]; B^n_{\iota}(0))$  is a
   bifurcation point along sequences of the problem (\ref{e:LagrGenerEu})--(\ref{e:LagrGenerBEu})
  with respect to the trivial branch $\{(\lambda,0)\,|\,\lambda\in\Lambda\}$ in $\Lambda\times C^1_{V_0\times V_1}([0,\tau]; B^n_{\iota}(0))$.
  Then $m^0(\tilde{\mathcal{E}}^\ast_{\mu}, 0)>0$.

\item[\rm (II)]{\rm (\textsf{Sufficient condition}):}
Suppose that $\Lambda$ is first countable and that there exist two sequences in  $\Lambda$ converging to $\mu$, $(\lambda_k^-)$ and
$(\lambda_k^+)$,  such that one of the following conditions is satisfied:
 \begin{enumerate}
 \item[\rm (II.1)] For each $k\in\mathbb{N}$, either $0$  is not an isolated critical point of $\tilde{\mathcal{E}}^\ast_{\lambda^+_k}$,
 or $0$ is not an isolated critical point of $\tilde{\mathcal{E}}^\ast_{\lambda^-_k}$,
 or $0$  is an isolated critical point of $\tilde{\mathcal{E}}^\ast_{\lambda^+_k}$ and $\tilde{\mathcal{E}}^\ast_{\lambda^-_k}$ and
  $C_m(\tilde{\mathcal{E}}^\ast_{\lambda^+_k}, 0;{\bf K})$ and $C_m(\tilde{\mathcal{E}}^\ast_{\lambda^-_k}, 0;{\bf K})$ are not isomorphic for some Abel group ${\bf K}$ and some $m\in\mathbb{Z}$.
\item[\rm (II.2)] For each $k\in\mathbb{N}$, there exists $\lambda\in\{\lambda^+_k, \lambda^-_k\}$ such that
$0$  is an either non-isolated or homological visible critical point of
$\tilde{\mathcal{E}}^\ast_{\lambda}$ , and
$$[m^-(\tilde{\mathcal{E}}^\ast_{\lambda_k^-}, 0),
m^-(\tilde{\mathcal{E}}^\ast_{\lambda_k^-}, 0)+
m^0(\tilde{\mathcal{E}}^\ast_{\lambda_k^-}, 0)]\cap[m^-(\tilde{\mathcal{E}}^\ast_{\lambda_k^+}, 0),
m^-(\tilde{\mathcal{E}}^\ast_{\lambda_k^+}, 0)+m^0(\tilde{\mathcal{E}}^\ast_{\lambda_k^+}, 0)]=\emptyset.
$$
\item[\rm (II.3)]  $[m^-(\tilde{\mathcal{E}}^\ast_{\lambda_k^-}, 0),
m^-(\tilde{\mathcal{E}}^\ast_{\lambda_k^-}, 0)+
m^0(\tilde{\mathcal{E}}^\ast_{\lambda_k^-}, 0)]\cap[m^-(\tilde{\mathcal{E}}^\ast_{\lambda_k^+}, 0),
m^-(\tilde{\mathcal{E}}^\ast_{\lambda_k^+}, 0)+m^0(\tilde{\mathcal{E}}^\ast_{\lambda_k^+}, 0)]=\emptyset$,
and either $m^0(\tilde{\mathcal{E}}^\ast_{\lambda_k^-}, 0)=0$ or $m^0(\tilde{\mathcal{E}}^\ast_{\lambda_k^+}, 0)=0$
for each $k\in\mathbb{N}$.
 \end{enumerate}
   Then there exists a sequence $\{(\lambda_k,x_k)\}_{k\ge 1}$ in  $\hat\Lambda\times C^2([0,\tau], \R^n)$
   converging to  $(\mu, 0)$ such that each $x_k\ne 0$  is a  solution of the problem (\ref{e:LagrGenerEu})--(\ref{e:LagrGenerBEu})
    with $\lambda=\lambda_k$, $k=1,2,\cdots$,
      where  $\hat{\Lambda}=\{\mu,\lambda^+_k, \lambda^-_k\,|\,k\in\mathbb{N}\}$.
  \end{enumerate}
  \end{theorem}

\begin{theorem}[\textsf{Existence for bifurcations}]\label{th:bif-existLagrGenerEu+}
Let $\Lambda$ be connected.  For $\lambda^-, \lambda^+\in\Lambda$
suppose that
one of the following conditions is satisfied:
 \begin{enumerate}
 \item[\rm (i)] Either $0$  is not an isolated critical point of $\tilde{\mathcal{E}}^\ast_{\lambda^+}$,
 or $0$ is not an isolated critical point of $\tilde{\mathcal{E}}^\ast_{\lambda^-}$,
 or $0$  is an isolated critical point of $\tilde{\mathcal{E}}^\ast_{\lambda^+}$ and $\tilde{\mathcal{E}}^\ast_{\lambda^-}$ and
  $C_m(\tilde{\mathcal{E}}^\ast_{\lambda^+}, 0;{\bf K})$ and
   $C_m(\tilde{\mathcal{E}}^\ast_{\lambda^-}, 0;{\bf K})$ 
   are not isomorphic for some Abel group ${\bf K}$ and some $m\in\mathbb{Z}$.

\item[\rm (ii)] $[m^-(\tilde{\mathcal{E}}^\ast_{\lambda^-}, 0),
m^-(\tilde{\mathcal{E}}^\ast_{\lambda^-}, 0)+ m^0(\tilde{\mathcal{E}}^\ast_{\lambda^-}, 0)]\cap[m^-(\tilde{\mathcal{E}}^\ast_{\lambda^+}, 0),
m^-(\tilde{\mathcal{E}}^\ast_{\lambda^+}, 0)+m^0(\tilde{\mathcal{E}}^\ast_{\lambda^+}, 0)]=\emptyset$,
and there exists $\lambda\in\{\lambda^+, \lambda^-\}$ such that $0$  is an either non-isolated or homological visible critical point of
$\mathcal{E}^\ast_{\lambda}$.

\item[\rm (iii)] $[m^-(\tilde{\mathcal{E}}^\ast_{\lambda^-}, 0),
m^-(\tilde{\mathcal{E}}^\ast_{\lambda^-}, 0)+
m^0(\tilde{\mathcal{E}}^\ast_{\lambda^-}, 0)]\cap[m^-(\tilde{\mathcal{E}}^\ast_{\lambda^+}, 0),
m^-(\tilde{\mathcal{E}}^\ast_{\lambda^+}, 0)+m^0(\tilde{\mathcal{E}}^\ast_{\lambda^+}, 0)]=\emptyset$,
and either $m^0(\mathcal{E}^\ast_{\lambda^+}, 0)=0$ or $m^0(\mathcal{E}^\ast_{\lambda^-}, 0)=0$.
 \end{enumerate}
Then for any path $\alpha:[0,1]\to\Lambda$ connecting $\lambda^+$ to $\lambda^-$ there exists
  a sequence $(t_k)\subset[0,1]$ converging to some $\bar{t}\in [0,1]$, and
   a nonzero solution $x_k$ of the problem (\ref{e:LagrGener})--(\ref{e:LagrGenerB})
  with $\lambda=\alpha(t_k)$ for each $k\in\mathbb{N}$ such that $\|x_k\|_{C^2([0,\tau];\mathbb{R}^n)}\to 0$
  as $k\to\infty$.
 Moreover,  $\alpha(\bar{t})$ is not equal to $\lambda^+$ (resp. $\lambda^-$) if $
 m^0(\tilde{\mathcal{E}}^\ast_{\lambda^+}, 0)=0$ (resp. $m^0(\tilde{\mathcal{E}}^\ast_{\lambda^-}, 0)=0$).
 \end{theorem}

  \begin{theorem}[\textsf{Alternative bifurcations of Rabinowitz's type}]\label{th:bif-suffLagrGenerEu}
    Let $\Lambda$ be a real interval and  $\mu\in{\rm Int}(\Lambda)$. Suppose that $m^0(\tilde{\mathcal{E}}^\ast_{\mu}, 0)>0$,
   and that $m^0(\tilde{\mathcal{E}}^\ast_{\lambda}, 0)=0$  for each $\lambda\in\Lambda\setminus\{\mu\}$ near $\mu$, and
  $m^-(\tilde{\mathcal{E}}^\ast_{\lambda}, 0)$ take, respectively, values $m^-(\tilde{\mathcal{E}}^\ast_{\mu}, 0)$ and
  $m^-(\tilde{\mathcal{E}}^\ast_{\mu}, 0)+ m^0(\tilde{\mathcal{E}}^\ast_{\mu}, 0)$
 as $\lambda\in\Lambda$ varies in two deleted half neighborhoods  of $\mu$.
Then  one of the following alternatives occurs:
\begin{enumerate}
\item[\rm (i)] The problem (\ref{e:LagrGenerEu})--(\ref{e:LagrGenerBEu})
 with $\lambda=\mu$ has a sequence of solutions, $x_k\ne 0$, $k=1,2,\cdots$,
which converges to $0$ in $C^2([0,\tau], \R^n)$.

\item[\rm (ii)]  For every $\lambda\in\Lambda\setminus\{\mu\}$ near $\mu$ there is a  solution $y_\lambda\ne 0$ of
(\ref{e:LagrGenerEu})--(\ref{e:LagrGenerBEu}) with parameter value $\lambda$,
 such that   $y_\lambda$  converges to zero in  $C^2([0,\tau], \R^n)$ as $\lambda\to \mu$.

\item[\rm (iii)] For a given neighborhood $\mathfrak{W}$ of $0\in C^1_{V_0\times V_1}([0,\tau]; B^n_{\iota}(0))$,
there is a one-sided neighborhood $\Lambda^0$ of $\mu$ such that
for any $\lambda\in\Lambda^0\setminus\{\mu\}$, the problem (\ref{e:LagrGenerEu})--(\ref{e:LagrGenerBEu}) with parameter value $\lambda$
has at least two distinct solutions in  $\mathfrak{W}$, $y_\lambda^1\ne 0$ and $y_\lambda^2\ne 0$,
which can also be required to satisfy $\tilde{\mathcal{E}}^\ast_\lambda(y^1_\lambda)\ne \tilde{\mathcal{E}}^\ast_\lambda(y^2_\lambda)$
provided that $m^0(\tilde{\mathcal{E}}^\ast_{\mu}, 0)>1$ and the problem
(\ref{e:LagrGenerEu})--(\ref{e:LagrGenerBEu}) with parameter value $\lambda$
has only finitely many solutions in $\mathfrak{W}$.
\end{enumerate}
\end{theorem}

Theorems~\ref{th:bif-nessLagrGener},~\ref{th:bif-suffLagrGener} are  derived from
Theorems~\ref{th:bif-nessLagrGenerEu},~\ref{th:bif-suffLagrGenerEu}, respectively.
We first admit them and postpone their proof to Section~\ref{sec:LagrBound.2}.
Theorem~\ref{th:bif-existLagrGenerEu+} can only lead to Theorem~\ref{th:bif-existLagrGener}
under the assumption (\ref{e:gamma2}). We shall directly prove Theorem~\ref{th:bif-existLagrGener}
in Section~\ref{sec:LagrBound.1.4}.

\subsubsection{Proofs of Theorems~\ref{th:bif-nessLagrGener},~\ref{th:bif-suffLagrGener}}\label{sec:LagrBound.1.2}

\begin{proof}[\bf Proof of Theorem~\ref{th:bif-nessLagrGener}]
\begin{description}
\item[(I)] By the assumption there exists a sequence in $\Lambda\times C^{1}_{S_0\times S_1}([0,\tau]; M)$
  converging to $(\mu,\gamma_\mu)$, $\{(\lambda_k, \gamma^k)\}_{k\ge 1}$,
  such that each $\gamma^k\ne\gamma_{\lambda_k}$
  is a solution of (\ref{e:LagrGener})--(\ref{e:LagrGenerB}) with $\lambda=\lambda_k$, $k=1,2,\cdots$.
After removing the finite terms (if necessary) we may assume that all $\gamma^k$ are contained in the image of the chart
$\Phi_{\overline{\gamma}}$ in (\ref{e:chart5.1}). Then for each $k\in\mathbb{N}$ there exists a unique
${\bf u}^k\in C^2_{V_0\times V_1}([0,\tau]; B^n_{\iota}(0))$ such that
$\Phi_{\overline{\gamma}}({\bf u}^k)=\gamma^k$.
Since $\gamma_{\lambda_k}\ne\gamma^k$, $d\mathcal{E}_{\lambda_k}(\gamma_{\lambda_k})=0$ and $d\mathcal{E}_{\lambda_k}(\gamma^k)=0$,
we obtain  ${\bf u}^k\ne {\bf u}_{\lambda_k}$, and $d\tilde{\mathcal{E}}_{\lambda_k}({\bf u}_{\lambda_k})=0$
and $d\tilde{\mathcal{E}}_{\lambda_k}({\bf u}^k)=0$.
Recall that we have assumed $M\subset\mathbb{R}^N$.
Assumption~\ref{ass:LagrGenerB} implies that $\gamma_\lambda-\gamma_\mu\to 0$ in $C^1([0,\tau];\mathbb{R}^N)$ as $\lambda\to\mu$.
Moreover, $\gamma^k\to\gamma_\mu$ in $C^{1}_{S_0\times S_1}([0,\tau]; M)\subset C^1([0,\tau];\mathbb{R}^N)$
as $k\to\infty$. Therefore $\|\gamma^k-\gamma_{\lambda_k}\|_{C^1([0,\tau];\mathbb{R}^N)}\to 0$ as $k\to\infty$.
This implies that $\|{\bf u}^k- {\bf u}_{\lambda_k}\|_{C^1([0,\tau];\mathbb{R}^n)}\to 0$ as $k\to\infty$.
In particular, there exists an integer $k_0>0$ such that $\|{\bf u}^k- {\bf u}_{\lambda_k}\|_{C^1([0,\tau];\mathbb{R}^n)}<\iota$
for all $k\ge k_0$. Since ${\bf u}^k=({\bf u}^k- {\bf u}_{\lambda_k}) + {\bf u}_{\lambda_k}$, by the arguments below (\ref{e:brakefunctGener*})
we get $d\tilde{\mathcal{E}}^\ast_{\lambda_k}({\bf u}^k-{\bf u}_{\lambda_k})=0$ for all $k\ge k_0$.
These show that $(\mu, 0)\in\Lambda\times C^1_{V_0\times V_1}([0,\tau]; B^n_{\iota}(0))$  is a  bifurcation point
along sequences of the problem
(\ref{e:LagrGenerEu})--(\ref{e:LagrGenerBEu})
in $\Lambda\times C^1_{V_0\times V_1}([0,\tau]; B^n_{\iota}(0))$ with respect to the trivial branch $\{(\lambda,0)\,|\,\lambda\in\Lambda\}$.
Then Theorem~\ref{th:bif-nessLagrGenerEu}(I) concludes $m^0(\tilde{\mathcal{E}}^\ast_{\mu}, 0)>0$,
and therefore $m^0({\mathcal{E}}_{\mu}, \gamma_\mu)>0$ by  (\ref{e:twoMorse}).

\item[(II)] Follow the above notations. By the assumption and (\ref{e:twoMorse})  we get that for all $k\in\mathbb{N}$,
\begin{eqnarray*}
&&[m^-(\tilde{\mathcal{E}}^\ast_{\lambda_k^-}, 0),
m^-(\tilde{\mathcal{E}}^\ast_{\lambda_k^-}, 0)+
m^0(\tilde{\mathcal{E}}^\ast_{\lambda_k^-}, 0)]\cap[m^-(\tilde{\mathcal{E}}^\ast_{\lambda_k^+}, 0),
m^-(\tilde{\mathcal{E}}^\ast_{\lambda_k^+}, 0)+m^0(\tilde{\mathcal{E}}^\ast_{\lambda_k^+}, 0)]\\
&&=[m^-_\tau(\mathcal{E}_{\lambda_k^-}, \gamma_{\lambda_k^-}), m^-_\tau(\mathcal{E}_{\lambda_k^-}, \gamma_{\lambda_k^-})+ m^0_\tau(\mathcal{E}_{\lambda_k^-}, \gamma_{\lambda_k^-})]\\
&&\hspace{30mm}\cap[m^-_\tau(\mathcal{E}_{\lambda_k^+}, \gamma_{\lambda_k^+}), m^-_\tau(\mathcal{E}_{\lambda_k^+}, \gamma_{\lambda_k^+})+m^0_\tau(\mathcal{E}_{\lambda_k^+}, \gamma_{\lambda_k^+})]=\emptyset
\end{eqnarray*}
 and either $m^0(\tilde{\mathcal{E}}^\ast_{\lambda_k^-}, 0)=m^0_\tau(\mathcal{E}_{\lambda_k^+}, \gamma_{\lambda_k^-})=0$ or $m^0(\tilde{\mathcal{E}}^\ast_{\lambda_k^+}, 0)=m^0_\tau(\mathcal{E}_{\lambda_k^+}, \gamma_{\lambda_k^+})=0$.
By Theorem~\ref{th:bif-nessLagrGenerEu}(II)
we have a sequence $\{(\lambda_k, {\bf v}^k)\}_{k\ge 1}\subset
\{\mu,\lambda^+_k, \lambda^-_k\,|\,k\in\mathbb{N}\}\times C^{2}_{V_0\times V_1}([0,\tau]; B^n_{\iota}(0))$
 such that
$\lambda_k\to\mu$ and $0<\|{\bf v}^k\|_{C^2}\to 0$,
   and  that each ${\bf v}^k$  is a solution of (\ref{e:LagrGenerEu})--(\ref{e:LagrGenerBEu}) with $\lambda=\lambda_k$, $k=1,2,\cdots$.
Therefore for $k$ large enough, $\gamma^k:=\Phi_{\overline{\gamma}}({\bf v^k}+{\bf u}_{\lambda_k})$
defined by $\Phi_{\overline{\gamma}}({\bf v^k}+{\bf u}_{\lambda_k})(t)=\phi_{\overline{\gamma}}(t,{\bf v^k}(t)+{\bf u}_{\lambda_k}(t))$
is a solution of (\ref{e:LagrGener})--(\ref{e:LagrGenerB}) with $\lambda=\lambda_k$,
$\gamma^k\ne \gamma_{\lambda_k}$, and as $k\to\infty$ we have
$\gamma^k\to\gamma_\mu$ in $C^{2}_{S_0\times S_1}([0,\tau]; \mathbb{R}^N)$
because $\Phi_{\overline{\gamma}}$ is also a $C^1$ map from $C^2_{V_0\times V_1}([0,\tau]; B^n_{2\iota}(0))$ to
$C^{2}_{S_0\times S_1}([0,\tau]; M)$ as noted below (\ref{e:chart5.1}).
 Theorem~\ref{th:bif-nessLagrGener}(II) is proved.
\end{description}
\end{proof}

\begin{proof}[\bf Proof of Theorem~\ref{th:bif-suffLagrGener}]
Follow the above notations. By the assumption and (\ref{e:twoMorse}),
we obtain that $m^0(\tilde{\mathcal{E}}^\ast_{\mu}, 0)\ne 0$,
   and that $m^0(\tilde{\mathcal{E}}^\ast_{\lambda}, 0)=0$ 
    for each $\lambda\in\Lambda\setminus\{\mu\}$ near $\mu$, and
  $m^-(\tilde{\mathcal{E}}^\ast_{\lambda}, 0)$ take, respectively, 
  values $m^-(\tilde{\mathcal{E}}^\ast_{\mu}, 0)$ and
  $m^-(\tilde{\mathcal{E}}^\ast_{\mu}, 0)+ m^0(\tilde{\mathcal{E}}^\ast_{\mu}, 0)$
 as $\lambda\in\Lambda$ varies in two deleted half neighborhoods  of $\mu$.
Therefore one of the conclusions (i)-(iii) in Theorem~\ref{th:bif-suffLagrGenerEu} occurs.

Let $(x_k)$ be as in (i) in Theorem~\ref{th:bif-suffLagrGenerEu}.
Since $\|x_k\|_{C^2}\to 0$, we can choose $k_0>0$ such that $\|x_k\|_{C^2}<\iota$ for $k\ge k_0$.
Then for each $k\ge k_0$, $\gamma^k:=\Phi_{\overline{\gamma}}(x_k+{\bf u}_{\mu})\ne\gamma_{\mu}$
is a solution of (\ref{e:LagrGener})--(\ref{e:LagrGenerB}) with $\lambda=\mu$, and as above we may deduce that
 $\gamma^k\to\gamma_\mu$ in $C^{2}_{S_0\times S_1}([0,\tau]; \mathbb{R}^N)$ as $k\to\infty$.
This is, (i) of Theorem~\ref{th:bif-suffLagrGener} occurs.

For $y_\lambda\ne 0$ in (ii) in Theorem~\ref{th:bif-suffLagrGenerEu}, we can shrink $\Lambda$ toward $\mu$
so that $\|y_\lambda\|_{C^2}<\iota$ for all $\lambda\in\Lambda$. Then
$\alpha_\lambda:=\Phi_{\overline{\gamma}}(y_\lambda+{\bf u}_{\lambda})\ne\gamma_{\lambda}$
is a  solution  of (\ref{e:LagrGener})--(\ref{e:LagrGenerB}) with parameter value $\lambda$,
 and  $\alpha_\lambda-\gamma_\lambda$
 converges to zero in  $C^2([0,\tau], \R^N)$ as $\lambda\to \mu$.
 Namely, (ii) of Theorem~\ref{th:bif-suffLagrGener} occurs.

For a given neighborhood $\mathcal{W}$ of $\gamma_\mu$ in $C^1([0,\tau], M)$,
let us choose a neighborhood $\mathfrak{W}$ of $0\in C^1_{V_0\times V_1}([0,\tau]; B^n_{\iota}(0))$
such that $\Phi_{\overline{\gamma}}\left({\bf u}_\mu+\mathfrak{W}\right)\subset\mathcal{W}$.
Let $\Lambda^0$, $y_\lambda^1\ne 0$ and $y_\lambda^2\ne 0$ be as in (iii) in Theorem~\ref{th:bif-suffLagrGenerEu}.
Put $\gamma^i_\lambda:=\Phi_{\overline{\gamma}}(y^i_\lambda+{\bf u}_{\lambda})\ne\gamma_{\lambda}$, $i=1,2$.
Both sit in $\mathcal{W}$ and are distinct solutions  of (\ref{e:LagrGener})--(\ref{e:LagrGenerB}) with parameter value $\lambda$.
Suppose that $m^0(\mathcal{E}_{\mu}, \gamma_\mu)=m^0(\tilde{\mathcal{E}}^\ast_{\mu}, 0)>1$ and
(\ref{e:LagrGener})--(\ref{e:LagrGenerB}) with parameter value $\lambda$
has only finitely many distinct solutions in $\mathcal{W}$.
Then the problem (\ref{e:LagrGenerEu})--(\ref{e:LagrGenerBEu}) with parameter value $\lambda$
has only finitely many solutions in $\mathfrak{W}$ as well.
In this case (iii) in Theorem~\ref{th:bif-suffLagrGenerEu} concludes that
the above $y_\lambda^1\ne 0$ and $y_\lambda^2\ne 0$
are chosen to satisfies $\tilde{\mathcal{E}}^\ast_\lambda(y^1_\lambda)\ne \tilde{\mathcal{E}}^\ast_\lambda(y^2_\lambda)$,
which implies $\mathcal{E}_{\lambda}(\gamma_\lambda^1)\ne \mathcal{E}_{\lambda}(\gamma_\lambda^2)$.
Hence (iii) in Theorem~\ref{th:bif-suffLagrGenerEu} occurs.
\end{proof}

\subsubsection{Proofs of Theorems~\ref{th:bif-nessLagrGenerEu},~\ref{th:bif-existLagrGenerEu+},~\ref{th:bif-suffLagrGenerEu}}\label{sec:LagrBound.2}

We need to make modifications for the Lagrangian $\tilde L^\ast$ in (\ref{e:ModifiedLEu}).

\begin{lemma}\label{lem:Gener-modif}
Given a positive number $\rho_0>0$ and a subset $\hat\Lambda\subset\Lambda$ which is either compact or sequential compact,
there exists a  continuous function $\check{L}:\hat\Lambda\times [0, \tau]\times B^n_{3\iota/4}(0)\times\mathbb{R}^n\to\R$
 satisfying the following properties  for some constants $\check\kappa>0$ and $0<\check{c}<\check{C}$:
\begin{enumerate}
 \item[\rm (L0)] The following partial derivatives
 $$
\partial_t\check{L}(\cdot),\;\partial_q\check{L}(\cdot),\;\partial_v\check{L}(\cdot),
\;\partial_{tv}\check{L}(\cdot),\;\partial_{vt}\check{L}(\cdot), \;\partial_{qv}\check{L}(\cdot),\;\partial_{vq}\check{L}(\cdot),\;
\partial_{qq}\check{L}(\cdot),\;\partial_{vv}\check{L}(\cdot)
$$
exist and depend continuously on $(\lambda, t, q, v)$. (These are all used in the proof of Proposition~\ref{prop:solutionLagr}.)

\item[\rm (L1)] $\check{L}$ and $\tilde{L}^\ast$ are equal in $\hat\Lambda\times [0,\tau]\times B^n_{3\iota/4}(0)\times B^n_{\rho_0}(0)$.

\item[\rm (L2)] $\partial_{vv}\check{L}_\lambda(t, q,v)\ge\check{c}I_n,\quad \forall (\lambda, t, q,v)\in \hat\Lambda\times [0,\tau]\times B^n_{3\iota/4}(0)\times\mathbb{R}^n$.

\item[\rm (L3)] $\bigl| \frac{\partial^2}{\partial q_i\partial
q_j}\check{L}_\lambda(t, q,v)\bigr|\le \check{C}(1+ |v|^2),\quad \bigl|
\frac{\partial^2}{\partial q_i\partial v_j}\check{L}_\lambda(t, q,v)\bigr|\le \check{C}(1+
|v|),\quad\hbox{and}\\
 \bigl| \frac{\partial^2}{\partial v_i\partial v_j}\check{L}_\lambda(t, q,v)\bigr|\le \check{C},
 \quad\forall (\lambda, t, q,v)\in\hat\Lambda\times [0,\tau]\times B^n_{3\iota/4}(0)\times\mathbb{R}^n$.
 
\item[\rm (L4)]  $\check{L}(\lambda,t, q, v)\ge \check{\kappa}|v|^2-\check{C}$
for all $(\lambda, t, q, v)\in\hat\Lambda\times [0,\tau]\times B^n_{3\iota/4}(0)\times\R^{n}$.
\item[\rm (L5)] $|\partial_{q}\check{L}(\lambda,t,q,v)|\le \check{C}(1+ |v|^2)$ and $|\partial_{v}\check{L}(\lambda,t,q,v)|\le \check{C}(1+ |v|)$
for all $(\lambda, t, q, v)\in \hat\Lambda\times [0, \tau]\times B^n_{3\iota/4}(0)\times\R^{n}$.
\item[\rm (L6)] $|\check{L}_\lambda(t, q,v)|\le \check{C}(1+|v|^2)$ for all $(\lambda, t, q,v)\in \hat\Lambda\times[0, \tau]\times {B}^n_{3\iota/4}(0)\times\mathbb{R}^n$.
 \end{enumerate}
\end{lemma}
\begin{proof}[\bf Proof]
\textsf{Step 1}. Fix a positive number $\rho_1>\rho_0$.
As in the proof of Lemma~\ref{lem:Lagr},
we may choose a $C^\infty$ convex function $\psi_{\rho_0,\rho_1}:[0, \infty)\to\R$
such that $\psi'_{\rho_0,\rho_1}(t)>0$ for
$t\in(\rho_0^2, \infty)$, $\psi_{\rho_0,\rho_1}(t)=0$ for $t\in[0, \rho_0^2)$ and
 $\psi_{\rho_0,\rho_1}(t)=\kappa t+ \varrho_0$ for $t\in [\rho_1^2, \infty)$,
where $\kappa>0$ and $\varrho_0<0$ are suitable constants. Define
$\tilde L^{\ast\ast}:\Lambda\times [0,\tau]\times B^n_{\iota}(0)\times\R^n\to\R$ by
\begin{equation}\label{e:tildeLtwo}
\tilde{L}^{\ast\ast}(\lambda,t,q,v)=\tilde{L}^\ast(\lambda,t,q,v)+ \psi_{\rho_0,\rho_1}(|v|^2).
 \end{equation}
It possess the same properties as $\check{L}$ in (L0) and also satisfies
\begin{equation}\label{e:tildeLtwo1}
\tilde{L}^{\ast\ast}(\lambda,t,q,v)=\tilde{L}^\ast(\lambda,t,q,v),\quad\forall (\lambda,t,q,v)\in\Lambda\times [0,\tau]\times B^n_{\iota}(0)\times B^n_{\rho_0}(0).
 \end{equation}
Since the closure of $B^n_{3\iota/4}(0)$ is a compact subset in $B^n_{\iota}(0)$,
and $\hat\Lambda$ is either compact or sequential compact,
 by Lemma~\ref{lem:Lagr} (or the proof of Lemma~\ref{lem:Lagr}(iv))
 there exists a constant $C'>0$ such that
    \begin{equation}\label{e:tildeLtwo2}
   \tilde{L}^{\ast\ast}(\lambda,t,q,v)\ge \kappa|v|^2-C',\quad\forall (\lambda, t, q, v)\in \hat\Lambda\times [0,\tau]\times \overline{B^n_{3\iota/4}(0)}\times\mathbb{R}^n.
    \end{equation}

\textsf{Step 2}. 
Take a smooth function $\Gamma:\R\to\R$ such that $\Gamma(s)=s$ for $s\le 1$ and 
that $\Gamma(s)$ is constant for $s\ge 2$.
Fix positive numbers $\rho>\rho_1$ and $\vartheta$ such that
$$
\vartheta\ge\max\{\tilde{L}^{\ast\ast}_\lambda(q, v)\,|\, (\lambda, q, v)\in \hat\Lambda\times [0,\tau]\times \overline{B^n_{3\iota/4}(0)}\times \bar{B}^n_\rho(0)\}.
$$
Define $\tilde{L}^{\ast\ast\ast}:\hat\Lambda\times [0,\tau]\times \overline{B^n_{3\iota/4}(0)}\times\mathbb{R}^n\to\R$ by
$\tilde{L}^{\ast\ast\ast}=\vartheta\Gamma(\tilde{L}^{\ast\ast}/\vartheta)$. 
Then the choice of $\vartheta$ implies
 \begin{equation}\label{e:tildeLtwo3}
 \tilde{L}^{\ast\ast\ast}(\lambda,t,q,v)=\tilde{L}^{\ast\ast}(\lambda,t,q,v),\quad\forall (\lambda,t,q,v)
  \in \hat\Lambda\times [0,\tau]\times \overline{B^n_{3\iota/4}(0)}\times B^n_\rho(0).
    \end{equation}
By (\ref{e:tildeLtwo2}), $\tilde{L}^{\ast\ast\ast}$
 is equal to  a constant $C''$ outside $\hat\Lambda\times [0,\tau]\times \overline{B^n_{3\iota/4}(0)}\times B^n_R(0)$ for a large $R>\rho$.
 Because of this fact and
   \begin{eqnarray}\label{e:tildeLtwo4}
\partial_{vv}\tilde{L}^{\ast\ast}(\lambda,t,q,v)[u,u]&=&\partial_{vv}\tilde{L}^{\ast}(\lambda,t,q,v)[u,u]+ 2\psi'_{\rho_0,\rho_1}(|v|^2)|u|^2\nonumber\\
&&+4\psi''_{\rho_0,\rho_1}(|v|^2)\bigl(({v},{u})_{\mathbb{R}^{n}}\bigr)^2
 \end{eqnarray}
  for each $(\lambda, t, q, v)\in \hat\Lambda\times [0,\tau]\times B^n_{\iota}(0)\times\R^n$, there exist positive constants $\Upsilon$ and $C_0'$ such that
\begin{eqnarray}
  &&\partial_{vv}\tilde{L}^{\ast\ast\ast}_\lambda(t, q,v)[u,u]\ge -\Upsilon |u|^2\quad\forall (\lambda,t, q,v,u)\in
  \hat\Lambda\times [0,\tau]\times \overline{B^n_{3\iota/4}(0)}\times\R^{n}, \label{e:positive1}\\
 && \tilde{L}^{\ast\ast\ast}(\lambda,t, q, v)\ge \kappa|v|^2-C'_0,\quad\forall
 (\lambda, t, q, v)\in\hat\Lambda\times [0,\tau]\times \overline{B^n_{3\iota/4}(0)}\times\R^{n}.
  \label{e:positive2}
\end{eqnarray}
 Choose a smooth function $\Xi:[0, \infty)\to\R$ such that:
\begin{center}
$\Xi'\ge 0$,
  $\Xi$ is convex on $[\rho_0^2, \infty)$,  vanishes in $[0, \rho_0^2)$, and is equal
to\\ the affine function $\Upsilon s+ \Theta$ on $[\rho^2, \infty)$, where $\Theta<0$ is a suitable constant.
\end{center}
(See \cite[Lemma~2.1]{Lu4} or \cite[\S5]{AbF}).
Define
\begin{equation}\label{e:tildeLtwo5}
\check{L}: \hat\Lambda\times [0,\tau]\times \overline{B^n_{3\iota/4}(0)}\times\R^{n}\to\R,\;(\lambda,t, q,v)\mapsto\tilde{L}^{\ast\ast\ast}(\lambda,t, q,v)+\Xi(|v|^2).
\end{equation}
Since $\tilde{L}^{\ast\ast\ast}=\vartheta\Gamma(\tilde{L}^{\ast\ast}/\vartheta)$, 
 it clearly satisfies (L0) by (\ref{e:tildeLtwo}).
(L1) may follow from (\ref{e:tildeLtwo1}), (\ref{e:tildeLtwo3})
 and the fact that $\Xi$ vanishes in $[0, \rho_0^2)$.
 (\ref{e:positive2}) leads to (L4) because $\Xi\ge 0$.

 Let us prove that $\check{L}$ satisfies (L2) and (L3).
\begin{enumerate}
\item[$\bullet$] 
 If $|v|<\rho$, by (\ref{e:tildeLtwo3}) $\check{L}(\lambda, t, q,v)=
 \tilde{L}^{\ast\ast\ast}(\lambda, t, q,v)+ \Xi(|v|^2)=
\tilde{L}^{\ast\ast}(\lambda, t, q,v)+ \Xi(|v|^2)$ and so
\begin{eqnarray*}
\partial_{vv}\check{L}(\lambda, t, q,v)[u,u]&=&\partial_{vv}\tilde{L}^{\ast\ast}(\lambda, t, q,v)[u,u]+ \partial_{vv}(\Xi(|v|^2))[u,u]\\
&=&\partial_{vv}\tilde{L}^{\ast}(\lambda,t,q,v)[u,u]+ 2\psi'_{\rho_0,\rho_1}(|v|^2)|u|^2\\
&&+4\psi''_{\rho_0,\rho_1}(|v|^2)\bigl(({v},{u})_{\mathbb{R}^{n}}\bigr)^2+ 2\psi'_{\rho_0,\rho_1}(|v|^2)|u|^2\\
&&+2\Xi'(|v|^2)|u|^2+4\Xi''(|v|^2)\bigl(({v},{u})_{\mathbb{R}^{n}}\bigr)^2
\end{eqnarray*}
because of (\ref{e:tildeLtwo4}) and the equality
$\partial_{vv}(\Xi(|v|^2))[u,u]=2\Xi'(|v|^2)|u|^2+4\Xi''(|v|^2)\bigl(({v},{u})_{\mathbb{R}^{n}}\bigr)^2$ for all $u\in\R^n$.
Recall that $\psi'_{\rho_0,\rho_1}\ge 0$, $\psi''_{\rho_0,\rho_1}\ge 0$, $\Xi'\ge 0$ and $\Xi''\ge 0$.
Since both $\partial_{vv}\check{L}(\lambda, t, q,v)$ and $\partial_{vv}\tilde{L}^{\ast\ast}(\lambda, t, q,v)$
depend continuously on $(\lambda, t, q, v)$ we deduce
\begin{equation}\label{e:positive3}
\partial_{vv}\check{L}(\lambda, t, q,v)\ge\partial_{vv}\tilde{L}^\ast(\lambda, t, q,v),\quad\forall (\lambda,t,q,v)
  \in \hat\Lambda\times [0,\tau]\times \overline{B^n_{3\iota/4}(0)}\times \overline{B^n_\rho(0)}.
\end{equation}

\item[$\bullet$] 
 If $|v|\ge\rho$ then (\ref{e:positive1}) may lead to
\begin{eqnarray}\label{e:tildeLtwo5+}
\partial_{vv}\check{L}(\lambda,t, q,v)[u,u]&=&\partial_{vv}\tilde{L}^{\ast\ast\ast}(\lambda,t, q,v)[u,u]+2\Xi'(|v|^2)|u|^2\nonumber\\
&\ge& \Upsilon|u|^2,\quad\forall u\in\R^n.
\end{eqnarray}
By Proposition~\ref{th:bif-LagrGenerEu}(c),  $\tilde L^\ast_\lambda(t, q, v)$ is
strictly convex in $v$. (L2) may follow from (\ref{e:tildeLtwo5+}) and (\ref{e:positive3}) because
 $\hat\Lambda\times [0,\tau]\times \overline{B^n_{3\iota/4}(0)}\times \overline{B^n_\rho(0)}$ is either compact or sequential compact.
Using the same reason we obtain that $\check{L}$ satisfies (L3) because
$$
\check{L}(\lambda,t, q,v)=\tilde{L}^{\ast\ast\ast}(\lambda, t, q,v)+\Xi(|v|^2)=C''+ \Upsilon |v|^2+ \Theta\quad\forall|v|>R.
$$
\end{enumerate}

Finally, since $\partial_{q}\check{L}(\lambda,t,0,0)$ and $\partial_{v}\check{L}(\lambda,t,0,0)$ are bounded
using the Taylor formula (L5) and (L6) easily follows from (L3). 
\end{proof}

 Consider the Banach subspace
 $$
 {\bf X}_{V_0\times V_1}:=\left\{\xi\in C^{1}([0,\tau]; \mathbb{R}^n)\,|\, (\xi(0),\xi(\tau))\in V_0\times V_1\right\}
 $$
 of  $C^1([0,\tau],\mathbb{R}^n)$, and
the Hilbert subspace
$$
{\bf H}_{V_0\times V_1}:=\left\{\xi\in W^{1,2}([0,\tau]; \mathbb{R}^n)\,|\, (\xi(0),\xi(\tau))\in V_0\times V_1\right\}
$$
of $W^{1,2}([0,\tau]; \mathbb{R}^n)$.  The spaces ${\bf H}_{V_0\times V_1}$ and ${\bf X}_{V_0\times V_1}$ have the following open subsets
\begin{eqnarray*}
\mathcal{U}:&=&W^{1,2}_{V_0\times V_1}\left([0,\tau]; B^n_{\iota/2}(0)\right)=\left\{\xi\in W^{1,2}\left([0,\tau];
B^n_{\iota/2}(0)\right)\,\big|\,(\xi(0),\xi(\tau))\in V_0\times V_1 \right\},\\
\mathcal{U}^X:&=&\mathcal{U}\cap{\bf X}_{V_0\times V_1}=C^1_{V_0\times V_1}([0,\tau]; B^n_{\iota/2}(0))
\end{eqnarray*}
respectively. Define a family of functionals $\check{\mathcal{E}}_\lambda:\mathcal{U}\to\R$ given by
\begin{equation}\label{e:checkEu}
\check{\mathcal{E}}_{\lambda}(x)=\int^{\tau}_0\check{L}_\lambda(t, x(t),\dot{x}(t))dt,\quad\lambda\in\hat\Lambda.
\end{equation}
Since  $\tilde{\mathcal{E}}^\ast_\lambda$ is defined on $C^1_{V_0\times V_1}([0,\tau]; B^n_{\iota}(0))$ and
$\mathcal{U}^X=C^1_{V_0\times V_1}([0,\tau]; B^n_{\iota/2}(0))$ is an open neighborhood  of
$0\in C^1_{V_0\times V_1}([0,\tau]; B^n_{\iota}(0))$, by (L1) in Lemma~\ref{lem:Gener-modif} we obtain
\begin{equation}\label{e:twofunctAgree}
\tilde{\mathcal{E}}^\ast_\lambda=\check{\mathcal{E}}_{\lambda}|_{\mathcal{U}^X}\quad\hbox{in}\quad
\{x\in\mathcal{U}^X\,|\,\|x\|_{C^1}<\rho_0\}\subset\mathcal{U}^X,
\end{equation}
and therefore  the following (\ref{e:MorseindexEu}).

\begin{proposition}\label{prop:funct-analy}
\begin{enumerate}
\item[\rm (i)] Each $\check{\mathcal{E}}_{\lambda}$ is
$C^{2-0}$ and twice G\^ateaux-differentiable, and $d\check{\mathcal{E}}_{\lambda}(0)=0$ and
 \begin{equation}\label{e:MorseindexEu}
 m^\star(\tilde{\mathcal{E}}^\ast_{\lambda}, 0)=m^\star(\check{\mathcal{E}}_{\lambda}|_{\mathcal{U}^X},0)=
m^\star(\check{\mathcal{E}}_{\lambda},0),\quad\star=-,0.
\end{equation}
\item[\rm (ii)] Each critical point of $\check{\mathcal{E}}_\lambda$ sits in
$C^2\left([0,\tau]; B^n_{\iota/2}(0)\right)\cap\mathcal{U}^X$, and satisfies
the boundary problem:
\begin{eqnarray}\label{e:LagrGenerEuTwo}
\left.\begin{array}{ll}
\frac{d}{dt}\big(\partial_v\check{L}_\lambda(t, x(t), \dot{x}(t))\big)-\partial_q \check{L}_\lambda(t, x(t), \dot{x}(t))=0,\\
(x(0), x(\tau))\in V_0\times V_1,\\
\partial_v\check{L}_\lambda(0, x(0), \dot{x}(0))[v_0]=0\quad\forall v_0\in V_0,\\
\partial_v\check{L}_\lambda(\tau, x(\tau), \dot{x}(\tau))[v_1]=0\quad\forall v_1\in V_1.
\end{array}\right\}
\end{eqnarray}

\item[\rm (iii)]  The gradient of $\check{\mathcal{E}}_{\lambda}$ at $x\in\mathcal{U}$, denoted by
 $\nabla\check{\mathcal{E}}_{\lambda}(x)$, is given by
\begin{eqnarray}\label{e:4.14}
\nabla\check{\mathcal{E}}_{\lambda}(x)(t)&=&e^t\int^t_0\left[
e^{-2s}\int^s_0e^{r}f_{\lambda,x}(r)dr\right]ds  + c_1(\lambda,x)e^t+
c_2(\lambda,x)e^{-t}\nonumber\\
&&\qquad +\int^t_0 \partial_v \check{L}_\lambda(s, x(s),\dot{x}(s))ds ,
\end{eqnarray}
where $c_1(\lambda,x), c_2(\lambda,x)\in\R^n$ are suitable constant vectors and %$f_\lambda(t)$ is
\begin{eqnarray}\label{e:4.13}
 f_{\lambda,x}(t)= - \partial_q \check{L}_\lambda(t,x(t),\dot{x}(t))+ \int^t_0
\partial_v \check{L}_\lambda(s, x(s),\dot{x}(s))ds.
\end{eqnarray}
\item[\rm (iv)] $\nabla\check{\mathcal{E}}_{\lambda}$ restricts to a $C^1$ map $A_\lambda$
from $\mathcal{U}^X$ to ${\bf X}_{V_0\times V_1}$.
\item[\rm (v)] $\nabla\check{\mathcal{E}}_{\lambda}$
 has the G\^ateaux derivative  $B_\lambda(\zeta)\in{\mathscr{L}}_s({\bf H}_{V_0\times V_1})$ at $\zeta\in\mathcal{U}$ given by
 \begin{eqnarray}\label{e:gradient4Lagr+}
 (B_\lambda(\zeta)\xi,\eta)_{1,2}
   = \int_0^{\tau} \Bigl(\!\! \!\!\!&&\!\!\!\!\!\partial_{vv}
     \check{L}_{{\lambda}}\bigl(t, \zeta(t),\dot{\zeta}(t)\bigr)
\bigl[\dot{\xi}(t), \dot{\eta}(t)\bigr]
+ \partial_{qv}
  \check{L}_{{\lambda}}\bigl(t, \zeta(t), \dot{\zeta}(t)\bigr)
\bigl[\xi(t), \dot{\eta}(t)\bigr]\nonumber \\
&& + \partial_{vq}
  \check{L}_{{\lambda}}\bigl(t,\zeta(t),\dot{\zeta}(t)\bigr)
\bigl[\dot{\xi}(t), \eta(t)\bigr] \nonumber\\
&&+  \partial_{qq} \check{L}_{{\lambda}}\bigl(t,\zeta(t),
\dot{\zeta}(t)\bigr) \bigl[\xi(t), \eta(t)\bigr]\Bigr) \, dt
\end{eqnarray}
for any  $\xi,\eta\in {\bf H}_{V_0\times V_1}$.
$B_\lambda(\zeta)$  is a self-adjoint Fredholm operator and
  has a decomposition
${B}_{\lambda}(\zeta)=\textsl{{P}}_{{\lambda}}(\zeta)+\textsl{{Q}}_{{\lambda}}(\zeta)$, where
$\textsl{P}_{{\lambda}}(\zeta)\in{\mathscr{L}}_s({\bf H}_{V_0\times V_1})$ is a positive
definitive linear operator defined by
\begin{eqnarray}\label{e:3.17}
(\textsl{P}_{{\lambda}}(\zeta)\xi, \eta)_{1,2}
   = \int_0^\tau \big(\partial_{vv}\check{L}_{{\lambda}}\bigl(t,\zeta(t),\dot\zeta(t)\bigr)
[\dot\xi(t), \dot\eta(t)]+ \bigl(\xi(t), \eta(t)\bigr)_{\R^n}\big)
\, dt,
\end{eqnarray}
and $\textsl{Q}_{{\lambda}}(\zeta)\in\check{\mathscr{L}}_s({\bf H}_{V_0\times V_1})$ 
is a compact self-adjoint linear operator.
Moreover, (L2) in Lemma~\ref{lem:Gener-modif} implies that  $(\textsl{P}_{{\lambda}}(\zeta)\xi, \xi)_{1,2}\ge\min\{\check{c},1\}\|\xi\|_{1,2}^2$
for all $\zeta\in \mathcal{U}$ and $\xi\in{\bf H}_{V_0\times V_1}$.
 \item[\rm (vi)] If $(\lambda_k)\subset\hat\Lambda$ and $(\zeta_k)\subset\mathcal{U}$ converge to $\mu\in\hat\Lambda$ and $0$, respectively,
then $\|\textsl{P}_{{\lambda_k}}(\zeta_k)\xi-\textsl{P}_{{\mu}}(0)\xi\|_{1,2}\to 0$ for each $\xi\in {\bf H}_{V_0\times V_1}$.
\item[\rm (vii)] $\mathcal{U}\ni\zeta\mapsto \textsl{Q}_{{\lambda}}(\zeta)\in\mathscr{L}_s({\bf H}_{V_0\times V_1})$
is uniformly continuous at $0\in\mathcal{U}$ with respect to $\lambda\in\hat\Lambda$ and
$\|\textsl{Q}_{\lambda_k}(0)-\textsl{Q}_{\mu}(0)\|\to 0$ as $(\lambda_k)\subset\Lambda$  converges to $\mu\in\hat\Lambda$.

\item[\rm (viii)] 
$\{\xi\in {\bf H}_{V_0\times V_1}\mid  B_\lambda(0)\xi=s\xi,\, s\le 0\}\subset{\bf X}_{V_0\times V_1}$ and 
$$
\{\xi\in {\bf H}_{V_0\times V_1}\mid  B_\lambda(0)\xi\in {\bf X}_{V_0\times V_1}\}\subset {\bf X}_{V_0\times V_1}.
$$
\end{enumerate}
\end{proposition}
\begin{proof}[\bf Proof]
(i) is obtained by  \cite[\S4]{Lu4} or \cite[\S3]{Lu5} and \cite{Lu1}.
(ii) follows from \cite[Theorem~4.5]{BuGiHi} because of conditions (L0), (L2) and (L4)-(L6) in Lemma~\ref{lem:Gener-modif}.
(iii) is obtained by  (4.13) and (4.14) in \cite{Lu4}.
(iv) and (v) are proved in \cite[\S4]{Lu4}.

In order to prove (vi), by (\ref{e:3.17}) we have
\begin{eqnarray*}
\|[\textsl{P}_{{\lambda_k}}(\zeta_k)-\textsl{P}_{{\mu}}(0)]\xi\|^2_{1,2}
  \le \int_0^\tau \big|[\partial_{vv}\check{L}_{{\lambda_k}}\bigl(t,\zeta_k(t),\dot\zeta_k(t)\bigr)-
  \partial_{vv}\check{L}_{\mu}\bigl(t, 0,0\bigr)]
\dot\xi(t)\big|^2_{\mathbb{R}^n}\, dt.
\end{eqnarray*}
Note that $\|\zeta_k\|_{1,2}\to 0$ implies $\|\zeta_k\|_{C^0}\to 0$.
Since $(\lambda, t, x,v)\mapsto\partial_{vv}\check{L}_{{\lambda}}(t, x,v)$ is continuous, by
the third inequality in (L3) in Lemma~\ref{lem:Gener-modif} we may apply \cite[Prop.~B.9]{Lu10} (\cite[Prop.~C.1]{Lu9}) to
$$
f(t,\eta;\lambda)=\partial_{vv}\check{L}(\lambda, t, \zeta_k(t),\dot\zeta_k(t))\eta
$$
 to get that
$$
\int_0^\tau \big|[\partial_{vv}\check{L}_{{\lambda_k}}\bigl(t,\zeta_k(t),\dot\zeta_k(t)\bigr)-
\partial_{vv}\check{L}_{{\mu}}\bigl(t, 0, 0\bigr)]
\dot\xi(t)\big|^2_{\mathbb{R}^n}\, dt\to 0.
$$
Moreover, the Lebesgue dominated convergence theorem also leads to
$$
\int_0^\tau \big|[\partial_{vv}\check{L}_{{\lambda_k}}\bigl(t, 0, 0\bigr)-
\partial_{vv}\check{L}_{{\mu}}\bigl(t, 0, 0\bigr)]
\dot\xi(t)\big|^2_{\mathbb{R}^n}\, dt\to 0.
$$
Hence $\|[\textsl{P}_{{\lambda_k}}(\zeta_k)-\textsl{P}_{\mu}(0)]\xi\|_{1,2}\to 0$.

Next, let us prove (vii).
Write $\textsl{Q}_\lambda(\zeta):=\textsl{Q}_{\lambda,1}(\zeta)+ \textsl{Q}_{\lambda,2}(\zeta)+
\textsl{Q}_{\lambda,3}(\zeta)$, where
\begin{eqnarray*}
&&(\textsl{Q}_{\lambda,1}(\zeta)\xi,  \eta)_{1,2}
   = \int_0^{\tau}\partial_{vq}   \check{L}_\lambda\big(t, \zeta(t),\dot{\zeta}(t)\big)
[\dot{\xi}(t), \eta(t)]dt,\\
 &&(\textsl{Q}_{\lambda,2}(\zeta)\xi,
  \eta)_{1,2}= \int_0^{\tau}\partial_{qv}\check{L}_\lambda\big(t, \zeta(t), \dot{\zeta}(t)\big)
[\xi(t), \dot{\eta}(t)]dt, \\
&&(\textsl{Q}_{\lambda,3}(\zeta)\xi,
  \eta)_{1,2} = \int_0^{\tau} \Bigl(\partial_{qq}\check{L}_\lambda\big(t, \zeta(t), \dot{\zeta}(t)\big) [\xi(t), \eta(t)]-
  \bigl(\xi(t), \eta(t)\bigr)_{\mathbb{R}^n}\Bigr) \, dt.
\end{eqnarray*}
As above the first claim follows from (L3) in Lemma~\ref{lem:Gener-modif} and \cite[Prop.~B.9]{Lu10} (\cite[Prop.~C.1]{Lu9}) directly.

In order to prove the second claim, as in the proof of \cite[page 571]{Lu1} we have
\begin{eqnarray*}
&&\|\textsl{Q}_{\lambda_k,1}(0)-\textsl{Q}_{\mu,1}(0)\|_{\mathscr{L}({\bf H})}\\
&\le&2(e^{\tau}+ 1)\left(\int^{\tau}_{0}\left|\partial_{vq}
  \check{L}_{\lambda_k}(s, 0, 0)-\partial_{vq}
  \check{L}_\mu(s, 0, 0)\right|^2  ds\right)^{1/2}.
\end{eqnarray*}
Because of the second inequality in ({\bf L2}), it follows from the Lebesgue dominated convergence theorem
that $\|\textsl{Q}_{\lambda_k,1}(0)-\textsl{Q}_{\mu,1}(0)\|_{\mathscr{L}({\bf H})}\to 0$.
Observe that
$(\textsl{Q}_{\lambda,2}(\zeta)\xi, \eta)_{1,2}
   =\bigr(\xi,  (\textsl{Q}_{\lambda,1}(\zeta))^\ast\eta\bigl)_{1,2}$.
Hence $\|\textsl{Q}_{\lambda_k,2}(0)-\textsl{Q}_{\mu,2}(0)\|_{\mathscr{L}({\bf H})}\to 0$.
Finally, it is easy to deduce that 
\begin{eqnarray*}
\|\textsl{Q}_{\lambda_k,3}(0)-\textsl{Q}_{\mu,3}(0)\|^2_{\mathscr{L}({\bf H})}\le
  \int_0^{\tau} \bigl|\partial_{qq}\check{L}_{\lambda_k}\big(t, 0, 0\big)
-\partial_{qq}\check{L}_\mu\big(t, 0, 0\big)\bigr|^2\, dt.
\end{eqnarray*}
By the Lebesgue dominated convergence theorem the right side converges to zero.
Then $\|\textsl{Q}_{\lambda_k,3}(0)-\textsl{Q}_{\mu,3}(0)\|_{\mathscr{L}({\bf H})}\to 0$ and therefore
 $\|\textsl{Q}_{\lambda_k}(0)-\textsl{Q}_{\mu}(0)\|\to 0$.

Finally, we prove (viii). Below \cite[(4.17)]{Lu4} we concluded that
these can be proved by almost repeating the arguments in \cite[\S3]{Lu1}.
The case $s<0$ in the first claim was proved in \cite[p.178]{Du}.
Here we present an improvement of the method from \cite[\S3]{Lu1}, 
which can be effectively applied to the case considered in Section~\ref{sec:LagrPPerio1}.
 Define 
\begin{align*}
&\mathfrak{P}(t)=\partial_{vv}\check{L}_{{\lambda}}\bigl(t, 0, 0\bigr), \quad
\mathfrak{Q}(t)=D_{{q}{v}}\check{L}_\lambda\big(t,0,0\big), \quad
\mathfrak{R}(t)=D_{{q}{q}}\check{L}_\lambda\big(t,0,0\big), \\
&\mathfrak{L}(t,\tilde{y},\tilde{v})=\frac{1}{2}\mathfrak{P}(t)\tilde{v}\cdot\tilde{v}+
\mathfrak{Q}(t)\tilde{y}\cdot\tilde{v}+\frac{1}{2}\mathfrak{R}(t)\tilde{y}\cdot\tilde{y},
\end{align*}
which are~$C^2$, and a functional~$\mathfrak{F}$ on~${\bf H}_{V_0\times V_1}$ by
$$
\mathfrak{F}(\tilde{y})=\int\limits_{0}^{\tau}
\mathfrak{L}\big(t,\tilde{y}(t),\dot{\tilde{y}}(t)\big)dt,\quad\forall\tilde{y}\in{\bf H}_{V_0\times V_1}. 
$$
Then $\mathfrak{F}$  is $C^{2}$-smooth on ${\bf H}_{V_0\times V_1}$ (and hence
on ${\bf X}_{V_0\times V_1}$),  and $\tilde{y}=0\in{\bf H}_{V_0\times V_1}$ 
 is a critical point of $\mathfrak{F}$.   It is also easily checked that
 \begin{align}\label{e:3.17+}
d\mathfrak{F}(\tilde{y})[\tilde{z}]=
d^{2}\mathfrak{F}(0)(\tilde{y}, \tilde{z})=(B_\lambda(0)\tilde{y}, \tilde{z})_{1,2},\quad
\forall \tilde{y}, \tilde{z}\in {\bf H}_{V_0\times V_1}.
\end{align}
 Therefore, $\tilde{y} \in \mathbf{H}_{V_0 \times V_1}$ belongs to $\ker B_\lambda(0)$~if and only if~$\tilde{y}$ is a critical point of~$\mathfrak{F}$, which means it satisfies (in the distributional sense) the differential equation given in the first line of the boundary problem: 
$$
\left.\begin{array}{l}
\frac{d}{dt}\big(\partial_v\mathfrak{L}(t, x(t), \dot{x}(t))\big)-\partial_q \mathfrak{L}(t, x(t), \dot{x}(t))=0,\\
(x(0), x(\tau))\in V_0\times V_1,\\
\partial_v\mathfrak{L}(0, x(0), \dot{x}(0))[v_0]=0\quad\forall v_0\in V_0,\\
\partial_v\mathfrak{L}(\tau, x(\tau), \dot{x}(\tau))[v_1]=0\quad\forall v_1\in V_1.
\end{array}\right\}
$$
Of course, $\tilde{y}$ is automatically $C^2$ and hence satisfies the above boundary problem in the classical sense (cf.~\cite[p.~178]{Du}).
 In particular, $\tilde{y}\in {\bf X}_{V_0\times V_1}$. 
 (As shown in \cite[p.~178]{Du}, such an argument is also applicable to the case considered in Section~\ref{sec:LagrPPerio1}.)
 
 The case where $s<0$ in the first claim can be reduced the case $s=0$.
 Assume that $\xi\in {\bf H}_{V_0\times V_1}$ satisfies $B_\lambda(0)\xi=s\xi$ for some $s<0$,
 and let the functional $\mathfrak{F}$ and the Lagrangian $\mathfrak{L}$ be as above.
 Define a new  Lagrangian~$\widehat{\mathfrak{L}}$ and the corresponding functional
 ~$\widehat{\mathfrak{F}}$ on~${\bf H}_{V_0\times V_1}$ by 
 \begin{align*}
\widehat{\mathfrak{L}}(t,\tilde{y},\tilde{v})&= \mathfrak{L}(t,\tilde{y},\tilde{v})-\frac{s}{2}(\tilde{v}, \tilde{v})_{1,2},\\
\widehat{\mathfrak{F}}(\tilde{y})&=\int\limits_{0}^{\tau}
\widehat{\mathfrak{L}}\big(t,\tilde{y}(t),\dot{\tilde{y}}(t)\big)dt,\quad\forall\tilde{y}\in{\bf H}_{V_0\times V_1}. 
\end{align*}
As above, we have $d\widehat{\mathfrak{F}}(\tilde{y})[\tilde{z}]=
(\widehat{B}_\lambda(0)\tilde{y}, \tilde{z})_{1,2}$ for all $\tilde{y}, \tilde{z}\in {\bf H}_{V_0\times V_1}$.
Let ${\rm I}_{{\bf H}_{V_0\times V_1}}$ be the identity operator on ${\bf H}_{V_0\times V_1}$,
and let $I_n$ be the identity matrix of order $n$. 
 It is easily checked that $\widehat{B}_\lambda(0)={B}_\lambda(0)-s{\rm I}_{{\bf H}_{V_0\times V_1}}$ and that
$$
 \widehat{\mathfrak{L}}(t,\tilde{y},\tilde{v})=\frac{1}{2}\widehat{\mathfrak{P}}(t)\tilde{v}\cdot\tilde{v}+
\widehat{\mathfrak{Q}}(t)\tilde{y}\cdot\tilde{v}+\frac{1}{2}\widehat{\mathfrak{R}}(t)\tilde{y}\cdot\tilde{y},
 $$
 where $\widehat{\mathfrak{P}}(t)={\mathfrak{P}}(t)-sI_n$, $\widehat{\mathfrak{Q}}(t)=
 {\mathfrak{Q}}(t)$ and $\widehat{\mathfrak{R}}(t)={\mathfrak{R}}(t)$.
 The key is that all $\widehat{\mathfrak{P}}(t)$ are positive definite because $s<0$.
 Therefore, by the above conclusion for $s=0$, we obtain $\ker~\widehat{B}_\lambda(0)\subset
 {\bf X}_{V_0\times V_1}$ and hence $\xi\in  {\bf X}_{V_0\times V_1}$.
 
  What remains is to prove the second inclusion (relation).
Let $\xi\in {\bf H}_{V_0\times V_1}$ be such that $\zeta:= B_\lambda(0)\xi\in {\bf X}_{V_0\times V_1}$.
By (\ref{e:3.17+}), $\nabla\mathfrak{F}(\xi)=B_\lambda(0)\xi=\zeta$.
 Applying the conclusions in (iii) to $\mathfrak{F}$, we arrive at
 \begin{eqnarray}\label{e:3.17++}
\zeta(t)=\nabla\mathfrak{F}(\xi)(t)&=&e^t\int^t_0\left[
e^{-2s}\int^s_0e^{r}\hat{f}_{\lambda,\xi}(r)dr\right]ds  + c_1(\lambda,\xi)e^t+
c_2(\lambda,\xi)e^{-t}\nonumber\\
&&\qquad +\int^t_0 \partial_v \mathfrak{L}(s, \xi(s),\dot{\xi}(s))ds ,
\end{eqnarray}
where $c_1(\lambda,\xi), c_2(\lambda,\xi)\in\R^n$ are suitable constant vectors and 
\begin{eqnarray*}
 \hat{f}_{\lambda,\xi}(t)= - \partial_q \mathfrak{L}(t,\xi(t),\dot{\xi}(t))+ \int^t_0
\partial_v \mathfrak{L}(s, \xi(s),\dot{\xi}(s))ds.
\end{eqnarray*}
 Differentiating both sides of equation (\ref{e:3.17++}) with respect to $t$ yields
   \begin{eqnarray*}
\dot{\zeta}(t)&=&e^t\int^t_0\left[
e^{-2s}\int^s_0e^{r}\hat{f}_{\lambda,\xi}(r)dr\right]ds+ e^{-t}\int^t_0e^{r}\hat{f}_{\lambda,\xi}(r)dr    \nonumber\\
&&\qquad + c_1(\lambda,\xi)e^t
-c_2(\lambda,\xi)e^{-t}+ \partial_v \mathfrak{L}(t, \xi(t),\dot{\xi}(t)).
\end{eqnarray*}
Since $\dot{\zeta}(t)$ is continuous, this equation implies that 
 $\partial_v \mathfrak{L}(t, \xi(t),\dot{\xi}(t))$ is continuous with respect to $t$.
But $\partial_v\mathfrak{L}(t,\xi(t),\dot{\xi}(t))=\mathfrak{P}(t)\dot{\xi}(t)+\mathfrak{Q}(t)\xi(t)$,
and so
$$
\dot{\xi}(t)=(\mathfrak{P}(t))^{-1}\partial_v\mathfrak{L}(t,\xi(t),\dot{\xi}(t))-
(\mathfrak{P}(t))^{-1}\mathfrak{Q}(t)\xi(t).
$$
It follows that $\dot{\xi}(t)$ is continuous in $t$ and hence $\xi\in {\bf X}_{V_0\times V_1}$.
The proof of (viii) is complete.
\end{proof}

In order to apply our abstract theory in \cite{Lu8, Lu10, Lu10} to the family of functionals in
(\ref{e:checkEu}) we also need two results.

\begin{proposition}\label{prop:continA}
Both maps $\hat\Lambda\times \mathcal{U}^X\ni (\lambda, x)\mapsto
\check{\mathcal{E}}_{\lambda}(x)\in\R$ and
$\hat\Lambda\times \mathcal{U}^X\ni (\lambda, x)\mapsto A_\lambda(x)\in{\bf X}_{V_0\times V_1}$ are continuous.
\end{proposition}
\begin{proof}[\bf Proof]
\textsf{{Step 1}}({\it Prove that $\hat\Lambda\times \mathcal{U}\ni (\lambda, x)\mapsto
\check{\mathcal{E}}_{\lambda}(x)\in\R$ is continuous, and therefore obtain the first claim}).
   Indeed, for any two points $(\lambda, x)$ and $(\lambda_0, x_0)$ in $\hat\Lambda\times \mathcal{U}$ we can write
\begin{eqnarray*}
\check{\mathcal{E}}_{\lambda}(x)-\check{\mathcal{E}}_{\lambda_0}(x_0)&=&
\left[\int^{\tau}_0\check{L}_\lambda(t, x(t),\dot{x}(t))dt-
\int^{\tau}_0\check{L}_{\lambda}(t, x_0(t),\dot{x}_0(t))dt\right]\\
&&+\left[\int^{\tau}_0\check{L}_\lambda(t, x_0(t),\dot{x}_0(t))dt-
\int^{\tau}_0\check{L}_{\lambda_0}(t, x_0(t),\dot{x}_0(t))dt\right].
\end{eqnarray*}
As $(\lambda, x)\to (\lambda_0, x_0)$, we derive from (L6) in Lemma~\ref{lem:Gener-modif}  and \cite[Prop.B.9]{Lu10}
or \cite[Proposition C.1]{Lu9} (resp. (L6) in Lemma~\ref{lem:Gener-modif}
 and the Lebesgue dominated convergence theorem)  that the first (resp. second) bracket on the right side converges to the zero.

\textsf{{Step 2}}({\it Prove that $\hat\Lambda\times \mathcal{U}^X\ni (\lambda, x)\mapsto
A_\lambda(x)\in{\bf X}_{V_0\times V_1}$ is continuous}).
By  \cite[(4.14)]{Lu4} we have
\begin{eqnarray}\label{e:4.16}
\frac{d}{dt}\nabla\check{\mathcal{E}}_{\lambda}(x)(t)&=&
e^t\int^t_0\left[ e^{-2s}\int^s_0e^{r}f_{\lambda,x}(r)dr\right]ds
+e^{-t}\int^t_0e^{r}f_{\lambda,x}(r)dr
\nonumber\\
&&\quad + c_1(\lambda,x)e^t -c_2(\lambda,x)e^{-t}+ \partial_{v} \check{L}_\lambda\left(t, x(t),\dot{x}(t)\right).
\end{eqnarray}
This and (\ref{e:4.14}) lead to
\begin{eqnarray}\label{e:4.17}
 2c_1(\lambda,x)e^t&=&-2e^t\int^t_0\left[e^{-2s}\int^s_0e^{r}f_{\lambda,x}(r)dr\right]ds
 -e^{-t}\int^t_0e^{r}f_{\lambda,x}(r)dr \nonumber\\
&&-\int^t_0 \partial_v \check{L}_\lambda(s, x(s),\dot{x}(s))ds- \partial_{v} \check{L}_\lambda\left(t, x(t),\dot{x}(t)\right)\nonumber \\
&& +\nabla\check{\mathcal{E}}_{\lambda}(x)(t)+ \frac{d}{dt}\nabla\check{\mathcal{E}}_{\lambda}(x)(t)
\end{eqnarray}
and
\begin{eqnarray}\label{e:4.18}
 2c_2(\lambda,x)e^{-t}&=& e^{-t}\int^t_0e^{r}f_{\lambda,x}(r)dr
-\int^t_0 \partial_v \check{L}_\lambda(s, x(s),\dot{x}(s))ds\nonumber\\
&&+ \partial_{v} \check{L}_\lambda\left(t, x(t),\dot{x}(t)\right)
+\nabla\check{\mathcal{E}}_{\lambda}(x)(t)-\frac{d}{dt}\nabla\check{\mathcal{E}}_{\lambda}(x)(t).
\end{eqnarray}
Moreover, since
\begin{eqnarray*}
&&d\check{\mathcal{E}}_{\lambda_1}(x)[\xi]-d\check{\mathcal{E}}_{\lambda_2}(y)[\xi]\\
&=&\int_0^{\tau}  \big( \partial_{q}\check{L}(\lambda_1, t, x(t),\dot{x}(t))-\partial_{q}\check{L}(\lambda_2, t, y(t),\dot{y}(t))\big)\cdot\xi(t)dt\\
&&+ \int_0^{\tau}\big(\partial_{v}\check{L}\big(\lambda_1, t, x(t),\dot{x}(t))-\partial_{v}\check{L}(\lambda_2, t, y(t),\dot{y}(t))\big)\cdot\dot{\xi}(t)
\big) \, dt,
\end{eqnarray*}
we have 
\begin{eqnarray*}
&&\|\nabla\check{\mathcal{E}}_{\lambda_1}(x)-\nabla\check{\mathcal{E}}_{\lambda_2}(y)\|_{1,2}\\
&\le&\left(\int_0^{\tau}  \left| \partial_{q}\check{L}(\lambda_1, t, x(t),\dot{x}(t))-\partial_{q}\check{L}(\lambda_2, t, y(t),\dot{y}(t))\right|^2dt\right)^{1/2}\\
&&+ \left(\int_0^{\tau}\left|\partial_{v}\check{L}(\lambda_1, t, x(t),\dot{x}(t))-\partial_{v}\check{L}(\lambda_2, t, y(t),\dot{y}(t))\right|^2dt\right)^{1/2}.
\end{eqnarray*}
Fix a point $(\lambda_1, x)\in\hat\Lambda\times\mathcal{U}^X$. Then $\{(\lambda_1, t, x(t), \dot{x}(t))\,|\, t\in [0,\tau]\}$
is a compact subset of $\hat\Lambda\times [0,\tau]\times B^n_{\iota/2}(0)\times\R^n$. Since
$\partial_{q}\check{L}$ and $\partial_{v}\check{L}$ are uniformly continuous in any compact neighborhood of this compact subset we deduce:
If $(\lambda_2, y)\in\hat\Lambda\times\mathcal{U}^X$ converges to $(\lambda_1, x)$ in $\hat\Lambda\times\mathcal{U}^X$, then
$$
\|\nabla\check{\mathcal{E}}_{\lambda_1}(x)-
\nabla\check{\mathcal{E}}_{\lambda_2}(y)\|_{1,2}\to 0\quad\hbox{and so}\quad
\|\nabla\check{\mathcal{E}}_{\lambda_1}(x)-\nabla\check{\mathcal{E}}_{\lambda_2}(y)\|_{C^0}\to 0.
$$
This fact, (\ref{e:4.13}) and (\ref{e:4.17})--(\ref{e:4.18}) imply that
$c_1(\lambda,x)$ and $c_2(\lambda,x)$ are continuous in $\hat\Lambda\times\mathcal{U}^X$.
From the latter claim, (\ref{e:4.14})--(\ref{e:4.13}) and (\ref{e:4.16}), it easily follows that as
 $(\lambda_2, y)\in\hat\Lambda\times\mathcal{U}^X$ converges to $(\lambda_1, x)$ in $\hat\Lambda\times\mathcal{U}^X$,
$$
\left\|\frac{d}{dt}\nabla\check{\mathcal{E}}_{\lambda_1}(x)-\frac{d}{dt}\nabla\check{\mathcal{E}}_{\lambda_2}(y)\right\|_{C^0}\to 0
$$
and hence $\|\nabla\check{\mathcal{E}}_{\lambda_1}(x)-\nabla\check{\mathcal{E}}_{\lambda_2}(y)\|_{C^1}\to 0$.
\end{proof}

\begin{proposition}\label{prop:solutionLagr}
For any given $\epsilon>0$ there exists $\varepsilon>0$ such that
if a critical point $x$ of $\check{\mathcal{E}}_\lambda$ satisfies
$\|x\|_{1,2}<\varepsilon$ then $\|x\|_{C^2}<\epsilon$. (\textsf{Note}: $\varepsilon$ is independent of $\lambda\in\hat\Lambda$.)
Consequently, if $0\in \mathcal{U}^X$ is an isolated critical point of $\check{\mathcal{E}}_{\lambda}|_{\mathcal{U}^X}$
then $0\in \mathcal{U}$ is also an isolated critical point of $\check{\mathcal{E}}_{\lambda}$.
\end{proposition}

By Proposition~\ref{prop:funct-analy}(ii),
$$
{\rm Crit}(\check{\mathcal{E}}):=\{(\lambda, x)\in\hat\Lambda\times\mathcal{U}\,|\, d\check{\mathcal{E}}_\lambda(x)=0\}\subset
\hat\Lambda\times C^2\big([0,\tau]; B^n_{\iota/2}(0)\big)\cap\mathcal{U}^X.
$$
Proposition~\ref{prop:solutionLagr} claims that $\hat\Lambda\times\mathcal{U}$ and
$\hat\Lambda\times C^2\big([0,\tau]; B^n_{\iota/2}(0)\big)\cap\mathcal{U}^X$ induce the equivalence topologies on
${\rm Crit}(\check{\mathcal{E}})$.

\begin{proof}[\bf Proof of Proposition~\ref{prop:solutionLagr}]
The second claim  may follow from the first one by contradiction.
Let us prove the first one. By Proposition~\ref{prop:funct-analy}(ii), $x$ is $C^2$.
Let $c_1(\lambda, x)$ and $c_2(\lambda,x)$ be given by (\ref{e:4.17}) and (\ref{e:4.18}), respectively.

\textsf{Step 1} ({\it Prove that both $|c_1(\lambda, x)-c_1(\lambda, 0)|$ and
$|c_2(\lambda, x)-c_2(\lambda, 0)|$   uniformly converge to zero 
in $\lambda\in\hat\Lambda$ as $\|x\|_{1,2}\to 0$}).

Since $\nabla\check{\mathcal{E}}_{\lambda}(x)(t)\equiv 0$ and $\frac{d}{dt}\nabla\check{\mathcal{E}}_{\lambda}(x)(t)\equiv 0$,
by (\ref{e:4.17}) for any $t\in [0,\tau]$ we have
\begin{eqnarray*}
&&2|c_1(\lambda,x)-c_1(\lambda,0)|\le 2|c_1(\lambda,x)e^t-c_1(\lambda,0)e^t|\\
&\le&2e^t\int^t_0\left[e^{-2s}\int^s_0e^{r}|f_{\lambda,x}(r)-f_{\lambda,0}(r)|dr\right]ds
 +e^{-t}\int^t_0e^{r}|f_{\lambda,x}(r)-f_{\lambda,0}(r)|dr \nonumber\\
&&+\int^t_0 |\partial_v \check{L}_\lambda(s, x(s),\dot{x}(s))-\check{L}_\lambda(s, 0, 0)|ds
+|\partial_{v} \check{L}_\lambda\left(t, x(t),\dot{x}(t)\right)-\partial_{v} \check{L}_\lambda\left(t, 0,0\right)|\\
&\le&
2e^{2\tau}\tau\int^\tau_0|f_{\lambda,x}(r)-f_{\lambda,0}(r)|dr
 +\int^\tau_0|f_{\lambda,x}(r)-f_{\lambda,0}(r)|dr \nonumber\\
&&+\int^\tau_0 |\partial_v \check{L}_\lambda(s, x(s),\dot{x}(s))-\check{L}_\lambda(s, 0,0)|ds
+|\partial_{v} \check{L}_\lambda\left(t, x(t),\dot{x}(t)\right)-\partial_{v} \check{L}_\lambda\left(t, 0,0\right)|.
\end{eqnarray*}
Integrating this inequality over $[0,\tau]$ yields
\begin{eqnarray*}
&&2\tau |c_1(\lambda,x)-c_1(\lambda,0)|\le 2\int^\tau_0|c_1(\lambda,x)e^t-c_1(\lambda,0)e^t|dt\\
&\le&
2e^{2\tau}\tau^2\int^\tau_0|f_{\lambda,x}(r)-f_{\lambda,0}(r)|dr
 +\tau\int^\tau_0|f_{\lambda,x}(r)-f_{\lambda,0}(r)|dr \nonumber\\
&&+(\tau+1)\int^\tau_0 |\partial_v \check{L}_\lambda(s, x(s),\dot{x}(s))-\partial_v\check{L}_\lambda(s, 0, 0)|ds,
\end{eqnarray*}
that is,
\begin{eqnarray}\label{e:4.17two}
 |c_1(\lambda,x)-c_1(\lambda,0)|&\le&
(e^{2\tau}\tau+1)\int^\tau_0|f_{\lambda,x}(r)-f_{\lambda,0}(r)|dr\nonumber\\
 &&+\frac{(\tau+1)}{2\tau}\int^\tau_0 |\partial_v \check{L}_\lambda(s, x(s),\dot{x}(s))-\partial_v\check{L}_\lambda(s, 0, 0)|ds.
\end{eqnarray}
Moreover, (\ref{e:4.13}) leads to
\begin{eqnarray}\label{e:4.13two}
\int^\tau_0|f_{\lambda,x}(t)-f_{\lambda,0}(t)|dt&\le&
 \int^\tau_0|\partial_q \check{L}_\lambda(t,x(t),\dot{x}(t))-\partial_q \check{L}_\lambda(t,0, 0)|dt\nonumber\\
 &&+\tau\int^\tau_0|\partial_v \check{L}_\lambda(t,x(t),\dot{x}(t))-\partial_v \check{L}_\lambda(t,0, 0)|dt.
 \end{eqnarray}
From this and (\ref{e:4.17two}) we derive
\begin{eqnarray}\label{e:4.17two+}
 &&|c_1(\lambda,x)-c_1(\lambda,0)|\le(e^{2\tau}\tau+1)\int^\tau_0|\partial_q \check{L}_\lambda(t,x(t),\dot{x}(t))-\partial_q \check{L}_\lambda(t,0, 0)|dt\nonumber\\
 &&\qquad+\left((e^{2\tau}\tau+1)\tau+ \frac{\tau+1}{2\tau}\right)
\int^\tau_0 |\partial_v \check{L}_\lambda(s, x(s),\dot{x}(s))-\check{L}_\lambda(s, 0, 0)|ds.
\end{eqnarray}

Similarly, by (\ref{e:4.18}) we obtain
\begin{eqnarray*}
 &&2|c_2(\lambda,x)-c_2(\lambda,0)|\le  \int^t_0e^{r}|f_{\lambda,x}(r)-f_{\lambda,0}(r)|dr\\
 &&\qquad +e^t\int^t_0 |\partial_v \check{L}_\lambda(s, x(s),\dot{x}(s))-
 \partial_v \check{L}_\lambda(s, 0, 0)|ds
+ e^t|\partial_{v} \check{L}_\lambda\left(t, x(t),\dot{x}(t)\right)-\partial_v \check{L}_\lambda(t, 0,0)|.
\end{eqnarray*}
Integrating this inequality over $[0,\tau]$ and using (\ref{e:4.13two}) lead to
\begin{eqnarray}\label{e:4.18two}
 2\tau|c_2(\lambda,x)-c_2(\lambda,0)|\le  e^\tau\int^\tau_0|\partial_q \check{L}_\lambda(t,x(t),\dot{x}(t))-\partial_q \check{L}_\lambda(t,0, 0)|dt\nonumber\\
+e^\tau(2\tau+1)\int^\tau_0|\partial_{v} \check{L}_\lambda\left(t, x(t),\dot{x}(t)\right)-\partial_v \check{L}_\lambda(t, 0,0)|dt.
\end{eqnarray}

Note that  (L5) of Lemma~\ref{lem:Gener-modif} and
 \cite[Proposition~B.9]{Lu10} (\cite[Proposition~C.1]{Lu9}) imply that
 \begin{eqnarray*}
 &&\int^\tau_0\big|\big[
\partial_q \check{L}_{\lambda}\bigl(s,x(s),\dot{x}(s)\bigr)-\partial_q \check{L}_{\lambda}\bigl(s, 0, 0\bigr)\big]\big|ds
 \to 0\quad\hbox{and}\\
 && \int^\tau_0\big|\big[
\partial_v \check{L}_{\lambda}\bigl(s,x(s),\dot{x}(s)\bigr)-\partial_v \check{L}_{\lambda}\bigl(s, 0, 0\bigr)\big]\big|ds\\
&&\le\sqrt{\tau}\Big(\int^\tau_0\Big|\big[
\partial_v \check{L}_{\lambda}\bigl(s,x(s),\dot{x}(s)\bigr)-\partial_v \check{L}_{\lambda}\bigl(s, 0, 0\bigr)\big]\big|^2ds\Big)^{1/2}
 \to 0
 \end{eqnarray*}
 uniformly in $\lambda\in\hat\Lambda$ as $\|x\|_{1,2}\to 0$.
The required claim follows from this and (\ref{e:4.17two+})-(\ref{e:4.18two}).

\textsf{Step 2}({\it Prove that for any given $\nu'>0$ there exists $\varepsilon'>0$
such that $\nabla\check{\mathcal{E}}_{\lambda}(x)=0$ and $\|x\|_{1,2}\le\varepsilon'$ imply $\|x\|_{C^1}<\nu'$}).

Since $\nabla\check{\mathcal{E}}_{\lambda}(x)=0$ and $\nabla\check{\mathcal{E}}_{\lambda}(0)=0$, by (\ref{e:4.16}) we have
\begin{eqnarray*}
0&=&e^t\int^t_0\left[ e^{-2s}\int^s_0e^{r}f_{\lambda,x}(r)dr\right]ds
+e^{-t}\int^t_0e^{r}f_{\lambda,x}(r)dr
\nonumber\\
&&\quad + c_1(\lambda,x)e^t -c_2(\lambda,x)e^{-t}+ \partial_{v} \check{L}_\lambda\left(t, x(t),\dot{x}(t)\right).
\end{eqnarray*}
and
\begin{eqnarray*}
0&=&
e^t\int^t_0\left[ e^{-2s}\int^s_0e^{r}f_{\lambda,0}(r)dr\right]ds
+e^{-t}\int^t_0e^{r}f_{\lambda,0}(r)dr
\nonumber\\
&&\quad + c_1(\lambda, 0)e^t -c_2(\lambda, 0)e^{-t}+ \partial_{v} \check{L}_\lambda\left(t, 0, 0\right).
\end{eqnarray*}
For each $0\le t\le\tau$, the former minus the latter gives rise to
\begin{eqnarray*}
&&|\partial_{v} \check{L}_\lambda\left(t, x(t),\dot{x}(t)\right)-\partial_{v} \check{L}_\lambda\left(t, 0, 0\right)|\le
e^t|c_1(\lambda,x)-c_1(\lambda,0)| + e^{-t}|c_2(\lambda,x)-c_2(\lambda,0)|\\
&&\qquad+e^t\int^t_0\left[ e^{-2s}\int^s_0e^{r}|f_{\lambda,x}(r)-f_{\lambda,0}(r)|dr\right]ds
+e^{-t}\int^t_0e^{r}|f_{\lambda,x}(r)-f_{\lambda,0}(r)|dr\\
&&\le e^\tau|c_1(\lambda,x)-c_1(\lambda,0)| + |c_2(\lambda,x)-c_2(\lambda,0)|
+e^\tau(\tau+1)\int^\tau_0|f_{\lambda,x}(r)-f_{\lambda,0}(r)|dr
\end{eqnarray*}
and (by (\ref{e:4.13two})) hence
\begin{eqnarray}\label{e:4.13two++}
&&|\partial_{v} \check{L}_\lambda\left(t, x(t),\dot{x}(t)\right)-
\partial_{v} \check{L}_\lambda\left(t, 0, 0\right)|\le
e^\tau|c_1(\lambda,x)-c_1(\lambda,0)| +|c_2(\lambda,x)-c_2(\lambda,0)|\nonumber\\
&&\qquad +e^\tau(\tau+1)\int^\tau_0|\partial_q \check{L}_\lambda(t,x(t),\dot{x}(t))-\partial_q \check{L}_\lambda(t,0, 0)|dt\nonumber\\
 &&\qquad +e^\tau(\tau+1)\tau\int^\tau_0|\partial_v \check{L}_\lambda(t,x(t),\dot{x}(t))-\partial_v \check{L}_\lambda(t,0, 0)|dt.
 \end{eqnarray}

Note that (L3) in Lemma~\ref{lem:Gener-modif} and  the mean value theorem of integrals may lead to
 \begin{eqnarray*}
 \check{c}|v|^2\le\int^1_0\left(\partial_{vv}\check{L}_\lambda(t, q, sv)[v], v\right)_{\mathbb{R}^n}ds
 =\left(\partial_v\check{L}_\lambda(t, q, v)-\partial_v\check{L}_\lambda(t, q, 0), v\right)_{\mathbb{R}^n}
  \end{eqnarray*}
 and so
  $\check{c}|v|\le \left|\partial_v\check{L}_\lambda(t, q, v)-\partial_v\check{L}_\lambda(t, q, 0)\right|$
  for any $(\lambda, t, q,v)\in \hat\Lambda\times [0,\tau]\times B^n_{3\iota/4}(0)\times\mathbb{R}^n$. In particular, for all $t\in [0,\tau]$ we have
 \begin{eqnarray*}
 \check{c}|\dot{x}(t)|&\le& \left|\partial_v\check{L}_\lambda(t, x(t), \dot{x}(t))-\partial_v\check{L}_\lambda(t, x(t), 0)\right|\nonumber\\
 &\le& \left|\partial_v\check{L}_\lambda(t, x(t), \dot{x}(t))-\partial_v\check{L}_\lambda(t, 0, 0)\right|
 +\left|\partial_v\check{L}_\lambda(t, x(t), 0)-\partial_v\check{L}_\lambda(t, 0, 0)\right|.
 \end{eqnarray*}
By this and (\ref{e:4.13two++}) we arrive at
\begin{eqnarray}\label{e:lv2.1two}
\check{c}|\dot{x}(t)|&\le& e^\tau|c_1(\lambda,x)-c_1(\lambda,0)| +|c_2(\lambda,x)-c_2(\lambda,0)|+\left|\partial_v\check{L}_\lambda(t, x(t), 0)-\partial_v\check{L}_\lambda(t, 0, 0)\right|\nonumber\\
&&+e^\tau(\tau+1)\int^\tau_0|\partial_q \check{L}_\lambda(t,x(t),\dot{x}(t))-\partial_q \check{L}_\lambda(t,0, 0)|dt\nonumber\\
 &&+e^\tau(\tau+1)\tau\int^\tau_0|\partial_v \check{L}_\lambda(t,x(t),\dot{x}(t))-\partial_v \check{L}_\lambda(t,0, 0)|dt.
 \end{eqnarray}
 Since $\|x\|_{C^0}\le (\sqrt{\tau}+1/\sqrt{\tau})\|x\|_{1,2}$,
as in the final proof of Step 1, the required claim may follow from 
(\ref{e:lv2.1two}) and  the conclusion in Step 1.

\textsf{Step 3}({\it Complete the proof for the first claim}).
Note that $x$ satisfies
\begin{eqnarray}\label{e:P-EL1}
0&=&\frac{d}{dt}\big(\partial_v\check{L}_\lambda(t, x(t), \dot{x}(t))\big)-\partial_q\check{L}_\lambda(t, x(t), \dot{x}(t))\nonumber\\
&=&\partial_{vv}\check{L}_\lambda(t, x(t), \dot{x}(t))\ddot{x}(t)+\partial_{vq}\check{L}_\lambda(t, x(t), \dot{x}(t))\dot{x}(t)
+\partial_{vt}\check{L}_\lambda(t, x(t), \dot{x}(t))\nonumber\\
&&-\partial_q\check{L}_\lambda(t, x(t), \dot{x}(t)).
\end{eqnarray}
In particular, taking $x=0$ we get
\begin{eqnarray}\label{e:P-EL2}
0=\partial_{vt}\check{L}_\lambda(t, 0, 0)-\partial_q\check{L}_\lambda(t, 0, 0).
\end{eqnarray}
(\ref{e:P-EL1}) minus (\ref{e:P-EL2}) gives rise to
\begin{eqnarray}\label{e:P-EL3}
0&=&\partial_{vv}\check{L}_\lambda(t,  x(t), \dot{x}(t))\ddot{x}(t)+\partial_{vq}\check{L}_\lambda(t,  x(t), \dot{x}(t))\dot{x}(t)
\nonumber\\
&&+\partial_{vt}\check{L}_\lambda(t, x(t), \dot{x}(t))-\partial_{vt}\check{L}_\lambda(t, 0, 0)\nonumber\\
&&-\partial_{q}\check{L}_\lambda(t, x(t), \dot{x}(t))+\partial_q\check{L}_\lambda(t, 0, 0).
\end{eqnarray}
Since (L2) of Lemma~\ref{lem:Gener-modif} implies
$|[\partial_{vv}\check{L}_\lambda(t,  x(t), \dot{x}(t))]^{-1}\xi|\le\frac{1}{\check{c}}|\xi|\;\forall \xi\in{\mathbb{R}^n}$,  (\ref{e:P-EL3}) and (L3) in Lemma~\ref{lem:Gener-modif} lead to
\begin{eqnarray}\label{e:P-EL4-}
|\ddot{x}(t)|&\le&
\frac{1}{\check{c}}|\partial_{vq}\check{L}_\lambda(t,  x(t), \dot{x}(t))|\cdot|\dot{x}(t)|
+\frac{1}{\check{c}}|\partial_{vt}\check{L}_\lambda(t, x(t), \dot{x}(t))-\partial_{vt}\check{L}_\lambda(t, 0, 0)|\nonumber\\
&&+\frac{1}{\check{c}}|\partial_{q}\check{L}_\lambda(t, x(t), \dot{x}(t))-\partial_q\check{L}_\lambda(t, 0, 0)|\nonumber\\
&\le&\frac{\check{C}}{\check{c}}(1+|\dot{x}(t)|)|\cdot|\dot{x}(t)|
+\frac{1}{\check{c}}|\partial_{vt}\check{L}_\lambda(t, x(t), \dot{x}(t))-\partial_{vt}\check{L}_\lambda(t, 0, 0)|\nonumber\\
&&+\frac{1}{\check{c}}|\partial_{q}\check{L}_\lambda(t, x(t), \dot{x}(t))-\partial_q\check{L}_\lambda(t, 0, 0)|.
\end{eqnarray}
Recall that $\hat\Lambda$ is either compact or sequential compact and that
$\partial_{q}\check{L}_\lambda(t,q,v)$ and $\partial_{vt}\check{L}_{\lambda}(t,q,v)$
is continuous in $(\lambda, t, q, v)$ by (L0) in Lemma~\ref{lem:Gener-modif}.
The desired claim easily follows from (\ref{e:P-EL4-}) and the result in Step 2.
\end{proof}

\begin{remark}\label{rm:funct-analy}
{\rm  \begin{enumerate}
\item[\rm (i)] The existence and continuity of the  partial derivative $\partial_t\check{L}(\cdot)$ in (L0) of Lemma~\ref{lem:Gener-modif}
are not used in the proofs of Propositions~\ref{prop:funct-analy},~\ref{prop:continA},~\ref{prop:solutionLagr}.
\item[\rm (ii)] The existence and continuity of the  partial derivatives $\partial_{tv}\check{L}(\cdot)$ and $\partial_{vt}\check{L}(\cdot)$ in (L0)
of Lemma~\ref{lem:Gener-modif} are not used in the proofs of Propositions~\ref{prop:funct-analy},~\ref{prop:continA};
but they are necessary for the proof of Proposition~\ref{prop:solutionLagr}.
\end{enumerate}
}
\end{remark}

\begin{proof}[\bf Proof of Theorem~\ref{th:bif-nessLagrGenerEu}(I)]
  Since there exist a sequence $\{(\lambda_k,x_k)\}_{k\ge 1}\subset\Lambda\times C^1_{V_0\times V_1}([0,\tau]; B^n_{\iota}(0))$
   converging to $(\mu,0)$ such that each $x_k$ is a nonzero solution of
   the problem (\ref{e:LagrGenerEu})--(\ref{e:LagrGenerBEu}) with $\lambda=\lambda_k$, i.e.,
   $\nabla\tilde{\mathcal{E}}^\ast_{\lambda_k}(x_k)=0$, $k=1,2,\cdots$,
    by (L1) in Lemma~\ref{lem:Gener-modif} with $\hat\Lambda=\{\mu,\lambda_k\,|\,k\in\mathbb{N}\}$ we deduce that $\nabla\check{\mathcal{E}}_{\lambda_k}(x_k)=0$ for $k$ large enough.
   Therefore $(\mu, 0)$  is a  bifurcation point along sequence of
$\nabla\check{\mathcal{E}}_{\lambda}(x)=0$ in $\hat\Lambda\times\mathcal{U}$.
By (i) and (v)-(vii) of Proposition~\ref{prop:funct-analy}
we see that the conditions of \cite[Theorem~3.1]{Lu8} (\cite[Theorem~C.6]{Lu11}) are satisfied with
$\mathcal{F}_\lambda= {\check{\mathcal{E}}}_{\lambda}$ and $H=X={\bf H}_{V_0\times V_1}$,  $U=\mathcal{U}$ and $\lambda^\ast=\mu$.
Then  $m^0_\tau(\check{\mathcal{E}}_{\mu}, 0)>0$ and so
$m^0(\tilde{\mathcal{E}}^\ast_\lambda, 0)>0$ by  (\ref{e:MorseindexEu}).
 \end{proof}

\begin{proof}[\bf Proof of Theorem~\ref{th:bif-suffLagrGenerEu}]
Since $\Lambda$ is a real interval and  $\mu\in{\rm Int}(\Lambda)$ we can take a small $\varepsilon>0$
so that $\hat\Lambda:=[\mu-\varepsilon,\mu+\varepsilon]\subset\Lambda$.
Propositions~\ref{prop:funct-analy},~\ref{prop:continA}
shows that $(\mathcal{U}, \mathcal{U}^X, \{\check{\mathcal{E}}_{\lambda}\,|\,\lambda\in\hat\Lambda\})$ satisfies the conditions
in \cite[Theorem~3.6]{Lu9}  with $\lambda^\ast=\mu$ except for the condition (f).
The latter may follow from  (\ref{e:MorseindexEu}) and the assumption, that is,
    $m^0(\check{\mathcal{E}}_{\mu}, 0)\ne 0$ and
   $m^0(\check{\mathcal{E}}_{\lambda}, 0)=0$ for all $\lambda\in\hat\Lambda\setminus\{\mu\}$ near $\mu$,
 and  $m^-(\check{\mathcal{E}}_{\lambda}, 0)$ take, respectively, values $m^-(\check{\mathcal{E}}_{\mu}, 0)$ and
  $m^-(\check{\mathcal{E}}_{\mu}, 0)+ m^0(\check{\mathcal{E}}_{\mu}, 0)$
 as $\lambda\in\hat\Lambda$ varies in two deleted half neighborhoods  of $\mu$.
 Therefore from \cite[Theorem~C.7]{Lu11} (\cite[Theorem~3.6]{Lu10})  we deduce that  one of the following  occurs:
\begin{enumerate}
\item[\rm (i)]  $\nabla\check{\mathcal{E}}_{\mu}$  has a sequence of nontrivial zero points converging to $0$ in $\mathcal{U}$.

\item[\rm (ii)]  For every $\lambda\in\hat\Lambda\setminus\{\mu\}$ near $\mu$,
$A_{\lambda}$ has a zero point  $y_\lambda\ne 0$, which  converge to zero in  $\mathcal{U}^X$ as $\lambda\to \mu$.

\item[\rm (iii)] For a given neighborhood $\mathfrak{M}$ of $0\in \mathcal{U}^X$, there is
a one-sided neighborhood $\Lambda^0$ of $\mu$ in
$\hat\Lambda$ (therefore in $\Lambda$) such that
for any $\lambda\in\Lambda^0\setminus\{\mu\}$,  $A_{\lambda}$ has
at least two distinct nontrivial zero points in $\mathfrak{M}$,  $y_\lambda^1$ and $y_\lambda^2$,
which can also be required to satisfy  $\check{\mathcal{E}}_{\lambda}(y^1_\lambda)\ne \check{\mathcal{E}}_{\lambda}(y^2_\lambda)$
provided that $m^0(\check{\mathcal{E}}_{\mu}, 0)>1$ and $A_{\lambda}$ has only
finitely many nontrivial zero points in $\mathfrak{M}$.
\end{enumerate}
 As above the required results may follow from Proposition~\ref{prop:solutionLagr} and (L0) in Lemma~\ref{lem:Gener-modif}.
\end{proof}

To prove Theorem~\ref{th:bif-nessLagrGenerEu}(II) and Theorem~\ref{th:bif-existLagrGenerEu+},
noting that because of Propositions~\ref{prop:funct-analy},~\ref{prop:continA},
(specially Proposition~\ref{prop:funct-analy}(iv) implies that
$\check{\mathcal{E}}_\lambda|_{\mathcal{U}^X}\in C^2(\mathcal{U}^X, \mathbb{R})$
  and $B_\lambda$ is continuous as a map from $\mathcal{U}^X$ to $\mathscr{L}_s({\bf H}_{V_0\times V_1})$
because $A_\lambda'=B_\lambda$),
we may, respectively, apply \cite[Theorem~C.4]{Lu11} and \cite[Theorem~C.5]{Lu11} to
$(\mathcal{U}, \mathcal{U}^X,  \{\check{\mathcal{E}}_{\lambda}\,|\, \lambda\in\hat\Lambda\})$ to obtain (I) and (II) of the following theorem.

\begin{theorem}\label{th:bif-existLagrGenerEu}
\begin{enumerate}
\item[\rm (I)]{\rm (\textsf{Sufficient condition}):}
  Suppose that $\hat\Lambda$ is first countable and that there exist two sequences in  $\Lambda$ converging to $\mu$, $(\lambda_k^-)$ and
$(\lambda_k^+)$,  such that one of the following conditions is satisfied:
 \begin{enumerate}
 \item[\rm (I.1)] For each $k\in\mathbb{N}$,
 either $0$  is not an isolated critical point of $\check{\mathcal{E}}_{\lambda^+_k}$,
 or $0$ is not an isolated critical point of $\check{\mathcal{E}}_{\lambda^-_k}$,
 or $0$  is an isolated critical point of $\check{\mathcal{E}}_{\lambda^+_k}$ and $\check{\mathcal{E}}_{\lambda^-_k}$ and
  $C_m(\check{\mathcal{E}}_{\lambda^+_k}, 0;{\bf K})$ and $C_m(\check{\mathcal{E}}_{\lambda^-_k}, 0;{\bf K})$ are not isomorphic for some Abel group ${\bf K}$ and some $m\in\mathbb{Z}$.  Moreover,  in the third case, ``$C_m(\check{\mathcal{E}}_{\lambda^+_k}, 0;{\bf K})$ and
  $C_m(\check{\mathcal{E}}_{\lambda^-_k}, 0;{\bf K})$''  may be replaced by
   ``$C_\ast(\check{\mathcal{E}}_{\lambda^+_k}|_{\mathcal{U}^X}, 0;{\bf K})$ and
   $C_\ast(\check{\mathcal{E}}_{\lambda^-_k}|_{\mathcal{U}^X}, 0;{\bf K})$''.
 \item[\rm (I.2)]
 For each $k\in\mathbb{N}$,
 $[m^-(\check{\mathcal{E}}_{\lambda^+_k}, 0), m^-(\check{\mathcal{E}}_{\lambda^+_k}, 0)+
m^0(\check{\mathcal{E}}_{\lambda^+_k}, 0)]\cap[m^-(\check{\mathcal{E}}_{\lambda^-_k}, 0),
m^-(\check{\mathcal{E}}_{\lambda^-_k}, 0)+m^0(\check{\mathcal{E}}_{\lambda^-_k}, 0)]=\emptyset$,
and there exists $\lambda\in\{\lambda^+_k, \lambda^-_k\}$ such that  $0$  is an either non-isolated or homological visible critical point of
$\check{\mathcal{E}}_{\lambda}$. Moreover, $\check{\mathcal{E}}_{\lambda}$ can be replaced by $\check{\mathcal{E}}_{\lambda}|_{\mathcal{U}^X}$
in the second condition.
\item[\rm (I.3)] For each $k\in\mathbb{N}$,
$[m^-(\check{\mathcal{E}}_{\lambda^+_k}, 0), m^-(\check{\mathcal{E}}_{\lambda^+_k}, 0)+
m^0(\check{\mathcal{E}}_{\lambda^+_k}, 0)]\cap[m^-(\check{\mathcal{E}}_{\lambda^-_k}, 0),
m^-(\check{\mathcal{E}}_{\lambda^-_k}, 0)+m^0(\check{\mathcal{E}}_{\lambda^-_k}, 0)]=\emptyset$,
and either $m^0(\check{\mathcal{E}}_{\lambda^+_k}, 0)=0$ or $m^0(\check{\mathcal{E}}_{\lambda^-_k}, 0)=0$.
 \end{enumerate}
   Then there exists a sequence $\{(\lambda_k, x_k)\}_{k\ge 1}$ in  $\check\Lambda\times (\mathcal{U}\cap C^2([0,\tau], \R^n))$
   such that $\lambda_k\to\mu$, $0<\|x_k\|_{C^2}\to 0$ and $\nabla\check{\mathcal{E}}_{\lambda}(x_k)=0$ for $k=1,2,\cdots$,
   where  $\check{\Lambda}=\{\mu,\lambda^+_k, \lambda^-_k\,|\,k\in\mathbb{N}\}$.
  In particular,  $(\mu, 0)$  is a  bifurcation point of the problem $\nabla\check{\mathcal{E}}_{\lambda}(x)=0$
  in  $\check\Lambda\times C^2([0,\tau], \R^n)$  with respect to the branch $\{(\lambda, 0)\,|\,\lambda\in\check\Lambda\}$
  (and so $\{(\lambda, 0)\,|\,\lambda\in\hat\Lambda\}$).

\item[\rm (II)]{\rm (\textsf{Existence for bifurcations}):}
For $\lambda^-, \lambda^+$ in a path-connected component of $\hat\Lambda$
suppose that one of the following conditions is satisfied:
 \begin{enumerate}
 \item[\rm (II.1)] Either $0$  is not an isolated critical point of $\check{\mathcal{E}}_{\lambda^+}$,
 or $0$ is not an isolated critical point of $\check{\mathcal{E}}_{\lambda^-}$,
 or $0$  is an isolated critical point of $\check{\mathcal{E}}_{\lambda^+}$ and $\check{\mathcal{E}}_{\lambda^-}$ and
  $C_m(\check{\mathcal{E}}_{\lambda^+}, 0;{\bf K})$ and $C_m(\check{\mathcal{E}}_{\lambda^-}, 0;{\bf K})$ are not isomorphic for some Abel group ${\bf K}$ and some $m\in\mathbb{Z}$.  Moreover, in the final case, ``$C_m(\check{\mathcal{E}}_{\lambda^+}, 0;{\bf K})$ and
  $C_m(\check{\mathcal{E}}_{\lambda^-}, 0;{\bf K})$''  may be replaced by
   ``$C_\ast(\check{\mathcal{E}}_{\lambda^+}|_{\mathcal{U}^X}, 0;{\bf K})$ and
   $C_\ast(\check{\mathcal{E}}_{\lambda^-}|_{\mathcal{U}^X}, 0;{\bf K})$''.

\item[\rm (II.2)] $[m^-(\check{\mathcal{E}}_{\lambda^+}, 0), m^-(\check{\mathcal{E}}_{\lambda^+}, 0)+
m^0(\check{\mathcal{E}}_{\lambda^+}, 0)]\cap[m^-(\check{\mathcal{E}}_{\lambda^-}, 0),
m^-(\check{\mathcal{E}}_{\lambda^-}, 0)+m^0(\check{\mathcal{E}}_{\lambda^-}, 0)]=\emptyset$,
and  there exists $\lambda\in\{\lambda^+, \lambda^-\}$ such that $0$  is an either non-isolated or homological visible critical point of
$\check{\mathcal{E}}_{\lambda}$. Moreover, in the second condition, $\check{\mathcal{E}}_{\lambda}$ can be replaced by $\check{\mathcal{E}}_{\lambda}|_{\mathcal{U}^X}$.

\item[\rm (II.3)] $[m^-(\check{\mathcal{E}}_{\lambda^+}, 0), m^-(\check{\mathcal{E}}_{\lambda^+}, 0)+
m^0(\check{\mathcal{E}}_{\lambda^+}, 0)]\cap[m^-(\check{\mathcal{E}}_{\lambda^-}, 0),
m^-(\check{\mathcal{E}}_{\lambda^-}, 0)+m^0(\check{\mathcal{E}}_{\lambda^-}, 0)]=\emptyset$,
and either $m^0(\check{\mathcal{E}}_{\lambda^+}, 0)=0$ or $m^0(\check{\mathcal{E}}_{\lambda^-}, 0)=0$.
 \end{enumerate}
 Then for any path $\alpha:[0,1]\to\hat\Lambda$ connecting $\lambda^+$ to $\lambda^-$
 there exists a sequence $\{(t_k, x_k)\}_{k\ge 1}$  in $[0, 1]\times \mathcal{U}$
   converging to $(\bar{t}, 0)$ for some $\bar{t}\in [0,1]$, such that each $x_k$
  is a nonzero solution of $\nabla\check{\mathcal{E}}_{\alpha(t_k)}(x)=0$, $k=1,2,\cdots$.
 (In fact, $\|x_k\|_{C^2}\to 0$ by Proposition~\ref{prop:solutionLagr}.)
  Moreover, $\alpha(\bar{t})$ is not equal to $\lambda^+$ (resp. $\lambda^-$) if $
 m^0(\check{\mathcal{E}}_{\lambda^+}, 0)=0$ (resp. $m^0(\check{\mathcal{E}}_{\lambda^-}, 0)=0$).
  \end{enumerate}
  \end{theorem}

\begin{proof}[\bf Proof]
\textsf{Step 1}({\it Prove} (I)).
Because of Propositions~\ref{prop:funct-analy},~\ref{prop:continA},
(specially Proposition~\ref{prop:funct-analy}(iv) implies that
$\check{\mathcal{E}}_\lambda|_{\mathcal{U}^X}\in C^2(\mathcal{U}^X, \mathbb{R})$
  and $B_\lambda$ is continuous as a map from $\mathcal{U}^X$ to $\mathscr{L}_s({\bf H}_{V_0\times V_1})$
since $A_\lambda'=B_\lambda$),
for the case (I.2) [resp. (I.3)] we  apply \cite[Theorem~C.4(B.1),(B.2)]{Lu11}
(resp. \cite[Theorem~C.4(B.3)]{Lu11}) to obtain:
\begin{enumerate}
 \item[\rm ($\ast$)] There exists a sequence $\{(\lambda_k,x_k)\}_{k\ge 1}\subset\hat\Lambda\times \mathcal{U}\setminus\{(\mu,0)\}$
converging to $(\mu, 0)$  such that $x_k\ne 0$ and $\nabla\check{\mathcal{E}}_{\lambda}(x_k)=0$ for $k=1,2,\cdots$.
  \end{enumerate}

For the case (I.1), if $\{\check{\mathcal{E}}_{\lambda^+_k}\}_{k\ge 1}$ or $\{\check{\mathcal{E}}_{\lambda^-_k}\}_{k\ge 1}$
has a subsequence such that each term of it has $0$ as a non-isolated critical point,
then we have ($\ast$) naturally.
Otherwise, for each sufficiently large integer $l$, $0$  is an isolated critical point of $\check{\mathcal{E}}_{\lambda^+_l}$ and $\check{\mathcal{E}}_{\lambda^-_l}$ and  $C_{m_l}(\check{\mathcal{E}}_{\lambda^+_l}, 0;{\bf K}_l)$ and 
$C_{m_l}(\check{\mathcal{E}}_{\lambda^-_l}, 0;{\bf K}_l)$ are not isomorphic 
for some Abel group ${\bf K}_l$ and some $m_l\in\mathbb{Z}$, which implies 
 that for each such an integer $l$,
$0$  is an isolated critical point of $\check{\mathcal{E}}_{\lambda^+_l}|_{\mathcal{U}^X}$ and
$\check{\mathcal{E}}_{\lambda^-_l}|_{\mathcal{U}^X}$ and  $C_{m_l}(\check{\mathcal{E}}_{\lambda^+_l}|_{\mathcal{U}^X}, 0;{\bf K}_l)$ and $C_{m_l}(\check{\mathcal{E}}_{\lambda^-_l}|_{\mathcal{U}^X}, 0;{\bf K}_l)$ are not isomorphic
because by Proposition~\ref{prop:funct-analy} we may use \cite[Corollary~2.8]{JM} to deduce that
$C_{m_l}(\check{\mathcal{E}}_{\lambda^+_l}|_{\mathcal{U}^X}, 0;{\bf K}_l)\cong
C_{m_l}(\check{\mathcal{E}}_{\lambda^+_l}, 0;{\bf K}_l)$ and $C_{m_l}(\check{\mathcal{E}}_{\lambda^-_l}|_{\mathcal{U}^X}, 0;{\bf K}_l)\cong
C_{m_l}(\check{\mathcal{E}}_{\lambda^-_l}, 0;{\bf K}_l)$.
Then we may use \cite[Theorem~C.4(A)]{Lu11} to get a contradiction provided that ($\ast$) is not true.

In summary, we get ($\ast$) in this case, and therefore the desired statements by Proposition~\ref{prop:solutionLagr}.

\textsf{Step 2}({\it Prove} (II)). Applying \cite[Theorem~C.5]{Lu11} to $(\mathcal{U}, \mathcal{U}^X,  \{\check{\mathcal{E}}_{\lambda}\,|\, \lambda\in\hat\Lambda\})$ a similar proof to that of Step 1 yields the required results.
\end{proof}

\begin{proof}[\bf Proof of Theorem~\ref{th:bif-nessLagrGenerEu}(II)]
Suppose that (II.1) is satisfied. If $\{\tilde{\mathcal{E}}^\ast_{\lambda_k^+}\}_{k\ge 1}$ or $\{\tilde{\mathcal{E}}^\ast_{\lambda_k^-}\}_{k\ge 1}$
has a subsequence such that each term of it has $0$ as a non-isolated critical point,
then the required result may follow from (\ref{e:twofunctAgree}) and Proposition~\ref{prop:solutionLagr}.
Otherwise, for each sufficiently large integer $l$, $0$  is an isolated critical point of $\tilde{\mathcal{E}}^\ast_{\lambda_l^+}$ and
$\tilde{\mathcal{E}}^\ast_{\lambda_l^-}$ and  $C_{m_l}(\tilde{\mathcal{E}}^\ast_{\lambda_l^+}, 0;{\bf K}_l)$ and $C_{m_l}(\tilde{\mathcal{E}}^\ast_{\lambda_l^-}, 0;{\bf K}_l)$ are not isomorphic for some Abel group ${\bf K}_l$ and some $m_l\in\mathbb{Z}$, which implies by (\ref{e:twofunctAgree}) that
$0$  is an isolated critical point of $\check{\mathcal{E}}_{\lambda^+_l}|_{\mathcal{U}^X}$ and
$\check{\mathcal{E}}_{\lambda^-_l}|_{\mathcal{U}^X}$ and  $C_{m_l}(\check{\mathcal{E}}_{\lambda^+_l}|_{\mathcal{U}^X}, 0;{\bf K}_l)$ and $C_{m_l}(\check{\mathcal{E}}_{\lambda^-_l}|_{\mathcal{U}^X}, 0;{\bf K}_l)$ are not isomorphic.
That is, the condition (I.1) in Theorem~\ref{th:bif-existLagrGenerEu} is satisfied.
Hence (\ref{e:twofunctAgree}) and the conclusion in Theorem~\ref{th:bif-existLagrGenerEu}(I) lead to the required result.

Next, let (II.2) be satisfied. If the assumptions in the second sentence in last paragraph are true we are done.
Otherwise, for each sufficiently large integer $l$, $0$  is an isolated critical point of $\tilde{\mathcal{E}}^\ast_{\lambda_l^+}$ and
$\tilde{\mathcal{E}}^\ast_{\lambda_l^-}$ and either $C_{m_l}(\tilde{\mathcal{E}}^\ast_{\lambda^+}, 0;{\bf K}_l)\ne 0$
for some Abel group ${\bf K}_l$ and some $m_l\in\mathbb{Z}$ or
and $C_{n_l}(\tilde{\mathcal{E}}^\ast_{\lambda_l^-}, 0;{\bf K}'_l)\ne 0$  for some Abel group ${\bf K}'_l$ and some $n_l\in\mathbb{Z}$.
Therefore  by (\ref{e:twofunctAgree}) $0$  is an isolated critical point of $\check{\mathcal{E}}_{\lambda^+_l}|_{\mathcal{U}^X}$ and
$\check{\mathcal{E}}_{\lambda^-_l}|_{\mathcal{U}^X}$ and either  $C_{m_l}(\check{\mathcal{E}}_{\lambda^+_l}|_{\mathcal{U}^X}, 0;{\bf K}_l)\ne 0$ or $C_{n_l}(\check{\mathcal{E}}_{\lambda^-_l}|_{\mathcal{U}^X}, 0;{\bf K}'_l)\ne 0$. These mean that
the condition (I.2) in Theorem~\ref{th:bif-existLagrGenerEu} is satisfied. As above, Theorem~\ref{th:bif-existLagrGenerEu}(I) and (\ref{e:twofunctAgree}) yield the desired conclusions.

For the case (I.3), by (\ref{e:MorseindexEu}) we see that the condition (I.3) in Theorem~\ref{th:bif-existLagrGenerEu} is satisfied.
The required statements are derived as above.
 \end{proof}

\begin{proof}[\bf Proof of Theorem~\ref{th:bif-existLagrGenerEu+}]
For the case (i) in Theorem~\ref{th:bif-existLagrGenerEu+}. The first two cases easily follow from
Proposition~\ref{prop:solutionLagr} and (\ref{e:twofunctAgree}).
For the third case, $0$  is also an isolated critical point of $\check{\mathcal{E}}_{\lambda^+}$
and $\check{\mathcal{E}}_{\lambda^-}$ by Proposition~\ref{prop:solutionLagr}  and (\ref{e:twofunctAgree}),
and $C_m(\check{\mathcal{E}}_{\lambda^+}|_{\mathcal{U}^X}, 0;{\bf K})$ and $C_m(\check{\mathcal{E}}_{\lambda^-}|_{\mathcal{U}^X}, 0;{\bf K})$ are not isomorphic.  Theorem~\ref{th:bif-existLagrGenerEu}(II) leads to the required results.

For the case (ii) in Theorem~\ref{th:bif-existLagrGenerEu+}, by (\ref{e:twofunctAgree})
 ``$0$  is an either non-isolated or homological visible critical point of
$\mathcal{E}^\ast_{\lambda}$'' is equivalent to
``$0$  is an either non-isolated or homological visible critical point of
$\check{\mathcal{E}}_{\lambda^+}|_{\mathcal{U}^X}$''.
Because of these and (\ref{e:MorseindexEu}), Theorem~\ref{th:bif-existLagrGenerEu}(II) yields the required results.

The case (iii) in Theorem~\ref{th:bif-existLagrGenerEu+} follows from (\ref{e:MorseindexEu}) and Theorem~\ref{th:bif-existLagrGenerEu}(II).
 \end{proof}

\subsubsection{Proof of Theorem~\ref{th:bif-existLagrGener}}\label{sec:LagrBound.1.4}

By contradiction suppose that there exists a path $\alpha:[0,1]\to\Lambda$ connecting $\lambda^+$ to $\lambda^-$ such that
each point $(\alpha(s), \gamma_{\alpha(s)})$, $s\in [0,1]$,  is not a  bifurcation point  of the problem (\ref{e:LagrGener})--(\ref{e:LagrGenerB}) in
 $\alpha([0,1])\times C^2_{S_0\times S_1}([0,\tau]; M)$  with respect to the branch $\{(\lambda, \gamma_\lambda)\,|\,\lambda\in\alpha([0,1])\}$.
  Then for some small $\epsilon>0$ we have:
\begin{equation}\label{e:contrad}
\left.\begin{array}{ll}
&\gamma\in C^{2}_{S_0\times S_1}([0,\tau]; M)\;\hbox{satisfies
$\|\gamma-\gamma_{\alpha(s)}\|_{C^2([0,\tau];\mathbb{R}^N)}\le\epsilon$}\\
&\hbox{and (\ref{e:LagrGener})--(\ref{e:LagrGenerB}) with $\lambda=\alpha(s)$ for some $s\in [0,1]$}\\
&\Longrightarrow\;\gamma=\gamma_\lambda\;\hbox{for some $\lambda\in\alpha([0,1])$.}
\end{array}\right\}
\end{equation}

Fix a point $\mu\in\alpha([0,1])$. Let $\overline{\gamma}$ be as in (\ref{e:gamma1}).
We have a compact neighborhood $\hat\Lambda$ of $\mu$ in $\alpha([0,1])$ such that
(\ref{e:gamma2}) is satisfied for all $(\lambda,t)\in\hat\Lambda\times [0,\tau]$, i.e.,
\begin{eqnarray}\label{e:gamma2contrad}
{\rm dist}_g(\gamma_\lambda(t), \overline{\gamma}(t))<\iota,\quad\forall (\lambda,t)\in
\hat\Lambda\times[0,\tau].
\end{eqnarray}
Then the reduction in Section~\ref{sec:LagrBound.1.1} is valid after we use $\hat\Lambda$ to replace $\Lambda$ therein.
Therefore for the functionals in (\ref{e:brakefunctGener*}), (\ref{e:contrad}) implies that for some $\bar{\epsilon}>0$ we have
\begin{equation}\label{e:contrad1}
\left.\begin{array}{ll}
&x\in C^1_{V_0\times V_1}([0,\tau]; B^n_{\iota}(0))\;\hbox{satisfies $\|x\|_{C^2([0,\tau];\mathbb{R}^N)}\le \bar{\epsilon}$}\\
&\hbox{and $d\tilde{\mathcal{E}}^\ast_\lambda(x)=0$ for some $\lambda\in\hat\Lambda$}\quad\Longrightarrow\;x=0.
\end{array}\right\}
\end{equation}
It follows from this fact, (\ref{e:twofunctAgree}) and Proposition~\ref{prop:solutionLagr} that
there exists $\hat{\epsilon}>0$ such that
\begin{equation}\label{e:contrad2}
\left.\begin{array}{ll}
&x\in \mathcal{U}\;\hbox{satisfies $\|x\|_{1,2}\le \hat{\epsilon}$ and}\\
&\hbox{$d\check{\mathcal{E}}_{\lambda}(x)=0$ for some $\lambda\in\hat\Lambda$}\quad\Longrightarrow\;x=0.
\end{array}\right\}
\end{equation}
Since we can shrink $\hat{\epsilon}>0$ so small that the ball
$\bar{B}_{\hat{\epsilon}}({\bf H}_{V_0\times V_1}):=\{\xi\in {\bf H}_{V_0\times V_1}\,|\, \|\xi\|_{1,2}\le\hat{\epsilon}\}$
is contained in $\mathcal{U}$, (\ref{e:contrad2}) means that $0\in\mathcal{U}$
is a unique critical point of $\check{\mathcal{E}}_{\lambda}$ in $\bar{B}_{\hat{\epsilon}}({\bf H}_{V_0\times V_1})$
 for each $\lambda\in\hat\Lambda$.

Take $\bar{s}\in [0, 1]$ such that $\alpha(\bar{s})=\mu$. We have a connected compact neighborhood $N(\bar{s})$ of $\bar{s}$ in $[0, 1]$
such that $\alpha(N(\bar{s}))\subset\hat\Lambda$.
Because of (\ref{e:contrad2}), as in the proof of \cite[Theorem~C.5]{Lu11} we have a correspondent result of \cite[(C.17)]{Lu11}
and therefore  obtain that for any Abel group ${\bf K}$ and any $s, s'\in N(\bar{s})$,
\begin{eqnarray}\label{e:Spli.2.4.5Lu10}
C_q(\check{\mathcal{E}}_{\alpha(s)}, 0;{\bf K})\cong C_q(\check{\mathcal{E}}_{\alpha(s')}, 0;{\bf K}),\;\;\forall q\in\mathbb{N}\cup\{0\}.
\end{eqnarray}
As in the previous proofs we may use \cite[Corollary~2.8]{JM} to deduce that
\begin{eqnarray*}
C_{q}(\check{\mathcal{E}}_{\lambda}|_{\mathcal{U}^X}, 0;{\bf K})\cong
C_{q}(\check{\mathcal{E}}_{\lambda}, 0;{\bf K}),\quad\forall\lambda\in\hat\Lambda.
\end{eqnarray*}
This, (\ref{e:twofunctAgree}) and (\ref{e:twofunctional}) lead to
\begin{eqnarray*}
C_{q}(\check{\mathcal{E}}_{\lambda}|_{\mathcal{U}^X}, 0;{\bf K})\cong
C_q(\tilde{\cal E}^\ast_{\lambda}, 0;{\bf K})\cong C_q({\cal E}_{\lambda},\gamma_\lambda ;{\bf K}),\quad\forall\lambda\in\hat\Lambda.
\end{eqnarray*}
Combining this with equation (\ref{e:Spli.2.4.5Lu10}), we obtain
\begin{eqnarray}\label{e:Spli.2.4.5Lu10+}
C_q({\cal E}_{\alpha(s)},\gamma_{\alpha(s)};{\bf K})\cong C_q({\mathcal{E}}_{\alpha(s')}, \gamma_{\alpha(s')};{\bf K}),\;\;\forall s, s'\in N(\bar{s}),\;\forall q\in\mathbb{N}\cup\{0\}.
\end{eqnarray}
Because the point $\mu\in\alpha([0,1])$ is arbitrary,  (\ref{e:Spli.2.4.5Lu10+}) implies
\begin{eqnarray}\label{e:Spli.2.4.5Lu10++}
C_q({\cal E}_{\lambda},\gamma_{\lambda};{\bf K})\cong C_q({\mathcal{E}}_{\lambda'}, \gamma_{\lambda'};{\bf K}),\;\;\forall \lambda, \lambda'\in \alpha([0,1]),\;\forall q\in\mathbb{N}\cup\{0\}.
\end{eqnarray}
Almost repeating the arguments below (C.18) in the proof of \cite[Theorem~C.5]{Lu11} we may see that
(\ref{e:Spli.2.4.5Lu10++}) contradicts to each of the conditions (i)-(iii) in Theorem~\ref{th:bif-existLagrGener}.
Hence the assumption above (\ref{e:contrad}) is not true. Theorem~\ref{th:bif-existLagrGener} is proved.
\hfill$\Box$

\subsection{Proof of Theorem~\ref{th:MorseBif}}

We first admit the following version of the Morse index theorem due to Duistermaat \cite{Du},
which can directly lead to the Morse index theorem  in Finsler geometry
(\cite[Corollary~9.7]{Lu12}).

\begin{theorem}\label{th:DustMorse}
Let the functionals $\mathcal{L}_{S_0,s}$ be as in (\ref{e:LagrFinEnergyPQ0}). Then
 $$
 m^-(\mathcal{L}_{S_0,\lambda}, \gamma_\lambda)=\sum_{0<s<\lambda}m^0(\mathcal{L}_{S_0,s}, \gamma_s),\quad\forall\lambda\in (0,\tau].
 $$
\end{theorem}

Theorem~\ref{th:DustMorse} implies that there  exist only finitely many  $\mu\in (0,\tau)$
such that $m^0(\mathcal{L}_{S_0,\mu}, \gamma_\mu)\ne 0$, and hence the conclusion (i) follows.

\begin{proof}[\bf Proof of Theorem~\ref{th:MorseBif}]
In the arguments above (\ref{e:ModifiedLEu}) taking $\gamma_\mu=\gamma$ and $S_1=\{\gamma(\tau)\}$
 we have a unique map
$$
{\bf u}\in C^1_{V_0\times \{0\}}([0,\tau]; B^n_{\iota}(0))\cap C^3([0,\tau]; B^n_{\iota}(0))
$$
  such that ${\bf u}(0)={\bf u}(\tau)=0$ and $\gamma(t)=\phi_{\overline{\gamma}}(t,{\bf u}(t))$ for all $t\in [0,\tau]$.
For $\lambda\in (0,\tau]$ let
$$
W^{1,2}_{V_0\times\{{\bf u}(\lambda)\}}([0, \lambda]; B^n_{2\iota}(0))=
\big\{u\in W^{1,2}_{V_0\times\{{\bf u}(\lambda)\}}([0, \lambda]; \mathbb{R}^n)\,\big|\, u([0,\lambda])\subset B^n_{2\iota}(0))\big\},
$$
which contains $C^1_{V_0\times\{{\bf u}(\lambda)\}}([0, \lambda]; B^n_{2\iota}(0))$ as a dense subset.
 Using \cite[Theorems~4.2,~4.3]{PiTa01}, as before we deduce that
$$
W^{1,2}_{S_0\times\{\gamma(\lambda)\}}([0, \lambda]; M):=\big\{\alpha\in W^{1,2}([0, \lambda]; M)\,\big|\,
\alpha(0)\in S_0, \alpha(\lambda)=\gamma(\lambda)\big\}
$$
is a $C^4$ Hilbert manifold and obtain a $C^2$ chart
\begin{eqnarray*}
\Phi_{\lambda}:W^{1,2}_{V_0\times\{0\}}([0, \lambda]; B^n_{\iota}(0))\to W^{1,2}_{S_0\times\{\gamma(\lambda)\}}([0, \lambda]; M)
\end{eqnarray*}
given by $\Phi_{\lambda}(\xi)(t)=\phi_{\overline{\gamma}}(t, {\bf u}_\lambda(t)+\xi(t))\;\forall t\in [0, \lambda]$ for each $\lambda\in (0,\tau]$,
where ${\bf u}_\lambda:={\bf u}|_{[0,\lambda]}$. Then
$\Phi_\lambda(0)=\gamma_\lambda:=\gamma|_{[0,\lambda]}$.
Note that $\Phi_{\lambda}:C^1_{V_0\times\{0\}}([0, \lambda]; B^n_{\iota}(0))\to C^{1}_{S_0\times\{\gamma(\lambda)\}}([0, \lambda]; M)$
is also a $C^2$ chart by the $\omega$-Lemma (cf. \cite{Moo17}).
Clearly, the Banach space isomorphism
$$
\Gamma_\lambda:W^{1,2}_{V_0\times\{0\}}([0, 1]; \mathbb{R}^n)\to W^{1,2}_{V_0\times\{0\}}([0,\lambda]; \mathbb{R}^n),\;
\xi\mapsto \xi(\lambda^{-1}\cdot)
$$
maps $W^{1,2}_{V_0\times\{0\}}([0,1]; B^n_{\iota}(0))$ onto $W^{1,2}_{V_0\times\{0\}}([0,\lambda]; B^n_{\iota}(0))$.
Put
$$
{\bf L}_\lambda:W^{1,2}_{V_0\times\{0\}}([0,1]; B^n_{\iota}(0))\to\mathbb{R},\;\xi\mapsto \mathcal{L}_{S_0,\lambda}\circ\Phi_\lambda\circ\Gamma_\lambda(\xi).
$$
It is easy to check that
\begin{eqnarray}\label{e:Lagr3.15Gener}
{\bf L}_\lambda(\xi)&=&\int^\lambda_0L\big(t,\Phi_{\lambda}(\Gamma_\lambda(\xi))(t),
\frac{d}{dt}\Phi_{\lambda}(\Gamma_\lambda(\xi))(t)\big)dt\nonumber\\
&=&\int^1_0\lambda L\big(\lambda s,\Phi_{\lambda}(\Gamma_\lambda(\xi))(\lambda s),
\frac{d}{dt}\Phi_{\lambda}(\Gamma_\lambda(\xi))(t)\Big|_{t=\lambda s}\big)ds.
\end{eqnarray}
Note that $\Phi_{\lambda}(\Gamma_\lambda(\xi))(t)=\phi_{\overline{\gamma}}(t,
{\bf u}_\lambda(t)+\Gamma_\lambda(\xi)(t))=\phi_{\overline{\gamma}}(t,{\bf u}_\lambda(t)+\xi(t/\lambda))$
for all $t\in [0, \lambda]$. We have
$$
\frac{d}{dt}\Phi_{\lambda}(\Gamma_\lambda(\xi))(t)=
D_1\phi_{\overline{\gamma}}(t, {\bf u}_\lambda(t)+\xi(t/\lambda))+ D_2\phi_{\overline{\gamma}}(t, {\bf u}_\lambda(t)+\xi(t/\lambda))[\dot{\bf u}_\lambda(t)+\frac{1}{\lambda}\dot{\xi}(t/\lambda)].
$$
It follows that for any $s\in [0,1]$,
\begin{eqnarray*}
&&\Phi_{\lambda}(\Gamma_\lambda(\xi))(\lambda s)=\phi_{\overline{\gamma}}(\lambda s,{\bf u}(\lambda s)+\xi(s)),\\
&&\frac{d}{dt}\Phi_{\lambda}(\Gamma_\lambda(\xi))(t)\Big|_{t=\lambda s}=
D_1\phi_{\overline{\gamma}}(\lambda s, {\bf u}(\lambda s)+\xi(s))+ D_2\phi_{\overline{\gamma}}(\lambda s,
{\bf u}(\lambda s)+\xi(s))[\dot{\bf u}(\lambda s)+\frac{1}{\lambda}\dot{\xi}(s)].
\end{eqnarray*}
Define  $\hat{L}:(0, \tau)\times [0, 1]\times  B^n_{\iota}(0)\times\R^n\to\R$  by
\begin{eqnarray}\label{e:Lagr3.15Gener*}
&&\hat L(\lambda, s, q,  v)=\hat L_\lambda(s, q,  v)\nonumber\\
&=&\lambda L\big(\lambda s, \phi_{\overline{\gamma}}(\lambda s, {\bf u}(\lambda s)+q),
D_1\phi_{\overline{\gamma}}(\lambda s, {\bf u}(\lambda s)+q)+
\frac{1}{\lambda}D_2\phi_{\overline{\gamma}}(\lambda s, {\bf u}(\lambda s)+ q)[v+\lambda\dot{\bf u}(\lambda s)]\big).\nonumber\\
\end{eqnarray}
Since both ${\bf u}$ and $L$ are $C^3$, $\hat{L}$ is $C^2$ and  fiberwise strictly convex. Clearly, (\ref{e:Lagr3.15Gener}) becomes
\begin{eqnarray}\label{e:Lagr3.15Gener**}
{\bf L}_\lambda(\xi)=\int^1_0\hat{L}_\lambda\left(t,\xi(t), \dot{\xi}(t)\right)dt.
\end{eqnarray}
Let $\gamma_\lambda:=\gamma|_{[0,\lambda]}$. Then $\Phi_\lambda\circ\Gamma_\lambda(0)=\gamma_\lambda$ and
\begin{eqnarray*}
&&d{\bf L}_\lambda(0)=d\mathcal{L}_{S_0,\lambda}(\gamma_\lambda)\circ d(\Phi_\lambda\circ\Gamma_\lambda)(0)=0,\\
&&d^2{\bf L}_\lambda(0)[\xi,\eta]=d^2\mathcal{L}_{S_0,\lambda}(\gamma_\lambda)\big[d\Phi_\lambda(0)[\Gamma_\lambda(\xi)],
d\Phi_\lambda(0)[\Gamma_\lambda(\eta)]\big],\quad\forall \xi,\eta\in W^{1,2}_{V_0\times\{0\}}([0,1]; \mathbb{R}^n).
\end{eqnarray*}
Let
$m^-({\bf L}_\lambda, 0)$ and $m^0({\bf L}_\lambda, 0)$ be  the Morse index and nullity of ${\bf L}_\lambda$ at $0$,
respectively. Then
\begin{equation}\label{e:DustMorse*}
m^-({\bf L}_\lambda, 0)=m^-(\mathcal{L}_{S_0,\lambda}, \gamma_\lambda)\quad\hbox{and}\quad
m^0({\bf L}_\lambda, 0)=m^0(\mathcal{L}_{S_0,\lambda}, \gamma_\lambda)
\end{equation}
because
$d\Phi_\lambda(0)\circ\Gamma_\lambda\big(C^1_{V_0\times\{0\}}([0,1]; \mathbb{R}^n)\big)$ is equal to
$$
C^1_{S_0\times\{\gamma_\lambda(\lambda)\}}(\gamma_\lambda^\ast TM):=\left\{X\in C^1(\gamma_\lambda^\ast TM)\,\big|\, X(0)\in
T_{\gamma_\lambda(0)}S_0=T_{\gamma(0)}S_0,\, X(\lambda)=0\right\}.
$$

 \begin{claim}~\label{cl:DustMorse}
  For $\mu\in (0, \tau]$,  $\mu$ is a bifurcation instant for $(S_0, \gamma)$
 if and only if  $(\mu, 0)\in (0, \tau]\times C^1_{V_0\times \{0\}}([0, 1]; B^n_{\iota}(0))$
 is a  bifurcation point of the problem
 \begin{eqnarray}\label{e:LagrGenerEuDust}
&&\frac{d}{dt}\big(\partial_v\hat{L}_\lambda(t, x(t), \dot{x}(t))\big)-\partial_x \hat{L}_\lambda(t, x(t), \dot{x}(t))=0,\\
&&\left.\begin{array}{ll}
x\in C^2([0, 1]; B^n_{\iota}(0)),\;(x(0), x(1))\in V_0\times \{0\}\quad\hbox{and}\quad\\
\partial_v\hat{L}_\lambda(0, x(0), \dot{x}(0))[v_0]=0\;\forall v_0\in V_0.\label{e:LagrGenerBEuDust}
\end{array}\right\}
\end{eqnarray}
 with respect to the trivial branch $\{(\lambda,0)\,|\, \lambda\in (0, \tau]\}$
in $(0, \tau]\times C^1_{V_0\times \{0\}}([0, 1]; B^n_{\iota}(0))$.
 \end{claim}
 \begin{proof}[\bf Proof]
 By Definition~\ref{def:bifurLagr} a real $\mu\in (0, \tau]$ is a bifurcation instant for $(S_0, \gamma)$
 if and only if  there exists a sequence $(\lambda_k)\subset (0,\tau]$ converging to $\mu$
and a sequence of Euler-Lagrange curves of $L$ emanating perpendicularly from $S_0$, $\gamma^k:[0, \lambda_k]\to M$,   such that
(\ref{e:Lgeobifu1}) and (\ref{e:Lgeobifu2}) are satisfied, i.e.,
$$
\gamma^k(\lambda_k)=\gamma(\lambda_k)\;\hbox{for all $k\in\N$},\quad
0<\|\gamma^k-\gamma|_{[0, \lambda_k]}\|_{C^1([0, \lambda_k],\mathbb{R}^N)}\to 0\;\hbox{ as $k\to\infty$}.
$$
Since $\lambda_k\to \mu$, from the latter we deduce that for each $k$ large enough
each $\gamma^k$ sits in the image of the map
$\Phi_{\lambda_k}:C^1_{V_0\times\{0\}}([0, \lambda_k]; B^n_{\iota}(0))\to C^{1}_{S_0\times\{\gamma(\lambda_k)\}}([0, \lambda_k]; M)$
and therefore  there exists a unique
${\bf u}^k\in C^1_{V_0\times\{0\}}([0, \lambda_k]; B^n_{\iota}(0))$
such that $\Phi_{\lambda_k}({\bf u}^k)=\gamma^k\in C^{1}_{S_0\times\{\gamma(\lambda_k)\}}([0, \lambda_k]; M)$.
Let ${\bf v}^k:=(\Gamma_{\lambda_k})^{-1}({\bf u}^k)\in C^1_{V_0\times \{0\}}([0, 1]; B^n_{\iota}(0))$.
It follows from $\Phi_{\lambda_k}(0)=\gamma|_{[0,\lambda_k]}$ that
$0<\|{\bf u}^k\|_{C^1([0, \lambda_k],\mathbb{R}^n)}\to 0$ and so
$0<\|{\bf v}^k\|_{C^1([0, 1],\mathbb{R}^n)}\to 0$
as $k\to\infty$. Moreover, $\Phi_{\lambda_k}\circ\Gamma_{\lambda_k}({\bf v}^k)=\gamma|_{[0, \lambda_k]}$ implies
 $$
 d{\bf L}_{\lambda_k}({\bf v}^k)=d\mathcal{L}_{S_0,\lambda_k}\big(\Phi_{\lambda_k}\circ\Gamma_{\lambda_k}({\bf v}^k)\big)
\circ d(\Phi_{\lambda_k}\circ\Gamma_{\lambda_k})({\bf v}^k)=0.
$$
These affirm the necessary.

Carefully checking the above arguments it is easily seen that the sufficiency also holds.
 \end{proof}

 Suppose that $\mu\in (0, \tau]$ is a bifurcation instant for $(S_0, \gamma)$.
Then it follows from  Claim~\ref{cl:DustMorse} and  Theorem~\ref{th:bif-nessLagrGener} (or Theorem~\ref{th:bif-nessLagrGenerEu})
that $m^0({\bf L}_{\mu}, 0)>0$ and therefore
$m^0(\mathcal{L}_{S_0, \mu}, \gamma_{\mu})>0$ by (\ref{e:DustMorse*}). That is,
$\mu$ is a $S_0$-focal point along $\gamma$. (ii) is proved.

Finally, let us prove (iii).  Suppose that $\mu\in (0, \tau)$ is a $S_0$-focal point along $\gamma$.
By (i) or Theorem~\ref{th:DustMorse} there only exist finitely many numbers in $(0,\tau)$,
$0<\mu_1<\cdots<\mu_m<\tau$, such that  $m^0(\mathcal{L}_{S_0,\mu_i}, \gamma_{\mu_i})>0$,
$i=1,\cdots,m$. Therefore
$\mu\in\{\mu_1,\cdots,\mu_m\}$.
Then by Theorem~\ref{th:DustMorse} and (\ref{e:DustMorse*})  we obtain that
\begin{eqnarray}\label{e:LagrFin3.19}
&&m^0({\bf L}_{\mu_i}, 0)\ne 0,\; i=1,\cdots,m\quad\hbox{and}\quad
m^0({\bf L}_{\lambda},0)=0\quad\hbox{for}\quad \lambda\in(0,\tau)\setminus\{\mu_1,\cdots,\mu_m\},\nonumber\\
&&m^-({\bf L}_{\lambda},0)=\sum_{0<s<\lambda}m^0({\bf L}_{s},0),\quad\forall\lambda\in (0,\tau].
 \end{eqnarray}
Let $\rho$ be the distance from $\mu$ to the set $\{0,\mu_1,\cdots,\mu_m,\tau\}\setminus\{\mu\}$.
Then for any $0<\epsilon<\rho$ it holds that $m^-({\bf L}_{\mu-\epsilon},0)\ne m^-({\bf L}_{\mu+\epsilon},0)$
and that $m^0({\bf L}_{\mu-\epsilon},0)= m^0({\bf L}_{\mu+\epsilon},0)=0$.
By Theorem~\ref{th:bif-nessLagrGenerEu}(II) and Claim~\ref{cl:DustMorse}
we deduce that  $\mu$ is a bifurcation instant for $(S_0, \gamma)$.
This completes the proof of the first claim in (iii).

It remains to prove others in (iii).
 Note that (\ref{e:LagrFin3.19}) gives rise to
\begin{equation*}
m^-({\bf L}_\lambda, 0)=\left\{\begin{array}{ll}
m^-({\bf L}_{\mu}, 0)\,&\,{\rm for}\;
\lambda<\mu\;\hbox{near $\mu$},\\
m^-({\bf L}_{\mu},0)+m^0({\bf L}_{\mu},0)\,&\,{\rm for}\;
\lambda>\mu\;\hbox{near $\mu$}.
\end{array}\right.
\end{equation*}
Then Theorem~\ref{th:bif-suffLagrGenerEu} may be applicable to $\tilde{\mathcal{E}}^\ast_\lambda={\bf L}_{\lambda}$ and $V_1=\{0\}\subset\mathbb{R}^n$, and therefore
  one of the following alternatives occurs:
\begin{enumerate}
\item[\rm (A)] The problem (\ref{e:LagrGenerEuDust})--(\ref{e:LagrGenerBEuDust})
 with $\lambda=\mu$ has a sequence of distinct solutions, ${\bf v}^k\ne 0$, $k=1,2,\cdots$,
which converges to $0$ in $C^2_{V_0\times\{0\}}([0,1]; B^n_{\iota}(0))$.

\item[\rm (B)] There exists a real $0<\sigma<\min\{\mu,\tau-\mu\}$ such that for every $\lambda\in [\mu-\sigma,\mu+\sigma]\setminus\{\mu\}$
the problem (\ref{e:LagrGenerEuDust})--(\ref{e:LagrGenerBEuDust}) with parameter value $\lambda$
has a  solution ${\bf v}^\lambda\ne 0$  to satisfy  $\|{\bf v}^\lambda\|_{C^2}\to 0$  as $\lambda\to \mu$.

\item[\rm (C)] For a given neighborhood $\mathfrak{W}$ of $0\in C^1_{V_0\times\{0\}}([0,1]; B^n_{\iota}(0))$,
there exists a real $0<\sigma<\min\{\mu,\tau-\mu\}$ such that for any $\lambda\in\Lambda^0\setminus\{\mu\}$, where $\Lambda^0=[\mu-\sigma, \mu]$
or $[\mu, \mu+\sigma]$, the problem
(\ref{e:LagrGenerEuDust})--(\ref{e:LagrGenerBEuDust}) with parameter value $\lambda$
has at least two distinct solutions in  $\mathfrak{W}$, ${\bf v}_1^\lambda\ne 0$ and ${\bf v}_2^\lambda\ne 0$,
which can also be required to satisfy ${\bf L}_{\lambda}({\bf v}_1^\lambda)\ne {\bf L}_{\lambda}({\bf v}_2^\lambda)$
provided that $m^0({\bf L}_{\mu}, 0)>1$ and the problem
(\ref{e:LagrGenerEu})--(\ref{e:LagrGenerBEu}) with parameter value $\lambda$
has only finitely many solutions in $\mathfrak{W}$.
\end{enumerate}

Let us prove that the cases (A) and (B) lead to (iii-1) and (iii-2) in Theorem~\ref{th:MorseBif}, respectively.

\underline{In the case of (A) above}, $C^2$ paths
$$
\alpha_k:[0,\mu]\to M,\;t\mapsto (\Phi_{\mu}\circ\Gamma_\mu)({\bf v}^k)(t)=\phi_{\overline{\gamma}}\big(t,{\bf u}(t)+ {\bf  v}^k(t/\mu)\big),\quad k=1,2,\cdots,
$$
are a sequence distinct $C^2$ Euler-Lagrange curves of $L$ emanating perpendicularly from $S_0$ and ending at $\gamma(\mu)$,
and each of them is not equal to $\gamma|_{[0,\mu]}$.
Moreover, since $\phi_{\overline{\gamma}}:[0,\tau]\times B^n_{2\iota}(0)\to M$
 and ${\bf u}:[0, \tau]\to B^n_{\iota}(0)$ are $C^5$ and $C^3$, respectively,
we have a continuous map
$$
C^2([0,\mu]; [0,\tau]\times B^n_{2\iota}(0))\to C^{2}([0, \mu]; M),\;w\mapsto \phi_{\overline{\gamma}}\circ w
$$
by Exercise 10 on the page 64 of \cite{Hir}, and a $C^\infty$ map
$$
\Theta:C^2([0,1]; B^n_{\iota}(0))\to C^2([0,\mu]; [0,\tau]\times B^n_{2\iota}(0))
$$
given by $\Theta(v)(t)=(t, {\bf u}(t)+v(t/\mu))$. Hence the composition map
$$
C^2([0,1]; B^n_{\iota}(0))\to C^{2}([0, \mu]; M),\;v\mapsto \phi_{\overline{\gamma}}\circ \Theta(v)=(\Phi_{\mu}\circ\Gamma_\mu)(v)
$$
is continuous, and therefore  $\alpha_k=(\Phi_{\mu}\circ\Gamma_\mu)({\bf v}^k)\to\Phi_{\mu}(0)=\gamma|_{[0,\mu]}$ in $C^2([0,\mu],\mathbb{R}^N)$ as $k\to\infty$.

\underline{In the case of (B) above}, for each $\lambda\in [\mu-\sigma,\mu+\sigma]\setminus\{\mu\}$,
$$
\alpha^\lambda:[0,\lambda]\to M,\;t\mapsto (\Phi_{\lambda}\circ\Gamma_\lambda)({\bf v}^\lambda)(t)=\phi_{\overline{\gamma}}\big(t,{\bf u}(t)+ {\bf  v}^\lambda(t/\lambda)\big)
$$
is a $C^2$ Euler-Lagrange curve of $L$ emanating perpendicularly from $S_0$ and ending at $\gamma(\lambda)$,
and not equal to $\gamma|_{[0,\lambda]}$.
We cannot prove the desired claim as above. But noting that we have assumed $M\subset\mathbb{R}^N$,
$\phi_{\overline{\gamma}}$ can be viewed as a $C^5$ map from $[0,\tau]\times B^n_{2\iota}(0)$ to $\mathbb{R}^N$.
A straight computation leads to
 \begin{eqnarray*}
(\alpha^\lambda)'(t)&=&D_1\phi_{\overline{\gamma}}\big(t,{\bf u}(t)+ {\bf  v}^\lambda(t/\lambda)\big)
+D_2\phi_{\overline{\gamma}}\big(t,{\bf u}(t)+ {\bf  v}^\lambda(t/\lambda)\big)
[{\bf u}'(t)+ \frac{1}{\lambda}({\bf  v}^\lambda)'(t/\lambda)],\\
(\alpha^\lambda)''(t)&=&D_1D_1\phi_{\overline{\gamma}}\big(t,{\bf u}(t)+ {\bf  v}^\lambda(t/\lambda)\big)
+D_2D_1\phi_{\overline{\gamma}}\big(t,{\bf u}(t)+ {\bf  v}^\lambda(t/\lambda)\big)
[{\bf u}'(t)+ \frac{1}{\lambda}({\bf  v}^\lambda)'(t/\lambda)]\\
&+&D_2D_1\phi_{\overline{\gamma}}\big(t,{\bf u}(t)+ {\bf  v}^\lambda(t/\lambda)\big)[{\bf u}'(t)+ \frac{1}{\lambda}({\bf  v}^\lambda)'(t/\lambda)]\\
&+&D_2\phi_{\overline{\gamma}}\big(t,{\bf u}(t)+ {\bf  v}^\lambda(t/\lambda)\big)
[{\bf u}''(t)+ \frac{1}{\lambda^2}({\bf  v}^\lambda)''(t/\lambda)]\\
&+&D_2D_2\phi_{\overline{\gamma}}\big(t,{\bf u}(t)+ {\bf  v}^\lambda(t/\lambda)\big)
[{\bf u}'(t)+ \frac{1}{\lambda}({\bf  v}^\lambda)'(t/\lambda), {\bf u}'(t)+ \frac{1}{\lambda}({\bf  v}^\lambda)'(t/\lambda)],
 \end{eqnarray*}
where we denote by $D_1$ and $D_2$ the partial derivatives of $\phi_{\overline{\gamma}}(t, x)$
with respect to the arguments $t$ and $x$, respectively, and in particular, the final term is equal to
$$
\frac{\partial^2}{\partial s_1\partial s_2}
\phi_{\overline{\gamma}}\Big(t,{\bf u}(t)+ {\bf  v}^\lambda(t/\lambda)+
(s_1+s_2)({\bf u}'(t)+ \frac{1}{\lambda}({\bf  v}^\lambda)'(t/\lambda))\Big)\Big|_{s_1=s_2=0}.
$$
Since $(\mu+\sigma)^i\le\lambda^i\le(\mu-\sigma)^i$, $i=1,2$, and $\|v^\lambda\|_{C^2}\to 0$ as $\lambda\to\mu$,
it follows from the above expressions that $\|\alpha_\lambda-\gamma|_{[0,\lambda]}\|_{C^2([0,\lambda],\mathbb{R}^N)}\to 0$ as $\lambda\to\mu$.

(As pointed out below (\ref{e:LagrCurve}) the above Euler-Lagrange curves of $L$, $\alpha_k$, $\alpha^\lambda$ and the following
$\beta^i_\lambda$ are actually $C^3$.)

\underline{Suppose that (iii-1) and (iii-2) in Theorem~\ref{th:MorseBif} do not hold}.
Then the above proofs show that the cases (A) and (B) do not occur.
That is, the case (C) must hold. Let us prove that it implies (iii-3) in Theorem~\ref{th:MorseBif}.
By the proof of the case (B) above we have positive numbers $\delta$
and $\sigma'<\min\{\mu,\tau-\mu\}$ such that
$$
\mathfrak{W}_0:=\big\{{\bf v}\in C^1_{V_0\times\{0\}}([0,1]; B^n_{\iota}(0))\,\big|\, \|{\bf v}\|_{C^1}\le\delta\big\}\subset\mathfrak{W}
$$
and that $C^1$-paths
$$
\alpha_{\lambda,\bf v}:[0,\lambda]\to M,\;t\mapsto \phi_{\overline{\gamma}}\left(t,{\bf u}(t)+ {\bf  v}(t/\lambda)\right)
$$
associated with $(\lambda, {\bf v})\in [\mu-\sigma', \mu+\sigma]\times \{{\bf v}\in C^1_{V_0\times\{0\}}([0,1]; B^n_{\iota}(0))\,|\, \|{\bf v}\|_{C^1}\le\delta\}$ satisfy
$$
\|\alpha_{\lambda,\bf v}-\gamma|_{[0,\lambda]}\|_{C^1}<\epsilon.
$$
 By (C) there exists a positive number $\sigma_0\le\min\{\sigma',\sigma\}$ such that
 the corresponding conclusions in (C) also hold true after
   $\mathfrak{W}$ and $\sigma$ are replaced by $\mathfrak{W}_0$ and $\sigma_0$, respectively.
For $\Lambda^*:=\Lambda^0\cap[\mu-\sigma_0,\mu+\sigma_0]$ and $\lambda\in\Lambda^\ast\setminus\{\mu\}$,
let ${\bf v}_1^\lambda\in\mathfrak{W}_0\setminus\{0\}$
and ${\bf v}_2^\lambda\in\mathfrak{W}_0\setminus\{0\}$ be two distinct solutions of
the problem (\ref{e:LagrGenerEuDust})--(\ref{e:LagrGenerBEuDust}) with parameter value $\lambda$.
Then
$$
\beta^i_\lambda:[0,\lambda]\to M,\;t\mapsto (\Phi_{\lambda}\circ\Gamma_\lambda)({\bf v}_i^\lambda)(t)=\phi_{\overline{\gamma}}\big(t,{\bf u}(t)+ {\bf  v}_i^\lambda(t/\lambda)\big),\quad i=1,2,
$$
are two distinct $C^2$ Euler-Lagrange curves of $L$ emanating perpendicularly from $S_0$ and ending at $\gamma(\lambda)$,
and not equal to $\gamma|_{[0,\lambda]}$. Moreover both satisfy
$\|\beta_i^{\lambda}-\gamma|_{[0,\lambda]}\|_{C^1}<\epsilon$, $i=1,2$.

Suppose further that $m^0(\mathcal{L}_{S_0,\mu}, \gamma_{\mu})=m^0({\bf L}_{\mu},0)>1$.
For $\lambda\in\Lambda^\ast\setminus\{\mu\}$, if there only exist
finitely many distinct $C^2$ Euler-Lagrange curves of $L$ emanating perpendicularly from $S_0$ and ending at $\gamma(\lambda)$,
$\alpha_1,\cdots,\alpha_m$, such that $\|\alpha_i-\gamma|_{[0,\lambda]}\|_{C^1}<\epsilon$, $i=1,\cdots, m$.
then the above arguments imply that
the problem (\ref{e:LagrGenerEuDust})--(\ref{e:LagrGenerBEuDust}) with parameter value $\lambda$
has only finitely many solutions in $\mathfrak{W}_0$. In this case the above
${\bf v}_1^\lambda$ and ${\bf v}_2^\lambda$ can be chosen to satisfy
${\bf L}_{\lambda}({\bf v}_1^\lambda)\ne {\bf L}_{\lambda}({\bf v}_2^\lambda)$, which implies (\ref{e:diffEnergy}).
\end{proof}

\begin{proof}[\bf Proof of Theorem~\ref{th:DustMorse}]
 Follow the first paragraph in the proof of Theorem~\ref{th:MorseBif}.
We have a $C^2$ chart
\begin{eqnarray*}
\Psi_{\lambda}:C^1_{V_0\times\{{\bf u}(\lambda)\}}([0, \lambda]; B^n_{2\iota}(0))\to C^{1}_{S_0\times\{\gamma(\lambda)\}}([0, \lambda]; M)
\end{eqnarray*}
given by $\Psi_{\lambda}(\xi)(t)=\phi_{\overline{\gamma}}(t, \xi(t))\;\forall t\in [0, \lambda]$ for each $\lambda\in (0,\tau]$.
 Then $\Psi_\lambda({\bf u}_\lambda(t))=\gamma_\lambda$, where ${\bf u}_\lambda:={\bf u}|_{[0,\lambda]}$
 and $\gamma_\lambda=\gamma|_{[0,\lambda]}$.
  Define $\tilde L: [0,\tau]\times B^n_{2\iota}(0)\times\R^n\to\R$ by
\begin{equation*}
\tilde L(t, q,  v)=\tilde L(t, q,  v)=L\left(t, \phi_{\overline{\gamma}}(t, q),
D_t\phi_{\overline{\gamma}}(t, q)+ D_q\phi_{\overline{\gamma}}(t, q)[v]\right).
\end{equation*}
By Assumption~\ref{ass:LagrGenerC},  $\tilde L$ is $C^3$
and  $\tilde{L}(t, q, v)$ is strictly convex in $v$ for each $(t, q)\in  [0,\tau]\times B^n_{2\iota}(0)$.
Therefore for each $\lambda\in (0, \tau]$ the functional
\begin{equation*}
C^1_{V_0\times\{{\bf u}(\lambda)\}}([0,\lambda]; B^n_{2\iota}(0))\ni x\mapsto
\tilde{\mathcal{E}}_{V_0,\lambda}(x)=\int^\lambda_0\tilde{L}(t, x(t), \dot{x}(t))dt\in\R
\end{equation*}
 is $C^2$, and satisfies $d\tilde{\mathcal{E}}_{V_0,\lambda}({\bf u}_\lambda)=0$ and
$$
\tilde{\cal E}_{V_0,\lambda}(\xi)={\cal L}_{S_0,\lambda}\left(\Psi_{\overline{\gamma}}(\xi)\right)\quad
\hbox{for all}\; \xi\in C^1_{V_0\times\{{\bf u}(\lambda)\}}([0,\lambda]; B^n_{2\iota}(0)).
$$
With the same reasoning as for (\ref{e:DustMorse*}) these yield
\begin{equation}\label{e:DustMorse}
m^-(\tilde{\mathcal{E}}_{V_0,\lambda}, {\bf u}_\lambda)=m^-(\mathcal{L}_{S_0,\lambda}, \gamma_\lambda)\quad
\hbox{and}\quad m^0(\tilde{\mathcal{E}}_{V_0,\lambda}, {\bf u}_\lambda)=m^0(\mathcal{L}_{S_0,\lambda}, \gamma_\lambda).
\end{equation}

Therefore \textsf{from now on we may assume that $M=\R^n$ and $S_0$ is a linear subspace  in $\R^n$.}
Then for  $0<\lambda\le\tau$ and $y,z\in C^1_{S_0\times\{0\}}([0,\lambda];\R^{n})$,
\begin{equation}\label{e:Linquadform}
D^2\mathcal{L}_{S_0,\lambda}(\gamma_\lambda)[y,z]=\int^\lambda_0\left[(\textsf{P}\dot{y}+
\textsf{Q} y)\cdot\dot{z}+\textsf{ Q}^\top\dot{y}\cdot z+ \textsf{R} y\cdot z\right]dt,
\end{equation}
where $\textsf{P}(t)=\partial_{vv}L\left(t, \gamma(t),\dot{\gamma}(t)\right)$,
$\textsf{Q}(t)=\partial_{xv}L\left(t, \gamma(t), \dot{\gamma}(t)\right)$ and
$\textsf{R}(t)=\partial_{xx}L\left(t, \gamma(t), \dot{\gamma}(t)\right)$.

Let $v=v(t,x,\xi)$ be the solution of $\xi=\partial_{v}L(t,x,v)$. Define
$$
H(t,x,\xi)=\langle v(t,x,\xi),\xi\rangle_{\R^n}-L(t, x, v(t,x,\xi)).
$$
Writing $x(t)=\gamma(t)$ and $\xi(t)=\partial_{v}L(t,x(t),v(t))$, we get that
(\ref{e:LagrCurve}) is equivalent to
\begin{equation}\label{e:HamCurve1}
\left.\begin{array}{ll}
&\frac{d}{dt}x(t)=\partial_{\xi} H(t,x(t),\xi(t)),\\
&\frac{d}{dt}\xi(t)=-\partial_x H(t, x(t), \xi(t))
\end{array}\right\}
\end{equation}
with boundary condition
\begin{equation}\label{e:HamCurve2}
\left.\begin{array}{ll}
&(x(0), x(\tau))\in S_0\times\{q\}\subset \R^n\times \R^n,\\
&(\xi(0),-\xi(\tau))\in S_0^\bot\times\R^n
\end{array}\right\}
\end{equation}
since $(T_{(x(0), x(\tau))}(S_0\times\{q\})^\bot=(S_0\times\{0\})^\bot=S_0^\bot\times\R^n$.
Note that $T^\ast\R^n=\R^n\times\R^n$. The natural projection $\pi:T^\ast M\to M$
becomes the projection from $\R^n\times\R^n$ onto the first factor $\R^n$.
The co-normal bundle of $S_0$,
i.e., $N^\ast S_0=\{(x,\xi)\in T^\ast M\,|\, x\in S_0, \langle\xi, v\rangle=0\;\forall v\in T_{x}S_0\}$,
becomes $S_0\times S_0^\bot$ and therefore its tangent space $U=S_0\times S_0^\bot$. Moreover
the vertical space $V$ is equal to $\{0\}\times\R^n\subset\R^n\times\R^n$.
By \cite[Proposition~4.5]{Du} with $\rho=U\times V$ we have
\begin{equation}\label{e:DuProp4.5}
m^-(D^2\mathcal{L}_{S_0,\lambda}(\gamma_\lambda))=\sum_{0<s<\lambda}\dim (U\cap\Phi(0,s)^{-1}(V)),
\end{equation}
where  $\Phi(0,t)$, according to (1.19) and (1.20)  in \cite{Du}, is the fundamental matrix solution of
\begin{equation}\label{e:LinHamCurve1}
\left(\begin{array}{c}
             \dot{x}(t)\\
             \dot{y}(t) \\
           \end{array}
         \right)= A(0,t)
         \left(\begin{array}{c}
             {x}(t)\\
             {y}(t) \\
           \end{array}
         \right)
\end{equation}
with
\begin{eqnarray*}
A(0,t)=\left(\begin{array}{cc}
            D^2_{x\xi}H(t,x(t),\xi(t)) & D^2_{\xi}H(t,x(t),\xi(t)) \\
             -D^2_{xx}H(t,x(t),\xi(t))& -D^2_{\xi x}H(t,x(t),\xi(t)) \\
           \end{array}
         \right)
\end{eqnarray*}
with $\xi(t)=\partial_{v}L(t,\gamma(t),\dot{\gamma}(t))$. Note that \cite[(1.13)-(1.14)]{Du} with $\mu=0$ corresponds to
the Jacobi equation of the functional $\mathcal{L}_{S_0,\tau}$, namely, the following linearized problem of (\ref{e:LagrCurve})
\begin{equation}\label{e:LinLagrCurve}
\left.\begin{array}{ll}
&\frac{d}{dt}\Big(\textsf{Q}(t)\cdot{x}(t)+ \textsf{P}(t)\cdot\dot{x}(t)\Big)=
\textsf{R}(t)\cdot x(t)+ \textsf{Q}^\top(t)\cdot\dot{x}(t),\\
&x(0)\in S_0,\;x(\tau)=0\quad\hbox{and}\quad
\textsf{Q}(0)\cdot{x}(0)+ \textsf{P}(0)\cdot\dot{x}(0)\in S_0^\bot
\end{array}\right\}
\end{equation}
whose solution space is equal to the kernel of $D^2\mathcal{L}_{S_0,\tau}(\gamma)$.
It was claimed in \cite[page 179]{Du} that (\ref{e:LinLagrCurve}) is equivalent to
(\ref{e:LinHamCurve1}) plus
with boundary condition
\begin{equation}\label{e:HamCurve23}
\left.\begin{array}{ll}
&(x(0), x(\tau))\in S_0\times\{0\}\subset \R^n\times \R^n,\\
&y(0)\in S_0^\bot.
\end{array}\right\}
\end{equation}
According to the deduction from \cite[(1.13)-(1.14)]{Du} with $\mu=0$ to \cite[(1.19)-(1.21)]{Du} with $\mu=0$,
(\ref{e:LinHamCurve1}) was obtained by putting
 $y(t):=\textsf{Q}(t)\cdot{x}(t)+ \textsf{P}(t)\cdot\dot{x}(t)$ in (\ref{e:LinLagrCurve}).
Hence (\ref{e:LinHamCurve1}) is exactly
\begin{equation}\label{e:LinHamCurve3}
\left(\begin{array}{c}
             \dot{x}(t)\\
             \dot{y}(t) \\
           \end{array}
         \right)=
    \left(\begin{array}{cc}
             -[\textsf{P}(t)]^{-1}\textsf{Q}(t) & [\textsf{P}(t)]^{-1} \\
             \textsf{R}(t)-[\textsf{Q}(t)]^\top[\textsf{P}(t)]^{-1}\textsf{Q}(t)& [\textsf{Q}(t)]^\top[\textsf{P}(t)]^{-1} \\
           \end{array}
         \right)
         \left(\begin{array}{c}
             {x}(t)\\
             {y}(t) \\
           \end{array}
         \right),
\end{equation}
that is, \begin{eqnarray*}
\left(\begin{array}{cc}
             -[\textsf{P}(t)]^{-1}\textsf{Q}(t) & [\textsf{P}(t)]^{-1} \\
             \textsf{R}(t)-[\textsf{Q}(t)]^\top[\textsf{P}(t)]^{-1}\textsf{Q}(t)& [\textsf{Q}(t)]^\top[\textsf{P}(t)]^{-1} \\
           \end{array}
         \right)=A(0,t)
\end{eqnarray*}
with $\xi(t)=\partial_{v}L(t,\gamma(t),\dot{\gamma}(t))$.
(Hence $\Phi(0,t)=d\phi^t(\gamma(0),\partial_vL(t,\gamma(0),\dot\gamma(0))):\R^{2n}\to\R^{2n}$,
where  $\phi^t$ is the flow of the Hamiltonian system (\ref{e:HamCurve1}); see
\cite[page 192]{Du}.)

Note that the map sending $(\bar{x},\bar{y})\in U\cap\Phi(0,\tau)^{-1}(V)$ to $x$, where
$(x(t),y(t))=\Phi(0,t)(\bar{x},\bar{y})$ with $0\le t\le\tau$, is a linear isomorphism
between $U\cap\Phi(0,\tau)^{-1}(V)$ and the space of solutions of (\ref{e:LinLagrCurve}). We obtain
$$
\dim (U\cap\Phi(0,\tau)^{-1}(V))=m^0(D^2\mathcal{L}_{S_0,\lambda}(\gamma_\tau))
$$
and so $\dim (U\cap\Phi(0, s)^{-1}(V))=m^0(D^2\mathcal{L}_{S_0, s}(\gamma_s))$ since
 $\tau$ may be replaced by any $0<\lambda\le\tau$. The desired conclusion
  follows from these and (\ref{e:DuProp4.5}).
\end{proof}

\section{Proofs of Theorems~\ref{th:bif-nessLagr*},~\ref{th:bif-existLagr*},~\ref{th:bif-suffLagr*}}\label{sec:LagrPPerio1}
\setcounter{equation}{0}

The proofs are completely similar to those of Section~\ref{sec:LagrBound}. We only outline main procedures.

\subsection{Reduction to Euclidean spaces}\label{sec:LagrPPerio1.1}

As in Section~\ref{sec:LagrBound.1.1}, we have  a positive number $\iota$ such that $3\iota$ is less than the injectivity
radius of $g$ at each point on $\gamma_\mu([0, \tau])$, and a path
$\overline{\gamma}\in C^{1}_{\mathbb{I}_g}([0, \tau]; M)\cap C^7([0,\tau];M)$  such that (\ref{e:gamma1}) is satisfied.
Then the injectivity radius of $g$ at each point on
$\overline{\gamma}([0,\tau])$ is at least $2\iota$. 
Then we assume that (\ref{e:gamma2}) is satisfied.
(For cases of Theorems~\ref{th:bif-nessLagr*},~\ref{th:bif-suffLagr*},
it is naturally satisfied after shrinking $\Lambda$ toward $\mu$.)

As in \cite[\S3]{Lu5}, starting with a unit orthogonal
 frame at $T_{\gamma_\mu(0)}M$ and using the parallel
transport along $\overline{\gamma}$ with respect to the Levi-Civita
connection of the Riemannian metric $g$ we get a unit orthogonal
parallel $C^5$ frame field $[0,\tau]\to \overline{\gamma}^\ast TM,\;t\mapsto
(e_1(t),\cdots, e_n(t))$.
 Note that  there exists a unique orthogonal matrix
 $E_{\overline{\gamma}}$ such that
 $(e_1(\tau),\cdots, e_n(\tau))=(\mathbb{I}_{g\ast}e_1(0),\cdots,
\mathbb{I}_{g\ast}e_n(0))E_{\overline{\gamma}}$.
Let $B^n_{2\iota}(0):=\{x\in\R^n\,|\, |x|<2\iota\}$ and $\exp$ denote
 the exponential map of $g$. Then
 \begin{eqnarray}\label{e:Lagr4II}
\phi_{\overline{\gamma}}:[0,\tau]\times B^n_{2\iota}(0)\to
M,\;(t,x)\mapsto\exp_{\overline{\gamma}(t)}\Big(\sum^n_{i=1}x_i
e_i(t)\Big)
 \end{eqnarray}
 is a $C^5$ map and  satisfies 
\begin{eqnarray*}
\phi_{{\overline{\gamma}}}(\tau,x)=\mathbb{I}_g\left(\phi_{\overline{\gamma}}(0, (E_{\overline{\gamma}}
x^\top)^\top)\right)\quad\hbox{and}\quad
d\phi_{\overline{\gamma}}(\tau,x)[(1,v)]=d\phi_{\overline{\gamma}}(0,
(E_{\overline{\gamma}} x^\top)^\top)[(1, (E_{\overline{\gamma}} v^\top)^\top)]
\end{eqnarray*}
for any $(t,x, v)\in [0, \tau]\times B^n_{2\iota}(0)\times\mathbb{R}^n$.
(Note that the tangent map $d\phi_{\overline{\gamma}}:T([0,\tau]\times B^n_{2\iota}(0))\to
TM$ is $C^4$.)
Consider the Hilbert subspace
\begin{equation}\label{e:Rot-space1}
W^{1,2}_{E_{\overline{\gamma}}}([0,\tau];\R^{n}):=\{u\in W^{1,2}([0,\tau];\mathbb{R}^{n})\,|\, u(\tau)=E_{\overline{\gamma}}u(0)\}
\end{equation}
of  $W^{1,2}([0,\tau];\mathbb{R}^{2n})$  equipped with $W^{1,2}$-inner product (\ref{e:innerP2}),
and Banach spaces
\begin{eqnarray}\label{e:Rot-space2}
C^i_{E_{\overline{\gamma}}}([0,\tau];\R^{n})=\{u\in C^i([0,\tau];\R^{n})\,|\,u(\tau)=E_{\overline{\gamma}}u(0)\}
\end{eqnarray}
 with the induced norm $\|\cdot\|_{C^i}$ from $C^i([0,\tau],\mathbb{R}^n)$ for $i\in\N\cup\{0\}$.
 By \cite[Theorem~4.3]{PiTa01}, we get a $C^2$ coordinate chart around $\overline{\gamma}$ on
 the $C^4$ Banach manifold $C^{1}_{\mathbb{I}_g}([0, \tau]; M)$,
\begin{eqnarray}\label{e:Lagr5}
\Phi_{\overline{\gamma}}:C^1_{{E}_{\overline{\gamma}}}([0,\tau],B^n_{2\iota}(0))=\{\xi\in
C^1_{{E}_{\overline{\gamma}}}([0,\tau],\mathbb{R}^n) \,|\,\|\xi\|_{C^0}<2\iota\} \to C^{1}_{\mathbb{I}_g}([0, \tau]; M)
\end{eqnarray}
given by $\Phi_{\overline{\gamma}}(\xi)(t)=\phi_{\overline{\gamma}}(t,\xi(t))$,  and
$$
d\Phi_{\overline{\gamma}}(0): C^1_{{E}_{\overline{\gamma}}}([0,\tau],\mathbb{R}^n)\to C^1_{\mathbb{I}_g}(\overline{\gamma}^\ast TM),\;
\xi\mapsto\sum^n_{j=1}\xi_je_j.
$$
By (\ref{e:gamma1}) and (\ref{e:gamma2}), for each $\lambda\in\Lambda$ there exists a unique map
 ${\bf u}_\lambda:[0,\tau]\to B^n_{\iota}(0)$  such that
$$
\gamma_\lambda(t)=\phi_{\overline{\gamma}}(t,{\bf u}_\lambda(t))=\exp_{\overline{\gamma}(t)}\Big(\sum^n_{i=1}{\bf u}_\lambda^i
e_i(t)\Big),\quad t\in [0,\tau].
$$
As in the proofs of Lemma~\ref{lem:twoCont},~\ref{lem:twoCont+} we have:

\begin{lemma}\label{lem:PeriCont}
 ${\bf u}_\lambda\in C^2([0,\tau]; B^n_{\iota}(0))$,
${\bf u}_\mu(0)=0={\bf u}_\mu(\tau)$ (and so ${\bf u}_\lambda\in C^1_{{E}_{\overline{\gamma}}}([0,\tau],B^n_{\iota}(0))$) and
$$
 (\lambda,t)\mapsto {\bf u}_\lambda(t),\quad (\lambda,t)\mapsto \dot{\bf u}_\lambda(t)\quad\hbox{and}\quad
 (\lambda,t)\mapsto \ddot{\bf u}_\lambda(t)
 $$
 are continuous as maps from  $\Lambda\times [0,\tau]$ to $\R^n$.
 \end{lemma}

Let $\tilde L^\ast:\Lambda\times [0,\tau]\times B^n_{\iota}(0)\times\R^n\to\R$ be given by
(\ref{e:ModifiedLEu}). It satisfies Proposition~\ref{th:bif-LagrGenerEu}.
Each functional
\begin{equation}\label{e:PerbrakefunctGener*}
\tilde{\mathscr{E}}^\ast_\lambda: C^1_{{E}_{\overline{\gamma}}}([0,\tau],B^n_{\iota}(0))\to\R,\;x\mapsto\int^1_0\tilde{L}^\ast_\lambda(t, x(t), \dot{x}(t))dt
\end{equation}
is $C^2$, and satisfies
\begin{equation}\label{e:Pertwofunctional}
\tilde{\mathscr{E}}^\ast_{\lambda}(x)=\mathscr{E}_{\lambda}\left(\Phi_{\overline{\gamma}}(x+{\bf u}_\lambda)\right)\;\forall x\in
C^1_{{E}_{\overline{\gamma}}}([0,\tau],B^n_{\iota}(0))\quad\hbox{and}\quad d\tilde{\mathscr{E}}^\ast_\lambda(0)=0.
\end{equation}
Hence for each $\lambda\in\Lambda$,
 $x\in C^1_{{E}_{\overline{\gamma}}}([0,\tau],B^n_{\iota}(0))$ satisfies
 $d\tilde{\mathscr{E}}^\ast_\lambda(x)=0$ if and only if $\gamma:=\Phi_{\overline{\gamma}}(x+{\bf u}_\lambda)$ satisfies
$d{\mathscr{E}}_\lambda(\gamma)=0$; and in this case $\gamma$ and $x$ 
have the same Morse indexes and nullities.
In particular,  for each $\lambda\in\Lambda$, it holds that
\begin{equation}\label{e:PertwoMorse}
m^-(\tilde{\mathscr{E}}^\ast_\lambda, 0)=m^-(\mathscr{E}_\lambda, \gamma_\lambda)\quad
\hbox{and}\quad
m^0(\tilde{\mathscr{E}}^\ast_\lambda, 0)=m^0(\mathscr{E}_\lambda, \gamma_\lambda).
\end{equation}
The critical points of $\tilde{\mathscr{E}}^\ast_\lambda$ 
correspond to the solutions of the following boundary problem:
\begin{eqnarray}\label{e:PLagrGenerEu}
&&\frac{d}{dt}\big(\partial_v\tilde L^\ast_\lambda(t, x(t), \dot{x}(t))\big)-\partial_q\tilde{L}^\ast_\lambda(t, x(t), \dot{x}(t))=0,\\
&&\left.\begin{array}{ll}
x\in C^2_{{E}_{\overline{\gamma}}}([0,\tau],B^n_{\iota}(0))\quad\hbox{and}\quad\\
({E}_{\overline{\gamma}}^\top)^{-1}\partial_v\tilde{L}^\ast_\lambda(0, x(0), \dot{x}(0))=\partial_v\tilde{L}^\ast_\lambda(\tau, x(\tau), \dot{x}(\tau))
\label{e:PLagrGenerBEu}
\end{array}\right\}
\end{eqnarray}
(\cite[Proposition~4.2]{BuGiHi}). Corresponding to Theorems~\ref{th:bif-nessLagrGenerEu},~\ref{th:bif-existLagrGenerEu+},~\ref{th:bif-suffLagrGenerEu},
 we have the following three theorems,  which also hold true provided that $\tilde{L}^\ast$ satisfies (a) in Proposition~\ref{th:bif-LagrGenerEu} and the weaker
 (b') in Remark~\ref{rm:bif-LagrGenerEu} as noted in Remark~\ref{rm:bif-LagrGenerEu}.

\begin{theorem}\label{th:Pbif-nessLagrGenerEu}
\begin{enumerate}
\item[\rm (I)]{\rm (\textsf{Necessary condition}):}
  Suppose that $(\mu, 0)\in\Lambda\times C^1_{{E}_{\overline{\gamma}}}([0,\tau],B^n_{\iota}(0))$  is a
   bifurcation point along sequences of the problem (\ref{e:PLagrGenerEu})--(\ref{e:PLagrGenerBEu})
  with respect to the trivial branch $\{(\lambda,0)\,|\,\lambda\in\Lambda\}$ in $\Lambda\times C^1_{{E}_{\overline{\gamma}}}([0,\tau],B^n_{\iota}(0))$.
  Then $m^0(\tilde{\mathscr{E}}^\ast_{\mu}, 0)>0$.

\item[\rm (II)]{\rm (\textsf{Sufficient condition}):}
Suppose that $\Lambda$ is first countable and that there exist two sequences in  $\Lambda$ converging to $\mu$, $(\lambda_k^-)$ and
$(\lambda_k^+)$,  such that one of the following conditions is satisfied:
 \begin{enumerate}
 \item[\rm (II.1)] For each $k\in\mathbb{N}$, either $0$  is not an isolated critical point of $\tilde{\mathscr{E}}^\ast_{\lambda^+_k}$,
 or $0$ is not an isolated critical point of $\tilde{\mathscr{E}}^\ast_{\lambda^-_k}$,
 or $0$  is an isolated critical point of $\tilde{\mathscr{E}}^\ast_{\lambda^+_k}$ and $\tilde{\mathscr{E}}^\ast_{\lambda^-_k}$ and
  $C_m(\tilde{\mathscr{E}}^\ast_{\lambda^+_k}, 0;{\bf K})$ and $C_m(\tilde{\mathscr{E}}^\ast_{\lambda^-_k}, 0;{\bf K})$
  are not isomorphic for some Abel group ${\bf K}$ and some $m\in\mathbb{Z}$.
\item[\rm (II.2)] For each $k\in\mathbb{N}$, there exists $\lambda\in\{\lambda^+_k, \lambda^-_k\}$ such that
$0$  is an either non-isolated or homological visible critical point of
$\tilde{\mathscr{E}}^\ast_{\lambda}$ , and
$$[m^-(\tilde{\mathscr{E}}^\ast_{\lambda_k^-}, 0),
m^-(\tilde{\mathscr{E}}^\ast_{\lambda_k^-}, 0)+
m^0(\tilde{\mathscr{E}}^\ast_{\lambda_k^-}, 0)]\cap[m^-(\tilde{\mathscr{E}}^\ast_{\lambda_k^+}, 0),
m^-(\tilde{\mathscr{E}}^\ast_{\lambda_k^+}, 0)+m^0(\tilde{\mathscr{E}}^\ast_{\lambda_k^+}, 0)]=\emptyset.
$$
\item[\rm (II.3)]  $[m^-(\tilde{\mathscr{E}}^\ast_{\lambda_k^-}, 0),
m^-(\tilde{\mathscr{E}}^\ast_{\lambda_k^-}, 0)+
m^0(\tilde{\mathscr{E}}^\ast_{\lambda_k^-}, 0)]\cap[m^-(\tilde{\mathscr{E}}^\ast_{\lambda_k^+}, 0),
m^-(\tilde{\mathscr{E}}^\ast_{\lambda_k^+}, 0)+m^0(\tilde{\mathscr{E}}^\ast_{\lambda_k^+}, 0)]=\emptyset$,
and either $m^0(\tilde{\mathscr{E}}^\ast_{\lambda_k^-}, 0)=0$ or $m^0(\tilde{\mathscr{E}}^\ast_{\lambda_k^+}, 0)=0$
for each $k\in\mathbb{N}$.
 \end{enumerate}
   Then $(\mu, 0)$  is a  bifurcation point of the problem (\ref{e:PLagrGenerEu})--(\ref{e:PLagrGenerBEu})
  in  $\hat\Lambda\times C^2_{{E}_{\overline{\gamma}}}([0,\tau],B^n_{\iota}(0))$  with respect to the branch $\{(\lambda, 0)\,|\,\lambda\in\hat\Lambda\}$
  (and so $\{(\lambda, 0)\,|\,\lambda\in\Lambda\}$), where
 $\hat{\Lambda}=\{\mu,\lambda^+_k, \lambda^-_k\,|\,k\in\mathbb{N}\}$.
  \end{enumerate}
  \end{theorem}

\begin{theorem}[\textsf{Existence for bifurcations}]\label{th:Pbif-existLagrGenerEu+}
Let $\Lambda$ be connected.  For $\lambda^-, \lambda^+\in\Lambda$
suppose that
one of the following conditions is satisfied:
 \begin{enumerate}
 \item[\rm (i)] Either $0$  is not an isolated critical point of $\tilde{\mathscr{E}}^\ast_{\lambda^+}$,
 or $0$ is not an isolated critical point of $\tilde{\mathscr{E}}^\ast_{\lambda^-}$,
 or $0$  is an isolated critical point of $\tilde{\mathscr{E}}^\ast_{\lambda^+}$ and $\tilde{\mathscr{E}}^\ast_{\lambda^-}$ and
  $C_m(\tilde{\mathscr{E}}^\ast_{\lambda^+}, 0;{\bf K})$ and $C_m(\tilde{\mathscr{E}}^\ast_{\lambda^-}, 0;{\bf K})$ are not isomorphic for some Abel group ${\bf K}$ and some $m\in\mathbb{Z}$.

\item[\rm (ii)] $[m^-(\tilde{\mathscr{E}}^\ast_{\lambda^-}, 0),
m^-(\tilde{\mathscr{E}}^\ast_{\lambda^-}, 0)+ m^0(\tilde{\mathscr{E}}^\ast_{\lambda^-}, 0)]\cap[m^-(\tilde{\mathscr{E}}^\ast_{\lambda^+}, 0),
m^-(\tilde{\mathscr{E}}^\ast_{\lambda^+}, 0)+m^0(\tilde{\mathscr{E}}^\ast_{\lambda^+}, 0)]=\emptyset$,
and there exists $\lambda\in\{\lambda^+, \lambda^-\}$ such that $0$  is an either non-isolated or homological visible critical point of
$\mathscr{E}^\ast_{\lambda}$.

\item[\rm (iii)] $[m^-(\tilde{\mathscr{E}}^\ast_{\lambda^-}, 0),
m^-(\tilde{\mathscr{E}}^\ast_{\lambda^-}, 0)+
m^0(\tilde{\mathscr{E}}^\ast_{\lambda^-}, 0)]\cap[m^-(\tilde{\mathscr{E}}^\ast_{\lambda^+}, 0),
m^-(\tilde{\mathscr{E}}^\ast_{\lambda^+}, 0)+m^0(\tilde{\mathscr{E}}^\ast_{\lambda^+}, 0)]=\emptyset$,
and either $m^0(\mathscr{E}^\ast_{\lambda^+}, 0)=0$ or $m^0(\mathscr{E}^\ast_{\lambda^-}, 0)=0$.
 \end{enumerate}
Then for any path $\alpha:[0,1]\to\Lambda$ connecting $\lambda^+$ to $\lambda^-$ there exists
  a sequence $(t_k)\subset[0,1]$ converging to some $\bar{t}\in [0,1]$, and
   a nonzero solution $x_k$ of the problem (\ref{e:PLagrGenerEu})--(\ref{e:PLagrGenerBEu})
  with $\lambda=\alpha(t_k)$ for each $k\in\mathbb{N}$ such that $\|x_k\|_{C^2([0,\tau];\mathbb{R}^n)}\to 0$
  as $k\to\infty$.
 Moreover,  $\alpha(\bar{t})$ is not equal to $\lambda^+$ (resp. $\lambda^-$) if $
 m^0(\tilde{\mathscr{E}}^\ast_{\lambda^+}, 0)=0$ (resp. $m^0(\tilde{\mathscr{E}}^\ast_{\lambda^-}, 0)=0$).
 \end{theorem}

  \begin{theorem}[\textsf{Alternative bifurcations of Rabinowitz's type}]\label{th:Pbif-suffLagrGenerEu}
    Let $\Lambda$ be a real interval and  $\mu\in{\rm Int}(\Lambda)$. Suppose that $m^0(\tilde{\mathscr{E}}^\ast_{\mu}, 0)>0$,
   and that $m^0(\tilde{\mathscr{E}}^\ast_{\lambda}, 0)=0$  for each $\lambda\in\Lambda\setminus\{\mu\}$ near $\mu$, and
  $m^-(\tilde{\mathscr{E}}^\ast_{\lambda}, 0)$ take, respectively, values $m^-(\tilde{\mathscr{E}}^\ast_{\mu}, 0)$ and
  $m^-(\tilde{\mathscr{E}}^\ast_{\mu}, 0)+ m^0(\tilde{\mathscr{E}}^\ast_{\mu}, 0)$
 as $\lambda\in\Lambda$ varies in two deleted half neighborhoods  of $\mu$.
Then  one of the following alternatives occurs:
\begin{enumerate}
\item[\rm (i)] The problem (\ref{e:PLagrGenerEu})--(\ref{e:PLagrGenerBEu})
 with $\lambda=\mu$ has a sequence of solutions, $x_k\ne 0$, $k=1,2,\cdots$,
which converges to $0$ in $C^2([0,\tau], \R^n)$.

\item[\rm (ii)]  For every $\lambda\in\Lambda\setminus\{\mu\}$ near $\mu$ there is a  solution $y_\lambda\ne 0$ of
(\ref{e:PLagrGenerEu})--(\ref{e:PLagrGenerBEu}) with parameter value $\lambda$,
 such that   $y_\lambda$  converges to zero in  $C^2([0,\tau], \R^n)$ as $\lambda\to \mu$.

\item[\rm (iii)] For a given neighborhood $\mathfrak{W}$ of $0\in C^1_{{E}_{\overline{\gamma}}}([0,\tau],B^n_{\iota}(0))$,
there is a one-sided neighborhood $\Lambda^0$ of $\mu$ such that
for any $\lambda\in\Lambda^0\setminus\{\mu\}$, the problem (\ref{e:PLagrGenerEu})--(\ref{e:PLagrGenerBEu}) with parameter value $\lambda$
has at least two distinct solutions in  $\mathfrak{W}$, $y_\lambda^1\ne 0$ and $y_\lambda^2\ne 0$,
which can also be required to satisfy $\tilde{\mathscr{E}}^\ast_\lambda(y^1_\lambda)\ne \tilde{\mathscr{E}}^\ast_\lambda(y^2_\lambda)$
provided that $m^0(\tilde{\mathscr{E}}^\ast_{\mu}, 0)>1$ and the problem
(\ref{e:PLagrGenerEu})--(\ref{e:PLagrGenerBEu}) with parameter value $\lambda$
has only finitely many solutions in $\mathfrak{W}$.
\end{enumerate}
\end{theorem}

As in the proofs of Theorems~\ref{th:bif-nessLagrGener},~\ref{th:bif-suffLagrGener},
Theorems~\ref{th:bif-nessLagr*},~\ref{th:bif-suffLagr*}
are  derived from Theorems~\ref{th:Pbif-nessLagrGenerEu},~\ref{th:Pbif-suffLagrGenerEu}, respectively.

\subsection{Proofs of Theorems~\ref{th:bif-nessLagr*},~\ref{th:bif-existLagr*},~\ref{th:bif-suffLagr*}}\label{sec:LagrPPerio1.2}

Let us write
\begin{eqnarray*}
&&{\bf H}_{E_{\overline{\gamma}}}:=W^{1,2}_{E_{\overline{\gamma}}}([0,\tau];\R^{n})\quad\hbox{and}\quad
 {\bf X}_{E_{\overline{\gamma}}}:=C^1_{E_{\overline{\gamma}}}([0,\tau];\R^{n}),\\
&&\mathcal{U}:=\big\{u\in W^{1,2}\big([0,\tau];
B^n_{\iota/2}(0)\big)\,\big|\,u(\tau)=E_{\overline{\gamma}}u(0) \big\},\\
&&\mathcal{U}^X:=\mathcal{U}\cap{\bf X}_{E_{\overline{\gamma}}}=\big\{u\in C^1\big([0,\tau];
B^n_{\iota/2}(0)\big)\,\big|\,u(\tau)=E_{\overline{\gamma}}u(0) \big\}
\end{eqnarray*}
The latter two sets are open subsets in the spaces ${\bf H}_{E_{\overline{\gamma}}}$ and ${\bf X}_{E_{\overline{\gamma}}}$, respectively.
 Let the continuous function $\check{L}:\hat\Lambda\times [0, \tau]\times B^n_{3\iota/4}(0)\times\mathbb{R}^n\to\R$
be given by Lemma~\ref{lem:Gener-modif}.
 Define a family of functionals $\check{\mathscr{E}}_\lambda:\mathcal{U}\to\R$ given by
\begin{equation}\label{e:PcheckEu}
\check{\mathscr{E}}_{\lambda}(x)=\int^{\tau}_0\check{L}_\lambda(t, x(t),\dot{x}(t))dt,\quad\lambda\in\hat\Lambda.
\end{equation}
Then (L1) in Lemma~\ref{lem:Gener-modif} implies
\begin{equation}\label{e:PtwofunctAgree}
\tilde{\mathscr{E}}^\ast_\lambda=\check{\mathscr{E}}_{\lambda}|_{\mathcal{U}^X}\quad\hbox{in}\quad
\{x\in\mathcal{U}^X\,|\,\|x\|_{C^1}<\rho_0\}\subset\mathcal{U}^X,
\end{equation}
and hence
\begin{equation}\label{e:PMorseindexEu}
 m^\star(\tilde{\mathscr{E}}^\ast_{\lambda}, 0)=m^\star(\check{\mathscr{E}}_{\lambda}|_{\mathcal{U}^X},0)=
m^\star(\check{\mathscr{E}}_{\lambda},0),\quad\star=-,0.
\end{equation}

Since ${\bf H}_{E_{\overline{\gamma}}}$ contains the subspace $\{u\in W^{1,2}([0,\tau];
\mathbb{R}^n)\,|\,u(\tau)=0=u(0)\}$, carefully checking the computation of \cite[(4.14)]{Lu4} it is easily seen that
 replacing $H_V$ by ${\bf H}_{E_{\overline{\gamma}}}$ we also obtain that
the gradient  $\nabla\check{\mathscr{E}}_{\lambda}(x)$ of $\check{\mathscr{E}}_{\lambda}$ at $x\in\mathcal{U}$
 is still given by
\begin{eqnarray}\label{e:P4.14}
\nabla\check{\mathscr{E}}_{\lambda}(x)(t)&=&e^t\int^t_0\left[
e^{-2s}\int^s_0e^{r}f_{\lambda,x}(r)dr\right]ds  + c_1(\lambda,x)e^t+
c_2(\lambda,x)e^{-t}\nonumber\\
&&\qquad +\int^t_0 \partial_v \check{L}_\lambda(s, x(s),\dot{x}(s))ds ,
\end{eqnarray}
where $c_1(\lambda,x), c_2(\lambda,x)\in\R^n$ are suitable constant vectors and $f_\lambda(t)$ is given by (\ref{e:4.13}).

\begin{proposition}\label{prop:Pfunct-analy}
Proposition~\ref{prop:funct-analy} is still effective after making  the following substitutions:
\begin{enumerate}
\item[\rm $\bullet$] The functionals $\tilde{\mathcal{E}}^\ast_{\lambda}$ and
$\check{\mathcal{E}}_{\lambda}$ are changed into
 $\tilde{\mathscr{E}}^\ast_{\lambda}$ and $\check{\mathscr{E}}_{\lambda}$, respectively.
\item[\rm $\bullet$] The spaces ${\bf H}_{V_0\times V_1}$ and ${\bf X}_{V_0\times V_1}$ are changed into
${\bf H}_{E_{\overline{\gamma}}}$ and ${\bf X}_{E_{\overline{\gamma}}}$,  respectively.

\item[\rm $\bullet$] The boundary problem (\ref{e:LagrGenerEuTwo}) is changed into
\begin{eqnarray}\label{e:PLagrGenerEuTwo}
\left.\begin{array}{ll}
\frac{d}{dt}\big(\partial_v\check{L}_\lambda(t, x(t), \dot{x}(t))\big)-\partial_q \check{L}_\lambda(t, x(t), \dot{x}(t))=0,\\
x(\tau)=E_{\overline{\gamma}}x(0).
\end{array}\right\}
\end{eqnarray}
\end{enumerate}
\end{proposition}

Then repeating proofs of Propositions~\ref{prop:continA},~\ref{prop:solutionLagr} we can obtain:

\begin{proposition}\label{prop:PcontinA}
Both maps $\hat\Lambda\times \mathcal{U}^X\ni (\lambda, x)\mapsto
\check{\mathscr{E}}_{\lambda}(x)\in\R$ and
$\hat\Lambda\times \mathcal{U}^X\ni (\lambda, x)\mapsto A_\lambda(x)\in{\bf X}_{E_{\overline{\gamma}}}$ are continuous.
\end{proposition}

\begin{proposition}\label{prop:PsolutionLagr}
For any given $\epsilon>0$ there exists $\varepsilon>0$ such that
if a critical point $x$ of $\check{\mathscr{E}}_\lambda$ satisfies
$\|x\|_{1,2}<\varepsilon$ then $\|x\|_{C^2}<\epsilon$. (\textsf{{Note}}: $\varepsilon$ is independent of $\lambda\in\hat\Lambda$.)
Consequently, if $0\in \mathcal{U}^X$ is an isolated critical point of $\check{\mathscr{E}}_{\lambda}|_{\mathscr{U}^X}$
then $0\in \mathcal{U}$ is also an isolated critical point of $\check{\mathscr{E}}_{\lambda}$.
\end{proposition}

Having these we can prove
\begin{description}
\item[$\bullet$] Theorem~\ref{th:Pbif-nessLagrGenerEu}(I) with \cite[Theorem~3.1]{Lu8} (\cite[Theorem~C.6]{Lu11})
as in the proof of Theorem~\ref{th:bif-nessLagrGenerEu}(I),
\item[$\bullet$] Theorem~\ref{th:Pbif-suffLagrGenerEu} with
\cite[Theorem~C.7]{Lu11} (\cite[Theorem~3.6]{Lu10})
 as in the proof of Theorem~\ref{th:bif-suffLagrGenerEu},
 \item[$\bullet$] a corresponding result of Theorem~\ref{th:bif-existLagrGenerEu}, from which
Theorem~\ref{th:Pbif-nessLagrGenerEu}(II) and Theorem~\ref{th:Pbif-existLagrGenerEu+} may be derived,

\item[$\bullet$] Theorem~\ref{th:bif-existLagr*} as in the proof of Theorem~\ref{th:bif-existLagrGener} in Section~\ref{sec:LagrBound.1.4}.
 \end{description}

\section{Proofs of Theorems~\ref{th:bif-nessLagrBrakeM},~\ref{th:bif-existLagrBrakeM},~\ref{th:bif-suffLagrBrakeM}}\label{sec:Lagr7}
\setcounter{equation}{0}

Following the paragraph above (\ref{e:gamma1})
we take a path  $\overline{\gamma}\in EC^7(S_\tau;M)$  such that
 \begin{eqnarray}\label{e:gamma1brake}
{\rm dist}_g(\gamma_\mu(t), \overline{\gamma}(t))<\iota\;\forall t\in
\mathbb{R}.
\end{eqnarray}
We first assume:
\begin{eqnarray}\label{e:gamma2brake}
d_g(\gamma_\lambda(t), \overline{\gamma}(t))<\iota,\quad\forall (\lambda,t)\in
\Lambda\times\mathbb{R}.
\end{eqnarray}
(For cases of Theorems~\ref{th:bif-nessLagrBrakeM},\ref{th:bif-suffLagrBrakeM},
by contradiction we may use nets to prove that (\ref{e:gamma2brake}) is satisfied after shrinking $\Lambda$ toward $\mu$.)
Then (\ref{e:gamma1brake}) and (\ref{e:gamma2brake}) imply
\begin{eqnarray}\label{e:gamma3brake}
d_g(\gamma_\lambda(t), \gamma_\mu([0,\tau]))\le
d_g(\gamma_\lambda(t), \gamma_\mu(t))<2\iota,\quad\forall (\lambda,t)\in
\Lambda\times\mathbb{R}.
\end{eqnarray}
By Claim~\ref{cl:pallVEctorF} we have
 a unit orthogonal parallel $C^5$ frame field
  $\R\to \overline{\gamma}^\ast TM,\;t\mapsto (e_1(t),\cdots, e_n(t))$
to satisfy
\begin{eqnarray}\label{e:evenframe}
(e_1(\tau\pm t),\cdots, e_n(\tau\pm t))=(e_1(t),\cdots, e_n(t))\quad\forall t\in \R.
\end{eqnarray}
Let $B^n_{2\iota}(0):=\{x\in\R^n\,|\, |x|<2\iota\}$ and $\exp$ denote the exponential map of $g$.   Then
 \begin{eqnarray*}
 \phi_{\overline{\gamma}}:\R\times B^n_{2\iota}(0)\to
M,\;(t,x)\mapsto\exp_{\overline{\gamma}(t)}\Big(\sum^n_{i=1}x_i
e_i(t)\Big)
 \end{eqnarray*}
 is a $C^5$ map and  satisfies 
 \begin{equation}\label{e:3.28}
\phi_{\overline{\gamma}}(\tau\pm t, x)=\phi_{\overline{\gamma}}(t, x)\quad\hbox{and}\quad
\frac{d}{dt}\phi_{\overline{\gamma}}(\tau\pm t, x)=\pm(\phi_{\overline{\gamma}})'_1(\tau\pm t, x)
\end{equation}
 for any $(t,x)\in \R\times B^n_{2\iota}(0)$.  By  \cite[Theorem~4.3]{PiTa01}, there exists
  a $C^2$ coordinate chart around $\overline{\gamma}$ on the $C^4$ Banach manifold
 $EC^{1}(S_\tau; M)$,
  \begin{eqnarray}\label{e:EBanachM**}
\Phi_{\overline{\gamma}}:EC^1(S_\tau; B^n_{2\iota}(0))=\{\xi\in
C^1(S_\tau; \mathbb{R}^n) \,|\,\|\xi\|_{C^0}<2\iota\} \to EC^{1}(S_\tau; M)
\end{eqnarray}
given by $\Phi_{\overline{\gamma}}(\xi)(t)=\phi_{\overline{\gamma}}(t,\xi(t))$,   and
$$
d\Phi_{\overline{\gamma}}(0):EC^1(S_\tau; \mathbb{R}^n)\to T_{\overline{\gamma}}EC^{1}(S_\tau; M),\;
\xi\mapsto\sum^n_{j=1}\xi_je_j.
$$
Define
$d_\infty(\gamma_1,\gamma_2)=\max_{t\in S_\tau}{\rm dist}_g(\gamma_1(t),\gamma_2(t))$
for $\gamma_i\in EC^1(S_\tau, M)$, $i=1,2$. Then
 \begin{eqnarray}\label{e:EBanachM**+}
\Phi_{\overline{\gamma}}(C^1(S_\tau; B^n_{2\iota}(0)))=
\mathcal{U}(\overline{\gamma}, 2\iota):=\{\gamma\in C^1(S_\tau, M)\,|\, d_\infty(\gamma,\overline{\gamma})<2\iota\},
\end{eqnarray}

Let $EC^{1}(S_\tau; \mathbb{R}^n)=\{\xi\in C^1(\R;\mathbb{R}^n)\,|\,\xi(t+\tau)=\xi(t)\,\&\, \xi(-t)=\xi(t)\;\forall t\in\R\}$ and
\begin{eqnarray}\label{e:EBanachM***}
EC^1(S_\tau; B^n_{2\iota}(0))=\{\xi\in EC^1(S_\tau; \mathbb{R}^n) \,|\,\|\xi\|_{C^0}<2\iota\}.
\end{eqnarray}
 Since $\overline{\gamma}\in C^7(S_\tau;M)\cap EC^1(S_\tau;M)$, it follows from (\ref{e:EBanachM**+}) and (\ref{e:EBanachM***}) that
 \begin{eqnarray*}
\Phi_{\overline{\gamma}}(EC^1(S_\tau; B^n_{2\iota}(0)))=
EC^{1}(S_\tau; M)\cap\mathcal{U}(\overline{\gamma}, 2\iota)=\{\gamma\in EC^1(S_\tau, M)\,|\, d_\infty(\gamma,\overline{\gamma})<2\iota\}.
\end{eqnarray*}
Hence we obtain  the following:

\begin{claim}\label{cl:Echart}
The restriction of $\Phi_{\overline{\gamma}}$ to $EC^1(S_\tau; B^n_{2\iota}(0)$
 gives rise to  a $C^2$ coordinate chart around $\overline{\gamma}$ on  $EC^{1}(S_\tau; M)$.
 And $\Phi_{\overline{\gamma}}(EC^1(S_\tau; B^n_{\iota}(0)))$ contains all
$\gamma_\lambda$ by (\ref{e:gamma2brake}).
 \end{claim}

Since the injectivity radius of $g$ at each point on $\overline{\gamma}(S_\tau)$ is at least $2\iota$,
by (\ref{e:gamma2brake}) there exists a unique map  ${\bf u}_\lambda:\R\to B^n_{\iota}(0)$  such that
$$
\gamma_\lambda(t)=\phi_{\overline{\gamma}}(t,{\bf u}_\lambda(t))=\exp_{\overline{\gamma}(t)}\Big(\sum^n_{i=1}{\bf u}_\lambda^i(t)
e_i(t)\Big)\quad\forall t\in \R.
$$
Note that $\overline{\gamma}$ is even and $\tau$-periodic. From (\ref{e:evenframe}) we derive
$$
\gamma_\lambda(\tau\pm t)=\exp_{\overline{\gamma}(\tau\pm t)}\Big(\sum^n_{i=1}{\bf u}_\lambda^i(\tau\pm t)
e_i(\tau\pm t)\Big)=\exp_{\overline{\gamma}(t)}\Big(\sum^n_{i=1}{\bf u}_\lambda^i(\tau\pm t)
e_i(t)\Big).
$$
Hence ${\bf u}_\lambda\in EC^1(S_\tau; B^n_{\iota}(0))\cap C^2(S_\tau; B^n_{\iota}(0))$.
By Lemmas~\ref{lem:twoCont},\ref{lem:twoCont+} we also see that
 $$
 (\lambda,t)\mapsto {\bf u}_\lambda(t),\quad (\lambda,t)\mapsto \dot{\bf u}_\lambda(t)\quad\hbox{and}\quad
 (\lambda,t)\mapsto \ddot{\bf u}_\lambda(t)
 $$
 are continuous maps from  $\Lambda\times \R$ to $\R^n$.

 Define $L^\ast_\lambda: \R\times B^n_{2\iota}(0)\times\R^n\to\R$ by
\begin{equation}\label{e:3.15}
 L^\ast_\lambda(t, x,  v)=L_\lambda\big(t, \phi_{\overline{\gamma}}(t, x),
(\phi_{\overline{\gamma}})'_1(t, x)+ (\phi_{\overline{\gamma}})'_2(t, x)[v]\big).
\end{equation}
Then  it follows from (\ref{e:3.28}) that
$L^\ast_\lambda(\tau+t, x, v)=L^\ast_\lambda(t, x, v)$
and
\begin{eqnarray}\label{e:3.31}
 L^\ast_\lambda(\tau-t, x, -v)&=&L_\lambda\left(\tau-t, \phi_{\overline{\gamma}}(\tau-t, x),
(\phi_{\overline{\gamma}})'_1(\tau-t, x)+ (\phi_{\overline{\gamma}})'_2(\tau-t, x)[-v]\right)\nonumber\\
&=&L_\lambda\left(\tau-t, \phi_{\overline{\gamma}}(t, x), -(\phi_{\overline{\gamma}})'_1(t, x)-
(\phi_{\overline{\gamma}})'_2(t, x)[v]\right)\nonumber\\
&=&L_\lambda\left(t, \phi_{\overline{\gamma}}(t,x), (\phi_{\overline{\gamma}})'_1(t,x)+ (\phi_{\overline{\gamma}})'_2(t, x)[v]\right)\nonumber\\
&=&L^\ast_\lambda(t, x, v),\quad\forall (t, x, v)\in  \R\times B^n_{2\iota}(0)\times\R^n.
\end{eqnarray}
As before we also define $\tilde L^\ast:\Lambda\times \mathbb{R}\times B^n_{\iota}(0)\times\R^n\to\R$ by
\begin{equation}\label{e:ModifiedLEuII}
\tilde L^\ast(\lambda, t, q,v)=\tilde L^\ast_{\lambda}(t, q,v)= {L}^\ast(\lambda, t, q+ {\bf u}_\lambda(t), v+\dot{\bf u}_\lambda(t)).
\end{equation}
It satisfies Proposition~\ref{th:bif-LagrGenerEu} after the interval $[0,\tau]$ is replaced by $\mathbb{R}$.
Moreover, (\ref{e:3.31}) leads to
\begin{eqnarray*}
\tilde L^\ast(\lambda,\tau-t, q,-v)&=& {L}^\ast(\lambda, \tau-t, q+ {\bf u}_\lambda(\tau-t), -v+\dot{\bf u}_\lambda(\tau-t))\\
 &=&{L}^\ast(\lambda, \tau-t, q+ {\bf u}_\lambda(t), -v-\dot{\bf u}_\lambda(t))\\
 &=&{L}^\ast(\lambda, t, q+ {\bf u}_\lambda(t), v+\dot{\bf u}_\lambda(t))
 =\tilde L^\ast(\lambda, t, q,v)
 \end{eqnarray*}
because ${\bf u}_\lambda(\tau-t)={\bf u}_\lambda(t)$ implies $-\dot{\bf u}_\lambda(\tau-t)=\dot{\bf u}_\lambda(t)$.

 Now  we have a family of $C^2$ functionals
\begin{equation*}
\tilde{\mathcal{L}}^E_\lambda: EC^1(S_\tau; B^n_{\iota}(0))\to\R,\; x\mapsto \int^\tau_0\tilde{L}^\ast_\lambda(t, x(t), \dot{x}(t))dt,\quad\lambda\in\Lambda.
\end{equation*}
(See the proof of the first claim in \cite[Proposition~4.2]{Lu9}).
For $\xi\in EC^1(S_\tau; B^n_{\iota}(0))$, since
\begin{equation*}
 \frac{d}{dt}\Phi_{\overline{\gamma}}(\xi+{\bf u}_\lambda)(t)=(\phi_{\overline{\gamma}})'_1(t,
\xi(t)+{\bf u}_\lambda(t))+ (\phi_{\overline{\gamma}})'_2(t, \xi(t)+{\bf u}_\lambda(t))[\dot{\xi}(t)+
\dot{\bf u}_\lambda(t)]\quad\forall t\in\R,
\end{equation*}
we deduce
\begin{equation}\label{e:3.32}
\tilde{\cal L}^E_{\lambda}(\xi)={\cal
L}^E_{\lambda}\left(\Phi_{\overline{\gamma}}(\xi+{\bf u}_\lambda)\right)\quad\forall\xi\in EC^1(S_\tau; B^n_{\iota}(0))
\end{equation}
and therefore that
\begin{equation*}
m^-_\tau(\tilde{\mathcal{L}}^E_\lambda, 0)=m^-_\tau(\mathcal{L}^E_\lambda, \gamma_\lambda)\quad
\hbox{and}\quad
m^0_\tau(\tilde{\mathcal{L}}^E_\lambda, 0)=m^0_\tau(\mathcal{L}^E_\lambda, \gamma_\lambda).
\end{equation*}

For the function $\tilde L^\ast$ in (\ref{e:ModifiedLEuII}),  a positive number $\rho_0>0$ and a subset $\hat\Lambda\subset\Lambda$ which is either compact or sequential compact, as in Lemma~\ref{lem:Gener-modif} using Lemmas~\ref{lem:Lagr},~\ref{lem:reduction}
we can construct a  continuous function $\check{L}:\hat\Lambda\times \mathbb{R}\times B^n_{3\iota/4}(0)\times\mathbb{R}^n\to\R$
 satisfying the  properties  (L0)-(L6) in  Lemma~\ref{lem:Gener-modif} on
 $\hat\Lambda\times \mathbb{R}\times B^n_{3\iota/4}(0)\times\mathbb{R}^n$ and the following equality
 \begin{eqnarray*}
\check{L}(\lambda, -t, q,-v)=\check{L}(\lambda, t, q,v)=\check{L}(\lambda,\tau+t, q,v)
 \end{eqnarray*}
 for all $(\lambda,t, q,v)\in\hat\Lambda\times \mathbb{R}\times B^n_{3\iota/4}(0)\times\mathbb{R}^n$.
  Let us write
\begin{eqnarray*}
&&{\bf H}:=W^{1,2}(S_\tau;\R^{n}),\quad
 {\bf X}:=C^1(S_\tau;\R^{n}),\\
&&\mathcal{U}:=W^{1,2}(S_\tau;B^n_{\iota/2}(0),\quad
\mathcal{U}^X:=\mathcal{U}\cap{\bf X}=EC^1\big(S_\tau;B^n_{\iota/2}(0)\big),\\
&&{\bf H}_e:=EW^{1,2}(S_\tau;\R^{n})=\{x\in W^{1,2}(S_\tau; \R^n)\,|\, x(-t)=x(t)\;\forall t\},
                                         \quad {\bf X}_e:=EC^1(S_\tau;\R^{n}),\\
&&\mathcal{U}_e:=\big\{u\in W^{1,2}\big(S_\tau;
B^n_{\iota/2}(0)\big)\,\big|\,u(-t)=u(t)\;\forall t \big\},\\
&&\mathcal{U}^X_e:=\mathcal{U}\cap{\bf X}=\big\{u\in C^1\big(S_\tau;
B^n_{\iota/2}(0)\big)\,\big|\,u(-t)=u(t)\;\forall t \big\}.
\end{eqnarray*}
 For each $\lambda\in\hat\Lambda$ let us define functionals
 \begin{eqnarray*}
\check{\mathcal{L}}_\lambda: \mathcal{U}\to\R,\; x\mapsto \int^\tau_0\check{L}_\lambda(t, x(t), \dot{x}(t))dt\quad\hbox{and}\quad
\check{\mathcal{L}}_\lambda^E:\mathcal{U}_e\to\R,\; x\mapsto \int^\tau_0\check{L}_\lambda(t, x(t), \dot{x}(t))dt.
\end{eqnarray*}
Clearly, $\check{\mathcal{L}}_\lambda|_{\mathcal{U}_e}=\check{\mathcal{L}}_\lambda^E$. As a special case,
$\{(\check{\mathcal{L}}_\lambda,  \mathcal{U}, \mathcal{U}^X)\,|\,\lambda\in\hat\Lambda\}$
 has the same properties as $\{(\check{\mathcal{E}}_\lambda,  \mathcal{U}, \mathcal{U}^X)\,|\,\lambda\in\hat\Lambda\}$
 in Section~\ref{sec:LagrPPerio1.2}.
 By the arguments in \cite{Lu1-}, $\{(\check{\mathcal{L}}_\lambda^E,  \mathcal{U}_e, \mathcal{U}^X_e)\,|\,\lambda\in\hat\Lambda\}$
 has also the same properties. (In fact, from the expression of $\nabla\check{\mathcal{L}}_\lambda$ it is not hard to prove that
 $\nabla\check{\mathcal{L}}_\lambda(x)$ is even for each $x\in\mathcal{U}_e$. Therefore
 $\nabla\check{\mathcal{L}}_\lambda^E(x)=\nabla\check{\mathcal{L}}_\lambda(x)$  for each $x\in\mathcal{U}_e$.)
 This implies that for each $x\in \mathcal{U}_e$ the operator $B_\lambda(x)\in\mathscr{L}_s({\bf H})$ defined by (\ref{e:gradient4Lagr+})
 maps ${\bf H}_e$ into ${\bf H}_e$.
  Almost repeating  the proofs in  Section~\ref{sec:LagrBound} we can obtain the required results.
Of course, \cite[\S8]{Lu12} also provides a way.

\section{Corollaries and examples}\label{sec:Cor-example}

In this section, we first derive some corollaries of the main results using the abstract framework in Appendix~\ref{app:Abst}, as these corollaries are more readily applicable. We then apply these results to physical Lagrangians---such as those involving potential and electromagnetic forces---and finally present bifurcation results for the planar forced simple pendulum.

\subsection{Corollaries}\label{sec:Cor-example-I}

\begin{assumption}\label{ass:6.A}
{\rm  Let $(M, g, \mathbb{I}_g)$ be as in Assumption~1.0 in Introduction.
For a real $\tau>0$, consider a function $W\in  C^2([0, \tau]\times M, \mathbb{R})$, and a $C^2$ Lagrangian $\mathbf{L}(t, q,v)$  on $[0, \tau]\times TM$ that is
 strictly convex with respect to $v$. 
 Let ${\bf N}$ be either ${\rm Graph}(\mathbb{I}_{g})$ or the product $S_0\times S_1$ of the submanifolds
 $S_0$ and $S_1$  in Assumption~\ref{ass:Lagr6}. 
Assume that  $\bar{\gamma}\in C^2([0, \tau]; M)$ satisfies 
Lagrangian  boundary value problem:
  \begin{equation*}
\left.\begin{array}{ll}
\frac{d}{dt}\big(\partial_v\mathbf{L}(t, \gamma(t), \dot{\gamma}(t))\big)-
\partial_q \mathbf{L}(t, \gamma(t), \dot{\gamma}(t))=0\;\;\forall
t\in [0,\tau],\\
(\gamma(0), \gamma(\tau))\in{\bf N}\quad\hbox{and}\quad\\
\partial_v\mathbf{L}(0, \gamma(0), \dot{\gamma}(0))[v_0]=\partial_v\mathbf{L}(\tau, \gamma(\tau), \dot{\gamma}(\tau))[v_1]\;\;\forall
(v_0,v_1)\in T_{(\gamma(0),\gamma(\tau))}{\bf N}
\end{array}\right\}
\end{equation*} 
Assume that for all $t$, $\partial_q W(t, \bar{\gamma}(t))=0$,
and that  the Hessian of each $W(t,\cdot)$ at the critical point $\bar{\gamma}(t)$ 
is either positive definite or negative definite.
}
\end{assumption}

Under Assumption~\ref{ass:6.A}, for $\lambda\in\mathbb{R}$ and $(t,q,v)\in [0, \tau]\times TM$, we define 
\begin{equation}\label{e:6.A}
L(\lambda, t, q, v)=L_\lambda(t, q, v)=\mathbf{L}(t,q,v)+\lambda W(t,q).
\end{equation}
It is easy to check that $L$ satisfies Assumptions~\ref{ass:Lagr6},~\ref{ass:Lagr7}.
With this $L$, let $\mathcal{E}_\lambda$ and $\mathscr{E}_\lambda$ be given by 
(\ref{e:LagrGener-}) and (\ref{e:Lagr0}), respectively. Then 
$\mathcal{E}_\lambda=\mathcal{E}_0+\lambda\mathcal{W}$ and 
$\mathscr{E}_\lambda= \mathscr{E}_0+ \lambda\mathscr{W}$, where
$$
\mathcal{E}_0(\gamma)=\int^\tau_0L_\lambda(t, \gamma(t), \dot{\gamma}(t))dt\quad\text{and}\quad
\mathcal{W}=\int^\tau_0W(t, \gamma(t))dt
$$
for $\gamma\in C^{1}_{S_0\times S_1}([0,\tau]; M)$, and
$$
\mathscr{E}_0(\gamma)=\int^\tau_0L_\lambda(t, \gamma(t), \dot{\gamma}(t))dt\quad\text{and}\quad
\mathscr{W}=\int^\tau_0W(t, \gamma(t))dt
$$
for $\gamma\in C^{1}_{\mathbb{I}_g}([0, \tau]; M)$.
When $\mathbf{N}$ is equal to $S_0\times S_1$ (resp.~$\mathrm{Graph}(\mathbb{I}_{g})$), $\bar{\gamma}$ is a critical point of $\mathcal{E}_0$ and $\mathcal{W}$ (resp.~$\mathscr{E}_0$ and $\mathscr{W}$), and hence of $\mathcal{E}_\lambda$ (resp.~$\mathscr{E}_\lambda$), for all $\lambda\in\mathbb{R}$, because of Assumption~\ref{ass:6.A}.
A standard computation shows that the second differentials of 
$\mathcal{W}$ and $\mathscr{W}$ at $\bar{\gamma}$ are, respectively,
\begin{align*}
D^{2}\mathcal{W}(\bar{\gamma})[\xi,\xi]&=\int_{0}^{\tau}
\langle \mathrm{Hess} W(t, \bar{\gamma})\xi(t), \xi(t)\rangle_g dt,\quad\xi\in W^{1,2}_{S_0\times S_1}(\bar{\gamma}^\ast TM),\\
D^{2}\mathscr{W}(\bar{\gamma})[\xi,\xi]&=\int_{0}^{\tau}
\langle \mathrm{Hess} W(t, \bar{\gamma})\xi(t), \xi(t)\rangle_g dt,\quad\xi\in
W^{1,2}_{\mathbb{I}_g}(\bar{\gamma}^\ast TM).
\end{align*}
Therefore, by the standard methods in functional analysis it is easily proved that
$D^{2}\mathcal{W}(\bar{\gamma})$ and $D^{2}\mathscr{W}(\bar{\gamma})$
are given by self-adjoint and compact operators 
$W^{1,2}_{S_0\times S_1}(\bar{\gamma}^\ast TM)$ and
$W^{1,2}_{\mathbb{I}_g}(\bar{\gamma}^\ast TM)$, respectively.
Moreover, since $\bar{\gamma}([0,\tau])\subset M$ is compact,
by modifying $L$ outside some compact neighborhood of $\bar{\gamma}([0,\tau])\subset M$ as done in this paper (if necessary), a localized argument shows 
that both $D^{2}\mathcal{E}_0(\bar{\gamma})$ and $D^{2}\mathscr{E}_0(\bar{\gamma})$
are self-adjoint Fredholm and have  finite Morse indexes (cf.~\cite[Proposition~3.1(iii)]{AS09}
and \cite{Du}). 
 (Recall that a continuous symmetric bilinear form on a Hilbert space $H$ is said to Fredholm if 
the associated self-adjoint operator is Fredholm.) 
% A continuous symmetric bilinear form $\alpha$ on the Hilbert space $\mathbb{H}$ is said to be \textbf{Fredholm} if the associated self-adjoint operator $A$ is Fredholm. Its \textbf{Morse index} is the dimension of the $A$-invariant subspace of $\mathbb{H}$ corresponding to the negative part of the spectrum of $A$. So the Morse index of $\alpha$ is finite if and only if the negative spectrum of $A$ consists of finitely many eigenvalues with finite multiplicity. All of these notions do not depend on the choice of the Hilbert product on $\mathbb{H}$.
Hence, for any $\lambda\in\mathbb{R}$, $D^{2}\mathcal{E}_\lambda(\bar{\gamma})$ and $D^{2}\mathscr{E}_\lambda(\bar{\gamma})$
are also self-adjoint Fredholm and have also finite Morse indexes.
By Proposition~\ref{prop:Abst} we immediately obtain:

\begin{proposition}\label{prop:6.A}
Under Assumption~\ref{ass:6.A}, suppose that $\mathbf{N}=S_0\times S_1$
and $m^0(\mathcal{E}_0, \bar{\gamma})> 0$. Then 
there exists $0 < \delta_0 < 1$ such that the following hold:
\begin{itemize}
\item[\rm (i)] If the Hessian of each $W(t,\cdot)$ at the critical point $\bar{\gamma}(t)$ 
is positive definite, then
\[
m^-(\mathcal{E}_\lambda, \bar{\gamma}) = \begin{cases}
m^-(\mathcal{E}_0, \bar{\gamma}), & 0 < \lambda < \delta_0,\\[4pt]
m^-(\mathcal{E}_0, \bar{\gamma}) + m^0(\mathcal{E}_0, \bar{\gamma}), & -\delta_0 < \lambda < 0.
\end{cases}
\]
\item[\rm (ii)] If the Hessian of each $W(t,\cdot)$ at the critical point $\bar{\gamma}(t)$ 
is  negative definite, then
\[
m^-(\mathcal{E}_\lambda, \bar{\gamma}) = \begin{cases}
m^-(\mathcal{E}_0, \bar{\gamma}) + m^0(\mathfrak{E}_0, \bar{\gamma}), & 0 < \lambda < \delta_0,\\[4pt]
m^-(\mathcal{E}_0, \bar{\gamma}), & -\delta_0 < \lambda < 0.
\end{cases}
\]
\end{itemize}
Similarly, if $\mathbf{N}=\mathrm{Graph}(\mathbb{I}_{g})$ and $m^0(\mathscr{E}_0, \bar{\gamma})> 0$, then
(i) and (ii) hold with $\mathcal{E}_\lambda$ and $\mathcal{E}_0$ replaced by $\mathscr{E}_\lambda$ and $\mathscr{E}_0$, respectively.
\end{proposition}

\begin{assumption}\label{ass:6.B}
{\rm  Let $(M, g, \mathbb{I}_g)$ be as in Assumption~1.0 in Introduction.
Consider a function $W\in  C^2(\mathbb{R}\times M, \mathbb{R})$, and a $C^2$ Lagrangian $\mathbf{L}(t, q,v)$  on $\mathbb{R}\times TM$ that is
 strictly convex with respect to $v$. For some $\tau>0$, we assume that both $W$ and
$\mathbf{L}$ are also $\mathbb{I}_g$-invariant in the
 sense of (\ref{e:Lagr8}), i.e., for all $(t,q,v)\in\mathbb{R}\times TM$,
\begin{equation*}
W(t, q)=W(t+\tau, \mathbb{I}_g(q)),\quad
\mathbf{L}(t+\tau, \mathbb{I}_g(q), d\mathbb{I}_g(q)[v])=\mathbf{L}(t, q,v).
\end{equation*}
Let $\bar{\gamma}\in C^2(\mathbb{R}, M)$ satisfy the problem 
\begin{equation*}
\left.\begin{array}{ll}
&\frac{d}{dt}\big(\partial_v\mathbf{L}(t, \gamma(t), \dot{\gamma}(t))\big)-\partial_q\mathbf{L}(t, \gamma(t), \dot{\gamma}(t))=0\;\;\forall t\in\mathbb{R},\\
&\mathbb{I}_g(\gamma(t))=\gamma(t+\tau)\quad\forall t\in\mathbb{R}.
\end{array}\right\}
\end{equation*}
Assume that for all $t$, $\partial_q W(t, \bar{\gamma}(t))=0$,
and that  the Hessian of each $W(t,\cdot)$ at the critical point $\bar{\gamma}(t)$ 
is either positive definite or negative definite.
}
\end{assumption}

Let $L(\lambda, t, q, v)=L_\lambda(t, q, v)=\mathbf{L}(t,q,v)+\lambda W(t,q)$ 
for $\lambda\in\mathbb{R}$ and $(t,q,v)\in\mathbb{R}\times TM$.
Then $L$ satisfies Assumption~\ref{ass:Lagr8}, and 
$\bar{\gamma}$ solves (\ref{e:Lagr10}) with this $L$, i.e., $\bar{\gamma}$ is a critical
point of each $C^2$ functional 
 $\mathfrak{E}_\lambda:=\mathfrak{E}_0+\lambda\mathfrak{W}$
 on the $C^4$ Banach manifold $\mathcal{X}^1_{\tau}(M, \mathbb{I}_g)$, where 
\begin{equation*}
\mathfrak{E}_0(\gamma)=\int_0^\tau \mathbf{L}(t, \gamma(t), \dot{\gamma}(t)) \, dt\quad\text{and}\quad
\mathfrak{W}(\gamma)=\int_0^\tau W(t, \gamma(t)) \, dt.
\end{equation*}
Note that
 the tangent space $T_{\gamma}\mathcal{X}^1_{\tau}(M, \mathbb{I}_g)=\{\xi\in C^1(\gamma^\ast TM\,|\,
 d\mathbb{I}_g(\gamma(t))[\xi(t)]=\xi(t+\tau)\}$
 for $\gamma\in\mathcal{X}^1_{\tau}(M, \mathbb{I}_g)$.
Let $W^{1,2}\left(T_{\gamma}\mathcal{X}^1_{\tau}(M, \mathbb{I}_g)\right)$
denote the completion of $T_{\gamma}\mathcal{X}^1_{\tau}(M, \mathbb{I}_g)$ with respect to 
the inner product given by (\ref{e:1.1}). 
Since $\partial_{vv}\mathbf{L}(t, \bar{\gamma}(t), \dot{\gamma}(t))$ is positive definite,
we have a new Riemannian metric on the vector bundle $\bar{\gamma}^\ast TM\to\mathbb{R}$ given by
$$
\langle u, v\rangle^\ast=\langle \partial_{vv}\mathbf{L}(t, \bar{\gamma}(t), \dot{\gamma}(t))u, v\rangle_{g},
\quad\forall u, v\in (\bar{\gamma}^\ast TM)_t=T_{\bar{\gamma}(t)}M,
$$ 
By \cite[Assertion~4.6]{Dav} (with the same proof as \cite[Proposition~5.1]{BoTr}),
there exists is a covariant derivative $D$ compatible with the metric such that
$$
\langle\!\langle \xi, \eta\rangle\!\rangle=\int^\tau_0
[\langle D\xi(t), D\eta(t)\rangle^\ast+ \langle \xi(t), \eta(t)\rangle^\ast]dt
$$
defines a new inner product on
$W^{1,2}\left(T_{\gamma}\mathcal{X}^1_{\tau}(M, \mathbb{I}_g)\right)$
that induces a norm equivalent to the one induced by the original inner product $\langle\cdot,\cdot\rangle_{1,2}$. As pointed out in \cite[Remark~4.14]{Dav},
under this new inner product,  $D^{2}\mathfrak{E}_0(\bar{\gamma})$
is self-adjoint Fredholm and has  finite Morse index. (See \cite{BoTr} and \cite{Lu5} for more details.)
By Claim~\ref{cl:Abst}, under the inner product given by (\ref{e:1.1}),
 $D^{2}\mathfrak{E}_0(\bar{\gamma})$
is also self-adjoint Fredholm and has also finite Morse index.
As above, we know that
\begin{align*}
D^{2}\mathfrak{W}(\bar{\gamma})[\xi,\xi]&=\int_{0}^{\tau}
\langle \mathrm{Hess} W(t, \bar{\gamma})\xi(t), \xi(t)\rangle_g dt,\quad\xi\in 
W^{1,2}\left(T_{\gamma}\mathcal{X}^1_{\tau}(M, \mathbb{I}_g)\right),
\end{align*}
can be given a compact self-adjoint operator on $W^{1,2}\left(T_{\gamma}\mathcal{X}^1_{\tau}(M, \mathbb{I}_g)\right)$. Hence Proposition~\ref{prop:Abst} yields:

\begin{proposition}\label{prop:6.B}
Under Assumption~\ref{ass:6.B}, suppose that $m^0(\mathfrak{E}_0, \bar{\gamma})> 0$. Then 
there exists $0 < \delta_0 < 1$ such that the following hold:
\begin{itemize}
\item[\rm (i)] If the Hessian of each $W(t,\cdot)$ at the critical point $\bar{\gamma}(t)$ 
is positive definite, then
\[
m^-(\mathfrak{E}_\lambda, \bar{\gamma}) = \begin{cases}
m^-(\mathfrak{E}_0, \bar{\gamma}), & 0 < \lambda < \delta_0,\\[4pt]
m^-(\mathfrak{E}_0, \bar{\gamma}) + m^0(\mathfrak{E}_0, \bar{\gamma}), & -\delta_0 < \lambda < 0.
\end{cases}
\]
\item[\rm (ii)] If the Hessian of each $W(t,\cdot)$ at the critical point $\bar{\gamma}(t)$ 
is  negative definite, then
\[
m^-(\mathfrak{E}_\lambda, \bar{\gamma}) = \begin{cases}
m^-(\mathfrak{E}_0, \bar{\gamma}) + m^0(\mathfrak{E}_0, \bar{\gamma}), & 0 < \lambda < \delta_0,\\[4pt]
m^-(\mathfrak{E}_0, \bar{\gamma}), & -\delta_0 < \lambda < 0.
\end{cases}
\]
\end{itemize}
\end{proposition}

\begin{assumption}\label{ass:6.C}
{\rm  Let $(M, g)$ be as in Assumption~1.0 in Introduction.
Consider a function $W\in  C^2(\mathbb{R}\times M, \mathbb{R})$, 
and a $C^2$ Lagrangian $\mathbf{L}(t, q,v)$  on $\mathbb{R}\times TM$ that is
 strictly convex with respect to $v$. For some $\tau>0$, we assume that $W$ and
$\mathbf{L}$ are also $\tau$-periodic in time $t$, and that $\mathbf{L}$ satisfies
 \begin{equation*}
\mathbf{L}(-t, q, -v)=\mathbf{L}(t, q, v)=\mathbf{L}(t+\tau, q, v)\quad\forall (t,q,v)\in\R\times TM.
\end{equation*}
Let $\bar{\gamma}\in C^2(\mathbb{R}, M)$ satisfy the problem 
\begin{equation*}
\left.\begin{array}{ll}
&\frac{d}{dt}\big(\partial_v\mathbf{L}(t, \gamma(t), \dot{\gamma}(t))\big)-\partial_q\mathbf{L}(t, \gamma(t), \dot{\gamma}(t))=0\;\;\forall t\in\mathbb{R},\\
&\gamma(-t)=\gamma(t)=\gamma(t+\tau)\quad\forall t\in\mathbb{R}.
\end{array}\right\}
\end{equation*}
Assume that for all $t$, $\partial_q W(t, \bar{\gamma}(t))=0$,
and that  the Hessian of each $W(t,\cdot)$ at the critical point $\bar{\gamma}(t)$ 
is either positive definite or negative definite.
}
\end{assumption}

As above, let $L(\lambda, t, q, v)=L_\lambda(t, q, v)=\mathbf{L}(t,q,v)+\lambda W(t,q)$ 
for $\lambda\in\mathbb{R}$ and $(t,q,v)\in\mathbb{R}\times TM$.
Then $L$ satisfies Assumption~\ref{ass:Lagrbrake}, and 
$\bar{\gamma}$ solves (\ref{e:PPerLagrBrakeorbit}) with this $L$, therefore is a critical
point of each $C^2$ functional ${\mathcal{L}}^E_\lambda:=
{\mathcal{L}}^E_0+ \lambda {\mathcal{W}}^E$
  on the $C^4$ Banach manifold $EC^{1}(S_\tau; M)$, where 
\begin{equation*}
{\mathcal{L}}^E_0(\gamma)=\int_0^\tau \mathbf{L}(t, \gamma(t), \dot{\gamma}(t)) \, dt\quad\text{and}\quad
{\mathcal{W}}^E(\gamma)=\int_0^\tau W(t, \gamma(t)) \, dt.
\end{equation*}
Let $EW^{1,2}(\bar{\gamma}^\ast TM)=\{\xi\in W^{1,2}_{\rm loc}(\bar{\gamma}^\ast TM)\,|\,\xi(-t)=\xi(t)=\xi(t+\tau)\quad\forall t\in\mathbb{R}\}$ with
the inner product given by (\ref{e:1.1}).
 Modifying the arguments in \cite[\S5.1]{BoTr} via the technique methods in Appendix~\ref{app:Exp},
on $EW^{1,2}(\bar{\gamma}^\ast TM)$, we can prove that $D^{2}\mathcal{L}^{E}_0(\bar{\gamma})$
is self-adjoint Fredholm and has also finite Morse index and that 
$D^{2}\mathcal{W}^E(\bar{\gamma})$ is  given a compact self-adjoint operator (cf. \cite{Lu0, Lu1-, Lu1}).
 Therefore, by Proposition~\ref{prop:Abst} we have:

\begin{proposition}\label{prop:6.C}
Under Assumption~\ref{ass:6.C}, suppose that $m^0(\mathcal{L}^{E}_0, \bar{\gamma})> 0$. Then 
there exists $0 < \delta_0 < 1$ such that the following hold:
\begin{itemize}
\item[\rm (i)] If the Hessian of each $W(t,\cdot)$ at the critical point $\bar{\gamma}(t)$ 
is positive definite, then
\[
m^-(\mathcal{L}^{E}_\lambda, \bar{\gamma}) = \begin{cases}
m^-(\mathcal{L}^{E}_0, \bar{\gamma}), & 0 < \lambda < \delta_0,\\[4pt]
m^-(\mathcal{L}^{E}_0, \bar{\gamma}) + m^0(\mathcal{L}^{E}_0, \bar{\gamma}), & -\delta_0 < \lambda < 0.
\end{cases}
\]
\item[\rm (ii)] If the Hessian of each $W(t,\cdot)$ at the critical point $\bar{\gamma}(t)$ 
is  negative definite, then
\[
m^-(\mathcal{L}^{E}_\lambda, \bar{\gamma}) = \begin{cases}
m^-(\mathcal{L}^E_0, \bar{\gamma}) + m^0(\mathcal{L}^{E}_0, \bar{\gamma}), & 0 < \lambda < \delta_0,\\[4pt]
m^-(\mathcal{L}^E_0, \bar{\gamma}), & -\delta_0 < \lambda < 0.
\end{cases}
\]
\end{itemize}
\end{proposition}

By Proposition~\ref{prop:6.A} and
Theorems~\ref{th:bif-suffLagrGener},~\ref{th:bif-suffLagr*}, we arrive at:

  \begin{theorem}[\textsf{Alternative bifurcations of Rabinowitz's type}]\label{th:6.A}
  Under Assumption~\ref{ass:6.A}, suppose that $\mathbf{N}=S_0\times S_1$,
  $\bar{\gamma}(0)\ne\bar{\gamma}(\tau)$
{\rm (}if $\dim S_0>0$ and $\dim S_1>0${\rm )}, 
and $m^0(\mathcal{E}_0, \bar{\gamma})> 0$. Then 
  one of the following alternatives occurs:
\begin{enumerate}
\item[\rm (i)] The problem (\ref{e:LagrGener})--(\ref{e:LagrGenerB})
 with $L$ given by (\ref{e:6.A}) and $\lambda=0$ has a sequence of solutions, $\gamma_k\ne \bar{\gamma}$, $k=1,2,\cdots$,
which converges to $\bar{\gamma}$ in $C^2([0,\tau], M)$.

\item[\rm (ii)] There exists $\delta^\ast>0$ such that
for every non-zero $\lambda\in(-\delta^\ast, \delta^\ast)$,  
the problem (\ref{e:LagrGener})--(\ref{e:LagrGenerB})
 with $L$ given by (\ref{e:6.A}) and parameter value $\lambda$
possesses  a  solution $\alpha_\lambda\ne \bar{\gamma}$ 
 such that  $\alpha_\lambda-\bar{\gamma}$
 converges to zero in  $C^2([0,\tau], \R^N)$ as $\lambda\to 0$.
{\rm (}Recall that we have assumed $M\subset\R^N$.{\rm )}

\item[\rm (iii)] For a given neighborhood $\mathcal{W}$ of $\bar{\gamma}$ in $C^1([0,\tau], M)$, 
there is  a one-sided  neighborhood $\Lambda^0$ of $0$ in $\mathbb{R}$ such that
for any $\lambda\in\Lambda^0\setminus\{\mu\}$, 
the problem (\ref{e:LagrGener})--(\ref{e:LagrGenerB})
 with $L$ given by (\ref{e:6.A}) and parameter value $\lambda$
has at least two distinct solutions in $\mathcal{W}$, $\gamma_\lambda^1\ne \bar{\gamma}$ and $\gamma_\lambda^2\ne \bar{\gamma}$,
which can also be chosen to satisfy $\mathcal{E}_{\lambda}(\gamma_\lambda^1)\ne \mathcal{E}_{\lambda}(\gamma_\lambda^2)$
provided that  $m^0(\mathcal{E}_{0}, \bar{\gamma})>1$ and (\ref{e:LagrGener})--(\ref{e:LagrGenerB}) with parameter value $\lambda$
has only finitely many distinct solutions in $\mathcal{W}$.
\end{enumerate}
Similarly, if $\mathbf{N}=\mathrm{Graph}(\mathbb{I}_{g})$ and $m^0(\mathscr{E}_0, \bar{\gamma})> 0$, then
one of the above alternatives (i)--(iii) occurs
 with (\ref{e:LagrGener})--(\ref{e:LagrGenerB}), and $\mathcal{E}_\lambda$ and $\mathcal{E}_0$ replaced by 
(\ref{e:Lagr2}), and $\mathscr{E}_\lambda$ and $\mathscr{E}_0$, respectively.
\end{theorem}

Proposition~\ref{prop:6.B} and Theorem~\ref{th:bif-suffLagr**} directly lead to:
  
\begin{theorem}[\textsf{Alternative bifurcations of Rabinowitz's type}]\label{th:6.B}
Under Assumption~\ref{ass:6.B}, suppose that $m^0(\mathfrak{E}_0, \bar{\gamma})> 0$. 
Then  one of the following alternatives occurs:
\begin{enumerate}
\item[\rm (i)] The problem (\ref{e:Lagr10}) with 
$L_\lambda(t, q, v)=\mathbf{L}(t,q,v)+\lambda W(t,q)$ and
 $\lambda=0$ has a sequence of solutions, $\gamma_k\ne \bar{\gamma}$, $k=1,2,\cdots$,
 such that $\gamma_k\to\bar{\gamma}$ in $\mathcal{X}^2_{\tau}(M, \mathbb{I}_g)$
 (or equivalently $\gamma_k|_{[0,\tau]}-\bar{\gamma}|_{[0,\tau]}$  converges to zero in  $C^2([0,\tau], \R^N)$).

\item[\rm (ii)]  There exists $\delta^\ast>0$ such that
for every non-zero $\lambda\in(-\delta^\ast, \delta^\ast)$,  
the problem (\ref{e:Lagr10}) with $L_\lambda(t, q, v)=\mathbf{L}(t,q,v)+\lambda W(t,q)$ and
parameter value $\lambda$ possesses  a  solution $\alpha_\lambda\ne \bar{\gamma}$
 such that  $\|(\alpha_\lambda-\bar{\gamma})|_{[0,\tau]}\|_{C^2([0,\tau], \R^N)}\to 0$  as $\lambda\to 0$.

\item[\rm (iii)] For a given neighborhood $\mathcal{W}$ of $\bar{\gamma}$ in $\mathcal{X}^1_{\tau}(M, \mathbb{I}_g)$,
there exists a one-sided neighborhood $\Lambda^0$ of $0$ in $\mathbb{R}$ such that
for any $\lambda\in\Lambda^0\setminus\{0\}$, the problem (\ref{e:Lagr10}) with 
$L_\lambda(t, q, v)=\mathbf{L}(t,q,v)+\lambda W(t,q)$ and parameter value $\lambda$
has at least two distinct solutions in $\mathcal{W}$, $\gamma_\lambda^1\ne \bar{\gamma}$ and $\gamma_\lambda^2\ne \bar{\gamma}$,
which can also be chosen to satisfy $\mathfrak{E}_{\lambda}(\gamma_\lambda^1)\ne\mathfrak{E}_{\lambda}(\gamma_\lambda^2)$
provided that  $m^0_\tau(\mathfrak{E}_{\mu}, \bar{\gamma})>1$ and the problem (\ref{e:Lagr10}) with parameter value
$\lambda\in\Lambda^0\setminus\{0\}$ has only finitely many distinct solutions in $\mathcal{W}$.
\end{enumerate}
\end{theorem}

By Proposition~\ref{prop:6.C} and Theorem~\ref{th:bif-suffLagrBrakeM} we obtain:

\begin{theorem}[\textsf{Alternative bifurcations of Rabinowitz's type}]\label{th:6.C}
Under Assumption~\ref{ass:6.C}, suppose that $m^0(\mathcal{L}^{E}_0, \bar{\gamma})> 0$. Then 
  one of the following alternatives occurs:
\begin{enumerate}
\item[\rm (i)] The problem (\ref{e:PPerLagrBrakeorbit})
 with  $L_\lambda(t, q, v)=\mathbf{L}(t,q,v)+\lambda W(t,q)$ and $\lambda=0$ has a sequence of solutions, $\gamma_k\ne \bar{\gamma}$, $k=1,2,\cdots$,
which converges to $\bar{\gamma}$ in $C^2(S_\tau, M)$.

\item[\rm (ii)] There exists $\delta^\ast>0$ such that
for every non-zero $\lambda\in(-\delta^\ast, \delta^\ast)$,  
the problem (\ref{e:PPerLagrBrakeorbit}) with  $L_\lambda(t, q, v)=\mathbf{L}(t,q,v)+\lambda W(t,q)$ and
parameter value $\lambda$, possesses  a  solution $\alpha_\lambda\ne \bar{\gamma}$ 
such that  $\alpha_\lambda-\bar{\gamma}$
 converges to zero in  $C^2(S_\tau, \R^N)$ as $\lambda\to 0$.
{\rm (}Recall that $M\subset\R^N$.{\rm )}

\item[\rm (iii)] For a given neighborhood $\mathcal{W}$ of $\bar{\gamma}$ in $C^2(S_\tau, M)$,
there exists a one-sided neighborhood $\Lambda^0$ of $0$ in $\mathbb{R}$ such that
for any $\lambda\in\Lambda^0\setminus\{0\}$, (\ref{e:PPerLagrBrakeorbit}) with 
 $L_\lambda(t, q, v)=\mathbf{L}(t,q,v)+\lambda W(t,q)$ and parameter value $\lambda$
has at least two distinct solutions in $\mathcal{W}$, $\gamma_\lambda^1\ne \bar{\gamma}$ and $\gamma_\lambda^2\ne \bar{\gamma}$,
which can also be chosen to satisfy $\mathcal{L}^E_{\lambda}(\gamma_\lambda^1)\ne\mathcal{L}^E_{\lambda}(\gamma_\lambda^2)$
provided that $m^0_\tau(\mathcal{L}^E_{\mu}, \bar{\gamma})>1$ and (\ref{e:PPerLagrBrakeorbit}) with parameter value $\lambda$
has only finitely many solutions in $\mathcal{W}$.
\end{enumerate}
\end{theorem}

In summary, the above discussion indicates that for a sufficiently smooth Lagrangian function that is strictly convex fiberwise, the corresponding Lagrangian system admits a perturbation of the Lagrangian function near a solution at which the action functional has a non-vanishing nullity, such that a Rabinowitz-type bifurcation occurs near that solution.

Consider  Physical Lagrangians (including  potential
 and electromagnetic forces)  given by
\begin{align}\label{e:6.B}
\mathbf{L}(t,q,v) = \frac{1}{2}\langle P(t,q)v, v\rangle_g+ \alpha(t, q)[v]- V(t, q),
\end{align}
where  $P(t,q):T_qM\to T_qM$ is self-adjoint and positive definite with respect to
$\langle \cdot,\cdot\rangle_g$, 
$\alpha$ is a  $t$-dependent one-form (the magnetic potential),
$\alpha(t, q)[v]$ denotes the value of $\alpha(t, q)\in T_q^\ast M$ at $v\in T_qM$,
and $V$ is a  $t$-dependent potential energy. 
They also satisfy the following respective conditions in the three cases:
\begin{itemize}
\item[\rm (i)] The functions $[0, \tau]\times M\to\mathbb{R},\; (t,q)\mapsto V(t,q)$ and
\begin{align*}
[0, \tau]\times TM\to\mathbb{R},\; (t,q,v)\mapsto \langle P(t,q)v, v\rangle_g,\quad
[0, \tau]\times TM\to\mathbb{R},\; (t,q,v)\mapsto \alpha(t, q)[v]
\end{align*}
 are $C^2$.
\item[\rm (ii)] The functions $\mathbb{R}\times M\to\mathbb{R},\; (t,q)\mapsto V(t,q)$ and
\begin{align*}
\mathbb{R}\times TM\to\mathbb{R},\; (t,q,v)\mapsto \langle P(t,q)v, v\rangle_g,\quad
\mathbb{R}\times TM\to\mathbb{R},\; (t,q,v)\mapsto \alpha(t, q)[v]
\end{align*}
 are $C^2$, and also satisfy for all $(t,q,v)\in\mathbb{R}\times TM$,
\begin{equation*}
P(t, q)=P(t+\tau, \mathbb{I}_g(q)),\quad V(t, q)=V(t+\tau, \mathbb{I}_g(q)),\quad
(\mathbb{I}_g)^\ast\alpha(t+\tau,\cdot)=\alpha(t,\cdot).
\end{equation*}

\item[\rm (iii)] $P$, $\alpha$ and $V$ satisfy the first three lines in (ii) and 
the following equalities 
 \begin{align*}
& P(-t, q)=P(t, q)=P(t+\tau, q),   \quad V(-t, q)=V(t, q)=V(t+\tau, q),\\
 &-\alpha(-t, q)=\alpha(t, q)=\alpha(t+\tau, q)\quad
\text{for all $(t,q)\in\R\times M$.}
\end{align*}
\end{itemize} 
Then Theorems~\ref{th:6.A},~\ref{th:6.B},~\ref{th:6.C}
are applicable to $\mathbf{L}$ in (\ref{e:6.B})
provided that  conditions (i), (ii) and (iii)
are satisfied, respectively, 

When $M$ is $\mathbb{R}^n$ or an open subset of $\mathbb{R}^n$,
we can construct more general examples than those in Theorems~\ref{th:6.A},~\ref{th:6.B},~\ref{th:6.C}.

For a real $\tau>0$, consider a $C^3$ function
 $L:(-1,1)\times [0,\tau]\times \mathbb{R}^n\times \mathbb{R}^n\to\R$.
 Suppose that $L(\lambda, t, q, v)$ is strictly convex in $v$ for
 all $(\lambda,t,q)$. Let $R$ be a linear subspace of $\mathbb{R}^n\times\mathbb{R}^n$,
 which is either the product $V_0\times V_1$ of  
  two linear subspaces $V_0$ and $V_1$ of $\mathbb{R}^n$
  or the graph $\mathrm{Graph}(E)$ of an orthogonal matrix $E$ of order $n$,
  i.e., $R=\{(x^\top, (Ex)^\top)^\top\,|\, x\in \mathbb{R}^n\}$.
For each $\lambda\in (-1,1)$, let $x_\lambda\in C^3([0,\tau]; \mathbb{R}^n)$ satisfy 
the following boundary value problem:
\begin{equation}\label{e:6.1}
\left.\begin{array}{ll}
\frac{d}{dt}\big(\partial_vL_\lambda(t, x(t), \dot{x}(t))\big)-\partial_q L_\lambda(t, x(t), \dot{x}(t))=0\;\forall
t\in [0,\tau],\\
x\in C^2([0, \tau]; \mathbb{R}^n),\; (x(0)^\top, x(\tau)^\top)^\top\in R\quad\hbox{and}\quad\\
\partial_vL_\lambda(0, x(0), \dot{x}(0))[v_0]=\partial_vL_\lambda(\tau, x(\tau), \dot{x}(\tau))[v_1]
\quad\forall (v_0^\top, v_1^\top)^\top\in R.
\end{array}\right\}
\end{equation}
That is, $x_\lambda$ is a critical point of the $C^2$ functional
\begin{equation}\label{e:6.2}
\mathcal{E}_\lambda: C^{1}_{R}([0,\tau]; \mathbb{R}^n)\to\mathbb{R},\;\gamma\mapsto \int^\tau_0L_\lambda(t, x(t), \dot{x}(t))dt.
\end{equation}
It is also assumed that $(-1,1)\times [0,\tau]\ni (\lambda,t)\mapsto x_\lambda(t)\in \mathbb{R}^n$ 
is $C^1$. Let $\mathbf{H}_R={\bf H}_{V_0\times V_1}$ if $R=V_0\times V_1$,
and $\mathbf{H}_R={\bf H}_{E}=W^{1,2}_{E}([0,\tau];\R^{n})$ if $R=\text{Graph}(E)$.
By Proposition~\ref{prop:funct-analy}, 
 the G\^ateaux derivative  $B_\lambda(x_\lambda)\in{\mathscr{L}}_s({\bf H}_{R})$ at $x_\lambda$ of  the gradient $\nabla{\mathcal{E}}_{\lambda}$ is a self-adjoint Fredholm operator  given by
\begin{eqnarray*}
 (B_\lambda(x_\lambda)\xi,\eta)_{1,2}
   = \int_0^{\tau} \Bigl(\!\! \!\!\!&&\!\!\!\!\!
   \bigl(P_\lambda(t)\dot{\xi}(t), \dot{\eta}(t)\bigr)_{\mathbb{R}^n}
+ \bigl(Q_\lambda(t)\xi(t), \dot{\eta}(t)\bigr)_{\mathbb{R}^n}\nonumber \\
&& + \bigl(\dot{\xi}(t), Q_\lambda(t)\eta(t)\bigr)_{\mathbb{R}^n}
  +\bigl(R_\lambda(t)\xi(t), \eta(t)\bigr)_{\mathbb{R}^n}\Bigr) \, dt
\end{eqnarray*}
for any  $\xi,\eta\in {\bf H}_{R}$, where 
${P}_\lambda(t)=\partial_{vv}{L}_{{\lambda}}\bigl(t, x_\lambda(t),
\dot{x}_\lambda(t)\bigr)$, ${Q}_\lambda(t)=D_{{q}{v}}{L}_{{\lambda}}\bigl(t, x_\lambda(t),
\dot{x}_\lambda(t)\bigr)$, ${R}_\lambda(t)=D_{{q}{q}}{L}_{{\lambda}}\bigl(t, x_\lambda(t),
\dot{x}_\lambda(t)\bigr)$. 
Since $L$ is $C^3$, it is easily verified that $B_\lambda(x_\lambda)$ is differentiable in $\lambda$ and
 \begin{eqnarray}\label{e:6.3}
 (\frac{d}{d\lambda}B_\lambda(x_\lambda)\xi,\eta)_{1,2}
   = \int_0^{\tau} \Bigl(\!\! \!\!\!&&\!\!\!\!\!
   \bigl(\frac{d}{d\lambda}P_\lambda(t)\dot{\xi}(t), \dot{\eta}(t)\bigr)_{\mathbb{R}^n}
+ \bigl(\frac{d}{d\lambda}Q_\lambda(t)\xi(t), \dot{\eta}(t)\bigr)_{\mathbb{R}^n}\nonumber \\
&& + \bigl(\dot{\xi}(t), \frac{d}{d\lambda}Q_\lambda(t)\eta(t)\bigr)_{\mathbb{R}^n}
  +\bigl(\frac{d}{d\lambda}R_\lambda(t)\xi(t), \eta(t)\bigr)_{\mathbb{R}^n}\Bigr) \, dt.
\end{eqnarray}
Now suppose ${\rm ker}B_0(x_0)\ne 0$. Then ${\rm ker}B_0(x_0)$ consists of solutions $\zeta$ of
the following linear boundary value problem:
\begin{equation}\label{e:6.4}
\left.\begin{array}{ll}
\frac{d}{dt}\big(P_0(t)\dot{\zeta}(t)+ Q_0(t)\zeta(t)\big)-
\big(Q_0(t)^\top\dot{\zeta}(t)+ R_0(t)\zeta(t)\big)=0\;\forall
t\in [0,\tau],\\
\zeta\in C^2([0, \tau]; \mathbb{R}^n),\;(\zeta(0)^\top, \zeta(\tau)^\top)^\top\in R\quad\hbox{and}\quad\\
(P_0(0)\dot{\zeta}(0)+ Q_0(0)\zeta(0), v_0)_{\mathbb{R}^n}=
(P_0(\tau)\dot{\zeta}(\tau)+ Q_0(\tau)\zeta(\tau), v_1)_{\mathbb{R}^n}
\quad\forall (v_0^\top, v_1^\top)\in R.
\end{array}\right\}
\end{equation}
By Proposition~\ref{prop:Abst}, we immediately obtain:

\begin{proposition}\label{prop:6.1}
Under the above assumptions, suppose that $m^0(\mathcal{E}_0, x_0) > 0$. Then there exists $0 < \delta_0 < 1$ such that the following hold:
\begin{itemize}
\item[\rm (i)] If the self-adjoint operator $\left.\frac{d}{d\lambda}B_\lambda(x_\lambda)\right|_{\lambda=0}$ defined by (\ref{e:6.3}) is positive definite on the solution space of (\ref{e:6.4}), then
\[
m^-(\mathcal{E}_\lambda, x_\lambda) = \begin{cases}
m^-(\mathcal{E}_0, x_0), & 0 < \lambda < \delta_0,\\[4pt]
m^-(\mathcal{E}_0, x_0) + m^0(\mathcal{E}_0, x_0), & -\delta_0 < \lambda < 0.
\end{cases}
\]
\item[\rm (ii)] If the operator $\left.\frac{d}{d\lambda}B_\lambda(x_\lambda)\right|_{\lambda=0}$
is negative definite on the solution space of (\ref{e:6.4}), then
\[
m^-(\mathcal{E}_\lambda, x_\lambda) = \begin{cases}
m^-(\mathcal{E}_0, x_0) + m^0(\mathcal{E}_0, x_0), & 0 < \lambda < \delta_0,\\[4pt]
m^-(\mathcal{E}_0, x_0), & -\delta_0 < \lambda < 0.
\end{cases}
\]
\end{itemize}
Similarly, if  $R=\text{Graph}(E)$ and $m^0(\mathcal{E}_0, x_0)> 0$, then
(i) and (ii) hold. 
\end{proposition}

This, together with the Theorems~\ref{th:bif-suffLagrGenerEu},~\ref{th:Pbif-suffLagrGenerEu},
leads to:

 \begin{theorem}[\textsf{Alternative bifurcations of Rabinowitz's type}]\label{th:6.D}
Suppose that $m^0(\mathcal{E}_0, x_0) > 0$, and that one of the conditions (i) and (ii) in Proposition~\ref{prop:6.1} holds. Then one of the following alternatives occurs:
\begin{enumerate}
\item[\rm (i)] The problem (\ref{e:6.1}) with $\lambda = 0$ has a sequence of nonzero solutions $\{x_k\}_{k=1}^\infty$ in $C^2([0,\tau], \mathbb{R}^n)$ that converges to the 
    trivial solution (i.e., $x_k \to 0$ in $C^2$).

\item[\rm (ii)] For every $\lambda$ sufficiently close to $0$ but $\lambda \neq 0$, there exists a nonzero solution $y_\lambda$ of (\ref{e:6.1}) with parameter $\lambda$, such that $y_\lambda \to 0$ in $C^2([0,\tau], \mathbb{R}^n)$ as $\lambda \to 0$.

\item[\rm (iii)] Given any neighborhood $\mathfrak{W}$ of $0$ in $C^1_{R}([0,\tau]; 
\mathbb{R}^n)$, there exists a one-sided neighborhood $\Lambda^0$ of $0$ (i.e., $\Lambda^0 \subset (0, \delta)$ or $\Lambda^0 \subset (-\delta, 0)$ for some $\delta > 0$) such that for every $\lambda \in \Lambda^0 \setminus \{0\}$, problem (\ref{e:6.1}) with parameter $\lambda$ has at least two distinct nonzero solutions $y_\lambda^1, y_\lambda^2 \in \mathfrak{W}$. Moreover, if $m^0(\mathcal{E}_{0}, 0) > 1$ and problem (\ref{e:6.1}) with parameter $\lambda$ has only finitely many solutions in $\mathfrak{W}$, then these two solutions can be chosen to satisfy $\mathcal{E}_\lambda(y^1_\lambda) \neq \mathcal{E}_\lambda(y^2_\lambda)$.
\end{enumerate}
Similarly, if  $R=\text{Graph}(E)$ and $m^0(\mathcal{E}_0, x_0)> 0$, then
one of the above alternatives (i)--(iii) occurs.
\end{theorem}

The results analogous to Theorems~\ref{th:6.B} and ~\ref{th:6.C} can be readily formulated.
We omit them. Instead, we consider a concrete example.

 \subsection{Example: planar simple pendulum}\label{sec:Cor-example-II}

 Consider the  planar simple pendulum equation:
\begin{equation}\label{e:6.9}
\ddot{u}+\frac{g}{l}\sin u=0, 
\end{equation}
where $g$ is the universal gravitational constant, $l$ is the length of the pendulum, and one end of the pendulum is fixed at the origin. It is well known that there are two kinds of nontrivial
motions of this planar simple pendulum: rotations around the origin and oscillations. 
The linearized equation of (\ref{e:6.9}) at the zero solution is
$l\ddot{u}(t)+ g u(t)=0$,
which has the  general solution  $u=C_1\cos(\omega t)+ C_2\sin(\omega t)$, where $\omega=\sqrt{\frac{g}{l}}$. These solutions have period $T:=\frac{2\pi}{\omega}$.
$T$-periodic solutions of (\ref{e:6.9}) are critical points of the functional
\begin{equation}\label{e:6.10}
\Phi(u)=\int^T_0\textsf{L}(t, u(t), \dot{u}(t)) dt
\end{equation}
on $W^{1,2}(\mathbb{R}/(T\mathbb{Z}), \mathbb{R})$, where
$\textsf{L}(t, x, y) =\frac{1}{2}l y^2+ g\cos x$.
From the above arguments, we see that  $0\in W^{1,2}(\mathbb{R}/(T\mathbb{Z}), \mathbb{R})$ is 
a  critical point of $\Phi$ with nullity $m^0(\Phi, 0)=2$.
Let $W(\lambda,t,x)$ be a $C^3$ function on $(-1, 1)\times\mathbb{R}^2$ that is $T$-periodic in time $t$.
We also assume 
\begin{equation}\label{e:6.11}
W_x(\lambda,t,0)\equiv 0,\quad W_{xx\lambda}(0, \cdot, 0)\ne 0\quad
\text{and either $W_{xx\lambda}(0, \cdot, 0)\ge 0$ or 
$W_{xx\lambda}(0, \cdot, 0)\le 0$}.
\end{equation}
A clear example of the function $W$ satisfying the above conditions is 
\begin{equation}\label{e:6.12}
W(\lambda,t,x)=(1+\sin(\frac{2\pi}{T}t))\sin\lambda\cos x.
\end{equation}
 Define $L(\lambda,t,x,y) := \textsf{L}(t,x,y) + W(\lambda, t,x)$, and let $\Phi_\lambda$ denote the corresponding action functional defined on  $W^{1,2}(\mathbb{R}/(T\mathbb{Z}), \mathbb{R})$. 
 The first equality in (\ref{e:6.11}) implies that 
  $0\in W^{1,2}(\mathbb{R}/(T\mathbb{Z}), \mathbb{R})$
 is a critical point of each $\Phi_\lambda$.
   Let $B_\lambda(0)$ be the G\^ateaux derivative  of  the gradient $\nabla{\Phi}_{\lambda}$
 at $0\in W^{1,2}(\mathbb{R}/(T\mathbb{Z}), \mathbb{R})$.  Then
 $\mathrm{ker}B_0(0)=\{u=C_1\cos(\omega t)+ C_2\sin(\omega t)\,|\,C_1, C_2\in\mathbb{R}\}$.
  Note that
 \begin{align*}
 &{P}_\lambda(t)=\partial_{yy}{L}_{{\lambda}}\bigl(t, 0, 0\bigr)=l,\qquad {Q}_\lambda(t)=D_{{x}{y}}{L}_{{\lambda}}\bigl(t, 0, 0\bigr)=0,\\ &{R}_\lambda(t)=D_{{x}{x}}{L}_{{\lambda}}\bigl(t, 0, 0\bigr)=-g+ W_{xx}(\lambda, t, 0).
\end{align*}
 By (\ref{e:6.3}), we obtain the following identity  for all $\xi, \eta\in W^{1,2}(\mathbb{R}/(T\mathbb{Z}), \mathbb{R})$,
  \begin{equation}\label{e:6.13}
 (\frac{d}{d\lambda}B_\lambda(x_\lambda)\big|_{\lambda=0}\xi,\eta)_{1,2}
   = \int_0^{T}W_{xx\lambda}(0, t, 0)\xi(t)\eta(t)\, dt.
\end{equation}
 It follows from the assumptions in (\ref{e:6.11}) that
 $\frac{d}{d\lambda}B_\lambda(x_\lambda)\big|_{\lambda=0}$ is either positive definite
 or negative definite.
  Since $m^0(\Phi, 0)=2>0$.
By apply Theorem~\ref{th:6.D} (with $E=I_n$ being the identity matrix of order $n$)
to this $L$, we readily deduce that  one of the following alternatives occurs:
\begin{enumerate}
\item[\rm (i)] The forced  simple pendulum
\begin{equation}\label{e:6.14}
m\ell\ddot{u}(t)+ g\sin u(t)=W_x(0, t, u(t))
\end{equation}
 has a sequence of nontrivial $T$-periodic solutions $\{u_k\}^\infty_{k=1}$
 such that  $u_k|_{[0, T]}$  converges to zero in  $C^2([0, T])$.

\item[\rm (ii)]  There exists $\delta^\ast>0$ such that
for every  $\lambda\in(-\delta^\ast, \delta^\ast)\setminus\{0\}$,  
the forced  simple pendulum
\begin{equation}\label{e:6.15}
m\ell\ddot{u}(t)+ g\sin u(t)=W_x(\lambda, t, u(t))
\end{equation}
 possesses  a nontrivial $T$-periodic solution $u_\lambda$
 such that  $u_\lambda|_{[0, T]}$  converges to zero in  $C^2([0, T])$)
   as $\lambda\to 0$.

\item[\rm (iii)] For a given neighborhood $\mathcal{W}$ of zero in 
$C^2([0, T])$, there exists a one-sided neighborhood $\Lambda^0$ of $0$ in $\mathbb{R}$ such that
for any $\lambda\in\Lambda^0\setminus\{0\}$, the problem (\ref{e:6.15}) 
has at least two distinct nontrivial solutions $u_\lambda^1$ and $u_\lambda^2$
whose restrictions to $[0, T]$ belong to $\mathcal{W}$.
Moreover, since $m^0(\Phi, 0)=2>1$, these two solutions  can also be chosen to satisfy 
$$
\int^T_0\textsf{L}(t, u_\lambda^1(t), \dot{u}_\lambda^1(t)) dt+ \int^T_0W(\lambda,t,
u_\lambda^1(t))dt\ne \int^T_0\textsf{L}(t, u_\lambda^2(t), \dot{\gamma}_\lambda^2(t)) dt+ \int^T_0W(\lambda,t, u_\lambda^2(t))dt
$$
provided that  the problem (\ref{e:6.15})  has only finitely many distinct solutions, 
and their restrictions to $[0, T]$ belong to $\mathcal{W}$.
\end{enumerate}

\begin{remark}\label{rm:6.1}
{\rm When $W(\lambda, t, x)$ has a form $\lambda \mathbf{W}(t, x)$, where
 $\mathbf{W}(t,x)$ is a $C^3$ function on $(-1, 1)\times\mathbb{R}$ that is $T$-periodic in time $t$
 and satisfies the following conditions:
\begin{equation}\label{e:6.16}
\mathbf{W}_x(t,0)\equiv 0,\quad \mathbf{W}_{xx\lambda}(\cdot, 0)\ne 0\quad
\text{and either $\mathbf{W}_{xx}(\cdot, 0)\ge 0$ or 
$\mathbf{W}_{xx}(\cdot, 0)\le 0$}.
\end{equation}
We have a $C^2$ functional $\mathfrak{W}$ on $W^{1,2}(\mathbb{R}/(T\mathbb{Z}), \mathbb{R})$ defined by
\begin{equation*}
\mathfrak{W}(u)=\int^T_0\mathbf{W}(t, u(t)) dt
\end{equation*}
with critical point $0\in W^{1,2}(\mathbb{R}/(T\mathbb{Z}), \mathbb{R})$. 
When $\mathbf{W}_{xx}(\cdot, 0)\ge 0$ (resp.  
$\mathbf{W}_{xx}(\cdot, 0)\le 0$), $D^2\mathfrak{W}(0)$ is positive definite (resp. negative definite)
on $\mathrm{ker}B_0(0)=\{u=C_1\cos(\omega t)+ C_2\sin(\omega t)\,|\,C_1, C_2\in\mathbb{R}\}$.
By applying \cite[Theorem~6.6]{MaWi} to $\alpha=\Phi$ and $\beta=\mathfrak{W}$, we obtain:\\
\noindent{\bf Conclusion}.\quad\emph{For each sufficiently small $\epsilon>0$, eauation
\begin{equation}\label{e:6.17}
m\ell\ddot{u}(t)+ g\sin u(t)=\lambda W_x(t, u(t))
\end{equation}
has at least two solutions $(\lambda_{\epsilon}, u_{\epsilon})$ such that
$\mathfrak{W}(u_\epsilon)=\epsilon$ (resp. $\mathfrak{W}(u_\epsilon)=-\epsilon$)
provided that $\mathbf{W}_{xx}(\cdot, 0)\ge 0$ (resp.  
$\mathbf{W}_{xx}(\cdot, 0)\le 0$). Moreover, $\lambda_\epsilon\to 0$ as $\epsilon\to 0$.}

Therefore, our above example cannot be derived from \cite[Theorem~6.6]{MaWi}.
}
\end{remark}

\appendix

\section{Proofs of some lemmas}\label{app:Exp}

\paragraph{Exponential map}
Let $M$ be a $n$-dimensional,  $C^k$-smooth  manifold.
Its tangent bundle $TM$ is a $C^{k-1}$-smooth manifold of dimension $2n$,
whose points are denoted by $(x,v)$, with $x\in M$ and $v\in T_xM$.
Let $g$ be a $C^{k-1}$ Riemannian metric on $M$.
Let $\varphi:U\to \varphi(U)\subset\mathbb{R}^n,\;q\mapsto\varphi(q)=(x^1(q),\cdots, x^n(q))$
be a coordinate chart on $M$,   $g_{ij}=g(\frac{\partial}{\partial x^i}, \frac{\partial}{\partial x^j})$
and $(g^{ij})=(g_{ij})^{-1}$. Then the Christoffel symbols
$$
\Gamma^k_{ij}=\frac{1}{2}g^{kj}\Big(\frac{\partial g_{li}}{\partial x^j}+ \frac{\partial g_{lj}}{\partial x^i}-
\frac{\partial g_{ij}}{\partial x^l}\Big)
$$
are $C^{k-2}$ and hence the exponential map $\exp:TM\to M$ is a $C^{k-2}$ map.
%(cf. page 209 in Schwart's nonlinear FA).
There exists a fibrewise convex neighborhood $\mathcal{U}(0_{TM})$ of the zero section of $TM$ such that  the map
\begin{eqnarray}\label{e:App.1}
\mathbb{F}:\mathcal{U}(0_{TM})\to M\times M,\;(q,v)\mapsto (q,\exp_qv)
\end{eqnarray}
is a $C^{k-2}$ immersion.

For a $W^{1,2}$-curve $\gamma:[a,b]\to M$, a $W^{1,2}$ vector field $V$ along $\gamma$ and $t\in \gamma^{-1}(U)$ we can write
$$
V(t)=\sum_jv^j(t)\frac{\partial}{\partial x^j}|_{\gamma(t)}\quad\hbox{and}\quad \dot{\gamma}(t)=\sum_j\dot{\gamma}^j(t)\frac{\partial}{\partial x^j}|_{\gamma(t)},
$$
where $\gamma^j(t)=x^j(\gamma(t))$. The ($W^{1,2}$) covariant derivative of $V$ along $\gamma$
\begin{equation}\label{e:covariantDerivative}
\frac{DV}{dt}=\sum_i\frac{Dv^i}{dt}\frac{\partial}{\partial x^i}|_{\gamma(t)}=
\sum_i\Big(\frac{dv^i}{dt}+\Gamma^{i}_{jk}(\gamma(t))\dot{\gamma}^j(t)v^k(t)\Big)\frac{\partial}{\partial x^i}|_{\gamma(t)}.
\end{equation}
The vector field $V$ is called parallel if $\frac{DV}{dt}\equiv 0$.

Let $\pi:TM\to M$ be the bundle projection and $\mathscr{U}=\pi^{-1}(U)$.
$\varphi$ induces a chart on $M$,
\begin{equation}\label{e:inducedChart}
\Phi:\mathscr{U}\to\varphi(U)\times\mathbb{R}^n,\;(q,v)\mapsto (x^1(q,v),\cdots, x^n(q,v), u^1(q,v),\cdots, u^n(q,v)),
\end{equation}
where $x^i(q,v)=x^i(q)$ and $u^i(q,v)=dx^i(q)[v]$ for $i=1,\cdots,n$, i.e.,
$v=\sum^n_{i=1}u^1(q,v)\frac{\partial}{\partial x^i}|_q$.
It is computed in the standard textbooks in Riemannian geometry that
\begin{eqnarray}\label{e:App.2}
d\mathbb{F}(q,0)\Big[\frac{\partial}{\partial x^i}(q,0)\Big]=\Big(\frac{\partial}{\partial x^i}(q), \frac{\partial}{\partial x^i}(q)\Big),\quad
d\mathbb{F}(q,0)\Big[\frac{\partial}{\partial u^i}(q,0)\Big]=\Big(0, \frac{\partial}{\partial x^i}(q)\Big).
\end{eqnarray}

\begin{claim}\label{cl:pallVEctorF}
Let $k>2$. For an even and $\tau$-periodic $C^{k-2}$-curve $\gamma:\R\to M$,
there exist unit orthogonal parallel $C^{k-2}$ frame fields along $\gamma$,
$\{e_1,\cdots, e_n\}$, such that for all $t\in\R$,
\begin{eqnarray}\label{e:evenframefield1}
&&(e_1(-t),\cdots, e_n(-t))=(e_1(t),\cdots, e_n(t)),\\
&&(e_1(t+\tau),\cdots, e_n(t+\tau))=(e_1(t),\cdots, e_n(t)).\label{e:evenframefield2}
\end{eqnarray}
\end{claim}
\begin{proof}[\bf Proof]
Starting with a unit orthogonal  frame at $T_{\gamma(0)}M$ and using the parallel
transport along $\overline{\gamma}$ with respect to the Levi-Civita
connection of the Riemannian metric $g$ we get a unit orthogonal
parallel $C^5$ frame field $\R\to {\gamma}^\ast TM,\;t\mapsto
(e_1(t),\cdots, e_n(t))$.

Firstly, we prove that (\ref{e:evenframefield1}) is satisfied.
In fact, let $(U; x^j)$ be a local coordinate system  around a point in ${\gamma}(\R)$.
Then we can write
$$
e_k(t)=\sum^n_{i=1}e^l_k(t)\frac{\partial}{\partial x^l}({\gamma}(t))\quad\forall t\in{\gamma}^{-1}(U),
\quad k=1,\cdots,n.
$$
Since $\{e_1,\cdots,e_n\}$ is a parallel frame field,
\begin{equation*}
\frac{d}{dt}e^i_l(t)+\Gamma^{i}_{jk}({\gamma}(t))\dot{{\gamma}}^j(t)e^k_l(t)=0,\quad\forall t\in{\gamma}^{-1}(U),\quad i=1,\cdots,n.
\end{equation*}
Note that ${\gamma}(-t)={\gamma}(t)$ implies $t\in{\gamma}^{-1}(U)$ if and only if $-t\in{\gamma}^{-1}(U)$.
We have
\begin{eqnarray*}
\frac{d}{dt}(e^i_l(-t))=-\dot{e}^i_l(-t)&=&\Gamma^{i}_{jk}({\gamma}(-t))\dot{{\gamma}}^j(-t)e^k_l(-t)\\
&=&-\Gamma^{i}_{jk}({\gamma}(-t))\frac{d}{dt}({\gamma}^j(-t))e^k_l(-t)\\
&=&-\Gamma^{i}_{jk}({\gamma}(t))\frac{d}{dt}({\gamma}^j(t))e^k_l(-t).
\end{eqnarray*}
 It follows that
$\{e_1(-\cdot),\cdots,e_n(-\cdot)\}$ is also a parallel frame field along ${\gamma}$.
Since $e_k(t)$ and $e_k(-t)$ agree at $t=0$, $k=1,\cdots,n$,
by the theorem of existence and uniqueness of ODE we obtain (\ref{e:evenframefield1}).

Next, we prove that for any $k\in\N$ there holds
\begin{eqnarray}\label{e:evenframefield3}
(e_1(k\tau-t),\cdots, e_n(k\tau-t))=(e_1(t),\cdots, e_n(t))\quad\forall t\in [0, k\tau].
\end{eqnarray}
Since for any $t\in [0, k\tau]$ it holds that $t\in{\gamma}^{-1}(U)$ if and only if $-t\in{\gamma}^{-1}(U)$, as above we get
\begin{eqnarray*}
\frac{d}{dt}(e^i_l(k\tau-t))=-\dot{e}^i_l(k\tau-t)&=&\Gamma^{i}_{jk}({\gamma}(k\tau-t))\dot{{\gamma}}^j(k\tau-t)e^k_l(k\tau-t)\\
&=&-\Gamma^{i}_{jk}({\gamma}(-t))\frac{d}{dt}({\gamma}^j(k\tau-t))e^k_l(k\tau-t)\\
&=&-\Gamma^{i}_{jk}({\gamma}(t))\frac{d}{dt}({\gamma}^j(t))e^k_l(k\tau-t),\quad l=1,\cdots,n.
\end{eqnarray*}
Therefore $[0, k\tau]\to e_l(t)$ and $[0, k\tau]\to e_l(k\tau-t)$
are parallel frame fields along $\gamma|_{[0, k\tau]}$ and have the same value at $t=k\tau/2$.
It follows that $e_l(k\tau-t)=e_l(t)$ for any $t\in[0, k\tau]$ and $l=1,\cdots,n$.

Finally, we prove (\ref{e:evenframefield2}). For any $t>0$, let us choose $k\in\N$ such that $t<k\tau$. Then
$$
e_l(\tau+t)=e_l((k+1)\tau-(\tau+t))=e_l(k\tau-t)=e_l(t).
$$
If $t<-\tau$, this and (\ref{e:evenframefield1}) lead to $e_l(\tau+t)=e_l(-\tau-t)=e_l(\tau+(-\tau-t))=e_l(-t)=e_l(t)$.
If $-\tau\le t<0$ then $0\le t+\tau<\tau$. By (\ref{e:evenframefield3}) and (\ref{e:evenframefield1}) we derive
$e_l(\tau+t)=e_l(\tau-(\tau+t))=e_l(-t)=e_l(t)$.
Summarizing these  (\ref{e:evenframefield2}) is proved.

\textsf{{Note}}: Since $\gamma$ is even, as a loop $\gamma:S_\tau\to M$ is contractible. Thus $\gamma^\ast TM\to S_\tau$ has an orthogonal
trivialization. This can only lead to the existence of an unit orthogonal  frame fields along $\gamma$
satisfying (\ref{e:evenframefield1}).
\end{proof}

\begin{proof}[\bf Proof of Lemma~\ref{lem:twoCont}]
By Lemma~\ref{lem:regu} it suffices to prove the third and fourth assertions.
Follow the notations above Lemma~\ref{lem:twoCont}.

{\it Step~1}(\textsf{Prove that $(\lambda,t)\mapsto {\bf u}_\lambda(t)$ is continuous, and $C^2$ with respect to $t$}).
By ($\spadesuit$) in the first paragraph in Section~\ref{sec:LagrBound.1.1} and (\ref{e:gamma1}), for all $t\in [0, \tau]$ we have
\begin{equation}\label{e:two-chart}
\mathbb{F}\Big(\overline{\gamma}(t),
\sum^n_{i=1}{\bf u}_\lambda^i(t)e_i(t)\Big)=\Big(\overline{\gamma}(t),\exp_{\overline{\gamma}(t)}\Big(\sum^n_{i=1}{\bf u}_\lambda^i(t)
e_i(t)\bigg)\Big)=\Big(\overline{\gamma}(t), \gamma_\lambda(t)\Big).
\end{equation}
By (\ref{e:gamma3}), $(\lambda,t)\mapsto\gamma_\lambda(t)$ is a continuous map from $\Lambda\times[0,\tau]$ into
the open subset ${\bf U}_{3\iota}(\gamma_\mu([0, \tau]))$ of $M$. These and ($\clubsuit$) in the first paragraph in Section~\ref{sec:LagrBound.1.1}
imply the composition
$$
\Lambda\times[0,\tau]\to TM,\;(\lambda,t)\mapsto \Big(\mathbb{F}|_{\mathcal{W}(0_{TM})}\Big)^{-1}(\overline{\gamma}(t),\gamma_\lambda(t))=
\Big(\overline{\gamma}(t),
\sum^n_{i=1}{\bf u}_\lambda^i(t)e_i(t)\Big)
$$
is continuous, and $C^2$ with respect to $t$. Since  $g$ is a $C^6$ Riemannian metric on $M$, we obtain
$$
\Lambda\times[0,\tau]\to \mathbb{R},\;(\lambda,t)\mapsto g\Big(\Big(\overline{\gamma}(t),
\sum^n_{i=1}{\bf u}_\lambda^i(t)e_i(t)\Big), \Big(\overline{\gamma}(t),
e_j(t)\Big)\Big)={\bf u}_\lambda^j(t)
$$
is continuous, and $C^2$ with respect to $t$, for each $j=1,\cdots,n$.

{\it Step~2}(\textsf{Prove that $(\lambda,t)\mapsto \dot{\bf u}_\lambda(t)$ is continuous}).
Fix a point $t_0\in [0,\tau]$. Let $\varphi:U\to \varphi(U)\subset\mathbb{R}^n$
be a coordinate chart centered at $\overline{\gamma}(t_0)$
with $\varphi(q)=(x^1(q),\cdots, x^n(q))$. For example, we can take $U=\phi_{\overline{\gamma}}(t_0,B^n_{2\iota}(0))$
and $\varphi=(\phi_{\overline{\gamma}}(t_0,\cdot))^{-1}$.
It has the induced chart $(\pi^{-1}(U),\Phi)$ on $TM$ as in (\ref{e:inducedChart}).
Let $J=(\overline{\gamma})^{-1}(U)$, which is an open neighborhood of $t_0$ in $[0,\tau]$.
For each $t\in J$, we have two bases of $T_{\overline{\gamma}(t)}M$,
$$
e_1(t),\cdots,e_n(t)\quad\hbox{and}\quad \frac{\partial}{\partial x^1}|_{\overline{\gamma}(t)},\cdots,\frac{\partial}{\partial x^n}|_{\overline{\gamma}(t)}.
$$
Hence there exists a unique non-degenerate matrix $(A_{ij}(t))$  of order $n$ such that
$$
e_i(t)=\sum^n_{j=1}A_{ij}(t)\frac{\partial}{\partial x^j}|_{\overline{\gamma}(t)},\quad i=1,\cdots,n.
$$
It follows that
$$
\sum^n_{i=1}{\bf u}_\lambda^i(t)e_i(t)=\sum^n_{i=1}{\bf u}_\lambda^i(t)\sum^n_{j=1}A_{ij}(t)\frac{\partial}{\partial x^j}|_{\overline{\gamma}(t)}=
\sum^n_{j=1}\Big(\sum^n_{i=1}{\bf u}_\lambda^i(t)A_{ij}(t)\Big)\frac{\partial}{\partial x^j}|_{\overline{\gamma}(t)}.
$$
Let $\varphi(\overline{\gamma}(t))=(x^1(\overline{\gamma}(t)),\cdots, x^n(\overline{\gamma}(t))$. 
Then for each $t\in J$ we have
\begin{eqnarray}\label{e:two-chart2}
\Phi\Big(\overline{\gamma}(t),
\sum^n_{i=1}{\bf u}_\lambda^i(t)e_i(t)\Big)=\Big(x^1(\overline{\gamma}(t)),\cdots, x^n(\overline{\gamma}(t),
\sum^n_{i=1}{\bf u}_\lambda^i(t)A_{i1}(t),\cdots,\sum^n_{i=1}{\bf u}_\lambda^i(t)A_{in}(t)\Big).\nonumber
\end{eqnarray}
It is clear that $\varphi\times\varphi: U\times U\to B^n_{2\iota}(0)\times B^n_{2\iota}(0)$ given by $\varphi\times\varphi(x,y)=(\varphi(x),\varphi(y))$
is a chart on $M\times M$ centered at $(\overline{\gamma}(t_0), \overline{\gamma}(t_0))$.
Take a small neighborhood $J_0\subset J$ of $t_0$ in $[0,\tau]$ such that
$d_g(\overline{\gamma}(t), \overline{\gamma}(t_0))<\iota$ for all $t\in J_0$. Then (\ref{e:gamma2}) implies
\begin{eqnarray}\label{e:gamma2*}
d_g(\gamma_\lambda(t), \overline{\gamma}(t_0))<2\iota,\quad\forall (\lambda,t)\in
\Lambda\times J_0.
\end{eqnarray}
Hence $\{\gamma_\lambda(t)\,|\, (\lambda,t)\in\Lambda\times J_0\}$ is contained in the chart $(U,\varphi)$. Let
$$
(\varphi\times\varphi)\left(\overline{\gamma}(t), \gamma_\lambda(t)\right)=
\left(x^1(\overline{\gamma}(t)),\cdots, x^n(\overline{\gamma}(t), x^1(\gamma_\lambda(t)),\cdots, x^n(\gamma_\lambda(t))\right).
$$
By this, (\ref{e:two-chart}) and (\ref{e:two-chart2}) we obtain that for all $(\lambda,t)\in\Lambda\times J_0$,
\begin{eqnarray}\label{e:two-chart3}
&&\Phi\circ(\mathbb{F})^{-1}\circ(\varphi\times\varphi)^{-1}\left(x^1(\overline{\gamma}(t)),\cdots, x^n(\overline{\gamma}(t), x^1(\gamma_\lambda(t)),\cdots, x^n(\gamma_\lambda(t))\right)\nonumber\\
&&=\Big(x^1(\overline{\gamma}(t)),\cdots, x^n(\overline{\gamma}(t),
\sum^n_{i=1}{\bf u}_\lambda^i(t)A_{i1}(t),\cdots,\sum^n_{i=1}{\bf u}_\lambda^i(t)A_{in}(t)\Big).
\end{eqnarray}
Note that $\Psi:=\Phi\circ\left(\mathbb{F}|_{\mathcal{W}(0_{TM})}\right)^{-1}\circ(\varphi\times\varphi)^{-1}$ is a $C^5$ diffeomorphism onto its image set.
Taking the derivative of $t$ for the equation  in (\ref{e:two-chart3}) we arrive at
{\small
\begin{eqnarray}\label{e:two-chart4}
&&D\Psi\Big(x^1(\overline{\gamma}(t)),\cdots, x^n(\overline{\gamma}(t), x^1(\gamma_\lambda(t)),\cdots, x^n(\gamma_\lambda(t))\Big)\Big[
\frac{d}{dt}x^1(\overline{\gamma}(t)),\cdots, \frac{d}{dt}
x^n(\overline{\gamma}(t), \frac{d}{dt}x^1(\gamma_\lambda(t)),\nonumber\\
&&\hspace{110mm}\cdots, \frac{d}{dt}x^n(\gamma_\lambda(t)\Big]
\nonumber\\
&&=\Big(\frac{d}{dt}x^1(\overline{\gamma}(t)),\cdots, \frac{d}{dt}x^n(\overline{\gamma}(t),
\sum^n_{i=1}{\bf u}_\lambda^i(t)\frac{d}{dt}A_{i1}(t),\cdots,\sum^n_{i=1}{\bf u}_\lambda^i(t)\frac{d}{dt}A_{in}(t)\Big)\nonumber\\
&&+\Big(\frac{d}{dt}x^1(\overline{\gamma}(t)),\cdots, \frac{d}{dt}x^n(\overline{\gamma}(t),
\sum^n_{i=1}\frac{d}{dt}{\bf u}_\lambda^i(t)A_{i1}(t),\cdots,\sum^n_{i=1}\frac{d}{dt}{\bf u}_\lambda^i(t)A_{in}(t)\Big)\nonumber\\
&&\quad\forall (\lambda,t)\in\Lambda\times J_0.
\end{eqnarray}}
Since all $A_{ij}$ are $C^5$, $D\Psi$ is $C^4$, $(\lambda,t)\mapsto{\bf u}_\lambda(t)$ is continuous (by Step 1), and
$$
(\lambda,t)\mapsto x^1(\gamma_\lambda(t))\quad\hbox{and}\quad
(\lambda,t)\mapsto\frac{d}{dt}x^1(\gamma_\lambda(t))
$$
are continuous in $\Lambda\times J_0$ (by Assumption~\ref{ass:LagrGenerB}),
it follows from (\ref{e:two-chart4}) that
{\small
$$
\Lambda\times J_0\ni (\lambda,t)\mapsto
\Big(\sum^n_{i=1}\frac{d}{dt}{\bf u}_\lambda^i(t)A_{i1}(t),\cdots,\sum^n_{i=1}\frac{d}{dt}{\bf u}_\lambda^i(t)A_{in}(t)\Big)=
\Big(\frac{d}{dt}{\bf u}_\lambda^1(t),\cdots,\frac{d}{dt}{\bf u}_\lambda^n(t)\Big)(A_{ij}(t))
$$}
is continuous. Moreover, all $(A_{ij}(t))$ are invertible. These lead to
$$
\Lambda\times J_0\ni (\lambda,t)\mapsto
\Big(\frac{d}{dt}{\bf u}_\lambda^1(t),\cdots,\frac{d}{dt}{\bf u}_\lambda^n(t)\Big)
$$
is continuous. Since $t_0\in [0,\tau]$ is arbitrarily chosen, the required claim is proved.
\end{proof}

\section{Morse indexes of a class of abstract functionals}\label{app:Abst}

Let $H$ be a separable real  Hilbert space with inner product $(\cdot,\cdot)$, and let $\mathcal{F}^{sa}(H)$
denote the set of all bounded linear self-adjoint Fredholm operators on $H$.
For a $C^0$ path $t\mapsto B_{t}\in \mathcal{F}^{sa}(H)$, $t\in (-\eta, \eta)$,
we have a family of functionals on
$H$, 
$$
f_t(x)=\frac{1}{2}(B_tx,x),\quad x\in H,
$$
for which  the origin $0\in H$ is a critical point with finite Morse index.
Suppose that $\dim{\rm ker}B_0=N>0$ and $t\mapsto B_{t}$ is differentiable at $0$.
The derivative of $B_t$ at $t=0$, denoted by $\dot{B}_{0}$, belongs to $\mathscr{L}_s(H)$.
Consider the symmetric bilinear form $Q$ on 
$\operatorname{ker} B_{0}$ given by
\[
Q(x, y) = \frac{\mathrm{d}}{\mathrm{d} t}(x, B_{t}(y))\big|_{t=0} = (x, \dot{B}_{0}(y)), \quad x, y \in \operatorname{ker} B_{0}.
\]
The following result is contained in the proof of \cite[Lemma~3.23]{Fur04}. 
For the sake of completeness, we present a detailed proof.
%The proof can be completed by modifying the proof of \cite[Lemma~3.23]{Fur04}.

\begin{proposition}\label{prop:Abst}
Under the above assumptions, suppose that $Q$ is  non-degenerate and has  the Morse index $m^-(Q)=q$.
Then there exists $0<\delta_0<\eta$ such that
 \begin{align}\label{e:Abst0}
 m^-(f_t,0)=
 \begin{cases}
 m^-(Q)+ m^-(f_0,0) &\;\text{for}\quad 0<t<\delta_0,\\
 m^0(f_0,0)-m^-(Q)+ m^-(f_0,0) &\;\text{for}\quad-\delta_0<t<0.
 \end{cases}
 \end{align}
In particular, we have the following special cases:
\[
m^{-}(f_{t},0)=\begin{cases}
m^{-}(f_{0},0), & 0<t<\delta_0,\\
m^{-}(f_{0},0)+m^{0}(f_{0},0), & -\delta_0<t<0,
\end{cases}
\quad \text{if } Q \text{ is positive definite,}
\]
and
\[
m^{-}(f_{t},0)=\begin{cases}
m^{-}(f_{0},0)+m^{0}(f_{0},0), & 0<t<\delta_0,\\
m^{-}(f_{0},0), & -\delta_0<t<0,
\end{cases}
\quad \text{if } Q \text{ is negative definite.}
\]
\end{proposition}
\begin{proof}[\bf Proof]
\noindent{\bf Step 1}. 
Consider the complexification $H_{\mathbb{C}} =\{z = x + iy\,|\, x, y \in H\}$ of the Hilbert space $H$, 
equipped with the complex linear structure and complex inner product defined by
  \begin{align*}
  & (a + ib)(x + iy) = (ax - by) + i(bx + ay), \quad  a, b \in \mathbb{R},\\
  &\langle z_1, z_2 \rangle_{\mathbb{C}} = (x_1, x_2)_\mathbb{R} + (y_1, y_2)_\mathbb{R} + i[(y_1, x_2)_\mathbb{R} - (x_1, y_2)_\mathbb{R}]
  \end{align*}
  for \( z_1 = x_1 + iy_1, \, z_2 = x_2 + iy_2 \). 
%An equivalent definition (more directly compatible with linear extension of the real inner product) is:
%\[\langle x + iy, x' + iy' \rangle_C := (x, x')_\mathbb{R} + (y, y')_\mathbb{R} + i[(y, x')_\mathbb{R} - (x, y')_\mathbb{R}],\]
%or more simply understood as: extending the original real inner product **complex linearly** in the first argument and conjugate linearly in the second:
%\[\langle z_1, z_2 \rangle_C := (x_1, x_2)_\mathbb{R} + (y_1, y_2)_\mathbb{R} + i(x_1, y_2)_\mathbb{R} - i(y_1, x_2)_\mathbb{R}.\]
%In practice, it is common to: extend the real inner product \((\cdot, \cdot)_\mathbb{R}\) as a bilinear form to complex coefficients, and then the inner product is:
% (2) Complexified operator \(\tilde{A}_t\)
The complexified operator \(\tilde{B}_t\) of $B_t$ is the complex linear extension of \(B_t\) defined 
for \(z = x + iy \in H_{\mathbb{C}}\) by
\[
\tilde{B}_t(z) = B_t x + iB_t y.
\]
It is a self-adjoint operator with respect to the complex inner 
product \(\langle z_1, z_2 \rangle_{\mathbb{C}}\).
% **Preservation of self-adjointness**:
%  Since \(B_t\) is self-adjoint with respect to the real inner product, \(\tilde{B}_t\) is a self-adjoint operator with respect to the complex inner product \(\langle \cdot, \cdot \rangle_C\):
%  \[
%  \langle \tilde{A}_t z_1, z_2 \rangle_C = \langle z_1, \tilde{A}_t z_2 \rangle_C.
%  \]
%- **Spectral relation**:
  The spectrum of \(\tilde{B}_t\) is the same as that of \(B_t\) (which is a subset of the real numbers), and for a real spectral point \(\lambda\), the complex eigenspace is the complexification of the real eigenspace.
Since $B_0$ is a bounded linear self-adjoint Fredholm operator,
$0\in\sigma(B_0)$ is an isolated point in $\sigma(B_0)$.
Therefore, we can choose $c>0$ such that $\pm c\notin\sigma(B_0)$.
From the  assumption of the continuous family \(\{B_t\}\) there exists
$\delta \in (0,\eta]$  such that
 $\pm c\notin\sigma(B_t)$ for all $|t|\le\delta$. Let $\Gamma_c=\{\zeta\in\mathbb{C}\,|\,
 |\zeta|=c\}$. Define 
the complex projection operators \(\tilde{P}_t\)  on \( H_{\mathbb{C}} \)  by
\[
\tilde{P}_t = \frac{1}{2\pi\sqrt{-1}} \int_{\Gamma_c} (\zeta I_{H_{\mathbb{C}}}- \tilde{B}_t)^{-1} \, d\zeta,\quad |t| \leq \delta,
\]
which are orthogonal projections (self-adjoint and idempotent with respect to the complex inner product).
All these $\tilde{P}_t$  have the constant rank equal to $N=\dim_{\mathbb{C}}\text{ker}\tilde{B}_0$, and 
$$
\tilde{E}_t:=\tilde{P}_t(H_{\mathbb{C}})=\sum_{|\zeta|<c} \text{ker}(\tilde{B}_t - \zeta
I_{H_{\mathbb{C}}})=E_t\oplus iE_t,
$$
where $E_t:=\sum_{-c<\zeta<c} \text{ker}({B}_t - \zeta I_H)$.  
%%%%%%%%%%%%%%%%%%%%%%%%%%%%%%%%%%%%%%%%%%%%%%%%%%%%%%%%%%%%%%%%%%%%%%%%%%%%%%%%%%%%%%%%%%%%%%%%%%%%%%%
% 4. Relation between \( P_t \) and the real space
%The image states that "range of each \( P_t = \sum_{|u|<c} \text{Ker}(A_t - u) \)", where the sum denotes the direct sum of generalized eigenspaces corresponding to eigenvalues \( u \) satisfying \( |u| < \varepsilon \). Since the spectrum is real, in fact only \( u = 0 \) satisfies \( |u| < \varepsilon \) (if we choose \( \varepsilon \) sufficiently small and 0 is the only spectral point inside that disk). Thus:
%\[
%\text{Ran}(P_t) = \text{Ker}(\tilde{A}_t) \quad (\text{in the complexified space})
%\]
%But note, if the generalized eigenspace for the eigenvalue 0 has finite dimension, the real form of \( \text{Ker}(\tilde{A}_t) \) is the complexification of \( \text{Ker}(A_t) \).
%From the image, \( \text{rank}(P_t) = \dim_C \text{Ran}(P_t) = \dim_\mathbb{R} \text{Ker}A_0 \), indicating that the complex rank of this projection equals the real dimension of the original real kernel. This implies that the real dimension of \( \text{Ker}A_t \) is the same for \( |t| \leq \delta \) as at \( t = 0 \), and that \( \text{Ran}(P_t) = \text{Ker}(A_t) \otimes_\mathbb{R} \mathbb{C} \) (complexification).
%Therefore, although \( P_t \) is defined via complexification, 
%%%%%%%%%%%%%%%%%%%%%%%%%%%%%%%%%%%%%%%%%%%%%%%%%
Since $\tilde{E}_t$  is the complexification of a real subspace $E_t$, and hence when 
restricted to \( H \subset H_\mathbb{C} \),  each \( \tilde{P}_t \) yields the real orthogonal projection 
${P}_t$ from the real space $H$ onto \( E_t \).
Moreover, it is not hard to prove that 
the path $t\mapsto \tilde{P}_{t}\in \mathscr{L}_s(H_{\mathbb{C}})$, $t\in (-\delta, \delta)$, is 
continuous and differentiable at $0$, with derivative given by
\[
\dot{\tilde{P}}_0 = \frac{1}{2\pi i} \int_{\Gamma_c} (\zeta I_{H_{\mathbb{C}}}- \tilde{B}_0)^{-1} \dot{\tilde{B}}_0 (\zeta I_{H_{\mathbb{C}}}- 
\tilde{B}_0)^{-1} d\zeta.
\]

\noindent{\bf Step 2}. 
 Consider the orthogonal complementary  $E_t^\bot$ of $E_t$ in $H$.
Since $B_t$ is self-adjoint, the eigenspaces corresponding to distinct eigenvalues of $B_t$
are orthogonal. Hence 
 $$
 \sum_{\zeta<-c} \text{ker}({B}_t - \zeta I_H)\oplus 
 \sum_{\zeta>c} \text{ker}({B}_t - \zeta I_H)\subset E_t^\bot.
 $$
 Let $F_t:=\sum_{\zeta<-c} \text{ker}({B}_t - \zeta I_H)$. Then 
 we have
 \begin{align}\label{e:Abst1}
 m^-(f_0,0)&=\dim F_0,\\
 m^-(f_t,0)&=m^-(f_t|_{E_t},0)+ m^-(f_t|_{F_t},0)\nonumber\\
 &=m^-(f_t|_{E_t},0)+ \dim F_t.\label{e:Abst2}
 \end{align}
Since $\sigma(B_0)\cap (-\infty, -c)$ consists of finitely many eigenvalues with finite (algebraic) multiplicity, say $\lambda_{0,1},\cdots,\lambda_{0, k}$, where each $\lambda_{0,i}$ has multiplicity $m_i$ for $i=1,\cdots,k$, by the arguments at the end of \cite[\S 135]{RieSz} we can shrink $\delta>0$ so that for each $t\in (-\delta,\delta)$, $\sigma(B_t)\cap (-\infty, -c)$ consists exactly of $k$ eigenvalues with finite (algebraic) multiplicity, $\lambda_{t,1},\cdots,\lambda_{t, k}$, and each $\lambda_{t,i}$ 
is near $\lambda_{0,i}$ and also has multiplicity $m_i$, $i=1,\cdots,k$. Hence
\begin{align}\label{e:Abst3}
\dim F_t=\dim F_0,\quad t\in (-\delta, \delta).
\end{align}

\noindent{\bf Step 3}. Define $p=N-q$. Let $\mu_1 \le \mu_2 \le \ldots \le \mu_N$ be all eigenvalues of $Q$, and let $e_1, \ldots, e_N$ be the unit eigenvectors corresponding to them, which form a unit
 orthogonal basis of $E_0$. Then $P_te_1, \ldots, P_te_N$ form a unit
 orthogonal basis of $E_t$ since all  $P_t$ are orthogonal projections. Hence 
 \begin{align}\label{e:Abst4}
 m^-(f_t|_{E_t},0)=\text{dimension of the negative definite space of matrix 
 $\big(\left(B_{t}\circ P_{t}(e_i),P_{t}(e_j)\right)\big)$.}
 \end{align}
   We can assume that $\mu_1 \le \mu_2 \le \ldots \le \mu_q < 0$ and $0 < \mu_{q+1} \le \ldots \le \mu_N$. 
 Define $\epsilon=\frac{1}{2}\min_j|\mu_j|$.
 
 Following the proof of \cite[Lemma~3.23]{Fur04}, for $x,y\in{\rm ker}B_{0}$ there holds:
\begin{align*}
&\left(\frac{1}{t}\cdot B_{t}\circ P_{t}(x), P_{t}(y)\right)-\left(\dot{B}_{0}(x), y\right)\\
=&\left(\left(\frac{1}{t}\cdot(B_{t}- B_{0})-\dot{B}_{0}\right)\circ P_{t}(x),P_{t}(y)\right)+\left(\frac{1}{t}\cdot(P_{t}(x)-x), B_{0}(P_{t}(y))\right)\\
&\quad+\left(\dot{B}_{0}(P_{t}(x)-x), P_{t}(y)\right)+\left(\dot{B}_{0}(x), P_{t}(y)-y\right).
\end{align*}
Let $\dot{P}_{0}\in\mathscr{L}_s(H)$ denote the derivative of the path $P_t$ at $t=0$. 
Since, when $ t \to 0 $, we have
\begin{align*}
 &\|{B_0 (P_t(e_i))}\| \to \|{B_0 (P_0(e_i))}\|=\|{B_0 e_i}\|= 0, \\
 &\|(P_{t}(e_i)-e_i)\|\to \|P_0e_i-e_i\|=0,\\
 &\|{\frac{1}{t} \cdot (P_t(e_i) - e_i)}\|\to\|\dot{P}_0e_i\|,
 \end{align*}
   there exists $ 0 < \delta_0 \leq \delta $ such that  for all $|t|<\delta_0$ and $i,j=1,\cdots, N$,
$$
 \left|\left(\frac{1}{t}\cdot(P_{t}(e_i)-e_i), B_{0}(P_{t}(e_j))\right)
+\left(\dot{B}_{0}(P_{t}(e_i)-e_i),P_{t}(e_j)\right)+\left(\dot{B}_{0}(e_i),P_{t}(e_j)-e_j\right)\right|
<\frac{\epsilon}{N},
$$
and hence
  \begin{align*}
\left|\left(\frac{1}{t}\cdot B_{t}\circ P_{t}(e_i),P_{t}(e_j)\right)-
\left(\dot{B}_{0}(e_i), e_j\right)\right|
<\frac{\epsilon}{N}.
\end{align*}
 Let $\lambda_1(t)\le \lambda_2(t)\le \ldots \le \lambda_N(t)$  be
 the eigenvalues of matrix 
 $\big(\left(B_{t}\circ P_{t}(e_i),P_{t}(e_j)\right)\big)$.
   By Weyl's Perturbation Theorem for  eigenvalues of Hermitian matrices
   (cf.~Problem~4.3.P1 on \cite[page~256]{HoJo13} or the first exercise on \cite[page~408]{HoJo13}),
       we obtain
 $$
 |\frac{1}{t}\lambda_i(t)-\mu_i|< \epsilon\quad\text{and hence}\quad
 |\frac{1}{t}\lambda_i(t)|\ge\epsilon
  $$
 for all $0<|t|<\delta_0$ and $i=1,\cdots,N$. 
  Then for $0<t<\delta_0$ there holds:
    $$
  \lambda_i(t)
  \begin{cases}
  <-\epsilon t, & i=1,\cdots,q,\\
  >\epsilon t, & i=q+1,\cdots, m;
  \end{cases}
  $$
   and for $-\delta_0<t<0$ there holds:
     $$
  \lambda_i(t)
  \begin{cases}
  >-\epsilon t, & i=1,\cdots,q,\\
  <\epsilon t, & i=q+1,\cdots, m.
  \end{cases}
  $$
   It follows from these and (\ref{e:Abst4}) that
  \begin{align}\label{e:Abst5}
 m^-(f_t|_{E_t},0)=\begin{cases}
 q &\;\text{for}\quad 0<t<\delta_0,\\
 p &\;\text{for}\quad-\delta_0<t<0.
 \end{cases}
 \end{align}
 This, together with (\ref{e:Abst1})-(\ref{e:Abst3}), leads to
  \begin{align*}
 m^-(f_t,0)=m^-(f_t|_{E_t},0)+ \dim F_0=
 \begin{cases}
 q+ m^-(f_0,0) &\;\text{for}\quad 0<t<\delta_0,\\
 p+ m^-(f_0,0) &\;\text{for}\quad-\delta_0<t<0.
 \end{cases}
 \end{align*}
(\ref{e:Abst0}) is proved.
\end{proof}

Let $B\in\mathcal{F}^{sa}(H)$, and let $\langle\cdot,\cdot\rangle$ be 
an alternative inner product on $H$ that induces a norm equivalent to the one induced by $(\cdot,\cdot)$. 
Since $\langle\cdot,\cdot\rangle$ is bounded and coercive with respect to the original norm, 
the Lax--Milgram theorem implies that there exists a bounded linear invertible operator
$\Xi:H\to H$ such that  for all $u, v\in H$,
\begin{align}\label{e:Abst6}
\langle u, v\rangle=(\Xi u, v)\quad\text{and hence}\quad
\langle \Xi^{-1}u, v\rangle=( u, v).
\end{align}

\begin{claim}\label{cl:Abst}
The operator $\Xi^{-1}B$ is also  Fredholm, and self-adjoint with respect to
the inner product $\langle\cdot,\cdot\rangle$.
\end{claim}
\begin{proof}[\bf Proof]
The first claim is clear since $\Xi$ (and 
 consequently $\Xi^{-1}$) is bounded and invertible.
 To prove the second one, let $(\Xi^{-1}B)^\ast$ denote the adjoint operator
 of $\Xi^{-1}B$ with respect to the inner product $\langle\cdot,\cdot\rangle$. Then 
 for any $u,v\in H$,
 \begin{align*}
 \langle (\Xi^{-1}B)^\ast u, v\rangle&=\langle  u, \Xi^{-1}Bv\rangle=
 \langle \Xi^{-1}Bv, u\rangle=(Bv, u)=(Bu,v)=\langle \Xi^{-1}Bu, v\rangle,
 \end{align*}
 where the third and the final equalities come from (\ref{e:Abst6}).
Hence $(\Xi^{-1}B)^\ast=\Xi^{-1}B$, i.e., 
$\Xi^{-1}B$ is self-adjoint with respect to  $\langle\cdot,\cdot\rangle$.
 \end{proof}

As a consequence of Claim~\ref{cl:Abst} and Proposition~\ref{prop:Abst}, we have:

\begin{corollary}\label{cor:Abst}
Under the above assumptions, suppose that $Q$ is  non-degenerate and has  the Morse index $m^-(Q)=q$.
Then there exists $0<\delta_0<\eta$ such that
 \begin{align*}
 m^-(f_t,0)=
 \begin{cases}
 m^-(Q)+ m^-(f_0,0) &\;\text{for}\quad 0<t<\delta_0,\\
 m^0(f_0,0)-m^-(Q)+ m^-(f_0,0) &\;\text{for}\quad-\delta_0<t<0.
 \end{cases}
 \end{align*}
\end{corollary}

\vspace{2mm}

%{\bf Disclosure statement}\\
%No potential conflict of interest was reported by the author(s).

%\noindent{\bf Acknowledgement}\;I am grateful to the referees for their valuable suggestions on improving the manuscript.\\

\noindent{\bf Funding}\; This work was supported by National Natural Science Foundation of China [12371108].\\
%This work was supported by National Natural Science Foundation of China [11271044 and 12371108].

\noindent{\bf Data Availability}\; We do not analyse or generate any datasets, because our work proceeds within a theoretical
and mathematical approach.\\

\noindent{\bf Declarations}\\

\noindent{\bf Conflict of interest}\; The author has no Conflict of interest to declare that are relevant to the content of this
article.

\renewcommand{\refname}{REFERENCES}

\medskip

\begin{tabular}{l}
 School of Mathematical Sciences, Beijing Normal University\\
 Laboratory of Mathematics and Complex Systems, Ministry of Education\\
 Beijing 100875, The People's Republic of China\\
 E-mail address: gclu@bnu.edu.cn\\
\end{tabular}

\end{document}